\documentclass[11pt]{amsart}

\usepackage[all,cmtip]{xy}

\usepackage[new]{old-arrows}
\usepackage{geometry}
\usepackage{graphicx}
\usepackage{lipsum}
\newsavebox{\myimage}

\usepackage{adjustbox}

\newcommand{\sk}{\smallskip}
\newcommand{\mk}{\medskip}
\newcommand{\bk}{\bigskip}

\usepackage{tikz}
\makeatletter
\newcommand{\xleftrightarrow}[2][]{\ext@arrow 3359\leftrightarrowfill@{#1}{#2}}
\makeatother

\newcommand{\xdasharrow}[2][->]{
\tikz[baseline=-\the\dimexpr\fontdimen22\textfont2\relax]{
\node[anchor=south,font=\scriptsize, inner ysep=1.5pt,outer xsep=2.2pt](x){#2};
\draw[shorten <=3.4pt,shorten >=3.4pt,dashed,#1](x.south west)--(x.south east);
}
}

\usepackage{amsmath}
\usepackage{amssymb}
\usepackage{amsthm}
\usepackage{latexsym}
\usepackage{pstricks}
\usepackage{amsbsy}
\usepackage{amscd}
\usepackage{mathrsfs}
\usepackage{tipa}
\usepackage{txfonts}
\usepackage{amsfonts}
\usepackage{graphicx}
\usepackage{tabularx}
\usepackage{listings}
\usepackage{tikz}
\usepackage{hyperref}

\usepackage{scalerel,stackengine}
\stackMath
\newcommand\reallywidehat[1]{%
\savestack{\tmpbox}{\stretchto{%
  \scaleto{%
    \scalerel*[\widthof{\ensuremath{#1}}]{\kern-.6pt\bigwedge\kern-.6pt}%
    {\rule[-\textheight/2]{1ex}{\textheight}}
  }{\textheight}%
}{0.5ex}}%
\stackon[1pt]{#1}{\tmpbox}%
}
\newtheorem{thm}{Theorem}[section]
\newtheorem{cor}[thm]{Corollary}
\newtheorem{lem}[thm]{Lemma}

\newtheorem{exm}{Example}

\newtheorem{prop}[thm]{Proposition}
\newtheorem{fact}[thm]{Fact}

\newtheorem{rem}[thm]{Remark}
\newtheorem{defn-prop}[thm]{Definition-Proposition}
\newtheorem{conjecture}[thm]{Conjecture}
\newtheorem{question}[thm]{Question}
\newtheorem{questions}[thm]{Questions}

\usepackage{graphicx,epsfig,amscd,graphics,xypic}

 \newcommand{\eq}[1][r]
   {\ar@<-3pt>@{-}[#1]
    \ar@<-1pt>@{}[#1]|<{}="gauche"
    \ar@<+0pt>@{}[#1]|-{}="milieu"
    \ar@<+1pt>@{}[#1]|>{}="droite"
    \ar@/^2pt/@{-}"gauche";"milieu"
    \ar@/_2pt/@{-}"milieu";"droite"}

 \newcommand{\incl}[1][r]
  {\ar@<-0.2pc>@{^(-}[#1] \ar@<+0.2pc>@{-}[#1]}

\author[L. Pirio]{Luc Pirio\textsuperscript{$\dagger$}}
\thanks{${}^{}$\hspace{-0.4cm}\textsuperscript{$\dagger$}\href{mailto:luc.pirio@uvsq.fr}{L.\,Pirio}, Laboratoire de Math\'ematiques de Versailles, Univ.\,Paris-Saclay -- UVSQ, CNRS, 78000 Versailles, France.}

\title[]
 {On the $(n+3)$-webs by rational curves induced by  \\ the  forgetful maps on the moduli spaces $\mathcal M_{0,n+3}$}

\begin{document}


%
%
\maketitle

\begin{abstract}
For $n\geq 2$, we discuss the curvilinear web $\boldsymbol{\mathcal W}_{0,n+3}$ on the moduli space  $\mathcal M_{0,n+3}$ 
 defined by the $n+3$ forgetful maps $
\mathcal M_{0,n+3}\rightarrow \mathcal M_{0,n+2}$. We recall classical results (first obtained by Room) which show that this web is linearizable when $n$ is odd, or is equivalent to a web by conics when $n$ is even.  We then turn to the abelian relations (ARs) of these webs. After recalling the classical and well-known case when $n=2$ (related to the famous 5-terms functional identity of the dilogarithm), we focus on the case of the 6-web $\boldsymbol{\mathcal W}_{{0,6}}$.  
We show that this web is isomorphic to the web formed by the projective lines contained in  Segre's cubic primal $\boldsymbol{S}\subset \mathbf P^4$ and that a kind of 'Abel's theorem' allows to describe the  ARs of $\boldsymbol{\mathcal W}_{{0,6}}$ by means of the  abelian 2-forms on the Fano surface  $F_1(\boldsymbol{S})\subset G_1(\mathbf P^4)$ of lines contained in $\boldsymbol{S}$. We deduce from this that $\boldsymbol{\mathcal W}_{{0,6}}$ has maximal rank with all its AR rational, and that these span a space which is an irreducible $\mathfrak S_6$-module. 
Then we take up an approach due to Damiano that we correct in the case when $n$ is odd: it leads to an abstract description of the space of ARs of $\boldsymbol{\mathcal W}_{0,n+3}$ as a $\mathfrak S_{n+3}$-representation. In particular, we obtain that this web has maximal rank  for any $n\geq 2$.  
Finally,  we consider  `Euler's abelian relation $\boldsymbol{\mathcal E}_n$', 
a particular AR for $\boldsymbol{\mathcal W}_{0,n+3}$ constructed by Damiano from a characteristic class on the grassmannian of 2-planes in $\mathbf R^{n+3}$ by means of Gelfand-MacPherson theory of polylogarithmic forms. We give an explicit conjectural formula for the components of $\boldsymbol{\mathcal E}_n$, which 
involves only rational (resp.\,rational and logarithmic) terms for $n$ odd (resp.\,for $n$ even).  By means of direct computations, we prove that our explicit formulas are indeed correct for $n$ less than or equal to $ 12$. 
\end{abstract}




\section{Introduction}
We work with analytic objects in the whole paper. The usual setting when studying webs 
with regard to their abelian relations and their rank is the complex analytic setting, but some basic constructions in \cite{D} rely on real analytic objects ({\it e.g.}\,real grassmannians, differential forms on these) hence both cases will be considered here, but mainly the latter. 
In most parts, the base field will be $\mathbf R$ but in some (which will be pointed out), it will be more natural to work over $\mathbf C$. In the sequel, $\mathbf K$ will stand for one of these two fields.\mk 

In the  Introduction, we first introduce the topic (webs, abelian relations, etc) before giving a quick overview of the content of Damiano's paper \cite{D}. We then state our results.  

\subsection{Curvilinear webs}
\label{SS:webs}
We introduce  basic notions about curvilinear webs. Our presentation below is very similar to 
those of   \cite{DThesis} and \cite{D}. For other references on webs (but where the focus is more on 1-codimensional webs), see the famous `Blaschke-Bol book' \cite{BB} and the more recent one \cite{Coloquio}. 

\subsubsection{\bf Webs.} 
\label{SSS:webs}
Let $U$ be a non empty domain of $\mathbf K^n$ and $d$ a positive integer. \sk 

A {\bf (curvilinear)  $d$-web} on $U$  is a $d$-tuple $\boldsymbol{\mathcal W}=(\mathcal F_1,\ldots,\mathcal F_d)$ of analytic foliations on $U$ whose leaves are in general position,  at any point of $U$ say (one can and one often only requires that this holds true generically on $U$). 
Since $U$ is simply connected, each foliation $\mathcal F_i$ of $\boldsymbol{\mathcal W}$ is defined by a global first integral, that is there exists a regular submersion 
$u_i:U\rightarrow \mathbf A^{n-1}$ the  fibers  of which are connected and coincide with the leaves of $\mathcal F_i$.   In this case, we also use the notation $\boldsymbol{\mathcal W}=\boldsymbol{\mathcal W}(u_1,\ldots,u_d)$.   For any $i=1,\ldots,d$, let 
$X_i$ be a non-vanishing vector field on $\Omega$ which generates the tangent distribution $T_{\mathcal F_i}={\rm Ker}(du_i) \subset T_\Omega$.  The general position hypothesis which  is required for the foliations of $\boldsymbol{\mathcal W}$  is that 
for any strictly increasing $n$-tuple $(i_1,\ldots,i_n)$ 
 of $\{1,\ldots,d\}$, the $n$ associated 1-dimensional distributions $T_{\mathcal F_{i_1}},\ldots,T_{\mathcal F_{i_n}}$ are in direct sum in $T_\Omega$, 
 which results in the more analytical fact that the field of  $n$-tangent vectors 
 $X_{i_1}\wedge \ldots \wedge X_{i_n}$ does not vanish on $U$. \sk 
 
Another (curvilinear)  $d$-web  $\boldsymbol{\mathcal W}'$ defined on another domain $U'$ is said to be {\bf equivalent}' to $\boldsymbol{\mathcal W}$ if there exists a germ of analytic isomorphism $\varphi: (U,u)\rightarrow (U',u')$  such that $\varphi^*(\boldsymbol{\mathcal W}')$ coincides with the germ of $\boldsymbol{\mathcal W}$ at $u$ 
 (possibly only up to relabeling the foliations).  {\bf Web geometry} consists in the study of webs up to this notion of equivalence. \sk 
 
 As a first  example of webs, but which is important for what is to come, let 
 us consider $n+1$ points $p_0,p_1,\ldots,p_n$ in general position in $\mathbf P^n$. Denoting by $\mathcal L_{p_i}$ the linear family of projective lines passing through $p_i$ for $i=0,\ldots,n$, 
  we get a linear   web $\boldsymbol{\mathcal L \mathcal  W}_{p_0,\ldots,p_n}=(\mathcal L_{p_0},\ldots,\mathcal L_{p_n})$ on $\mathbf P^n$\footnote{But the general position assumption holding true only on a certain Zariski open subset of $\mathbf P^n$.}. By definition, a  {\bf quadrilateral web} is a $(n+1)$-web which is isomorphic to a web of this kind.  
Assuming that the $p_i$'s are the vertices $e_0,\ldots,e_n$ of the standard  $n$-simplex in $\mathbf P^n$, one can give a simple explicit analytic model for 
the {\bf standard quadrilateral web}
$\boldsymbol{\mathcal Q \mathcal  W}=\boldsymbol{\mathcal L \mathcal  W}_{e_0,\ldots,e_n}$: if $x_1,\ldots,x_n$ stand for the standard affine  coordinates, then 
for $i=1,\ldots,n$, 
the $i$-th linear standard projection  
$\pi_i: \mathbf A^n\rightarrow \mathbf A^{n-1}, \, (x_s)_{s=1}^n\mapsto (x_1,\ldots,\widehat{x_{i}},\ldots,x_n)$ 
is a first integral for $\mathcal L_{e_i}$ whereas 
$\pi_{0}: (x_s)_{s=1}^n\mapsto (x_t/x_n)_{t=1}^{n-1}$ works for $\mathcal L_{e_0}$. 
 \mk 
  
  A $d$-web with $d\geq n+1$ is said to be quadrilateral if all its $(n+1)$-subwebs are quadrilateral. Finally, for a planar 3-web, being quadrilateral can be characterized by the closure of any small `hexagonal figures' which can be traced on the definition domain of the considered web by traveling along the leaves of its foliations (see for instance  \cite[\S1.2]{Coloquio}). Hence the term   `quadrilateral' never appears when considering planar webs, one  uses `hexagonal' instead.
  
\subsubsection{\bf Abelian relations and ranks.} 
\label{SSS:AR-rank}
 An important notion for the study of webs is that of `abelian relation' (ab.\,AR). This notion  
  is interesting  first since it is well-behaved (invariant) modulo equivalences and also because it is linked to classical objects and results of projective algebraic geometry.\footnote{The web-theoretic notion of abelian relation is related to that of  abelian differentials and 
to fundamental results of algebraic geometry such   
as   Abel's addition theorem and its converse, etc. See \cite{Coloquio}  and the references therein  for an overview on this perspective on web geometry.} In order to recall this notion, given a web $\boldsymbol{\mathcal W}=\boldsymbol{\mathcal W}(u_1,\ldots,u_d)$, one denotes by $u_i^1,\ldots, u_i^{n-1}$ the components of the first integral $u_i$ for each $i$. 
 Given $k\in \{0,\ldots,n-1\}$, we set $\boldsymbol{I}_n^k$ (or just $\boldsymbol{I}$ when both $n$ and $k$ are  unambiguously fixed) for the set of $k$-tuples  $I=(i_1,\dots,i_k)$ with $1\leq i_1<i_2<\cdots< i_k\leq n$. 
  Then for any $i=1,\ldots,d$ and any tuple $I\in \boldsymbol{I}$,  we set $\wedge^{I} \!du_i=du_i^{i_1} \wedge\cdots \wedge du_i^{i_k} \in \Omega^k(U)$ and we denote by  $u_i^*(\Omega^k)$  the $\mathbf K$-vector space of differential $k$-forms on $U$ spanned by $\{ \wedge^{I} \!du_i\, \lvert \, I\in \boldsymbol{I} \}$. We then define the
{\bf space of $k$-th abelian relations} (ARs) for $\boldsymbol{\mathcal W}$ as the $\mathbf K$-vector space, denoted by  $ \boldsymbol{AR}^{(k)}(\boldsymbol{\mathcal W})$,   of $d$-tuples of elements $\omega_i^k\in 
u_i^*(\Omega^k)$ summing up to 0: 
$$
 \boldsymbol{AR}^{(k)}\big(\boldsymbol{\mathcal W}\big)=\left\{
 (\omega_i^k)_{i=1}^d\in \prod_{i=1}^d u_i^*\big(\Omega^k\big) 
 \, \big\lvert \, \sum_{i=1}^d \omega_i^k=0 \mbox{ in } \, \Omega^k(U)\, 
 \right\}\, . 
$$
It is easily seen that the above definition does not really depend on the first integrals $u_i$ but only on the associated foliations. A $k$-AR can be written quite explicitly in terms of the first integrals $u_i$: such an object corresponds to a family $(F_i^I)$  indexed by pairs $(i,I)\in \underline{d}\times \boldsymbol{I}$, 
of functions $F_i^I: {\rm Im}(u_i)\rightarrow \mathbf K$, 
  such that the following relation between differential forms holds true identically: 
$$
\sum_{i=1}^d \sum_{I\in \boldsymbol{I} } F^I_i(u_i) \,\big( \wedge^{I} \hspace{-0.1cm} du_i \big) = \sum_{i=1}^d
\hspace{-0.1cm}
 \sum_{\substack{ I=(i_1,\ldots,i_k)\\ 1\leq i_1<\cdots< i_k\leq n } }
\hspace{-0.2cm}
 F^I_i(u_i)  \,\Big( du_i^{i_1}\wedge  \cdots \wedge du_i^{i_k}\Big)= 0\, . 
$$

By definition, the {\bf $\boldsymbol{k}$-th rank} of $\boldsymbol{\mathcal W}$ is 
$${\rm rk}^{k}(\boldsymbol{\mathcal W})=\dim_{\mathbf K}   \boldsymbol{AR}^{k}\big(\boldsymbol{\mathcal W}\big) \in \mathbf N\cup \{\infty \}\, .$$ 
It is an invariant attached to $\boldsymbol{\mathcal W}$ (two equivalent webs have the same $k$-rank).  
\mk

In this text, we will essentially deal only with ARs of top degree, {\it i.e.} in the case when $k=n-1$. Except in other situations in which the complete notation will be used,  we will drop the superscript $n-1$ everywhere and just speak of ARs of webs. \sk 

\label{Page-3}
An important property of ARs (at least of top degree) is that they extend globally but  as multivalued objects.\footnote{The proof of this is standard and is an easy generalization of that in the case of planar webs ({\it cf.}\,\cite[\S1.2.2]{PirioThese}). Details are left to the reader.} This is of importance when considering webs globally defined on varieties with non trivial topology, as are those under scrutiny in this paper (see \S\ref{SSS:Webs-W_0-n+3} below). For such a web, the ARs organize themselves into a local system whose monodromy may be non trivial.
\begin{center}
$\star$\sk 
\end{center}

Let us illustrate the notion of 
abelian relation 
 with an explicit example, which will happen to be important in the whole article. For $i=0,\ldots,n$, we denote by $\pi_i^k$  (with $k=1,\ldots,n-1$)   the components of the first integrals  $\pi_i$ given above for the standard quadrilateral web  $\boldsymbol{\mathcal L \mathcal  W}_{p_0,\ldots,p_n}$ and we set 
$\Pi_i=d\pi_i^1\wedge \cdots \wedge d\pi_i^{n-1} \in \pi_i^*(\Omega^{n-1})$.  It is then not difficult to verify that the following relation is identically satisfied 
\begin{equation}
\label{Eq:QuadrilateralWeb}
\sum_{i=0}^n (-1)^i \frac{\Pi_i  }{\pi_i^1\cdots \pi_i^{n-1}}
=0\, 
\end{equation}
which is equivalent to saying that 
 $\big( (-1)^i \Pi_i/(\pi_i^1\cdots \pi_i^{n-1})\big)_{i=0}^n$ is an element of  $\boldsymbol{AR}\big(\boldsymbol{\mathcal L \mathcal  W}_{p_0,\ldots,p_n}\big) $.\mk

As for many kinds of webs, there are universal bounds on the rank of curvilinear webs. 
The following result has been obtained by Damiano, and is of crucial importance considering the purpose of this paper:\footnote{This result is only stated in \cite{D}. However a detailed proof is given  \S3.3 in Damiano's thesis  \cite{DThesis}. In the case when $n=3$ and $d\leq 5$, the majoration \eqref{Eq:QuadrilateralWeb} has been obtained long before 
by K\"ahler, see ${\boldsymbol{\mathsf S}_4}$  in \cite{Blaschke-1-Rank}.}

\begin{prop}
\label{Prop:Damiano'sBound}
 For any curvilinear $d$-web $\boldsymbol{\mathcal  W}_d$ on a domain of $\mathbf C^n$, one has: 
\begin{equation}
\label{Eq:Bound-on-the-rank}
{\rm rk}
\big(\boldsymbol{\mathcal  W}_d\big)\leq  \sum_{\sigma=0}^{d-n-1} {  n-2+\sigma \choose  \sigma } \Big( d-n-\sigma \Big)^+
\, 
\end{equation}
where $m^+$ stands for $\max\{ 0, m\}$ for any $m\in \mathbf Z$. 
\end{prop}

The case of a quadrilateral web is interesting. When $d=n+1$, the RHS of \eqref{Eq:Bound-on-the-rank} is equal to 1 for any $n\geq 2$ and since \eqref{Eq:QuadrilateralWeb} corresponds to a non-trivial AR for $\boldsymbol{\mathcal Q \mathcal  W}$, it comes that $\boldsymbol{\mathcal Q \mathcal  W}$ has {\bf maximal rank}, that is is such that \eqref{Eq:Bound-on-the-rank} actually is an equality. 

\subsubsection{\bf Algebraic webs.} 
\label{SSS:Algebraic-webs}
What makes the relevance of the notion of web of  maximal rank is that, by means of a classical construction relying on basic and very well-known results of algebraic geometry, one can associate such a web to any plane algebraic curve.\mk 

The construction, which goes back to the early developments of web geometry \cite{Blaschke1933}, goes as follows: let $C\subset \mathbf P^2$ be
any reduced algebraic curve of degree $d\geq 3$. For any generic line $L_0$ intersecting $C$ transversally, one can find $d$ germs of algebraic maps $P_i: (\check{\mathbf P}^2,L_0)\rightarrow C$ such that, as 0-cycles, one has $C\cdot L=\sum_{i=1}^d P_i(L)$ for any line $L$ sufficiently close to $L_0$ (see Figure \ref{fig:method} below). The $P_i$'s are first integrals of the (germ at $L_0$ of the)  {\bf algebraic web  associated to $\boldsymbol{C}$}, denoted by $\boldsymbol{\mathcal  W}_C$.\sk

\begin{figure}[h!]
\centering
\includegraphics[width=70mm]{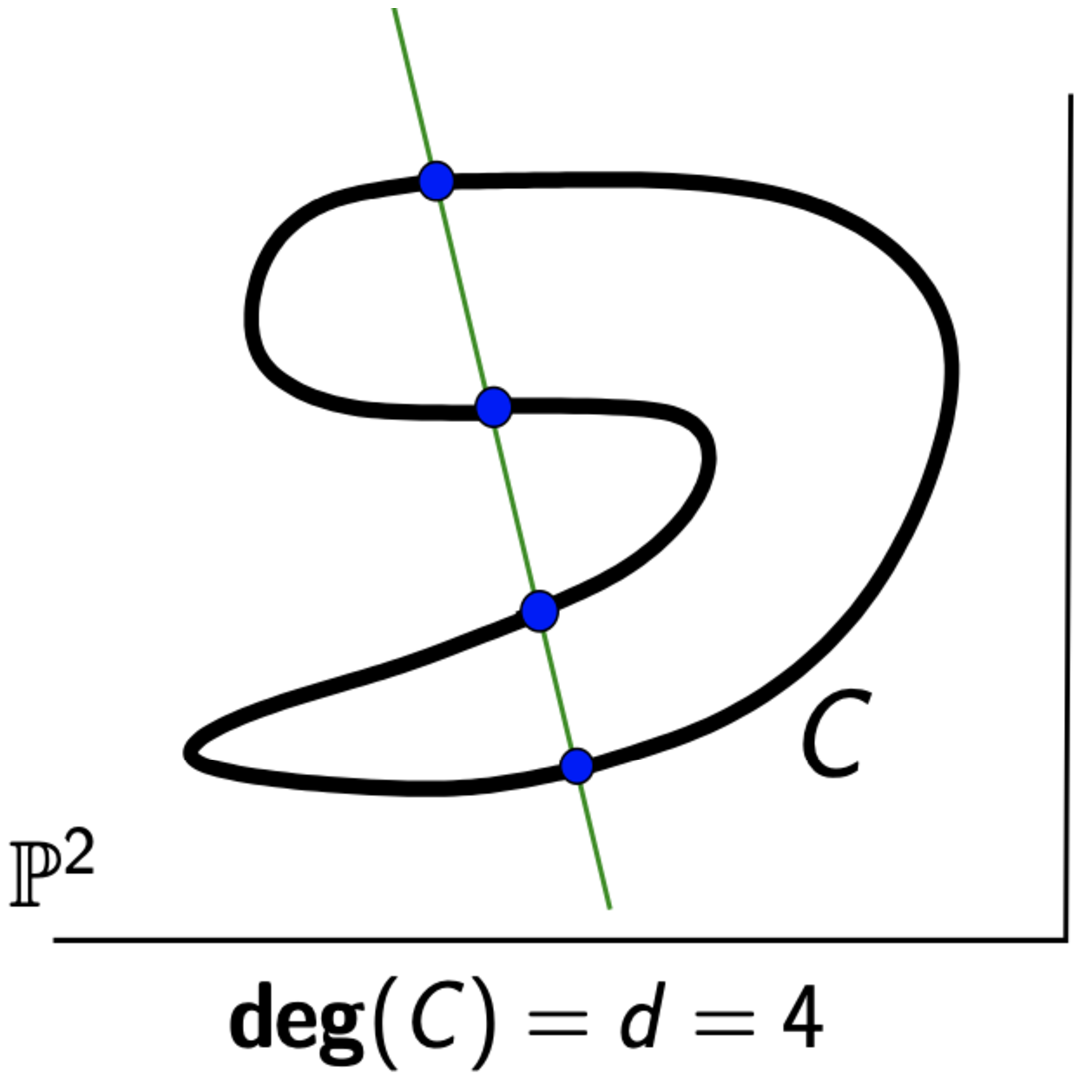}
\vspace{-0.3cm}
\caption{}
\label{fig:method}
\end{figure}

 Then it follows from Abel's addition theorem that for any abelian differential $\omega$ on $C$, one has $\sum_{i=1}^dP_i^*(\omega)=0$ from which it can be deduced that  
 $\omega\mapsto (P_i^*(\omega))_{i=1}^d$ is a well-defined injective linear map $\boldsymbol{H}^0(C,\omega^1_C)\rightarrow  \boldsymbol{AR}\big(\boldsymbol{\mathcal W}_C\big)$. It follows that $\boldsymbol{\mathcal W}_C$ has rank at least equal to $h^0(\omega^1_C)=p_a(C)=(d-1)(d-2)/2$, a quantity which coincides with the RHS of \eqref{Eq:Bound-on-the-rank} when $n=2$. It follows that the algebraic web 
 $\boldsymbol{\mathcal  W}_C$ 
  is always of maximal rank.\mk 
  
 What makes the interest of the notions of abelian relation and rank in the case of planar webs, is that these allow, in some cases, to prove algebraization results for webs of maximal rank ({\it cf.}\,\cite{Coloquio} for an extensive exposition from this perspective).  
 But the real interest of all these notions  lies precisely  in the fact that there exist planar webs of maximal rank which however do not come from a plane algebraic curve.


\subsubsection{\bf Abel's identity of the dilogarithm and Bol's web.}
Several authors of the XIXth and XXth centuries have independently discovered equivalent versions of the nowadays so-called Abel's 5-terms equation
$$
\big(\mathcal Ab\big) 
\hspace{3cm}
R(x)-R(y)-R\left(\frac{x}{y}\right)-R\left(\frac{1-y}{1-x}\right)
+R\left(\frac{x(1-y)}{y(1-x)}\right)=0\,
\hspace{3cm} {}^{}
$$
which is identically satisfied for  any $(x,y)\in \mathbf R^2$ such that $0<x<y<1$, by the famous \href{https://mathworld.wolfram.com/RogersL-Function.html}{\it Rogers' dilogarithm} $R$ defined by 
\begin{equation}
\label{Eq:R}
R(x)= {\bf L}{\rm i}_2(x) + \frac{1}{2}\log(x)\log(1 - x) - 
{\pi^2}/{6}\, 
 \end{equation}
for $x\in ]0,1[$, where ${\bf L}{\rm i}_2$ stands for the classical bilogarithm.\footnote{We recall that ${\bf L}{\rm i}_2(z)$ is defined  as the sum of the series $\sum_{k\geq 0} z^k/k^2$ which converges on the unit disk $\{ z\in \mathbf C\,, \, \lvert z\lvert<1\, \}$.} \mk 

Abel's identity (or more precisely, its total derivative with respect to $x$ and $y$) gives rise to an AR, denoted by $\boldsymbol{\mathcal Ab}$, for the so-called {\bf Bol's web $\boldsymbol{\mathcal B}$} which is the planar 5-web defined by the rational functions appearing as arguments of $R$ in $(\mathcal Ab)$: 
\begin{equation}
\label{Eq:B}
\boldsymbol{\mathcal B}=\boldsymbol{\mathcal W}\left( \, x
\, , \,  y 
\, , \,   \frac{x}{y}
\, , \,  \frac{1-y}{1-x}\, , \, 
 \frac{x(1-y)}{y(1-x)}\, 
\right)\, . 
\end{equation}
Of course Bol's web admits other ARs than the dilogarithmic AR $\boldsymbol{\mathcal Ab}$: for instance, its 3-subweb $\boldsymbol{\mathcal W}(  x
 ,  y  , {x}/{y}) $ is hexagonal and carries the AR associated to the basic functional equation 
 $${\rm Log}(x)-{\rm Log}(y)-{\rm Log}
 \big(x/y\big)
  =0$$
 of the logarithm (the total derivative of which is exactly the identity \eqref{Eq:QuadrilateralWeb} in the case when $n=2$).\mk

Bol's web is very particular and it has been known for a long time that it enjoys the following remarkable properties as a planar web: \mk
\begin{equation}
\label{Page:Properties-Bol's-Web}
\begin{tabular}{l}
{\bf [\,Hexagonality\hspace{0.03cm}].} {\it Bol's web is hexagonal and all its  3-terms ARs are logarithmic.}\\{\it  Moreover, the space $\boldsymbol{AR}_{Log}(\boldsymbol{\mathcal B})$ that these 
 abelian relations span  is 5-dimensional.}\mk \\ 
{\bf [\,Maximality of the rank\hspace{0.03cm}].}
{\it Since  the 5-terms abelian relation $\boldsymbol{\mathcal Ab}$ is dilogarithmic, it}\\
{\it does not belong to 
the space of logarithmic ARs hence 
Bol's web 
has maximal rank 6}\\
{\it  and there is a decomposition in direct sum} \sk\\ 
${}^{}$ 
\hspace{5cm} 
$\boldsymbol{AR}\big(\boldsymbol{\mathcal B}\big)=\boldsymbol{AR}_{Log}\big(\boldsymbol{\mathcal B}\big)\oplus \big\langle \boldsymbol{\mathcal Ab} \big\rangle \, .
$\mk \\
{\bf [\,Non linearizability\hspace{0.03cm}].} {\it $\boldsymbol{\mathcal B}$ is not linearizable hence not equivalent to an algebraic web.}\mk \\
{\bf [\,Characterization\hspace{0.03cm}].} {\it For $d\geq 3$, a  hexagonal planar $d$-web  either is linearizable and}\\ {\it  equivalent to a web formed by $d$ pencils of lines  or $d=5$ and it is equivalent to $\boldsymbol{\mathcal B}$.}
\end{tabular}
\end{equation}
%
%
\sk 

Bol's web has been the first known example of an {\bf exceptional web}, that is a web of maximal rank but which is not equivalent to an algebraic one.  Since its discovery in the 1930s by Blaschke and Bol (see \cite[\S4]{Blaschke1933} and \cite{Bol}), several new examples of planar exceptional webs have been discovered (see the sixth chapter of \cite{Coloquio} for a recent overview).  \mk 

The interest of the notion of `exceptional web' comes from the facts that first, such webs exist in numbers; and second, that one can mimic for them some very classical constructions of algebraic geometry  (canonical map, etc) which ask many interesting questions 
about the similarity between plane algebraic curves and planar exceptional webs. This is possibly what motivated several authors (such as Chern and Griffiths, {\it cf.}\,\cite[p.\,83]{CG2}) to qualify as one of the main problems in web geometry the following one: 
\begin{equation}
\begin{tabular}{l}
{\it Determine/classify the webs of maximal rank, especially the non-algebraic ones.}
\end{tabular}
\end{equation}

Most of the works regarding this problem concerned webs of codimension 1. For curvilinear webs, we are only aware of old (but quite remarkable) results by Blaschke and Walberer about skew curvilinear 3-webs 
regarding their $k$-ranks \cite{BlaschkeK,BW}, and the recent results of Damiano  which we return to in this article.\mk 

We note that what might be an `algebraic curvilinear web' has still not been precisely defined, even if it is known since \cite{BW} that some webs certainly deserve to be 
characterized as such. We will shed more light (but just a little) about this later on in the Introduction.

%


\subsubsection{\bf The webs $\mathcal W_{{0,n+3}}$ for $n\geq 2$}
\label{SSS:Webs-W_0-n+3}
Its is well-known and very easy to verify that Bol's web admits also the following nice geometric description: the web defined by the five rational first integrals in \eqref{Eq:B} is nothing else but a model in local coordinates,   of the web on the 2-dimensional moduli space ${\mathcal M_{0,5}}$ defined by the five 
rational fibrations  ${\mathcal M_{0,5}}\rightarrow {\mathcal M_{0,4}}\simeq \mathbf P^1\setminus \{0,1,\infty\}$ induced by forgetting a point among the five of any configuration element of 
${\mathcal M_{0,5}}$. 
\mk

Describing Bol's web in such a geometric way is interesting since it suggests immediately the following generalization:  for  $n\geq 2$ fixed and any $i=1,\ldots,n+3$, the forgetful map
$$\varphi_i: {\mathcal M_{0,n+3}}\rightarrow {\mathcal M_{0,n+2}} 
 $$ consisting in forgetting the $i$-th point of a configuration of $n+3$ points on $\mathbf P^1$, is a surjective rational map whose fibers are smooth rational curves with $n+2$ punctures.  The corresponding foliations of dimension 1 satisfy the 
 `general position assumption' of \S\ref{SSS:webs} hence form a curvilinear $(n+3)$-web on ${\mathcal M_{0,n+3}}$, which we will denote by 
\begin{equation} 
\label{Eq:W-0-n+3}
\boldsymbol{\mathcal W}_{{0,n+3}}=
\boldsymbol{\mathcal W}\big( \, \varphi_1\, ,  \ldots, 
 \varphi_{n+3}\, \big)\, . 
 \end{equation} 
 
 As explained just above, $\boldsymbol{\mathcal W}_{{0,5}}$ is a geometric model for Bol's web $\boldsymbol{\mathcal B}$. The latter being so important  regarding the study of non-linearizable webs of maximal rank, it appears quite natural to study the whole family of webs 
 $\boldsymbol{\mathcal W}_{{0,n+3}}$'s from the same perspective. \mk 
 
 The webs  $\boldsymbol{\mathcal W}_{{0,n+3}}$'s have been first mentioned by Burau in 
 \cite{Burau}, a paper on which we will come back further in \S\ref{SS:RoomBurau}. The study of the 
 $\boldsymbol{\mathcal W}_{{0,n+3}}$'s  for $n\geq 2$ arbitrary 
 with regard to their abelian relations, their ranks, etc, is  much more recent and  is due to Damiano hence we are going to  describe Damiano's work in some detail in the next subsection. \mk 
 
 To end this short presentation, let us mention the following obvious generalizations of the webs   $\boldsymbol{\mathcal W}_{{0,n+3}}$'s: for any fixed $k\in \{1,\ldots,n-1\}$, 
 there are  ${ n+3 \choose k}$ ways to forget $k$ points among all those of a configuration of $n+3$ points on $\mathbf P^1$ hence the corresponding forgetful maps 
 $ {\mathcal M_{0,n+3}}\rightarrow {\mathcal M_{0,n+3-k}}$  are the first integrals of 
 a ${ n+3 \choose k}$-web of dimension $k$ on $ {\mathcal M_{0,n+3}}$, denoted by 
 $\boldsymbol{\mathcal W}_{{0,n+3}}^k$.   The 1-codimensional webs $\boldsymbol{\mathcal W}_{{0,n+3}}^{n-1}$ have been studied in \cite{Pereira} where the author proved that they carry many logarithmic and dilogarithmic ARs and that  these are sufficiently many to obtain that, in some sense,  each such web has maximal rank.\footnote{Using the terminology introduced in \cite{ClusterWebs}, one can say that each web $\boldsymbol{\mathcal W}_{{0,n+3}}^{n-1}$ is  `AMP'.}
 \sk 
 
 It would be interesting to study and even quite better, to give a complete description of the $\ell$-ARs and the $\ell$-ranks of the webs $\boldsymbol{\mathcal W}_{{0,n+3}}^k$, this  for any $n\geq 2$, $k\leq n-1$ and $\ell\leq n-k$.  The present paper can be seen as a contribution to this wide open 
 question, in the specific case when $k=1$ and $\ell=n-1$. 
 
\subsection{Damiano's work}
\label{SS:Damiano's-work}
In \cite{GelfandMacPherson}, Gelfand and MacPherson describe a geometric construction of some differential forms on the spaces of projective configurations in (real) projective spaces from characteristic classes. Their general construction admits as a particular and very interesting case (the differential version of) Abel's dilogarithmic identity  $(\boldsymbol{\mathcal Ab})$. \sk

In his thesis \cite{DThesis} under the supervision of MacPherson (published in condensed form as the paper \cite{D}), Damiano applies Gelfand-Macpherson's approach to study the 
webs $\boldsymbol{\mathcal W}_{0,n+3}$'s for $n$ arbitrary, especially from the point of view of their rank and abelian relations. After having established several basic results about curvilinear webs such as Proposition \ref{Prop:Damiano'sBound} and studied carefully quadrilateral $(n+1)$-webs of curves and their abelian relations, Damiano focused on  the web $\mathcal W_{0,n+3}$'s for $n\geq 2$. We give below a short overview of the very interesting results he got/claims he got about these webs. 
 \sk 

He started by  noticing that any  $(n+1)$-subweb $\boldsymbol{\mathcal W}_{\widehat{\imath\jmath}}$ of $\mathcal W_{0,n+3}$ given by disregarding only the $i$-th and  $j$-th foliations of 
$\boldsymbol{\mathcal W}_{0,n+3}$,  is quadrilateral 
 and consequently has rank 1 (according to the remark just after Proposition \ref{Prop:Damiano'sBound} above). 
The corresponding ${ n+3\choose 2}=(n+3)(n+2)/2$ ARs  are called `combinatorial abelian relations' of $\boldsymbol{\mathcal W}_{0,n+3}$ and their span is denoted by $ \boldsymbol{AR}_{C}\big(\boldsymbol{\mathcal W}_{{0,n+3}}\big)$.  This is an a priori  quite  big subspace of the total space of ARs of $\boldsymbol{\mathcal W}_{{0,n+3}}$ on which the symmetric group 
$\mathfrak S_{n+3}$ naturally acts linearly. In the case when $n=2$ and up to  the equivalence of webs $ \boldsymbol{\mathcal W}_{{0,5}} \simeq  \boldsymbol{\mathcal B}$,  
$ \boldsymbol{AR}_{C}\big(\boldsymbol{\mathcal W}_{{0,5}}\big)$ coincides with the 
 space $ \boldsymbol{AR}_{Log}\big( \boldsymbol{\mathcal B} \big)$ 
considered above in \eqref{Page:Properties-Bol's-Web}. 
\sk 

Next, Damiano applies Gelfand-MacPherson approach 
to the construction of another AR for $\boldsymbol{\mathcal W}_{0,n+3}$, not a combinatorial one. The geometric starting point is a special case of the {\it Gelfand-MacPherson correspondence}, namely 
 that the moduli space $\mathcal M_{0,n+3}$ identifies naturally with the quotient 
 of a certain dense open subset $\widetilde {G}_{2}(\mathbf R^{n+3})$ of the grassmannian variety of 2-planes in $\mathbf R^{n+3}$ 
by the linear action of the identity component $H_{n+3}^0\simeq (\mathbf R_{>0})^{n+2}$ of the Cartan torus $H_{n+3}\subset 
 {\rm GL}_{n+3}(\mathbf R^{n+3})$. 
Actually (and this is important for what is to come) likely in order to work uniformly regarding the parity of $n$, Damiano works with the corresponding open subset 
$\widetilde {G}_{2}^{or}(\mathbf R^{n+3})$ of the grassmannian variety of oriented 2-planes in $\mathbf R^{n+3}$ and the associated quotient $\mathcal E\hspace{-0.02cm}\mathcal M_{0,n+3}= \widetilde {G}_{2}^{or}(\mathbf R^{n+3})/H_{n+3}^0$ which is seen as a space of `{\it enhanced  configurations}' of $n+3$ points on $\mathbf P^1$ 
(coming with a covering map $\mathcal E\hspace{-0.02cm}\mathcal M_{0,n+3}\rightarrow \mathcal M_{0,n+3}$).   
\sk 

The forgetful maps $\varphi_i: \mathcal M_{0,n+3}\rightarrow \mathcal M_{0,n+2}$ admit natural lifts between the corresponding spaces of enhanced configurations (denoted by $\varphi_i$ as well) and also between the (open subsets of the) corresponding oriented grassmannians, denoted by 
$\tilde \varphi_i: \widetilde {G}_{2}^{or}(\mathbf R^{n+3})\rightarrow \widetilde {G}_{2}^{or}(\mathbf R^{n+2}_i)$ where ${G}_{2}^{or}(\mathbf R^{n+2}_i)$ stands for the grassmann subvariety formed by oriented 2-planes contained in the $i$-th coordinate hyperplane $\mathbf R^{n+2}_i=\{ 
 \, x_i=0\, \}$ of $\mathbf R^{n+3}$. \sk
 
 Denoting by $\pi_{n+3}$ the quotient map $ \widetilde {G}_{2}^{or}\big(\mathbf R^{n+3}\big)\rightarrow \mathcal E\hspace{-0.02cm}\mathcal M_{0,n+3}=
 \widetilde {G}_{2}^{or}\big(\mathbf R^{n+3}\big)/ H_{n+3}^0$, it can be verified that the latter factors through $\pi_{n+2}$ when restricted along $ \widetilde {G}_{2}^{or}(\mathbf R^{n+2}_i)$  and these quotient maps together with the forgetful maps defined in the previous paragraph all  fit into the following commutative diagram: 
$$
  \xymatrix@R=0.4cm@C=1.5cm{ 
  \widetilde {G}_{2}^{or}\big(\mathbf R^{n+3}\big) 
\ar@{->}[dd]_{\pi_{n+3}}  \ar@{->}[r]^{ \hspace{0.1cm}\widetilde \varphi_i} & \widetilde {G}_{2}^{or}\big(\mathbf R^{n+2}_i\big)\, 
\ar@{->}[dd]^{\pi_{n+2}}
\\ & \\
\mathcal E\hspace{-0.02cm}\mathcal M_{0,n+3} \ar@{->}[d] 
 \ar@{->}[r]^{ \hspace{0.1cm} 
 \varphi_i }  & \mathcal E\hspace{-0.02cm}\mathcal M_{0,n+2}\ar@{->}[d] 
 \\
\mathcal M_{0,n+3}  \ar@{->}[r]_{ \hspace{0.1cm}\varphi_i} 
& \mathcal M_{0,n+2} \, . 
  }
$$

Let $\mathcal T$ be  the tautological rank 2 vector bundle over ${G}_{2}^{or}(\mathbf R^{n+3}$. Its \href{https://en.wikipedia.org/wiki/Euler_class}{Euler class}  is a characteristic class element of $ H^2\big({G}_{2}^{or}(\mathbf R^{n+3}),\mathbf R\big)$  which can be represented by a closed ${\rm SO}_{n+3}$-invariant 2-form, denoted by $\boldsymbol{e}_{n}$.  For any $k\geq 1$, the integral
 along the fibers of $\pi_{n+3}$ of the $k$-th wedge power $\boldsymbol{e}_{n}^k$ exists and is a $(2k-n-2)$-differential form on $ \mathcal E\hspace{-0.02cm}\mathcal M_{0,n}$, denoted by 
$\widetilde{\boldsymbol{e}}_{n}^k
=\pi_*(\boldsymbol{e}_{n}^k)$. It is an example of  a {\it `generalized dilogarithm form'} according to the terminology introduced in  \cite{GelfandMacPherson}. \sk

Damiano proves that (1) $\widetilde{\boldsymbol{e}}_{n}^k$ vanishes, except  when $k=n+1$ since it this case $\boldsymbol{e}_{n}^{n+1}$ is an invariant volume form on 
${G}_{2}^{or}(\mathbf R^{n+3})$; (2) for any $i$, the restriction
 of $\boldsymbol{e}_{n}$ along 
${G}_{2}^{or}(\mathbf R^{n+2}_i)$ 
 coincides (up to a sign corresponding to the compatibility between natural orientations) with the 2-form $\boldsymbol{e}_{n-1}$ intrinsically associated to this oriented grassmannian. Combining this (in the case when $k=n$) with the key technical result of 
 \cite{GelfandMacPherson} which is a differential identity between generalized dilogarithms, Damiano obtains that 
 $$ 
\sum_{i=1}^{n+3} (-1)^{i} \varphi_i^*\big( \,\, \widetilde{\boldsymbol{e}}_{n-1}^n\,  \big)=0
$$
holds true identically on any connected components of $\mathcal E\hspace{-0.02cm}\mathcal M$. This allows to consider  the $(n+3)$-tuple of $(n-1)$-differential forms  $ \big( (-1)^{i }\varphi_i^*( \, \widetilde{\boldsymbol{e}}_{n-1}^n\, ) \big)_{i=1}^{n+3}$  as an AR for $\boldsymbol{\mathcal W}_{0,n+3}$, called the {\bf Euler abelian relation} in \cite{D} and which we will denote by $ \boldsymbol{\mathcal E}_n$ here. When $n=2$, it can be verified (see \cite{GelfandMacPherson}, \cite[\S8.4]{DThesis} or \S\ref{SS:En-integral-representation} below) that $ \boldsymbol{\mathcal E}_2$ coincides with 
Abel's abelian relation $\boldsymbol{\mathcal Ab}$ up to the natural identification of 
$\boldsymbol{\mathcal W}_{0,5}$ with Bol's web $\boldsymbol{\mathcal B}$. 
\mk 

The material above being introduced, it is now possible to state the main results claimed by Damiano about the webs $\boldsymbol{\mathcal W}_{0,n+3}$'s as the following  ones: 
\mk

\label{MainClaims}
\hspace{-0.4cm}{\bf Main claims in \cite{D}.} {\it 
Let $n$ be an integer bigger than or equal to 2. 
\begin{enumerate}
\item[{\bf 1.}] 
 The web $\boldsymbol{\mathcal W}_{{0,n+3}}$ is quadrilateral and non linearizable.\sk 
\item[{\bf 2.}] $ \boldsymbol{AR}_{C}\big(\boldsymbol{\mathcal W}_{{0,n+3}}\big)$  has dimension bigger than or equal to 
$n(n+3)/2$.  \sk 
\item[{\bf 3.}] The Euler abelian relation $\boldsymbol{\mathcal E}_n$ is a non-trivial element  of $ \boldsymbol{AR}\big(\boldsymbol{\mathcal W}_{{0,n+3}}\big)$.
\sk
\item[{\bf 4.}] 
 There is a decomposition in direct sum 
 \begin{equation}
 \label{Eq:Direct-Sum}
 \boldsymbol{AR}\big(\boldsymbol{\mathcal W}_{{0,n+3}}\big)=
 \boldsymbol{AR}_{C}\big(\boldsymbol{\mathcal W}_{{0,n+3}}\big)\oplus \big\langle  \boldsymbol{\mathcal E}_n\big\rangle 
\end{equation}
  and consequently $\dim \boldsymbol{AR}_{C}\big(\boldsymbol{\mathcal W}_{{0,n+3}}\big)= {n(n+3)}/{2}$ and  $\boldsymbol{\mathcal W}_{{0,n+3}}$ has maximal rank, {\it i.e.} 
  $${\rm rk}\big(\boldsymbol{\mathcal W}_{{0,n+3}}\big)=(n+1)(n+2)/2\, .$$ 
\item[{\bf 5.}] The natural action of 
 $\mathfrak S_{n+3}$ on $\mathcal M_{0,n+3}$ induces an action on 
the space of abelian relations of $\boldsymbol{\mathcal W}_{{0,n+3}}$
 whose decomposition in irreducible factors corresponds to the decomposition in direct sum \eqref{Eq:Direct-Sum}. Moreover:  \sk 
 \begin{enumerate}
 \item[(i).] $\boldsymbol{AR}_{C}\big(\boldsymbol{\mathcal W}_{{0,n+3}}\big)$ is the $\mathfrak S_{n+3}$-representation with Young diagram $\big[221^{n-1}\big]$;\sk
 \item[(ii).] $\big\langle  \boldsymbol{\mathcal E}_n \big\rangle 
$ is the sign $\mathfrak S_{n+3}$-representation 
(with Young diagram $\big[1^{n+3}\big]${\rm )}.
 \end{enumerate}\mk
\item[{\bf 6.}] A quadrilateral curvilinear $d$-web either is equivalent to a web formed by 
the lines passing through $d$ points in $\mathbf P^n$ 
or $d=n+3$ and the considered web is equivalent to $\boldsymbol{\mathcal W}_{0,n+3}$.  
\end{enumerate}
}\mk 

These remarkable statements are generalizations (with precisions/refinements) of 
the well-known properties of Bol's web listed in \eqref{Page:Properties-Bol's-Web}.  
However, some of them were previously known\footnote{For instance, {\bf 6.}\,was already known to Blaschke in the case when $n=3$,  as the reading of \cite[\S51]{Blaschke1955} shows.} and even worse, some  are not correct. It is what we are going to explain by stating our results in the next subsection.

\subsection{Results}
The purpose of this text is to revisit the statements discussed above. 

In short, mainly we 
\begin{itemize}
\item recall a classical result implying that any web $\boldsymbol{\mathcal W}_{0,n+3}$ is linearizable when $n$ is odd; 
\mk 
\item  discuss in depht the web $\boldsymbol{\mathcal W}_{0,6}$ (case when $n=3$): 
realizing that it is isomorphic to the web by lines on Segre's cubic hypersurface $\boldsymbol{S}\subset \mathbf P^4$, we describe its ARs algebraically in terms of the global sections of the dualizing sheaf $\omega_\Sigma^2$ of the associated Fano surface 
$\Sigma=F_1(\boldsymbol{S})\subset G_1(\mathbf P^4)$. In particular, we get that 
it is quite justified to say that 
$\boldsymbol{\mathcal W}_{0,6}$ is (equivalent to) an `algebraic web'. 
\mk 
\item  prove that some of the claims \eqref{Eq:Direct-Sum} are not correct 
for any odd integer  $n\geq 3$. In this case,  we give 
corrected versions of them;  
\mk 
\item   give an explicit expression for the (components of the) Euler abelian relation $\boldsymbol{\mathcal E}_n$, which is rational when $n$ is odd, whereas it involves logarithmic terms as well when $n$ is even.  This formula  is conjectural in full generality but we have verified that it indeed holds true  for $n$ less than or equal to $12$.
\end{itemize}
\mk

We describe below in more detail the results obtained (or just conjectured) in this text and how they are related  to the claims \eqref{Eq:Direct-Sum}. 

\subsubsection{\bf Linearization}
First we revisit the linearization problem for $\boldsymbol{\mathcal W}_{{0,n+3}}$ by recalling a nice and rather simple (but seemingly forgotten) construction and results  by Room  rediscovered  independently by 
Burau more than thirty years later (precise references will be given in \S\ref{SS:RoomBurau}). \mk



For any integer $n$, we set $\delta_n=1$ if $n$ is odd and $\delta_n=2$ if it is even. 
\bk

\hspace{-0.45cm}{\bf Theorem A.} (Room-Burau) 
{\it {\bf 1.} For any $n\geq 2$, the web $\boldsymbol{\mathcal W}_{{0,n+3}}$ is realizable 
 as a web of rational curves of degree $\delta_n$ on a certain projective variety $V_n$.}\sk 
 
{\it {\bf 2.} In particular, $\boldsymbol{\mathcal W}_{{0,n+3}}$ is linearizable for any odd integer $n\geq 3$.}
\bk

%

The second point of this theorem shows that the second part of assertion 
\eqref{Eq:Direct-Sum}.{\bf 1.}
is wrong  when $n$ is odd  
(see \S\ref{SSS:linearizability}   below for a description of the main flaw in the arguments advanced in \cite{DThesis}  for proving  \eqref{Eq:Direct-Sum}.{\bf 1}). 
 We believe that $\boldsymbol{\mathcal W}_{{0,n+3}}$ is indeed non linearizable when $n$ is even but this is still conjectural. 
\sk 

An interesting feature of the preceding theorem (which was of course stated by Room and by Burau only in terms of projective algebraic geometry),  is that it is constructive: there exists an explicit linear system $\mathcal L_n$ on $\mathbf P^n$ 
giving rise to a map 
$\varphi_{n}: \mathbf P^n\dashrightarrow  
 \mathbf P^{N_n} $
  birational onto its image 
$V_n={\rm Im}(\varphi_{n})$ such that 
$$\boldsymbol{W}_{0,n+3}=\big(\varphi_{n}\big)_*\big(\boldsymbol{\mathcal W}_{{0,n+3}}\big)$$ 
is a web by rational curves of degree $\delta_n$ on $V_n$.\sk

The $V_n$'s together with the associated $\mathcal L_n$'s
form an interesting  family of projective varieties and linear systems, studied in classical papers and in a few recent publications as well. The first two cases $n=2$ and $n=3$ are very classical and have really been studied a lot: for instance, $V_2$ is the \href{https://en.wikipedia.org/wiki/Del_Pezzo_surface}{\it del Pezzo quintic surface} in $\mathbf P^5$ and 
$\boldsymbol{W}_{0,5}=\big(\varphi_{2}\big)_*\big(\boldsymbol{\mathcal W}_{{0,5}}\big)$ is the web  formed by the five fibrations by conics on it.
The case $n=3$ is not less classical than the previous one, and given the importance it will have regarding our approach, it deserves to be stated as the following 
\bk 

\hspace{-0.45cm}{\bf Proposition B.}  
{\it The variety $V_3$ is \href{https://en.wikipedia.org/wiki/Segre_cubic}{\it Segre's cubic primal} $\boldsymbol{S}$, that is the 
(projectively unique) irreducible cubic hypersurface  in $\mathbf P^4$ with 10 nodes. And the push-forward web $\boldsymbol{W}_{0,6}=\big(\varphi_{3}\big)_*\big(\boldsymbol{\mathcal W}_{{0,6}}\big)$ coincides with the linear web  on 
$\boldsymbol{S}$  formed by the six covering families of lines contained in it.
}

\subsubsection{\bf The web $\mathcal W_{0,6}$, its abelian relations and Segre's cubic primal}
What makes the previous result important for us is that once aware of it, it becomes unavoidable to relate the web-theoretic questions we are interested in to some very nice and nowadays well-know results about 3-dimensional cubic hypersurfaces.  We recall this material below and explain how it is related to curvilinear webs, referring to 
\S\ref{SS:Cubic-Hypersurfaces}
 further for more details and references.\mk 

Let $X\subset \mathbf P^4$ be a cubic hypersurface which we assume here to be smooth (essentially for simplicity). It is known that  through a general point $x$ of $X$ pass
 six pairwise distinct lines 
 included in $X$.  
  Since  three such lines necessarily span a 3-plane\footnote{This follows easily from Bezout's theorem.}, the general position assumption of  \S\ref{SSS:webs} is satisfied hence these lines are the leaves of a linear 6-web on  (a certain Zariski open subset $X^0$ of) $X$, denoted by $\boldsymbol{\mathcal L\hspace{-0.05cm}\mathcal W}_X$, and which is canonically defined on $X$.  
Looking at the global geometric picture is quite relevant here, and can be better understood by considering the \href{https://encyclopediaofmath.org/wiki/Fano_scheme}{Fano scheme} $F=F_1(X)$ of lines contained in $X$,  which is a 
surface naturally embedded in $G_1(\mathbf P^4)$. 
\sk

\vspace{-0.5cm}
 First,  remark that $F$ is naturally the space of leaves of the foliations composing 
 $\boldsymbol{\mathcal L\hspace{-0.05cm}\mathcal W}_X$ locally: 
 at any $x_0\in X^0$, one can define regular germs $L_i: (X,x_0)\rightarrow F$ for $i=1,\ldots,6$,  such that for any $x\in X$ sufficiently close to $x_0$,  the $L_i(x)$'s correspond to the six lines contained in $X$ and passing through $x$. These six (germs of) maps are local first integrals for  $\boldsymbol{\mathcal L\hspace{-0.05cm}\mathcal W}_X$: locally  at   $x_0$, one has 
 $$\boldsymbol{\mathcal L\hspace{-0.05cm}\mathcal W}_X=\boldsymbol{\mathcal W}\Big(L_1,\ldots,L_6\Big)\, .$$
 In general (for instance when  $X$ is non singular), $F$ is a smooth irreducible surface with remarkable properties some of which allow to describe the ARs of $\boldsymbol{\mathcal L\hspace{-0.05cm}\mathcal W}_X$ quite nicely.  The properties which will be crucial 
for our purpose are the following two: 
 \begin{itemize}
    \item[{\it (i).}] the space $\boldsymbol{H}^0(F,\Omega_F^2)$ of global holomorphic 2-forms on $F$  has dimension 10;
    \sk
 \item[{\it (ii).}] for any $\omega \in \boldsymbol{H}^0(F,\Omega_F^2)$, its trace, defined locally 
 by ${\rm Tr}(\omega)=\sum_{i=1}^6 L_i^*(\omega)$, 
  vanishes on $X$.

 \end{itemize}
 The second property has to be seen as an analogue, for the incidence between point and lines included in the hypercubic $X$, of the classical Abel's addition theorem for the abelian differentials on algebraic curves mentioned in \S\ref{SSS:Algebraic-webs} above. As in the case of any algebraic planar web, one first deduce from {\it (ii)} that the following `Trace map' 
 \begin{align}
 \label{Al:Trace-map}
 {\rm Tr} : \boldsymbol{H}^0\big(F,\Omega_F^2\big) & \longrightarrow \boldsymbol{AR}\Big( 
\boldsymbol{\mathcal L\hspace{-0.05cm}\mathcal W}_X
 \Big)\\
 \omega & \longmapsto \big(\,  L_i^*(\omega)\,\big)_{i=1}^6
 \nonumber
 \end{align}
is well-defined. Since it is obviously injective, it follows that 
the rank of $\boldsymbol{\mathcal L\hspace{-0.05cm}\mathcal W}_X$ 
is bigger than or equal to $h^0(F,\Omega_F^2)$, which is 10 according to {\it (i).} Hence  considering the majoration \eqref{Eq:Bound-on-the-rank} in case $n=3$ and $d=6$, we deduce that ${\rm rk}\big( \boldsymbol{\mathcal L\hspace{-0.05cm}\mathcal W}_X\big)=h^0(F,\Omega_F^2)=10$ is maximal. 
 \mk 

What has been obtained  above can be stated in condensed form as the following: 
 \bk
 
\hspace{-0.45cm}{\bf Proposition C.}  
{\it For any   sufficiently general ({\it e.g.}\,smooth) cubic hypersurface $X\subset \mathbf P^4$:  
\sk 

1. The map \eqref{Al:Trace-map} induces a linear isomorphism 
$\boldsymbol{H}^0(F,\Omega^2_F)\simeq \boldsymbol{AR}(\boldsymbol{\mathcal L\hspace{-0.05cm}\mathcal W}_X)$; 
\sk 

2. Consequently $\boldsymbol{\mathcal L\hspace{-0.05cm}\mathcal W}_X$ is a linear 6-web of maximal 2-rank $10=h^0\big(F,\Omega_F^2\big)$. }
\bk

Considering the way a web of the form  $(\boldsymbol{\mathcal L\hspace{-0.05cm}\mathcal W}_X)$ is defined and in view of the preceding result, it is more than reasonable 
 to say that such webs are `{\it algebraic}'.\footnote{{\it Cf.}\,the generalization of the notion of `algebraic web' introduced in \cite[\S1.3]{EAW}.}
 \mk

{The preceding proposition shows in particular that, as a linear web, $\boldsymbol{W}_{0,6}=\boldsymbol{\mathcal L\hspace{-0.05cm}\mathcal W}_{\boldsymbol{S}}$ is just a special element of a family of algebraic webs, which generically are of maximal rank with all their ARs coming from holomorphic 2-forms on the corresponding Fano surfaces. A natural question which immediately arises is whether the isomorphism \eqref{Al:Trace-map} also has a specialization for $\boldsymbol{S}$. \sk

\vspace{-0.4cm}
It is known that many results/constructions concerning smooth hypercubics in $\mathbf P^4$ generalize to some singular cubics. In particular, some cases of nodal cubics have been considered by several authors, especially the case of 1-nodal cubics threefolds.  The case of Segre's cubic is quite specific. For instance, 
the Fano surface $\Sigma=F_1(\boldsymbol{S})$ of Segre's cubic is a rational surface with 21 irreducible components, which is in sharp contrast to the case of the Fano surface of 
a smooth hypercubic.  We have not been able to localize 
in the huge existing literature 
on the subject 
a place  where the suitable generalization of property {\it (ii)} above which could apply to the case of Segre's cubic is proved. However, by explicit computations, we have verified that it is indeed the case, which implies in particular that 
$\boldsymbol{AR}(\boldsymbol{\mathcal W}_{{0,6}})$
 is isomorphic to the space $\boldsymbol{H}^0(\Sigma,\omega_\Sigma^2)$
 of global abelian 2-forms  on $\Sigma$. 
\sk


 Actually, it is easily seen that the isomorphism  of complex vector spaces 
\begin{equation}
\label{Eq:ARWM06-H0KSigma}
\boldsymbol{H}^0\Big(\Sigma,\omega_\Sigma^2\Big)
\simeq 
\boldsymbol{AR}\Big(\boldsymbol{\mathcal W}_{{0,6}}\Big)
\end{equation}
 induced by \eqref{Al:Trace-map} 
 is an isomorphism of $\mathfrak S_6$-modules, a fact from which interesting consequences 
 can be deduced. For instance, it is classically known that 
 $\boldsymbol{H}^0(\Sigma,\omega_\Sigma^2)$ hence $\boldsymbol{AR}(\boldsymbol{\mathcal W}_{{0,6}})$, is irreducible as a $\mathfrak S_6$-representation,  an observation which immediately appears  as contradicting statement {\bf 5.}\,in 
 \eqref{Eq:Direct-Sum}.  But actually even more can be obtained:  from an explicit description that one can give of a basis of $\boldsymbol{H}^0\big(\Sigma,\omega_\Sigma^2\big)$
 and using the fact that $
 \boldsymbol{\mathcal W}_{{0,6}}=\varphi_3^*\big( \boldsymbol{\mathcal L\hspace{-0.05cm}\mathcal W}_{\boldsymbol{S}}\big) $, one obtains the following 
\mk 

 \noindent 
{\bf {Theorem D}.}  \, 
{\it {$\bf 1.$}  One has  $
 \boldsymbol{\mathcal W}_{{0,6}}=\varphi_3^*\big( \boldsymbol{\mathcal L\hspace{-0.05cm}\mathcal W}_{\boldsymbol{S}}\big) $ hence  $ \boldsymbol{\mathcal W}_{{0,6}}$ is an algebraizable (hence linearizable) web of maximal rank 10, and \eqref{Eq:ARWM06-H0KSigma}
actually   is an isomorphism of $\mathfrak S_6$-modules. \sk
}
  
{\it 
{$\bf 2.$}  Regarding the abelian relations of $\boldsymbol{\mathcal W}_{{0,6}}$, 
the following assertions hold true: \sk
\begin{enumerate}
\item[{\bf a.}]
\vspace{-0.15cm}
  One has $\boldsymbol{AR}_C\big( \boldsymbol{\mathcal W}_{{0,6}} \big)=\boldsymbol{AR} \big(\boldsymbol{\mathcal W}_{{0,6}} \big)$; in particular Euler's AR  $\boldsymbol{\mathcal E}_3$  is combinatorial;
\sk
\item[{\bf b.}] As a $\mathfrak S_{6}$-module, $\boldsymbol{AR}_{C}\big( \boldsymbol{\mathcal W}_{{0,6}} \big)=\boldsymbol{AR}\big( \boldsymbol{\mathcal W}_{{0,6}} \big)$ is irreducible with Young diagram $\big[31^3\big]$. 
\end{enumerate} 
}
}

 {Assertions {\bf a.} and {\bf b.} above contradict statements  {\bf 4.} and {\bf 5.} of \eqref{Eq:Direct-Sum} in the case when $n=3$.}

\subsubsection{\bf The general case}
Motivated by the preceding result, we have investigated the general case taking up and verifying Damiano's approach. 
  We correct his main theorem 
   by establishing 
\sk

\noindent
{\bf {Theorem E}.}  \, {\it Let $n$ be any integer bigger than or equal to 2.}

\begin{itemize}
\item[]  ${}^{}$ \hspace{-1cm} {\bf 1.} \, {\it When $n$ is even, all the claims in \eqref{Eq:Direct-Sum} hold true}.
\mk 
 
 \item[]  ${}^{}$ \hspace{-1cm} {\bf 2.}\,  {\it On the other hand, this is not true when $n$ is odd since  then the 
 following holds true:} \mk
{\it
 \begin{enumerate}
  \item[{\bf a.}] The web $ \boldsymbol{\mathcal W}_{0,n+3}$ is linearizable.\mk 
 \item[{\bf b.}] The  space of combinatorial ARs 
  has dimension $(n+1)(n+2)/2$ hence the rank of  $ \boldsymbol{\mathcal W}_{0,n+3}$ is  indeed maximal but one has $\boldsymbol{AR}\big( \boldsymbol{\mathcal W}_{0,n+3}\big)=\boldsymbol{AR}_C\big( \boldsymbol{\mathcal W}_{0,n+3}\big)$.
 \mk 
 \item[{\bf c.}]  The Euler abelian relation is combinatorial: {\it i.e.} one has $\boldsymbol{\mathcal E}_n\in \boldsymbol{AR}_C\big( \boldsymbol{\mathcal W}_{0,n+3}\big)$.
 \mk 
 \item[{\bf d.}] The representation of 
$\mathfrak S_{n+3}$ on $\boldsymbol{AR}_C\big( \boldsymbol{\mathcal W}_{0,n+3}\big)= \boldsymbol{AR}\big( \boldsymbol{\mathcal W}_{0,n+3}\big)$ is irreducible with associated Young symbol $\big[3,1^n\big]$.
 \end{enumerate}}
 \end{itemize}
 \mk 
 
 Of course, the first part of this theorem is fully due to Damiano. Only the second one (when $n$ is odd) is new. However, it is fair to mention that it is proved just by  adding an elementary (but new) fact to Damiano's argumentation.

\subsubsection{\bf Explicit formulas for Euler's abelian relation}

We have also investigated more in depth Euler's abelian relation $\boldsymbol{\mathcal E}_n$.  Here are the main results we have obtained about it,  relatively to a certain explicit choice of rational first integrals for $\boldsymbol{\mathcal W}_{{0,n}} $: 
{\it \begin{itemize}
\item  We give a simple closed integral formula for the components of $\boldsymbol{\mathcal E}_n$ for any $n\geq 2$;
\mk 
\item  
For any odd integer $n\geq 3$, we give an explicit conjectural rational formula for the components of $\boldsymbol{\mathcal E}_n$. We prove this formula for $n$ sufficiently small (e.g.  for $n\leq 11$);
\mk 
\item  For any even integer $n\geq 2$,  we give an explicit conjectural formula for the components  of $\boldsymbol{\mathcal E}_n$, involving only rational and logarithmic quantities. We prove that this formula is valid for $n$ small enough (e.g. for $n\leq 12$). 
\end{itemize}
} 

We refer respectively to Proposition \ref{P:Prop-em-integral-formula}, Proposition \ref{P:5.20} and Proposition \ref{P:properties--en} for more precise statements, 
and to  \eqref{Eq:e-m(u)}, 
\eqref{Eq:tilde-e}
and 
\eqref{Eq:Euler-varepsilon}
for the corresponding formulas.

\subsection{Plan of the paper}
The  current Introduction  constitutes the first section of this paper. 

 The sequel 
  is organized as follows: 
\begin{itemize}
\item{} In Section \S\ref{S:M0n+3-and-W0n+3}, we start by discussing a few basic facts about the moduli spaces $\mathcal M_{0,n+3}$ and the webs $\mathcal W_{\mathcal M_{0,n+3}}$, for $n\geq 2$ arbitrary. We recall  Room and Burau's results in \S\ref{SS:RoomBurau}, 
from which one deduces Theorem A.  The corresponding flaw in the argumentation of  \cite{DThesis} is 
briefly discussed in 
\S\ref{SSS:linearizability}. 
\mk 
\item Section \S\ref{S:n=3} is entirely devoted to the case $n=3$ which is studied in depth. 
We start by discussing in 
\S\ref{SS:Cubic-Hypersurfaces}
the 6-web by projective lines carried by a general cubic hypersurface in $\mathbf P^4$. 
We explain that Proposition C follows from well-known results of Clemens and Griffiths about cubic threefolds.   The case of Segre's cubic $\boldsymbol{S}$ is discussed in \S\ref{SS:Segre's-cubic}. By means of elementary  explicit computations,  we prove  that, as for smooth cubics, the trace gives rise to the isomorphism \eqref{Eq:ARWM06-H0KSigma} for $\boldsymbol{S}$ as well. From this, we deduce  an explicit basis of 
$\boldsymbol{AR}(\boldsymbol{\mathcal W}_{0,6})$, all the elements of which are rational abelian relations. 
The structure of $\boldsymbol{AR}(\boldsymbol{\mathcal W}_{0,6})$ as a 
 ${\mathfrak S}_6$-module 
 is studied in \S\ref{SS:ARW06-as-a-S6-module}. Theorem D is proved there. 
\mk 
\item  In Section \S\ref{S:Conjectures-ARs}, 
Damiano's approach for studying  
$\boldsymbol{AR}(\boldsymbol{\mathcal W}_{0,n+3})$ as a $\mathfrak S_{n+3}$-module 
 is taken up in detail. 
Essentially all the material here is taken from \cite{DThesis}, the single novelty being Lemma \ref{L:n-odd}. Although its statement as well as its proof are elementary, the second part of Theorem E follows quite easily from it.  We get an explicit basis for the 
space $\boldsymbol{AR}_C\big( \boldsymbol{\mathcal W}_{0,n+3}\big)$ of combinatorial ARs.
\mk 
\item Section \S\ref{S:Euler-AR} is about Euler's abelian relation, 
the construction of which is taken up in detail.  
After having given a concise integral formula for the components of $\boldsymbol{\mathcal E}_n$ for $n$  arbitrary in \S\ref{SS:En-integral-representation}, 
$\boldsymbol{\mathcal E}_n$,  we deduce from 
some dihedral invariance properties of $\boldsymbol{\mathcal E}_n$ two transformations formulas that its components must satisfy (in \S\ref{SSS:Dihedral-Invariance-Properties}). 
We then turn to the case when $n$ is odd: in \S\ref{SSS:En-n-odd}, using the two 
just mentioned transformation formulas, we give an explicit rational expression for the 
components of $\boldsymbol{\mathcal E}_n$, which is conjectural in full generality but is proved to be the right one for any odd integer $n\leq 11$.  The case when $n$ is even is considered just after in \S\ref{SS:En-n-even}. 
We start by dealing with the case $n=4$. We first remark that Abel's method for solving abelian functional equations applies quite well as well for determining the ARs of any given curvilinear web. Applying this approach, we are able to give an explicit formula for the components of 
$\boldsymbol{\mathcal E}_4$ ({\it cf.}\,\eqref{Eq:E4(x2x3x4)}), from which we conjecture a closed formula involving  only rational and logarithmic quantities for any even integer $n\geq 2$ (see \eqref{Eq:Euler-varepsilon}). We prove that this formula is indeed valid for any even integer $n\leq 12$. 
\mk 
\item In the last section \S\ref{S:Problems} of the paper,  we formulate some questions  that we find interesting about curvilinear webs. Some are in relation to what has been discussed before, others are not. 
A few of them  concern the study of some projective varieties a better understanding of which could enlighten us  on the abelian relations of the 
$\boldsymbol{\mathcal W}_{{0,n+3}}$'s when $n$ is odd. \sk
\end{itemize}

Two appendices have been added at the end:
\begin{itemize}
\item In Appendix A, we investigate the 
 1-abelian relations of $\boldsymbol{\mathcal W}_{0,6}$ corresponding to the abelian 1-differentials on (a certain desingularization $\widetilde \Sigma$ of) the Fano surface  $\Sigma$ of Segre's cubic $\boldsymbol{S}$. 
 After explaining conceptually why the global sections of $\omega^1_{\widetilde \Sigma}$ give rise to 1-abelian relations for $\boldsymbol{\mathcal W}_{0,6}$  (via the specialization to Segre's cubic of a certain version of `Abel's theorem' for the holomorphic 1-forms on Fano surfaces of smooth cubics in $\mathbf P^4$), we explicitly determine these latter 
 using the $\mathfrak S_6$-structure of the space 
 $\boldsymbol{AR}^{(1)}\big(\boldsymbol{\mathcal W}_{0,6} \big)$
 they span.  
\sk 
\item Finally in Appendix B,  
we study the web $
\boldsymbol{\mathcal L\hspace{-0.05cm} \mathcal W}_{\mathscr C}$  formed by the lines contained in the so-called `chordal cubic' $\mathscr C\subset \mathbf P^4$.  This is a degenerate case in which
$\boldsymbol{\mathcal L \hspace{-0.05cm}\mathcal W}_{\mathscr C}$
is a 3-web (and not a 6-web as in the 
case of a 
generic cubic threefold). Using Abel's method, we determine the 1-abelian relations of $
\boldsymbol{\mathcal L\hspace{-0.05cm} \mathcal W}_{\mathscr C}$. It turns out that for this specific example, this web is equivalent to Blaschke-Walberer web defined on the variety of triangles included in $\mathscr C$. We take the opportunity offered by this coincidence to recall here the very nice classical but seemingly forgotten  results by Blaschke and Walberer about curvilinear 3-webs in dimension 3 with maximal rank.
\end{itemize}
\bk

\subsection*{\bf Acknowledgements.} The author is grateful to 
\href{https://imag.umontpellier.fr/~bolognesi/}{Michele Bolognesi} and 
\href{https://sites.google.com/view/castravet/home}{Ana-Maria Castravet}
for several interesting discussions about the geometry of the moduli spaces of marked rational curves.
  
Enfin, et bien s\^ur, l'auteur est  \'egalement tr\`es reconnaissant et redevable \`a Brubru pour ses  relectures nombreuses et ses corrections passionn\'ees.


%
%
%

%
%

%
%
%
%

\newpage
\section{\bf The moduli spaces $\mathcal M_{0,n+3}$ and 
the  webs $\mathcal W_{{0,n+3}}$ on them}
\label{S:M0n+3-and-W0n+3}
We start by recalling some basic facts about the moduli spaces $\mathcal M_{0,n+3}$ before 
describing  in several ways the webs ${\boldsymbol{\mathcal W}}_{{0,n+3}}$ on them. In \S\ref{SS:RoomBurau}, we recall Room-Burau's result and its consequence (namely Theorem A) for the webs ${\boldsymbol{\mathcal W}}_{{0,n+3}}$'s. We finish in \S\ref{SSS:linearizability}  by briefly  discussing the flawed reasoning used in \cite{DThesis} to wrongly conclude  that none of webs ${\boldsymbol{\mathcal W}}_{{0,n+3}}$ is linearizable. \mk 


All the material presented in this section is fairly standard and well known or well referenced in the literature. For this reason, no proof is given below, it seemed preferable to give specific references instead.

\subsection{\bf Basic facts.}
\label{SS:Basic-Facts}
We start by discussing $\mathcal M_{0,n+3}$ (for $n\geq 1$ arbitrary) over the field of complex numbers before discussing the web $W_{0,n+3}$ on it.  We end by 
considering its version over the reals, this in order to fit better with the setting of 
\cite{D}. Everything in \S\ref{SS:Basic-Facts} is standard and well-known.
\sk 

\vspace{-0.5cm}
In what follows, $n$ stands for a fixed integer bigger than or equal to 1.

\subsubsection{\bf The moduli space $\mathcal M_{0,n+3}$.}  
Working over $\mathbf C$, $\mathcal M_{0,n+3}$ stands for the moduli space of projective configurations of $n+3$ pairwise distinct points on the complex projective line $\mathbf P^1$: 
$$
\mathcal M_{0,n+3}=\Big\{ \, \big( z_i\big)_{i=1}^{n+3}\in (\mathbf P^1\big)^{n+3}\, 
\big\lvert \hspace{0.15cm} z_i\neq z_j\, \mbox{ for all }\, i\neq j
\hspace{0.15cm} 
 \Big\}_{\Big/\, {\rm PGL_2(\mathbf C)}}\, .
$$
As is well known, it is a smooth irreducible rational affine complex variety of dimension $n$. 
In order to get an  an explicit rational coordinates system on $\mathcal M_{0,n+3}$, let us consider the affine arrangement in 
 $\mathbf C^n$, denoted by $A_n$,  defined as the union of the $2n+n(n-1)/2=n(n+3)/2$ affine hyperplanes cut out by the equations 
  $X_i=0$, $X_i-1=0$ and $X_i-X_j=0$ 
 for $i,j=1,\ldots,n$ with  $i<j$. 
 Then, as is well known, the following rational map 
 \begin{align}
 \label{Eq:Psi-n}
 \psi_n : \, \mathbf C^n \setminus A_n &\longrightarrow   \mathcal M_{0,n+3}\\
 \nonumber
 \big(x_i\big)_{i=1}^n  & \longmapsto \Big[ \, 0,1,\infty ,x_1,\ldots,x_n \, \Big]
 \end{align}
is an isomorphism of affine varieties whose inverse map can be make explicit quite easily: 
for  
 $\boldsymbol{z}=[z_1,\ldots,z_{n+3}]\in \mathcal M_{0,n+3}$, 
 let $g_{\boldsymbol{z}}$ be the projective automorphism of $\mathbf P^1$ 
such that $g_{\boldsymbol{z}}(z_1)=0$, $g_{\boldsymbol{z}}(z_2)=1$ and $g_{\boldsymbol{z}}(z_3)=\infty$, namely 
$$
g_{\boldsymbol{z}}(\zeta)
=
\frac{(\zeta-z_1)(z_2-z_3)}{(\zeta-z_3)(z_2-z_1)}
$$
for any $\zeta \in \mathbf P^1$ (where this formula must be suitably interpreted when one of the $z_i$'s involved is equal to $\infty$).
Then all the $g_{\boldsymbol{z}}(z_k)$'s for $k=4,\ldots,n+3$ belong to $\mathbf C\setminus \{0,1\}$ and are pairwise distinct. 
Consequently, the map 
 \begin{align*}
 \phi_n : \, \mathcal M_{0,n+3}  &\longrightarrow  \mathbf C^n \setminus A_n   \\
 \boldsymbol{z}=\big[ z_i \big]_{i=1}^{n+3}  & \longmapsto \Big( g_{\boldsymbol{z}}(z_k) \Big)_{k=4}^{n+3}
 \end{align*}
is well-defined and  the composition $\phi_n\circ \psi_n$ obviously is the restriction of the identity to $\mathbf C^n \setminus A_n$.\mk

\subsubsection{\bf Action of $\mathfrak S_{n+3}$ by automorphisms and birational transformations.}
\label{SSS:S3-action}
For any $\sigma \in \mathfrak S_{n+3}$ and $\zeta=(\zeta_i)_{i=1}^{n+3}\in (\mathbf P^1)^{n+3}$, one sets $\sigma\zeta=(\zeta_{\sigma(i)}\big)_{i=1}^{n+3}$. Clearly, the map 
$\zeta=(\zeta_i)_{i=1}^{n+3}\mapsto \sigma\zeta=(\zeta_{\sigma(i)}\big)_{i=1}^{n+3}$ on $(n+3)$-tuples of points of $\mathbf P^1$ is ${\rm PGL}_2(\mathbf C)$-equivariant hence induces an automorphism of $\mathcal M_{0,n+3}$, denoted by $\gamma_\sigma$.  One gets a morphism of groups 
\begin{align*}
\mathfrak  S_{n+3} & \longrightarrow {\rm Aut}\big(\mathcal M_{0,n+3}\big)\\
\sigma & \longmapsto \gamma_\sigma
\end{align*}
which turns out to be an isomorphism.\footnote{This is a classical  fact of Teichm\"uller therory, see the remark at the very end of \cite[\S4]{Nag} for instance.}
\sk

Conjugating by the isomorphism $\phi_n$, one gets a group  embedding of 
$\mathfrak  S_{n+3}$ into the group ${\bf Bir}_n$ of birational maps in $n$ variables
\begin{align}
\label{Eq:G-sigma}
\mathfrak S_{n+3} & 
\longhookrightarrow
 {\bf Bir}_n\\
\sigma & \longmapsto  G_{\sigma}=\phi_n\circ  \gamma_\sigma \circ  \psi_n 
\nonumber 
\end{align}
 which is easy to describe explicitly: given a permutation $\sigma\in 
\mathfrak S_{n+3}$, for any 
   $x=(x_i)_{i=1}^n\in \mathbf C^n$ we set  
$\xi=(\xi_s)_{s=1}^{n+3}=(0,1,\infty,x_1,\ldots,x_n)$ and 
   $\sigma \xi=(\xi_{\sigma(s)})_{s=1}^{n+3}$.
  It is then immediate to verify that 
   the birational map $G_\sigma: \mathbf C^n\dashrightarrow \mathbf C^n$ is given by 
\begin{align}
\label{Eq:GG} 
G_\sigma (x)=
 \Bigg( 
\frac{\big(\xi_{\sigma(k)}-\xi_{\sigma(1)}\big)\big(\xi_{\sigma(2)}-\xi_{\sigma(3)}\big)}{\big(\xi_{\sigma(k)}-\xi_{\sigma(3)}\big)\big(\xi_{\sigma(2)}-\xi_{\sigma(1)}\big)}
\Bigg)_{k=4}^{n+3}\, .
\end{align}
\sk
\begin{exm} 
1. By way of illustration, let us consider the case of the $(n+3)$-cycle $c=(1\ldots n+3)$. For $\zeta=(0,1,\infty,x_1,\ldots,x_n)$ with $x=(x_1,\ldots,x_n)\in \mathbf C^n$, one has $c\zeta=(x_n,0,1,\infty,x_1,\ldots,x_{n-1})$ hence 
$g_{c\zeta}$ is given by $g_{c \zeta }(y)=(y-x_n)/(x_n (y-1))$ {\rm (}for $y\in \mathbf P^1${\rm )}.
It follows that $G_c$ is given 
 by
\begin{equation*}
G_c(x)=
\left( \, \frac{1}{x_n} \, , \,  \frac{x_1-x_n}{x_n(x_1-1)}\, , \, 
\ldots,  \frac{x_{n-1}-x_n}{x_n(x_{n-1}-1)}\,
\right) \, . 
\end{equation*}\sk\\
2. Some other explicit examples of birational maps $G_\sigma$ are given in Table  \ref{Table=Tr(Msigma)} of Appendix A.
\end{exm}

Formula \eqref{Eq:GG} is of importance for our purpose since at many places in the sequel, 
it will allow us 
 to explicitly compute  pull-backs of abelian relations on $\mathcal M_{0,n+3}$ under automorphisms $\gamma_\sigma$ for $\sigma\in \mathfrak S_{0,n+3}$,  in the rational coordinates $x_i$'s defined by means of \eqref{Eq:Psi-n}.

\subsubsection{\bf The web $\boldsymbol{\mathcal W}_{{0,n+3}}$.} 
\label{SSS:W0n+3}
We now assume that $n\geq 2$. 
In the Introduction, $\boldsymbol{\mathcal W}_{{0,n+3}}$ has been defined as the curvilinear $(n+3)$-web on $\mathcal M_{{0,n+3}}$ whose $i$-th foliation is the fibration  (by $(n+3)$-punctured smooth rational curves) induced by the map  consisting in forgetting the $i$-th point
\begin{equation}
\label{Eq:Forget-map-phi-i}
 \varphi_i: \mathcal M_{{0,n+3}}\longrightarrow \mathcal M_{{0,n+2}}\, .
 \end{equation}
   Here we want to make explicit an affine birational model of $\boldsymbol{\mathcal W}_{{0,n+3}}$.  
\sk 

Let $p_1,\ldots,p_{n+2}$ be $n+2$ fixed points in general position in $\mathbf P^n$. A natural choice for the $p_i$'s is to take the vertices of the standard simplex in $\mathbf P^n$, namely  
\begin{equation}
\label{Eq:points-p-i}
p_i=\big[ \,\delta_i^1:\cdots: \delta_i^{n+2}\,\big]\quad \mbox{for }\; i=1,\ldots,n+2 
\qquad \mbox{ and } \qquad p_{n+2}=\big[\,1:\cdots:1\,\big]\, .
\end{equation}
Given a generic point $p\in \mathbf P^n$, there exists a unique rational normal curve (RNC) of degree $n$ in $\mathbf P^n$ passing through all the $p_i$'s and through $p$. Denoting by $\nu_p: \mathbf P^1\rightarrow \mathbf P^n$ a projective parametrization of this curve, one obtains 
a $(n+3)$-tuple of points $\mu(p)=\big(\nu_p^{-1}(p_1),\ldots, \nu_p^{-1}(p_{n+2}), \nu_p^{-1}(p)\big)\in (\mathbf P^1)^{n+3}$ whose class  modulo ${\rm PGL}_2(\mathbf C)$ only depends on $p$. This gives us a well-defined map  
\begin{align}
\label{Eq:bb}  
\xymatrix@R=0.1cm@C=0.4cm{
\mathbf P^n \ar@{-->}[rr]& & \mathcal M_{0,n+3} \\
p  
\ar@{|->}[rr]& &
 \big[\mu(p)\big]\, .
}
\end{align}
Conversely, given $\boldsymbol{z}=[z_1:\cdots:z_{n+3}]\in \mathcal M_{0,n+3}$, the map 
$v_{\boldsymbol{z}} : \mathbf P^1\rightarrow \mathbf P^n,\, t\mapsto \big[   (z_{n+2}-z_i)/(t-z_i)\big]_{i=1}^{n+1}$ is a bijective parametrization of a RNC of degree $n$ which sends 
$z_i$ to $p_i$ for any $i$ ranging from $1$ to $n+2$. We get that way a rational  map
$\mathcal M_{0,n+3}\dashrightarrow \mathbf P^n,\, \boldsymbol{z}\mapsto v_{\boldsymbol{z}}(z_{n+3})=\big[   (z_{n+2}-z_i)/(z_{n+3}-z_i)\big]_{i=1}^{n+1}$ which is easily verified to be the inverse map of \eqref{Eq:bb}  which, as a result, is proved to be birational. 
\sk

\vspace{-0.4cm}
We will now give an explicit description of the pull-back of $\boldsymbol{\mathcal W}_{0,n+3}$ on $\mathbf P^n$ under  \eqref{Eq:bb}, that we will denote (a bit abusively) using 
 the same notation.  First 
remark that the projection from one  of its points of
a degree $n$ RNC in $\mathbf P^n$,  is a RNC of degree one less (namely $n-1$)  in $\mathbf P^{n-1}$. From this, one deduces easily that for $i=1,\ldots,n+2$,  
what corresponds on $\mathbf P^n$ (modulo  \eqref{Eq:bb}) to the $i$-th forgeful map \eqref{Eq:Forget-map-phi-i} is nothing else but the linear projection $\pi_{p_i}: \mathbf P^n\dashrightarrow \mathbf P^{n-1}$ from $p_i$. Finally, one verifies easily that the leaf of the $(n+3)$-th foliation of $\boldsymbol{\mathcal W}_{0,n+3}$ on $\mathbf P^n$ trough a generic point is precisely the RNC of degree $n$ passing through this point as well as through all the $p_i$'s. \sk 

  We thus get the following geometric description of the pull-back  of  $\boldsymbol{\mathcal W}_{0,n+3}$ under \eqref{Eq:bb}: 
\begin{equation}
\label{Eq:Exceptional-web}
\begin{tabular}{l}
{\it considered on $\mathbf P^n$, the web $\boldsymbol{\mathcal W}_{0,n+3}$ is formed 
 by the 
  families of lines passing through}\\
 {\it      one of the  $p_i$'s (for $i=1,\ldots,n+2$) plus the 
 family of degree $n$ rational normal curves }\\
 {\it   containing all  these points.}
\end{tabular}
\end{equation}  

From this description of $\boldsymbol{\mathcal W}_{0,n+3}$, it is straightforward to verify 
  that in the affine coordinates associated to the embedding $\mathbf C^n\hookrightarrow \mathbf P^n$, $u=(u_i)_{i=1}^n\mapsto [u_1:\cdots:u_n:1]$, this web  admits the rational functions $U_i$ 
   as first integrals, 
 where 
$U_i(u)=(u_1,\ldots,\widehat{u_i},\ldots,u_n)$ for $ i=1,\ldots,n$ and 
\begin{equation}
\label{Eq:U-n+1-n+2-n+3}
U_{n+1}(u)=
\left(  \frac{u_j}{u_n}
\right)_{j=1}^{n-1}
,\qquad 
U_{n+2}(u)=
\left(  \frac{u_j-1}{u_n-1}
\right)_{j=1}^{n-1}
\qquad \mbox{and}\qquad 
U_{n+3}(u)=
\left(  \frac{u_n(u_j-1)}{u_j(u_n-1)}
\right)_{j=1}^{n-1}. \sk
\end{equation}

It is obvious that  $\boldsymbol{\mathcal W}_{0,n+3}$ is invariant (as an unordered web) by the whole automorphism group ${\rm Aut}(\mathcal M_{0,n+3}) \simeq \mathfrak S_{n+3} $ and that 
its action on the set of foliations of this web is isomorphic to the standard action of the symmetric group on $\{1,\ldots,n+3\}$.  

\subsubsection{\bf Combinatorial abelian relations.}
For any distinct elements $i,j\in \{1,\ldots,n+3\}$ (say such that $i<j$), we denote by $\boldsymbol{\mathcal W}_{i,j}$ the $(n+1)$-subweb of $\boldsymbol{\mathcal W}_{0,n+3}$ formed by all its foliations except the $i$-th and the $j$-th ones.  Considering the last paragraph of the preceding subsection, it follows that all the subwebs $\boldsymbol{\mathcal W}_{i,j}$'s are equivalent (up to global automorphisms of $\mathcal M_{0,n+3}$).
\sk

Working with the affine coordinates $u_i$'s, it follows from \eqref{Eq:U-n+1-n+2-n+3} that 
$\boldsymbol{\mathcal W}_{n+2,n+3}$ coincides with the standard quadrilateral web considered in \S\ref{SSS:AR-rank}.  Hence it has rank equal to 1 
({\it cf.}\,\eqref{Eq:QuadrilateralWeb}) and since all the $(n+1)$-subwebs of $\boldsymbol{\mathcal W}_{0,n+3}$  are equivalent, any of them has rank 1 too:  for any $i,j$ with $i< j$, 
$\boldsymbol{\mathcal W}_{i,j}$ carries a nontrivial abelian relation, which will de denoted by $AR_{i,j}$. This AR is unique up to multiplication by a nonzero scalar (hence 
there is some ambiguity in its definition, but that doesn't really matter and therefore we will not pay attention to it).\sk

The $(n+3)(n+2)/2$ abelian relations $AR_{i,j}$'s  are all rational\footnote{By this we mean that  when expressed in the $u_i$'s, the components of these ARs are 
 rational 
  $(n-1)$-forms. } and span  what Damiano calls the subspace of `combinatorial abelian relations' of $\boldsymbol{\mathcal W}_{0,n+3}$, denoted by 
$$
\boldsymbol{AR}_C\big( \, \boldsymbol{\mathcal W}_{0,n+3} \, \big)=\left\langle \, 
AR_{i,j}
\, \big\lvert\, 1\leq i<j\leq n+3\,  
\right\rangle \subset \boldsymbol{AR}\big( \, \boldsymbol{\mathcal W}_{0,n+3} \, \big)\, .
$$

The term `combinatorial' comes from the fact that in their span, the $AR_{i,j}$'s satisfy relations of combinatorial nature which give rise to a matroid (see \S4.1 and \S8.2 in \cite{DThesis}). 
Although this aspect is interesting, we will not consider it here.

\subsubsection{\bf Real locus.}
\label{SSS:Real-Locus}
While most of the considerations in this paper are within the complex framework (which according to us is most natural), it is the real framework that was considered in both \cite{DThesis} and \cite{D}.  Working over $\mathbf R$ requires some subtleties, for instance regarding the domain on which the webs we are working with are considered, which were not very precisely explicited according to us. Hence we discuss briefly this aspect below.\footnote{Note that 
  it has some importance, for instance for  properly defining the $\mathfrak S_{n+3}$-action on  $\boldsymbol{AR}(  \boldsymbol{\mathcal W}_{0,n+3})$ or, as explained in 
  \S\ref{SS:Damiano's-work} (see also \S\ref{S:Euler-AR} further), when performing integration along fibers to obtain generalized dilogarithmic forms \`a la Gelfand and McPherson.} For some recent references, we mention \cite{Brown}, \cite[\S2.3]{FockGoncharov-XInfinity}
 as well as  the recent preprint \cite{AHL}.\mk

  The real moduli space $\mathcal M_{0,n+3}({\mathbf R})$ is the set of $(n+3)$-tuples $(x_1,\ldots,x_{n+3})$ of pairwise distinct points on the real projective line $\mathbf P^1_{\mathbf R}$ (that we identify with a circle $S^1$), modulo the diagonal action of ${\rm PSL}_2(\mathbf R)$. Its is a smooth  manifold of dimension $n$, but with several connected components contrarily to its complex version $\mathcal M_{0,n+3}$.   Clearly, $\mathcal M_{0,n+3}({\mathbf R})$ is invariant by the action of $\mathfrak S_{n+3}$ discussed in \S\ref{SSS:S3-action} and this group acts by real-analytic diffeomorphisms on it. 
\sk 
  
 By definition, the {\it `positive part'}
 $\mathcal M_{0,n+3}^{\,\,  >0}$ of  $\mathcal M_{0,n+3}(\mathbf R)$  is the subset  formed by {\it `positive configurations'} that is configurations coming from  tuples $(x_1,\ldots,x_{n+3})$ 
 positively cyclically ordered on the circle $S^1$ (with respect to a chosen orientation of $S^1$, previously  fixed once for all).  It is a connected component of $\mathcal M_{0,n+3}(\mathbf R)$ which is isomorphic to a ball whose topological closure in 
 a certain compactification 
  of the full real moduli space, is isomorphic (as a stratified real-analytic manifold) to a polytope, the  well-known  $n$-th \href{https://en.wikipedia.org/wiki/Associahedron}{\it `associahedron'} (or {\it `Stasheff polytope'}). \sk

The positive part $\mathcal M_{0,n+3}^{\, \, >0}$ is a connected component of $\mathcal M_{0,n+3}(\mathbf R)$,  from which all the others can be constructed.  
 Indeed, first one verifies that the set of permutations  elements of $\mathfrak S_{n+3}$ 
 letting $\mathcal M_{0,n+3}^{\,\,  >0}$ globally invariant is a  subgroup $D_{0,n+3}\subset \mathfrak S_{n+3}$ isomorphic to the dihedral group of order $2(n+3)$.  Consequently, 
 for each element $\boldsymbol{\sigma}$ of the coset space 
 $\mathfrak K_n=\mathfrak S_{n+3}/D_{0,n+3}$, there exists a  connected component $\mathcal M_{0,n+3}^{\, \, >0} ({\boldsymbol{\sigma}})$ of 
 $\mathcal M_{0,n+3}(\mathbf R)$, defined as the image of $\mathcal M_{0,n+3}^{\, \, >0}$ by any representative of $\boldsymbol{\sigma}$ in $ \mathfrak S_{n+3}$ (it is obviously well-defined): one has 
 $$
 \mathcal M_{0,n+3}^{\, \, >0} ({\boldsymbol{\sigma}})=\sigma\left( 
  \mathcal M_{0,n+3}^{\, \, >0}
 \right) \subset  \mathcal M_{0,n+3}(\mathbf R)\, .
 $$
 When $n$ is fixed and that there is no risk of any ambiguity, we will denote
 $ \mathcal M_{0,n+3}^{\, \, >0} ({\boldsymbol{\sigma}})$  by 
  $\mathcal M({\boldsymbol{\sigma}})$. 
 
For $\boldsymbol{\sigma}$ ranging in $\mathfrak S_{n+3}$, 
one gets  all the connected components of $\mathcal M_{0,n+3}(\mathbf R)$. 
  In other terms, there is  a disjoint union
 \begin{equation}
  \label{Eq:M0n+3>0}
  \mathcal M_{0,n+3}(\mathbf R)
=   \bigsqcup_{ {\boldsymbol{\sigma}} \in \mathfrak K_n}
  \mathcal M_{0,n+3}^{\, \, >0} ({\boldsymbol{\sigma}})
\end{equation}
 where any of  the $\lvert \mathfrak K_n\lvert=(n+2)!/2$ components $ \mathcal M_{0,n+3}^{\, \, >0} ({\boldsymbol{\sigma}})$ is  isomorphic to the positive part $\mathcal M_{0,n+3}^{\, \, >0}$  hence is topologically trivial. \sk
 
 \begin{exm}
 \label{Ex:n=2-on-IR}
 By way of illustration, let us consider the case when $n=2$ in more detail. Let $\mathscr A_2$ be the arrangement of five lines in $\mathbf R^2$ cut out by $xy(x-1)(y-1)(x- y)=0$. 
 Then $F : \mathbf R^2\setminus \mathscr A_2 \rightarrow 
\mathcal M_{0,5}(\mathbf R),\, (x,y)\mapsto [0: x : y:1:\infty]$ is an isomorphism and the 12 connected components of  $\mathcal M_{0,5}(\mathbf R)$ correspond to the 12 regions  $U_L $ with label $L$ ranging from $I$ to $XII$ on Figure \ref{Fig:FFIIGG} below. 
\mk 

\begin{figure}[h]
\begin{center}
\resizebox{3in}{2.5in}{
 \includegraphics{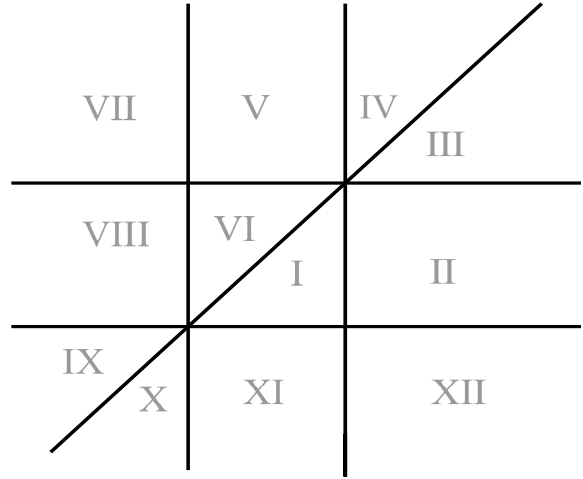}}
 \vspace{-0.3cm}
 \caption[caption]{Affine realization of the real moduli space $\mathcal M_{0,5}(\mathbf R)$ with \\\hspace{\textwidth} ${}^{}$ \hspace{1.4cm} a labeling from I to XII of its connected components.}
\label{Fig:FFIIGG}
\end{center}
\end{figure}

The region $U_I$ is the triangle $\{ (x,y)\in \mathbf R^2\, \big\lvert \, 0<x<y<1\, \}$ and it corresponds to (that is,  $F(U_I)$ is) the positive part $\mathcal M_{0,5}^{\, \, >0}$.  
  Both the cyclic shift $c: p=[p_1,\ldots,p_5]\longmapsto [p_5,p_1,\ldots, p_4]$ and  the dihedral reflection 
$r : [p_1,\ldots,p_5]\longmapsto [p_4,p_3,p_2,p_1,p_5]$ let  $\mathcal M_{0,5}^{\, \, >0}$ invariant and their pull-backs under $F$ 
 are respectively 
the order 5 birational map $C:(x,y)\mapsto \big( 1-y , (1-y)/(1-x)\big)$ and the affine involution $R: (x,y)\mapsto (1-y,1-x)$. The two maps $C$ and $R$  generate a subgroup of ${\bf Bir}_2$  isomorphic to
the automorphism group of a pentagon, that is the  dihedral group of order 10.
\end{exm}

 Being aware of the decomposition \eqref{Eq:M0n+3>0}
makes some of the subtleties arising when studying  the real version of 
$\boldsymbol{\mathcal W}_{0,n+3}$ on $\mathcal M_{0,n+3}(\mathbf R)$ 
more apparent, which was not sufficiently discussed in \cite{D} in our opinion. Indeed, \eqref{Eq:M0n+3>0} indicates that there is no global web in the real setting but several, namely 
one for each $\boldsymbol{\sigma}\in \mathfrak K_n$, denoted by $\boldsymbol{\mathcal W}_{0,n+3}^{\,  \boldsymbol{\sigma}}$,
 which is the one defined by the (restriction of the) forgetful maps 
\eqref{Eq:Forget-map-phi-i} 
on the component $ \mathcal M_{0,n+3}^{\, \, >0} ({\boldsymbol{\sigma}})$.  
 For each such $\boldsymbol{\sigma}$, 
let  $\boldsymbol{AR}(\boldsymbol{\sigma})=
\boldsymbol{AR}(\boldsymbol{\mathcal W}_{0,n+3}^{\,  \boldsymbol{\sigma}})$ be the space of abelian relations of $\boldsymbol{\mathcal W}_{0,n+3}^{\,  \boldsymbol{\sigma}}$
 which are globally defined on the whole definition domain of  this web. Moreover, since   
 this domain is  (homeomorphic to ) a ball, $\boldsymbol{AR}(\boldsymbol{\sigma})$ is naturally identified with the space of ARs of the restriction of 
 $\boldsymbol{\mathcal W}_{0,n+3}$ on any open subset of $ \mathcal M_{0,n+3}^{\, \, >0} ({\boldsymbol{\sigma}})$. \sk 
 
 Since $ \mathcal M_{0,n+3}^{\, \, >0} ({\boldsymbol{\sigma}})$  is stable under the action of 
 a subgroup $D_{0,n+3}({\boldsymbol{\sigma}})$ conjugate to $D_{0,n+3}$, each 
 $\boldsymbol{AR}(\boldsymbol{\sigma})$ carries 
a structure of  $D_{0,n+3}$-representation (well-defined up to conjugations).  But since there is no a priori natural identification between $\boldsymbol{AR}(\boldsymbol{\sigma})$ and
$\boldsymbol{AR}(\boldsymbol{1})$, it is not clear at all whether the 
 structure of $D_{0,n+3}$-module of $\boldsymbol{AR}(\boldsymbol{1})$ extends to a representation of the whole symmetric group.  To put it briefly,  without any supplementary argument, it is not clear how  to make  $\mathfrak S_{n+3}$ act on  a space of ARs ({\it eg.}\,$\boldsymbol{AR}(\boldsymbol{1})$) in order to get a linear representation of this group. 
\sk

But the situation has some subtleties 
even in the complex setting.   In this case indeed,   the ARs  of $\boldsymbol{\mathcal W}_{0,n+3}$ are sections of a local system on 
$\mathcal M_{0,n+3}$ ({\it cf.}\,page \pageref{Page-3}). But since its  topology  is not trivial and because there is no configurations invariant by all elements of $\mathfrak S_{n+3}$, 
one cannot straightaway define an action of $\mathfrak S_{n+3}$ (induced by pull-backs) on a  vector space  constructed from some  ARs of $\mathcal M_{0,n+3}$. 
 \begin{center}
\vspace{-0.3cm}
 $\star$
 \mk 
 \end{center}
   Note that,  in contrast with the whole space
   of abelian relations (whether in a real or complex setting), the situation is quite simpler 
 if one restricts to    the
  subspace of combinatorial ones. Indeed, since all the components of \eqref{Eq:QuadrilateralWeb} are rational functions, all the combinatorial abelian relations $AR_{i,j}$'s are rational hence 
$\mathfrak S_{n+3}$ naturally acts on   $\boldsymbol{AR}_C(\boldsymbol{\mathcal W}_{0,n+3})$. One verifies easily that this representation is defined over $\mathbf Q$ (but we will not use this arithmetic fact in the sequel). 
\sk

In the complex setting, a possibly bigger $\mathfrak S_{n+3}$-representation (actually obtained from the one of the previous paragraph by successive extensions), may be obtained by considering the local system of ARs of $\boldsymbol{\mathcal W}_{0,n+3}$ with unipotent monodromy, denoted by $\boldsymbol{{AR}}_{U}(\boldsymbol{\mathcal W}_{0,n+3})$.  The latter comes with an increasing filtration $ \boldsymbol{{AR}}_{U}^\bullet(\boldsymbol{\mathcal W}_{0,n+3})$ 
whose $0$-th piece is the vector space 
$\boldsymbol{AR}_C(\boldsymbol{\mathcal W}_{0,n+3})$  and where it is 
not difficult  to figure how are defined the pieces of higher degree.\footnote{Very summarily, for any $k\geq 1$, the stalk  $ \boldsymbol{{AR}}_{U,\boldsymbol{z}}^{\, \, \leq k}(\boldsymbol{\mathcal W}_{0,n+3})$ of the $k$-th piece of this filtration at any configuration $\boldsymbol{z}\in \mathcal M_{0,n+3}$ is the space of (germs of) 
abelian relations 
$\boldsymbol{\Omega}$ 
at this point whose analytic continuation  along any loop $\gamma$ based at $\boldsymbol{z}$, denoted by  $\boldsymbol{\Omega}^\gamma$, is such that 
$\boldsymbol{\Omega}^\gamma- \boldsymbol{\Omega}\in  \boldsymbol{{AR}}_{U,\boldsymbol{z}}^{\,   \leq k-1}(\boldsymbol{\mathcal W}_{0,n+3})$. Details are left to the reader.} Then one verifies easily that the associated graded space ${\bf Gr}^k 
\boldsymbol{{AR}}_{U}^\bullet(\boldsymbol{\mathcal W}_{0,n+3})$ is a  $\mathfrak S_{n+3}$-module for any $k\geq 0$ hence one obtains a representation 
\begin{equation}
\label{Eq:Gr}
{\bf Gr}
\boldsymbol{{AR}}_{U}^\bullet(\boldsymbol{\mathcal W}_{0,n+3})
= \oplus_{k\geq 0}  {\bf Gr}^k 
\boldsymbol{{AR}}_{U}^\bullet(\boldsymbol{\mathcal W}_{0,n+3})
\end{equation}
which is a priori bigger than the representation $\boldsymbol{AR}_C(\boldsymbol{\mathcal W}_{0,n+3})$ discussed in the preceding paragraph since obviously the latter coincides with the  $0$-th piece of \eqref{Eq:Gr}. 

\begin{exm} 
\label{Ex:Gr-n=2}
Let us again consider the web $\boldsymbol{\mathcal W}_{0,5}$, 
but this time within the complex setting (compare with Example \ref{Ex:n=2-on-IR} just above).  In this case, up to the natural identification of 
$\boldsymbol{\mathcal W}_{0,5}$ with Bol's web $\boldsymbol{\mathcal B}$ already mentioned above, 
$\boldsymbol{{AR}}_{C}(\boldsymbol{\mathcal W}_{0,5})$ corresponds to 
what was noted by $\boldsymbol{{AR}}_{Log}(\boldsymbol{\mathcal B})$ in the Introduction and the filtration $\boldsymbol{{AR}}_{U}^\bullet(\boldsymbol{\mathcal W}_{0,5})$ has only two pieces, with $\boldsymbol{{AR}}_{U}^{\, \leq 1}(\boldsymbol{\mathcal W}_{0,5})$ being obtained by adjoining the AR corresponding to the Abel's identity $(\boldsymbol{{\mathcal A}b})$ to the combinatorial ARs. 
\sk

\vspace{-0.35cm}
In the case under scrutiny, we then have $
{\bf Gr}
\boldsymbol{{AR}}_{U}^\bullet(\boldsymbol{\mathcal W}_{0,5})=
\boldsymbol{{AR}}_{C}(\boldsymbol{\mathcal W}_{0,5})\oplus \langle \boldsymbol{{\mathcal A}b} \rangle $ which shows that the decomposition in direct sum  
$ \boldsymbol{{AR}}(\boldsymbol{\mathcal B})=
\boldsymbol{{AR}}_{Log}(\boldsymbol{\mathcal B})\oplus \langle \boldsymbol{{\mathcal A}b} \rangle $
mentioned in  \eqref{Page:Properties-Bol's-Web} can actually be interpreted as a decomposition of $\mathfrak S_5$-representations.
\end{exm}

\subsection{Room-Burau's linear systems and consequences regarding the linearizability of the ${\mathcal W}_{0,n+3}$'s.}
\label{SS:RoomBurau}
Here we recall some nice constructions/results of classical algebraic geometry.
 Then we explain some rather immediate 
consequences of them for the webs ${\mathcal W}_{0,n+3}$'s, in what concerns their realizability as webs by rational curves of low degree.

Albeit different in certain aspects, the cases  when $n$ is even or $n$ is odd share some similarities as well, henceforth both  will be discussed alongside. As before, the material below already appeared in several references, hence no proof is given. Classical references are: the paper by Room \cite{Room1934}, the content of which has been generalized in arbitrary dimension in his book \cite{RoomBook} and Burau's article \cite{Burau}.  For  much more recent papers (without any reference to Room's or Burau's classical works), see \cite{Kumar2003} and 
\cite{BM}.

\subsubsection{}
\label{SSS:Ln-varphin}
  We continue to use the notation introduced before:  $p_1,\ldots,p_{n+2}$ are $n+2$ points in general position in $\mathbf P^n$ with $n\geq 2$, 
   etc.  
Our goal is to recall the definition of a linear system $\lvert \boldsymbol{\mathcal L}_n\lvert$ on $\mathbf P^n$, first considered by Room and Burau for $n$ arbitrary, the associated rational map of which will allow to give a nice projective model for ${\mathcal W}_{0,n+3}$. 
\sk

We set 
$P= \sum_{i=1}^{n+2} p_i $ (as a $0$-cycle of degree $n+2$ on $\mathbf P^n$). The linear system $ \boldsymbol{\mathcal L}_n$  we are interested in is defined as follows, according to the parity of $n$: one has
\begin{align*}
 \boldsymbol{\mathcal L}_n =& \bigg\lvert 
\, d_n H-\nu_n
 P  \, 
\bigg\lvert=\left\{ 
\begin{tabular}{l}
hypersurfaces of degree $d_n$ in $\mathbf P^n$ vanishing\\
 at the order $\nu_n$ at $p_i$ for all $i=1,\ldots,n+2$
\end{tabular}
\right\}\, \subset \big\lvert 
\mathcal O_{\mathbf P^n}\big(d_n\big)
\big\lvert \, , 
\end{align*}
with 
$$
\big(\,  d_n\, , \, \nu_n\, \big)=\begin{cases}\,\,
\big(n+1,n-1\big) \hspace{1.6cm} \mbox{if } n \mbox{ is even}\, ;\vspace{0.15cm}\\
\,\, \Big(  {(n+1)}/{2}  \, ,  {(n-1)}/{2}\Big)\hspace{0.3cm} \mbox{if } n \mbox{ is odd}\, .
\end{cases}
$$


We denote by $N_n$ (or just by $N$ when there is no ambiguity) the projective dimension of $\lvert \boldsymbol{\mathcal L}_n\lvert$ 
 and by 
$\varphi_n$ the associated rational map: 
$$\varphi_n = \varphi_{ \boldsymbol{\mathcal L}_n}: \, \mathbf P^n \dashrightarrow \mathbf P^N \simeq \big\lvert \mathcal L_n\big\lvert^\vee\, . $$
Finally, we set 
$$V_n={\rm Im}\big(\varphi_n\big) =\varphi_n\big(\mathbf P^n\big)\subset \mathbf P^N\, .$$
 It is a non degenerate projective variety in $\mathbf P^N$ which satisfies several nice properties, that we have all gathered in the following result 
(recall that $\delta_n$ is 1 if $n$ is odd, and  $2$ if $n$ is even).


\begin{prop} 
\label{Pr:Vn}
The following assertions hold true for any  $n\geq 2$. 
\begin{enumerate}
\item[{\bf 1.}]  The map $\varphi_n: \mathbf P^n \dashrightarrow V_n$ is birational  and its image $V_n$ is isomorphic to $(\mathbf P^1)^{n+3}\hspace{-0.06cm}/\hspace{-0.06cm}/{\rm PGL}_2(\mathbf C)$ (the GIT quotient  of $n+3$ copies of
 the projective line
  by the diagonal action of ${\rm PGL}_2(\mathbf C)${\rm )}.
\mk 
\vspace{-0.4cm}
\item[{\bf 2.}]  For any foliation $\mathcal F$ of $\boldsymbol{\mathcal W}_{0,n+3}$,  its image $(\varphi_n)_*(\mathcal F) $ by $\varphi_n$ is a curvilinear foliation on $V_n$, whose general leaf is an irreducible rational curve of degree $\delta_n$. These are exactly the covering families of $V_n$ by lines (for $n$ odd) or by conics (for $n$ even). 
\mk 
\item[{\bf 3.}]  For any $n\geq 2$, the automorphism group 
${\rm Aut}(V_n)$ of $V_n$ identifies with $\mathfrak S_{n+3}$.
\mk 
\item[{\bf 4.}] When $n$ is odd,  one has 
$$N=N_n={n+1 \choose (n+1)/2}-{n+1\choose (n+1)/2-2}-1\, $$ 
and  the following statements are satisfied: 
\begin{itemize}
\item[$(i).$] $V_n$ is singular, with singular set consisting of ${ n+2 \choose (n+1)/2}$ isolated singular points.
\item[$(ii).$] 
${\rm Pic}(V_n)$ is free of rank 1. 
Moreover, $K_{V_n}\simeq \mathcal O_{V_n}(-1)$ hence $V_n$ is a Fano variety.
\end{itemize}
\mk 
\item[{\bf 5.}] When $n$ is even one has 
\begin{equation}
\label{N:n-even}
N=N_n=\sum_{s=0}^{ \lfloor (n+1)/3\rfloor}
(-1)^s {n+2\choose s }{2n-3s+1  \choose n }-1
\end{equation}
and  the following statements are satisfied:
\begin{itemize}
\item[$(i).$] $V_n$ is smooth; 
\item[$(ii).$] one has ${\rm Pic}(V_n)\simeq \mathbf Z^{n+3}$. Moreover $K_{V_n}\simeq \mathcal O_{V_n}(-2)$ hence $V_n$ is a Fano manifold.
\end{itemize} 
\bk
\end{enumerate}
\end{prop}

The results above are quite nice generalizations for $n$ even and odd respectively, of some  well-known properties of  the del Pezzo quintic surface in $\mathbf P^5$ (case $n=2$) and of Segre's cubic hypersurface $\boldsymbol{S}$ in $\mathbf P^4$ (case $n=3$, discussed in more details \S\ref{SS:Segre's-cubic} further).   
If we will only use the first two points of this proposition in this article, we have mentioned the others  to show that the $V_n$'s form an interesting family of Fano varieties generalizing very classical examples, 
 which should be studied further in our opinion. 
 Focusing especially on the case $n=5$, we will briefly come back  on this  at the end in \S\ref{SS:Case-n=5}.\sk

 In order to help the reader who might be interested in the statements above, 
let us indicate quickly where these can be found in the existing (classical or more recent) literature.  First, the  case $n=5$  is studied by Room  in \cite{Room1934} by classical projective methods. For $n$ arbitrary, the fact 
that $\varphi_n$ is birational onto its image is not explicitly stated in \cite{RoomBook} or \cite{Burau} but we think that it is because this fact was more or less obvious to these authors. 
Another classical reference for the case when $n$ is odd is \cite{Coble}.
For a modern treatment, see Theorem 3.1 (for $n$ odd) and Theorem 3.4 (for all cases) in \cite{Kumar2003}.  Point {\it 2.} is rather immediate once it is known that $\varphi_n$ is birational onto its image. 
Regarding the determination of the dimension $N$ as a function of $n$, see \cite[p.\,363]{RoomBook}, \cite[\S4.a]{Burau} and \cite[\S3.1]{Kumar2003} for $n$ odd. 
A formula for $N$ in the case when $n$ is even is a bit more involved and can be found in \cite[\S4.b]{Burau} or in  \cite[\S3.3]{Kumar2003}.
 The other mentioned statements, namely that $V_n$ coincides with $(\mathbf P^1)^{n+3}\hspace{-0.06cm}/\hspace{-0.06cm}/{\rm PGL}_2(\mathbf C)$, that ${\rm Aut}(V_n)$ is naturally isomorphic to $ \mathfrak S_{n+3}$, or the description of the singularities, of the Picard group and of the  canonical class of $V_n$,  all have been obtained quite recently,  in \cite{BM}.

\subsubsection{}
\label{SSS:linearizability}
From the first two points of Proposition \ref{Pr:Vn}, one imediately deduces the 
\begin{cor} 
\label{Cor:Linearizable}
{\bf 1.} For any $n\geq 2$, the push-forward $\boldsymbol{W}_{0,n+3}=(\varphi_n)_*\big(\boldsymbol{\mathcal W}_{0,n+3}\big)$  is a web on $V_n$ which is  equivalent to $\boldsymbol{\mathcal W}_{0,n+3}$ and  whose leaves are rational curves of degree $\delta_n$.\sk 
 
{\bf 2.} In particular, $\boldsymbol{\mathcal W}_{0,n+3}$ is linearizable when $n$ is odd.
\end{cor}

The second point contradicts a result claimed by Damiano, namely that 
$\boldsymbol{\mathcal W}_{0,n+3}$ is not linearizable for any $n\geq 2$,  
which is one of the points stated in  
 Proposition 1.6 of \cite{D}.  
 But not even an indication of proof appears in the whole paper,  except in the classical case $n=2$  for which  standard references are given (namely \S29, \S30 and \S31 of \cite{BB}).  More details about the reasoning leading Damiano to state this 
can be found in his thesis. Since realizing that $\boldsymbol{\mathcal W}_{0,6}$ actually is linearizable has been the departure point of the work leading to the present paper, we believe it is worth explaining precisely where Damiano's reasoning is not correct  according to our understanding.\sk 

Actually, only a few lines are devoted in 
 \cite{DThesis}
to  establish that 
$\boldsymbol{\mathcal W}_{0,n+3}$ is not linearizable for any $n$ and can be found pages 20 and 21 of Damiano's dissertation. His reasoning is by {\it reductio ad absurdum} by means of  an induction on $n$ and  can be decomposed as follows:  
\begin{enumerate}
\item[{\it 1.}]  {\it If $\boldsymbol{\mathcal W}_{0,n+3}$ were linearizable, 
it would admit a model $\boldsymbol{\mathfrak W}_n$ in $\mathbf R^n$ with linear leaves;}
\sk 
\item[{\it 2.}] {\it Then 
the $(n+2)$-web $\boldsymbol{\mathfrak W}_n'$ in $H\simeq \mathbf R^{n-1}$ obtained by 
projecting $\boldsymbol{\mathfrak W}_n$ onto a hyperplane $H$ transverse to one of the foliations (and then disregarding this foliation)  would also be linearizable;}\sk 
\item[{\it 3.}] {\it But $\boldsymbol{\mathfrak W}_n'$ is equivalent to $\boldsymbol{\mathcal W}_{0,n+2}$ (case $n-1$);}
\sk
\item[{\it 4.}] {\it Hence assuming that $\boldsymbol{\mathcal W}_{0,n+3}$ is linearizable would allow to deduce by induction that $\boldsymbol{\mathcal W}_{0,5}$ is linearizable as well, but this is not the case.}\mk
\end{enumerate}

%
%
%

The idea for arguing by recurrence on $n$ (step {\it 3.}) is justified by the following fact which is easy to prove (details are left to the reader): if  $\pi_{{\mathcal F}_i}$ stands for the map corresponding to quotienting by the $i$-th foliation $\mathcal F_i$ of $\boldsymbol{\mathcal W}_{0,n+3}$, then $(\pi_{{\mathcal F}_i})_*(\boldsymbol{\mathcal W}_{0,n+3})=\big( 
(\pi_{{\mathcal F}_i})_*({\mathcal F}_j)
\big)_{j\neq i}$  is a $(n+2)$-web in dimension $n-1$ which is equivalent to 
$\boldsymbol{\mathcal W}_{0,n+2}$. The problem in the reasoning above lies in the second step: the projection onto a hyperplane $H$ which is considered therein is taken for a model of a  quotient map  with respect to a foliation $\mathfrak F_i$ of the linear 
model $\boldsymbol{\mathfrak W}_n$ of $\boldsymbol{\mathcal W}_{0,n+3}$. But for that to be the case, it is necessary that this foliation be precisely that of lines through a point (possibly at infinity) and it is not obvious at all that $\boldsymbol{\mathcal W}_{0,n+3}$ admits  (even locally) a linear model in $\mathbf R^n$ whose (at least)  one of the foliations enjoys this property. And it turns out that it is not the case for instance when $n=3$ since  $\boldsymbol{\mathcal W}_{0,6}$ is linearizable (according to 
Corollary \ref{Cor:Linearizable}) whereas $\boldsymbol{\mathcal W}_{0,5}$ is not. 
\sk

We believe that what might have been the source of the mistake 
discussed in the previous paragraph is the model \eqref{Eq:Exceptional-web} of $\boldsymbol{\mathcal W}_{0,n+3}$ in $\mathbf P^n$: at the exception of the last one, all its foliations are formed by the lines going through one of the $p_i$'s hence for this model, quotienting by such a foliation (which in this case makes sense even globally and corresponds to the $i$-th forgeful map \eqref{Eq:Forget-map-phi-i}) is indeed given by the linear projection from the corresponding $p_i$ and onto a hyperplane. But of course, this projection does not linearize the last foliation since the generic leaf of the latter is a RNC of degree $n$.  

All our remarks above show that, although it is a very appealing approach,  one has to be careful when trying to investigate the 
webs  $\boldsymbol{\mathcal W}_{0,n+3}$'s inductively by quotienting along one of its foliations.
\mk 

A general idea that emerges from the results of the present article 
is that the $\boldsymbol{\mathcal W}_{0,n+3}$'s 
seem to form two distinct classes of webs according to the parity of $n$, and those 
belonging to the same class seem to share most of their properties. For this reason, the validity of the following conjecture (already mentioned in \cite{Burau}) seems more than plausible: 
%

\begin{conjecture}[Burau]  For any $n\geq 2$ even, the web 
$\boldsymbol{\mathcal W}_{0,n+3}$ is not linearizable.
\end{conjecture}

That $\boldsymbol{\mathcal W}_{0,5}$ is not linearizable is very classical and goes back to the early beginning of web geometry ({\it cf.}\,\cite[p.\,262]{BB}). Since then, no progress has been done regarding this conjecture. We confess having no idea how to handle it, even in the special case $n=4$.

\begin{rem} 
As we finished writing this article, we realized that $\boldsymbol{\mathcal W}_{0,6}$ is linearizable is explicitly mentioned in the booklet on webs  \cite{Blaschke1955} published by Blaschke in 1955, see  the last paragraph of \S54 in it. It is unfortunate that Damiano missed this reference as well as Burau's paper \cite{Burau}.
\end{rem}


\section{\bf A thorough study of the case $n=3$.}
\label{S:n=3}
Our goal in this section is to study in depth the web $\boldsymbol{\mathcal W}_{0,6}$ 
by showing it can be seen as a particular case of a quite large class of interesting curvilinear webs coming from algebraic geometry. \sk

\vspace{-0.4cm}
We first recall below some well-known facts about lines on cubic threefolds and then discuss how these facts are interpreted in terms of webs. The study of the Fano surface of lines in a cubic hypersurface  $X$ of 
$\mathbf P^4$ is  a very nice and nowadays classical part of algebraic geometry when $X$ is smooth.  In subsection \S\ref{SS:Cubic-Hypersurfaces} below, we first recall  some points of this theory  and explain their consequences for web geometry. Most of these results have been extended to the case of singular cubic hypersurfaces, especially nodal cubics.  But as far as we are aware of, only the case of cubics with only a few nodes (typically one node) has been concretely considered in the literature.\footnote{However we would not be surprised 
if  all nodal cases, from zero (smooth) to ten nodes (Segre's cubic) were 
considered as fully understood by experts.}
For this reason, we discuss in detail the case of Segre's cubic (which carries 10 nodes, which is the maximal number of nodes that a hypercubic of $\mathbf P^4$ can have) in a separate subsection, namely in \S\ref{SS:Segre's-cubic}.

\subsection{Webs by lines on smooth cubic hypersurfaces of $\mathbf P^4$}
\label{SS:Cubic-Hypersurfaces}
The theory of lines on smooth cubic threefolds has been  considered in 
many papers. 
In addition of Fano's original paper \cite{Fano}, 
we mention \cite{Gherardelli}, 
\cite{BombieriSwinnerton-Dyer}, 
\cite{ClemensGriffiths},  
\cite{Tyurin}, 
\cite{Murre} 
and 
\cite{AltmanKleiman}. 
Among these, our main references are 
 \cite{ClemensGriffiths} and \cite{AltmanKleiman},  to which the reader may refer for justifications of the assertions below.
\bk

Here $X$ stands for a smooth cubic hypersurface in $\mathbf P^4$. The underlying set of its Fano scheme $F_1(X)$, denoted below just by $F$ to simplify, is 
the subset of the grassmannian variety of lines in the ambiant 
$\mathbf P^4$
 formed by the lines contained in $X$: one has 
$$
F=F_1(X)=\Big\{ \hspace{0.15cm}  \ell \in G_1\big(\mathbf P^4\big) \hspace{0.15cm}  \big\lvert \hspace{0.15cm}  \ell \subset X \hspace{0.15cm}  \Big\}\, . 
$$
It is not difficult to exhibit polynomial equations cutting $F$ in affine charts of $G_1(\mathbf P^4)$, which shows that $F$ is naturally an algebraic subvariety of this grassmannian ({\it cf.}\,\cite[(1.14)]{AltmanKleiman}).  Defining the associated incidence variety as the 
 algebraic subvariety $T
 =\big\{\, \big( x,\ell \big) \in X\times F \, \big\lvert \, x\in \ell\, 
\big\}$ of $X\times F$, both $X$ and $F$ naturally fit into the following incidence diagram: 
\begin{equation}
\label{Eq:Incid}
\xymatrix@R=0.4cm@C=0.4cm{
&  \ar@{->}[dl]_p \ar@{->}[dr]^q T
&   \\
X & & F
}
\end{equation}
where $p$ and $q$ stand   respectively for the restrictions to $T$ of the standard projections of $X\times F$ onto its first and second factors.  It will be useful to use the following notation below: we assume that the 
ambiant projective space is the projectivization of a  5-dimensional vector space $V$: $\mathbf P^4=\mathbf P(V)$.  To simplify, we will write $G$
for the grassmanian variety $G_1(\mathbf P^4)=G_2(V)$.  We denote by $\mathcal T_G$ the rank 2 \href{https://en.wikipedia.org/wiki/Tautological_bundle}{tautological bundle} over $G$ and $\rho: G\hookrightarrow \mathbf P(\wedge ^2 V)=\mathbf P^9$ stands for the Pl\"ucker embedding.  Finally,  $\iota: F\hookrightarrow G$ denotes the natural inclusion of $F$ into the grassmannian and we write 
$\varrho= \rho\circ \iota: F\hookrightarrow \mathbf P^9$ for the corresponding embedding of $F$ into  $\mathbf P^9$.

\subsubsection{}
\label{SSS:PropertiesOfF}
Here is a list of properties of $F$ and of the incidence correspondance \eqref{Eq:Incid},  many of which will have interesting consequences when interpreted in terms of webs: \mk 
\begin{enumerate}
\item[{\bf i.}] A line $\ell \subset X$ can be of two types, according to decomposition of its normal bundle in $X$: up to the identification $\ell\simeq \mathbf P^1$, either $N_{\ell/X}=\mathcal O_{\mathbf P^1}(0)^{\oplus 2}$ and $\ell$ is of the `{\it first type}', or 
$N_{\ell/X}=\mathcal O_{\mathbf P^1}(-1){\oplus}\mathcal O_{\mathbf P^1}(1)$ and 
$\ell$ is of the `{\it second type'} ({\it cf.}\,\cite[Proposition 6.19]{ClemensGriffiths}). In both cases, one has $h^0(N_{\ell/X})=2$ and $h^1(N_{\ell/X})=0$.
\mk
\item[{\bf ii.}] The 
\href{https://en.wikipedia.org/wiki/Fano_surface}{Fano variety}
$F$ is a smooth irreducible surface (see \cite[Lemma 7.7]{ClemensGriffiths}).  Hence $T$ is smooth as well. Moreover $q$ makes of $T$ a $\mathbf P^1$-bundle over $F$ hence  $T$ has dimension 3.\mk
\item[{\bf iii.}]  Through a general point $x$ of $X$ pass exactly 6 pairwise distinct lines 
 included in $X$, which we will denote (with arbitrary labels) by $\ell_1(x),\ldots,\ell_6(x)$ (see \cite[(1.7)]{AltmanKleiman}). In other terms, the map $p:T\rightarrow X$ in 
\eqref{Eq:Incid} 
  is generically 6 to 1.
 \mk
\item[{\bf iv.}]  The embedding $\varrho: F\hookrightarrow \mathbf P^9$ is canonical: 
if  $\mathcal O_F(1)=\varrho^*\big( 
\mathcal O_{\mathbf P^9}(1)\big)$ stands for the corresponding  line bundle, 
one has an isomorphism: 
$$
K_F=\Omega^2_F \simeq \mathcal O_F(1)
$$
 (see \cite[(1.8)]{AltmanKleiman}). 
 Moreover, $\iota(F)$ is non degenerate in $\mathbf P^9$ ({\it i.e.}\,no hyperplane of 
 $\mathbf P^9$ contains 
 $\iota(F)$, see \cite[Lemma 10.2]{ClemensGriffiths}) hence 
 there is a natural identification of vector spaces
\begin{equation}
H^0\big( F,\Omega_F^2\big)\simeq H^0\big( \mathbf P^9,\mathcal O_{\mathbf P^9}(1)\big)=\big(\wedge^2 V\,\big)^\vee\, 
\end{equation}
from which it follows that  $h^0\big(\Omega_F^2\big)
=10$ (see \cite[Lemma 10.2]{ClemensGriffiths}). \mk
\label{lala}
\item[{\bf v.}]  The pull-back under $\iota$ 
  of the dual of the tautological bundle 
$\mathcal T=\mathcal T_G$ 
 identifies with $\Omega_F^1$: 
\begin{equation}
\label{Eq:TBT}
\Omega^1_F= \iota^*\big(  \mathcal T^\vee \big)
\end{equation}
 (see \cite[(1.10)-(ii)]{AltmanKleiman}).\footnote{The isomorphism 
 \eqref {Eq:TBT} is known as the {\it `tangent bundle theorem'}. It is considered as one of the deepest results about the hypercubic $X$ and its Fano variety $F$ and has been obtained by several authors: originally by Fano (in a rough form). Modern proofs have been given independently by Clemens-Griffiths (see \cite[Corollary 12.38]{ClemensGriffiths}) and Tyurin.} Moreover,  the restriction map $H^0\big(G,\mathcal T^\vee\big)\rightarrow H^0\big(F, \mathcal T^\vee\lvert_F\big)$ induces an isomorphism 
 \begin{equation*}
V^\vee=H^0\big(G,\mathcal T^\vee\big)\simeq H^0\big(F, \Omega_F^1\big)\, .
\end{equation*}
In particular, one has 
$h^0\big(F,\Omega^1_F\big)=5$.\sk
\mk
\item[{\bf vi.}]  It follows that the \href{https://encyclopediaofmath.org/wiki/Albanese_variety}{Albanese variety} ${\rm Alb}(F)= H^0(F,\Omega^1_F)\big)^\vee / H_1(F, \mathbf Z)
$ is of dimension 5. Moreover (for any base point $x_0\in F$,) the Albanese map 
$a=a_{x_0} \, :\, F \rightarrow {\rm Alb}(F)$, $x\longmapsto \big(  a^x=a_{x_0}^x: H^0(F,\Omega^1_F)\big)\longrightarrow \mathbf C,\, \omega\mapsto \int_{x_0}^x \omega \big)$ 
is an embedding.
\mk 
\item[{\bf vii.}]  The cup-product mapping $H^0\big(F,\Omega_F^1\big)
\wedge 
H^0\big(F,\Omega_F^1\big)\rightarrow H^0\big(F,\Omega_F^2\big)$ is injective  (\cite[Lemma 9.13]{ClemensGriffiths}) hence it is an isomorphism (for dimensional reasons). 
\sk
\item[{\bf viii.}] Actually,  for $k=1,2$, the Albanese mapping 
$a :\, F \rightarrow {\rm Alb}(F)$
induces an  isomorphism 
\begin{equation*}
\label{Eq:a-Isom}
a^* :\, H^0\Big( {\rm Alb}(F), \Omega^k _{ {\rm Alb}(F)}\Big) \simeq 
H^0\big(F, 
\Omega^1_F\big)\, . 
\end{equation*}
\item[{\bf ix.}]  A line $\ell\subset X$ is of the second type if and only if there exists a 2-plane $P\subset \mathbf P^4$ such that $X\cdot P=2\ell+\ell'$ as 1-cycles on $X$, for another line $\ell'\subset X$ ({\it cf.}\,\cite[Lemma (1.14)]{Murre}).  The  lines $\ell \subset X$ of second type  form a non-singular subset $C$ of  $F$, of dimension 1 (see \cite[Corollary (1.9)]{Murre}). 
\mk 
\item[{\bf x.}]   The set $\mathcal E$ of `Eckardt points' on $X$, that is points $x\in X$ such that  $p^{-1}(x)=\big \{\, \ell\in F\, \lvert \, x\in \ell\,\big\}$ has positive dimension, is finite ({\it cf.}\,\cite[Lemma 8.1]{ClemensGriffiths}).\footnote{Actually, it is known that ${\rm Card}(\mathcal E)\leq 30$ and that  $X$ is the Fermat cubic threefold in case of equality.}
\sk
\item[{\bf xi.}]  One sets $X'=X\setminus \mathcal E$ and $T'=p^{-1}(X')$. 
Then  the restriction $p': T'\rightarrow X'$ of $p$ over $X'$  
is  unramified outside $S'=q^{-1}(C)\cap T'$ and 
has simple ramification along a Zariski open subset of each irreducible component of $S'$ 
 (see \cite[Lemma 10.18]{ClemensGriffiths}).
\sk 
\item[{\bf xii.}]  The map $\mathscr T: X'\rightarrow {\rm Alb}(F)$, $x\mapsto \sum_{i=1} a\big( \ell_i(x)\big)$ 
is  well-defined and holomorphic.  Since $\mathcal E$ is finite, it extends holomorphically to the whole $X$. Since the latter is simply-connected, it follows that the map  $\mathscr T$ is  constant 
({\it cf.}\,(11.10) in \cite{ClemensGriffiths}). 
\sk 
\item[{\bf xiii.}]  For $k=1,2$, there is a global trace map 
$${\rm Tr}^{(k)}=p_*\circ  q^* : \Omega^k_{\mathbf C(F)}\rightarrow  \Omega^k_{\mathbf C(X)}
\footnotemark
\footnotetext{Here for any algebraic variety $Z$ and any $k\geq 1$, $\Omega^k_{\mathbf C(Z)}$ stands for the $\mathbf C(Z)$-module of global rational $k$-forms on $Z$.}
$$ which locally (on a neighborhood of any $x\in X'$) is written ${\rm Tr}^{(k)}(\omega)=\sum_{i=1}^6 \ell_i^*(\omega)$. 
\mk 
\item[{\bf xiv.}]
The map $\mathscr T: X'\rightarrow {\rm Alb}(F)$, $x\mapsto \sum_{i=1} a\big( \ell_i(x)\big)$ 
is  well-defined and holomorphic.  Since $\mathcal E$ is finite, it extends holomorphically to the whole $X$. Since the latter is simply-connected, it follows that the map  $\mathscr T$ is  constant 
({\it cf.}\,(11.10) in \cite{ClemensGriffiths}). 
\mk 
\label{lolo}
\item[{\bf xv.}] For $k=1,2$ and for any global holomorphic $k$-differential $\omega\in H^0\big(F,\Omega_F^k\big)$, 
the $k$-form ${\rm Tr}^{(k)}(\omega)$ (which can be written ${\rm Tr}^{(k)}(\omega)=\sum_{i=1}^6 \ell_i^*(\omega)$ locally) vanishes identically on $X$. 
\mk 
\item[{\bf xvi.}]  Let $\overline{\Delta}_X$ be the variety of triangles in $X$, namely the 3-dimensional subvariety of $F^3$ 
formed by the 3-tuples  of pairwise distinct lines $(l_1,l_2,l_3)$ such that $l_1+l_2+l_3=X\cdot P$ (equality between 1-cycles) for a certain 2-plane $P\subset \mathbf P^4$.  Then 
  there 
exists $\alpha_1\in {\rm Alb}(F)$ such that $a(l_1)+a(l_2)+a(l_3)=\alpha_1$ for any generic triangle $(l_1,l_2,l_3) \in \overline{\Delta}_X$  ({\it cf.}\,\cite[(11.9)]{ClemensGriffiths}).  
\mk 
\end{enumerate}

From {\bf xi.}, it follows that $X^*=X'\setminus p(S')$ is a Zariski open set of $X$ (it is the complement of a finite union of surfaces contained in $X$) such that through any $x^*\in X^*$, pass  exactly six pairwise distinct lines included in $X$. Consequently, the linear 6-web 
$\boldsymbol{\mathcal L\hspace{-0.05cm}\mathcal W}_X$ defined by the lines included in $X$ is 
well-defined on  $X^*$.  Now, it is straightforward to state in web-theoretic terms for the web $\boldsymbol{\mathcal L\hspace{-0.05cm}\mathcal W}_X$, some of the sixteen properties {\bf i.}\,to {\bf xvi.}\,stated just above. We get the following 
\begin{prop}
\label{P:LW-X}
\begin{enumerate}
\item[] 
${}^{}$\hspace{-1cm}{\bf 1.}
 $\boldsymbol{\mathcal L\hspace{-0.05cm}\mathcal W}_X$ is an irreducible 6-web which is skew and  linearizable.
\mk 
\item[{\bf 2.}]  
 For $k=1,2$ and $\omega\in {H}^k\big(  F , \Omega^k_F\big)$, one has 
${\rm Tr}^{(k)}(\omega)=0$ hence the linear map  
\begin{align*}
{\rm Tr}^{(k)} \, : \, 
{H}^0\big(  F , \Omega^k_F\big) & 
\longrightarrow \boldsymbol{AR}^{(k)}\big( \boldsymbol{\mathcal L\hspace{-0.05cm}\mathcal W}_X\big)\\
 \omega &\longmapsto \Big( \ell_i^*\big(\omega\big)\Big)_{i=1}^6
\end{align*}
is well-defined and injective.  \mk 
\item[{\bf 3.}]  Consequently,  the following statements hold true: \mk

\begin{enumerate}
\item[{\bf a.}] The second trace map ${\rm Tr}^{(2)}$ induces an isomorphism 
$\boldsymbol{H}^0\big(  F , \Omega^2_F\big)\simeq \boldsymbol{AR}^{(2)}\big( \boldsymbol{\mathcal L\hspace{-0.05cm}\mathcal W}_X\big)$.
These spaces are of dimension 10 hence $\boldsymbol{\mathcal L\hspace{-0.05cm}\mathcal W}_X$ has maximal 2-rank;
\mk  
\item[{\bf b.}]  The image $\boldsymbol{AR}^{(1)}_{ab}\big( \boldsymbol{\mathcal L\hspace{-0.05cm}\mathcal W}_X\big)={\rm Im}({\rm Tr}^{(1)})$ of 
the first  trace map
is a 5-dimensional subspace of 
$\boldsymbol{AR}^{(1)}\big( \boldsymbol{\mathcal L\hspace{-0.05cm}\mathcal W}_X\big)$,
which is such that the wedge map 
\begin{align*}\wedge^2   \boldsymbol{AR}^{(1)}_{ab}\big( \boldsymbol{\mathcal L\hspace{-0.05cm}\mathcal W}_X\big)& \longrightarrow \boldsymbol{AR}^{(2)}\big( \boldsymbol{\mathcal L\hspace{-0.05cm}\mathcal W}_X\big)\\ 
\big(\omega_i \big)_{i=1}^6 \wedge  \big(\varpi_i \big)_{i=1}^6\hspace{0.1cm}   & \longmapsto \hspace{0.2cm}  \big(\omega_i \wedge \varpi_i \big)_{i=1}^6
\end{align*}
 is well-defined and is an isomorphism. 
\end{enumerate}
\end{enumerate}
\end{prop}
\begin{proof} 
The first point of {\bf 1} is rather obvious: the irreducibility of 
$\boldsymbol{\mathcal L\hspace{-0.05cm}\mathcal W}_X$ as a web follows immediately from that of $F$ as an algebraic surface. Regarding the linearizability of $\boldsymbol{\mathcal L\hspace{-0.05cm}\mathcal W}_X$, 
it suffices to consider its local image in $\mathbf P^3$ by a generic linear projection from a point of $\mathbf P^4$.   As for the 
skewness of $\boldsymbol{\mathcal L\hspace{-0.05cm}\mathcal W}_X$, 
it is also easy to establish: assuming that this web is not skew would imply that 
through a general point of $X$ passes a doubly-ruled surface $\Sigma $ included in $X$. Since 
$\Sigma$ is not a plane (otherwise the set $\mathcal E$ of Eckardt points of $X$ would be infinite), it has to be a quadric surface spanning a hyperplane. But then $ \langle \Sigma\rangle \cap X=\Sigma\cup P$  for a 2-plane $P\subset X$ residual to $\Sigma$, which is not possible since, again,  $\mathcal E$ is known to be finite.\footnote{A more direct way to conclude may have been to recall that a 
 cubic containing a 2-plane is necessariky singular (see the paragraph just before Proposition 2.2 in \cite{DolgachevSegre}.}
\sk 

To end the proof of {\bf 1}, it remains to show that the lines  included in $X$ passing through a general point $x$, denoted by $\ell_1(x),\ldots,\ell_6(x)$,  are in general position. To do so, we recall that the tangent directions at $x$ of the 
$\ell_i(x)$'s all six belong to the second fundamental form $II_{X,x}$ of $X$ at the considered point, that 
we see here as a conic in the projectivized tangent space $\mathbf P(T_xX)\simeq \mathbf P^2$. 
Then assume  that $x\in X$ is generic and such that there are three pairwise distinct lines in $X$ through $x$ which are coplanar.  Then their associated tangent directions  at this point span a line in $\mathbf P(T_xX)$, denoted by $L_x$.  Because  $L_x \cap II_{X,x}$ has cardinality (at least) three, it follows that $L_x  \subset II_{X,x}$, which gives us that the conic 
$II_{X,x}$ 
is not reducible hence has rank $\leq 2$. But this would contradict (in the case when $n=3$) the well-known fact that, 
given a $n$-dimensional smooth complex projective hypersurface with $n\geq 2$, its second fundamental form at a generic point, viewed as a quadric in $\mathbf P^{n-1}$,    has maximal rank equal to 
$n$.\footnote{It is worth mentioning that this `well-known fact' 
 is rather subtle and in any case less obvious than it seems at first sight.
 For instance, all the considered hypotheses are necessary: there exist complex affine 
or real projective 
hypersurfaces which are smooth but with a degenerate second fundamental form at any point.}   It follows that $\boldsymbol{\mathcal L\hspace{-0.05cm}\mathcal W}_X$ satisfies the general position assumption of  \S\ref{SS:webs} which finishes the proof of {\bf 1}.
\sk

The second and third points of the proposition are nothing else but reinterpretations in terms of  $\boldsymbol{\mathcal L\hspace{-0.05cm}\mathcal W}_X$ of some the points  stated above (mainly the points 
{\bf iv}, {\bf vi}, {\bf vii} and  {\bf xv}).
\end{proof}
Let $\boldsymbol{\mathcal L\hspace{-0.05cm}\mathcal W}_{X,x}$ be the germ of 
$\boldsymbol{\mathcal L\hspace{-0.05cm}\mathcal W}_{X}$ at a generic point $x\in X^*$.  
Using linear projections centered at sufficiently generic points of  $\mathbf P^4$, one gets a five dimensional family of linear models for $\boldsymbol{\mathcal L\hspace{-0.05cm}\mathcal W}_{X,x}$.  It seems more than likely that two such local linear models are generically not projectively equivalent  hence that this (germ of) web admits a 4-dimensional family of non projectively equivalent linearizations. A rigourous proof that it is indeed the case  would be welcome. \sk

Another property of $\boldsymbol{\mathcal L\hspace{-0.05cm}\mathcal W}_{X}$ which is not completely clear yet concerns its 1-rank: from {\bf 3.a.}\,in the above proposition, we know that ${\rm rk}^{(1)}\big( \boldsymbol{\mathcal L\hspace{-0.05cm}\mathcal W}_{X}\big)\geq h^0(\Omega_F^1)=5$. Is this majoration actually an equality? In other terms, 
does the space $\boldsymbol{AR}^{(1)}_{ab}\big( \boldsymbol{\mathcal L\hspace{-0.05cm}\mathcal W}_X\big)$ 
 coincide with the whole space 
$\boldsymbol{AR}^{(1)}\big( \boldsymbol{\mathcal L\hspace{-0.05cm}\mathcal W}_X\big)$ 
of 1-ARs of the web $ \boldsymbol{\mathcal L\hspace{-0.05cm}\mathcal W}_X$? 
We believe that the answer is affirmative (still assuming that $X$ is smooth) but we don't have any argument to offer  justifying this. 
\begin{center}
$\star$
\end{center}

In view of Proposition \ref{P:LW-X} and considering to which extent it is very similar to the very classical corresponding statement  holding true for classical `algebraic webs', we believe it is justified to say that  $
 \boldsymbol{\mathcal L\hspace{-0.05cm}\mathcal W}_X$ is `algebraic as well' (see also 
 our discussion in \S1.3 of \cite{EAW}). \sk 
 
 Since the web $\boldsymbol{\mathcal W}_{0,6}$ is (birationally) equivalent to the web 
$
 \boldsymbol{\mathcal L\hspace{-0.05cm}\mathcal W}_{\boldsymbol{\mathcal S}}$ formed by the six covering families of lines in Segre's cubic $\boldsymbol{\mathcal S}$ and because the latter can be seen as a degeneration of smooth cubic hypersurfaces, it is natural to wonder 
 whether it is justified to say that it is an algebraic web as well.  If this is obvious in what concerns its geometric definition (see \S\ref{SS:Segre's-cubic} just below) this is not as clear with regard to the ARs of $  \boldsymbol{\mathcal L\hspace{-0.05cm}\mathcal W}_{\boldsymbol{\mathcal S}}$ since it is not evident that  the statements in \S\ref{SSS:PropertiesOfF} used to obtain Proposition \ref{P:LW-X} admit counterparts for a singular cubic.  Indeed, some of these properties are not satisfied by singular cubics, as the first example considered in the next subsection shows.

\subsection{The case of (some) singular cubics.}
\label{SS:je-sais-pas-quoi}
 We first discuss the case of a special  cubic threefold  with a singular set of dimension 1. Then 
  we say a few words about those cubics with only nodal singularities. We will focus 
   on the case of Segre's cubic just after in \S\ref{SS:Segre's-cubic}. \mk

\subsubsection{\bf The chordal cubic.}
\label{SS:je-sais-pas-quoi}
Let  $[Z_0 : \cdots  : Z_4]$ stand for fixed homogeneous coordinates  on $\mathbf P^4$ 
 and denote by $[\cdot]: \mathbf C^5\setminus \{0\}\rightarrow \mathbf P^4$ the projectivisation map. We fix a rational normal quartic curve $\Gamma\subset \mathbf P^4$. Then the 2-secants to $\Gamma$  (which are the lines intersecting $\Gamma$ in two points counted with multiplicity) fill out a 
cubic hypersurface $\mathscr C\subset \mathbf P^4$, known as the {\it `chordal cubic threefold'} (see \cite{Segre,Collino1}{\rm )}. Taking the Zariski closure of  the projectivisation of the image of $
\nu_4: \mathbf C\ni t\mapsto (1,t,t^2,t^3,t^4)\in \mathbf C^5$ for the quartic curve $\Gamma$, one gets that $\mathscr C$ is cut out 
by the following cubic equation
\begin{equation}
\label{Eq:Chordal-cubic-equation}
\det \begin{bmatrix}
Z_0 & Z_1 & Z_2 \\
Z_1 & Z_2 & Z_3 \\
Z_2 & Z_3 & Z_4 \\
\end{bmatrix}  = Z_0\big( Z_2Z_4-Z_3^2\big) -
\big( Z_2^3 -2Z_1Z_2Z_3-Z_2^3+Z_1^2Z_4\big)= 0\, .
\end{equation}

The hypersurface $\mathscr C$ is singular along $\Gamma$ and is the unique cubic threefold of $\mathbf P^4$ singular along a rational quartic curve (up to projective equivalence). 
  Note also that $\mathscr C$  is semistable and that it appears as a special degenerate point in the GIT moduli space of cubic threefolds (see \cite{Allcock}).  The chordal cubic is discussed in the paragraphs  \S44 and \S45 of C. Segre's memoir \cite{Segre} and in the first part of the recent paper \cite{Collino1}, two references the reader interested by more details is invited to consult.  
 \mk

It is known 
that the Fano surface $F=F_1(\mathscr C)$ of lines included in $\mathscr C$ has two components denoted by $F'$ and $F''$ : the former $F'$ is the family of lines which are 2-secant with $\Gamma$ and the latter $F''$ is the set of lines $\ell\subset \mathbf P^4$ such that the restriction to $\Gamma$ of the linear projection $\pi_\ell : \mathbf P^4\dashrightarrow \mathbf P^2$ from $\ell$ has a ramification divisor of degree 2 on $\Gamma$ (or equivalently, $\pi_\ell(\Gamma)$ is a conic).  Both $F'$ and $F''$ are isomorphic to the symmetric square of $\Gamma$ hence both are isomorphic to $
\Gamma^{[2]}\simeq
(\mathbf P^1)^{[2]}\simeq \mathbf P^2$.\footnote{For any set $A$, $A^{[2]}$ stands for its symmetric product, that is  $A^{[2]}=(A\times A)/\mathfrak S_2$.} The two components $F'$ and $F''$  intersect along the smooth rational curve ${ \mathcal K}\subset G_1(\mathbf P^4)$ whose points are the lines tangent to   $\Gamma$ (${ \mathcal K}$ is the image of the Gauss map 
$\Gamma\rightarrow G_1(\mathbf P^4)$, $\gamma \mapsto T_\gamma \,\Gamma$).  Of course, one has 
${ \mathcal K} \simeq { \Gamma} \simeq \mathbf P^1$.
\sk 

As a scheme, it can be verified that $F_1(\mathscr C)$ contains $F'$ and $F''$ with multiplicity 4 and 1 respectively ({\it cf.}\,the bottom of p.\,214 in \cite{Collino1}{\rm )} from which it follows that through a general point  $c$ of $ \mathscr C$ pass exactly three pairwise distinct lines included in $ \mathscr C$: two lines  $\ell_1''(c)$ and $\ell_2''(c)$ corresponding to two distinct points of $F''$ and the 2-secant $\ell'(c)$ to 
$\Gamma$ passing through $c$, which comes with multiplicity 4.  From this it follows that  
when smooth cubics degenerate to $\mathscr C$, the associated 
6-webs by lines degenerate onto a web that it is justified to denote by  $  \boldsymbol{\mathcal L\hspace{-0.05cm}\mathcal W}_{\mathscr C}$ but which is a 3-web. \sk

\vspace{-0.3cm}
It is not difficult to make  $  \boldsymbol{\mathcal L\hspace{-0.05cm}\mathcal W}_{\mathscr C}$ explicit. 
Indeed, from the very definition of $\mathscr C$ it follows that 
\begin{equation}
\label{Eq:mu}
 \mu : \mathbf C^3\longrightarrow \mathbf P^4\, , \hspace{0.2cm} (s,t)=\big(s,t_1,t_2\big)\longmapsto \Big[ \nu_4(t_1)+s^2\nu_4(t_2)\Big]
 \end{equation}
  is an affine parametrization\footnote{The reason why 
the square of $s$ is used for parametrizing the secant between $\nu_4(t_1)$ and $\nu_4(t_2)$ is that it allows to get rational first integrals for the web $  \boldsymbol{\mathcal L\hspace{-0.05cm}\mathcal W}_{\boldsymbol{\mathscr C}}$ in the variables $t_1,t_2$ and $s$  ({\it cf.}\,\eqref{Eq:UUU}). Note that $\mu$ is a  2-1 ramified covering onto its image 
$\mu(\mathbf C^3)\subset  \mathscr C$ 
with ramification locus the coordinate hyperplane $s=0$.} 
 of (a Zariski-open subset of) $\mathscr C$ and it can be verified that 
the three lines contained in $\mathscr C$ and passing through a generic point $c=\mu(s,t)$ 
are 
\begin{itemize}
\item the  line $\ell'(c)$ with parametrization $\mathbf P^1 \ni \sigma \mapsto \mu(\sigma,t)\in \mathscr C\subset \mathbf P^4$ which is the 2-secant line $\langle  
\nu_4(t_1),\nu_4(t_2)
\rangle$. It corresponds to a 
point  of $F'=\Gamma^{[2]}$ hence comes with multiplicity 4;\sk
\item two lines elements of $F''$, denoted by $\ell_+''(c)$ and   $\ell_-''(c)$. These lines are those   joining $c$ to a point $C_\pm(s,t)\in \mathbf P^4$ whose coordinates are in $\mathbf Z[i](s,t)$ and can be explicited (moreover, one can manage things in such a way that 
$C_{-}(s,t)=\overline{C_+(s,t)}$). 
\end{itemize}

From the explicit formulas one can get for the lines  $\ell''_\pm(c)$ and $\ell'(c)$  for $c=\mu(s,t)$ generic in $\mathscr C$, it is not difficult to deduce some explicit rational first integrals on $\mathbf C^3$ for 
the pull-back under $\mu$ of the web $  \boldsymbol{\mathcal L\hspace{-0.05cm}\mathcal W}_{\boldsymbol{\mathscr C}}$. Indeed, it can be verified that this pull-back   (again denoted with the same notation to simplify), admits the following three rational first integrals in the coordinates $s,t_1,t_2$: 
\begin{equation}
\label{Eq:UUU}
U'= \big(t_1,t_2\big) \, , \hspace{0.3cm}  U''_+= \left( 
\frac{s\,t_2+i \,t_1 }{s+i} \, , \, \frac{s\,t_2^2+i \,t_1^2}{s\,t_2+i\,t_1}
 \right) 
 \hspace{0.35cm} \mbox{and} 
\hspace{0.35cm}
 U''_-=\overline{U''_+}= \left( 
\frac{s\,t_2-i \,t_1 }{s-i} \, , \, \frac{s\,t_2^2-i\, t_1^2}{s\,t_2-i\,t_1}
 \right) 
 \, . 
\end{equation}
From these explicit formulas, one gets that $  \boldsymbol{\mathcal L\hspace{-0.05cm}\mathcal W}_{\boldsymbol{\mathscr C}}$ is a 3-web by rational curves on $\mathbf C^3$ and it is straightforward to verify that some rational vector fields defining it  are 
$X'$,  $X''_+$ and $X''_-=\overline{X''_+}$ with 
$$
X'=\frac{\partial }{\partial s}\qquad \mbox{ and } \qquad 
X''_+=\ 2\,\frac{\partial }{\partial s}+ \left(\frac{t_1-t_2}{s+i}\right)\frac{\partial }{\partial t_1} + i
 \left( \frac{t_1-t_2}{s(s+i)}\right)\frac{\partial }{\partial t_2} 
 \, .
 $$
By direct computations, one gets the following formulas for the Lie brackets of these vector fields: 
\begin{align}
\big[  X_+''\, ,\, X_-''  \big] = & \, 
\frac{3s+i}{s(s+i)}\,X_+''- \frac{3s-i}{s(s-i)}\,X_-''- \frac{8i}{1+s^2}\,X'
\nonumber
\\
\big[  X_+''\, ,\, X'  \big] = & \, 
\frac{3s+i}{2s(s+i)}\,X_+''- \frac{s-i}{2s(s+i)}\,X_-''- \frac{2}{s}\,X'
\label{Eq:LW-C--skew}
\\
\mbox{and } \quad 
\big[  X_-''\, ,\, X'  \big] = & \, 
\frac{3s-i}{2s(s-i)}\,X_+''- \frac{s+i}{2s(s-i)}\,X_-''- \frac{2}{s}\,X'\, 
\nonumber
\end{align}
from which one deduces that  the web $  \boldsymbol{\mathcal L\hspace{-0.05cm}\mathcal W}_{\boldsymbol{\mathscr C}}$ is skew.  Its pull-back   by $\mu$ 
is formed by three global foliations on $\mathbf C^3$ but it is not the same for $  \boldsymbol{\mathcal L\hspace{-0.05cm}\mathcal W}_{\boldsymbol{\mathscr C}}$ itself
which is the union on (its definition domain in) $\mathscr C$ of the  foliation with global rational  
first integral $\ell' : X\dashrightarrow F'$, $c\mapsto \ell'(c)$ and of the global irreducible 2-web 
whose leaves through a general point $c\in \mathscr C$ are the two lines $\ell''_\pm(c)$.
\mk 

The remarks above indicate 
which statement corresponds to the first point of 
Proposition 
\ref{P:LW-X} when the cubic threefold $X$ specializes to the singular one $\mathscr C$. 
The other points of this Proposition concern abelian relations and clearly everything  stated in Proposition \ref{P:LW-X} about 2-ARs does not admit a counterpart for $  \boldsymbol{\mathcal L\hspace{-0.05cm}\mathcal W}_{\boldsymbol{\mathscr C}}$ since it is known that the 2-rank of any curvilinear 3-web in dimension 3 is zero.  This example shows that what has to be the statement corresponding to Proposition 
\ref{P:LW-X} when the considered cubic is singular is not self-evident. \mk 

Even if it is a bit off topic, it is interesting to consider the 1-ARs of $  \boldsymbol{\mathcal L\hspace{-0.05cm}\mathcal W}_{\boldsymbol{\mathscr C}}$, since it turns out that this web has maximal 1-rank and that its 1-ARs as well come from the abelian 1-forms on the Fano surface of $\mathscr C$. We will 
discuss this a little further in Appendix B. 
\begin{center}
$\star$
\end{center}

We have taken time to discuss the case of the chordal cubic $\mathscr C$ first because 
we find it interesting per se\footnote{It is worth mentioning that in \cite{Collino1},  Collino  has used degenerations towards the chordal cubic  $\mathscr C$ in order to deduce informations about the fundamental group of Fano surfaces of smooth cubic threefolds.}, but also  because it is not self-evident to show how to extend the nice theory discussed above in \S\ref{SSS:PropertiesOfF} to singular cubics as well as its web-theoretic consequence, namely Proposition \ref{P:LW-X}.  However we believe that in what concerns the webs associated to cubic threefolds, there is not much difference between the smooth case and that of cubics with only finitely many singular points. 
This is briefly discussed in the next subsection before considering more thoroughly the case  of Segre's cubic in \S\ref{SS:Segre's-cubic}.

\subsubsection{\bf Cubic threefolds with finitely many singular points.}
\label{SS:Nodal-cubics}


It is known that cubic hypersurfaces
in $\mathbf P^4$ with only finitely many singularities share many properties with the smooth cubics,
but may also present some distinct features.\sk 

The first and perhaps main difference between smooth and singular cubics is about their rationality. The smooth cubics are known to be irrational according to a famous result of the field (due
to Clemens and Griffiths \cite{ClemensGriffiths})
\footnote{In several publications circa 1945, Fano claimed having established that a general cubic hypersurface in $\mathbf P^4$ is not
rational. But his proof does not seem to have been accepted as correct ({\it cf.}\,the discussion in 
\cite[\S2.1]{Pukhlikov}).}, contrarily to any singular cubic for which the projection from
one of its singular points induces a birational map from it onto $\mathbf P^3$.
Another example of distinct feature depending on the singularities a cubic threefold can have,
is given by considering the case of cubics with finitely many ordinary nodes. Indeed, for such
a cubic $X$, it is known that the number $k$ of nodes is less than or equal to 10 and that the Fano
surface of lines $F_1(X)$ is generally irreducible when $k \leq  5$ but is not for any $k \geq  6$.\mk 

If classical geometers considered the case of cubic threefolds with arbitrary many ordinary (or
even more involved) nodes\footnote{See for instance   Segre's memoir \cite{Segre} (from \S12 to \S27),  Fano's papers \cite{Fano1} and \cite[\S9]{Fano} or Snyder's article 
\cite{Snyder},  where some of the assertions mentioned above are proved by means of classical geometric methods.} this is not the case in more recent papers. The results of 
Hodge-theoretic
nature listed in  \ref{SSS:PropertiesOfF}
  were established in view of proving the main result of the field,
namely the non rationality of a smooth cubic threefold and therefore were essentially considered
for smooth cubic threefolds. Actually, most of the techniques used by modern authors actually
apply to the so-called 'Lefschetz cubic hypersurfaces', that is cubics with at most one ordinary node as singularity. And indeed, in the recent papers where the material of \ref{SSS:PropertiesOfF}  is discussed in the 
case of a singular cubic, only the case of cubic threefolds with one ordinary node has been really
considered (in addition to \cite{ClemensGriffiths}) 
  where the case of Lefschetz cubics of $\mathbf P^4$ is explicitly considered,
see for instance \cite{CollinoMurre} or \cite{GeerKouvidakis}). Recently, cubic threefolds with more involved singularities
have been considered (for instance in some papers by Allock (2003) or Casalaina-Martin \& Laza (2009)) but with regard to their stability in the aim of studying their moduli theory. In particular, as far as we know, the theory thoroughly described in \ref{SSS:PropertiesOfF}  is  not discussed in a systematic
way for singular cubics with more than one ordinary node in the existing literature, especially
some points (such as {\bf xv.} page \pageref{lolo}, which has to be seen as a kind of Abel's theorem within this
context) which are crucial in view of establishing a result similar to Proposition \ref{P:LW-X} for the webs of lines on singular cubics.\mk

We consider below the case of singular cubics with only finitely many singular points. We do
not assume that the singular points are ordinary nodes, just that they are isolated. We deduce
the 'Abel's addition result for lines included in such a cubic' from the corresponding one holding true 
for smooth cubics first by means of a very natural deformation argument and also by using the fact
that the Fano surfaces of such cubics satisfy properties similar to some listed in \S\ref{SSS:PropertiesOfF} for smooth
cubics ({\it cf.} \eqref{Eq:X-singular-Properties}).  
Our arguments below are rather standard and easy to follow hence we believe
that very likely the 'Abel's addition result for lines on a cubic threefold with isolated
singularities' we prove below will appear to the experts as a well-known folkloric generalization of
the corresponding result for smooth cubics, which is classical. Proposition  \ref{P:LW-X-singular}  and its proof have been included here not for claiming any kind of originality but just for the sake of completeness.\bk

\paragraph{\bf  Abel's theorem for abelian 2-forms on the Fano surface.} 
Let us set precisely the notation and
assumptions we will deal with. Here $ X \subset  \mathbf P^4$  stands for a fixed irreducible cubic threefold with
finitely many singularities. We moreover assume that $  \boldsymbol{\mathcal L\hspace{-0.05cm}\mathcal W}_{X}$ is a well-defined 6-web on $X$: for $x \in  X$  generic, the six lines $\ell_1(x),\ldots,\ell_6(x) $ passing through $x$ and contained in $X$ are pairwise
distinct and are in general position, namely any three of them span a hyperplane in the ambiant $\mathbf P^4$. Here again $F = F_1(X)$ stands for the Fano scheme parametrizing the lines included in $X$. The
key fact behind the proof of Proposition \ref{P:LW-X-singular} 
 discussed below is the following generalization of
iv. of \S\ref{SSS:PropertiesOfF}  to the case of cubic threefolds with isolated singularities. Indeed, according to  \cite{AltmanKleiman} 
 ({\it cf.}\,(1.4), (1.8) and (1.15) therein), although not smooth, $F$ satisfies the following properties:
\begin{equation}
\label{Eq:X-singular-Properties}
\begin{tabular}{l}
 \hspace{-0.1cm} $1.$ $F=F_1(X)\subset G_1(\mathbf P^4)$ is still a reduced algebraic surface   
 (but with singularities and
 \\ 
${}^{}$ \hspace{0.0cm}
 even possibly several irreducible components);\mk\\
 \hspace{-0.1cm} $2.$  one has $\omega_F^2=\mathcal O_F(1)$, {\it i.e.}\,the dualizing sheaf of $F$ coincides with 
 the invertible  sheaf \\ 
 ${}^{}$ \hspace{0.1cm}
  induced by $\mathcal O_{\mathbf P^9}(1)$ via the Pl\"ucker 
embedding $F\subset G_1(\mathbf P^4)\hookrightarrow \mathbf P^9$;\mk  \\
 \hspace{-0.1cm} $3.$  the canonical map $\boldsymbol{H}^0(\mathbf P^9,\mathcal O_{\mathbf P^9}(1))\rightarrow \boldsymbol{H}^0(F,\mathcal O_{F}(1))\simeq \boldsymbol{H}^0(F,\omega^2_{F}) $ is an isomorphism.
\end{tabular}
\end{equation}
\mk

As is well-known (see \cite{Barlet}), the canonical sheaf 
$\omega_F^2$  can be identified with that of abelian differential
2-forms on $F$, namely the subsheaf of meromorphic 2-forms giving rise to $\overline{\partial}$-closed currents
on their definition domain. In particular, 
$\boldsymbol{H}^0(F,\omega^2_{F}) $ 
 can naturally be seen as a subspace of the vector
space $\Omega^2_{\mathbf 
C(F)}$  of rational 2-forms on $F$.  Given $\omega \in \boldsymbol{H}^0(F,\omega^2_{F}) $, the 
 sum $\sum_{i=1}^6 \ell_i^*(\omega)$  (where now
$\ell_i: (X, x) \rightarrow  F$ are considered as local algebraic first integrals for 
 $  \boldsymbol{\mathcal L\hspace{-0.05cm}\mathcal W}_{X}$) is the germ at $x$ of a
global rational 2-form on $X$,  called the 'trace of $\omega$' and denoted by ${\rm Tr}(\omega)$.

Our aim here is to establish that the points 2 and 3.a of Proposition  \ref{P:LW-X} 
 hold true as well
for the singular cubic under scrutiny:
 
\begin{prop}
\label{P:LW-X-singular} 
  \begin{enumerate}
\item[] ${}^{}$ \hspace{-1.2cm}1. For any abelian differential 2-form $\omega \in \boldsymbol{H}^0(F,\omega^2_{F})$, one has ${\rm Tr}(\omega)=0$.\mk 
  \item[{\it 2.}]  The well-defined induced linear map ${\rm Tr} : \boldsymbol{H}^0(F,\omega^2_{F}) 
  \rightarrow 
 \boldsymbol{AR}\big(    \boldsymbol{\mathcal L\hspace{-0.05cm}\mathcal W}_{X}\big)  $ is injective.
  \end{enumerate}
\end{prop}

  Since $10 = h^0  \big( F,\omega^2_{F} \big)\leq {\bf rk}\big(    \boldsymbol{\mathcal L\hspace{-0.05cm}\mathcal W}_{X}\big) \leq 10$, one gets immediately the
\begin{cor}
\label{P:LW-X-singular} 
 One has  $\boldsymbol{H}^0(F,\omega^2_{F}) 
  \simeq 
 \boldsymbol{AR}\big(    \boldsymbol{\mathcal L\hspace{-0.05cm}\mathcal W}_{X}\big)$ 
 hence the web $\boldsymbol{\mathcal L\hspace{-0.05cm}\mathcal W}_{X}$ has maximal 2-rank.
\end{cor}

In view of Proposition \ref{P:LW-X}, it is natural to wonder whether the statements therein about the 1-forms of the Fano surface/the 1-ARs of the web by lines on the cubic are also verified when X
has isolated singularities or not. This is more subtle than it is for 2-forms/2-ARs but we think that it is indeed the case. We will discuss this in the next paragraph.\sk 
 
The key fact behind Proposition \ref{P:LW-X-singular} is the point {\bf xv} of 
 \S\ref{SSS:PropertiesOfF} which follows from the fact that
when X is smooth, the trace of a global holomorphic 2-form on $F_1(X)$ admits a holomorphic
extension to the whole $X$ hence necessarily vanishes.
 \footnote{That a smooth cubic threefold $X$ does not carry any non trivial global holomorphic 2-from ({\it i.e.}\,that $h^0(
\Omega^2_
X) = 0$)
follows easily from Hodge theory (computation of the Hodge number $h^{p,q}(X)$'s). 
It is worth mentioning that the vanishing of the trace of any 2-forms on $F_1(X)$ also follows from a famous theorem by Mumford, {\it cf.}\,\cite[p.\,200]{Mumford}.} 
  It is not immediately clear in which way this approach
can be extended to singular cubics. For instance, even for a cubic with simple singularities (nodes) such as Segre's cubic $\boldsymbol{S}$, proving a priori that the trace of an abelian differential 2-form on  $\Sigma=
F_1(\boldsymbol{S})$ is abelian seems a bit delicate. We have chosen to follow another approach to prove the proposition above, which is quite natural since it relies on the smooth case via a local deformation
of the considered cubic.\sk 

We continue to use the notations introduced before in the proof below: 

\label{zolo}
\begin{proof}[Proof of Proposition \ref{P:LW-X-singular}]   For $\eta \in  \boldsymbol{H}^0(F,\omega^2_{F}) $
arbitrary, we first want to prove that $\sum_{i=1}^6 \ell_i^*(\eta)$ is identically zero. This can be obtained
quite easily by considering a smoothing of $(X, x)$ as a pointed cubic hypersurface. So let
$\{(X_t, x_t)\}_{t\in (\mathbf C,0)}$  be a local analytic family of pointed cubic threefolds in $\mathbf P^4$ such that $X_0 = X$ and
$x_0 = x$, with $X_t$  smooth for $t \neq  0$ (the local existence of such a smoothing family is obvious). 
Then the germ
$\boldsymbol{\mathcal L\hspace{-0.05cm}\mathcal W}_{X_t,x_t}$ of 
$\boldsymbol{\mathcal L\hspace{-0.05cm}\mathcal W}_{X_t}$  at $x_t$ is a well-defined germ of linear 6-web for $t$ sufficiently close to $0$,  and the set 
of them form an (local) analytic deformation of 
$\boldsymbol{\mathcal L\hspace{-0.05cm}\mathcal W}_{X,x}= \boldsymbol{\mathcal L\hspace{-0.05cm}\mathcal W}_{X_0,x_0}$. 
\footnote{This is not rigorously correct since the germs of webs $\boldsymbol{\mathcal L\hspace{-0.05cm}\mathcal W}_{X_t,x_t}$ do not live on the same space. But it is easy
to solve this, for instance by considering a generic linear projection $P : \mathbf P^4\rightarrow  \mathbf P^3$  and an analytic family $\{g_t\}_{t\in (\mathbf C,0)}$  of
projective transforms $g_t \in  {\rm PGL}_4(\mathbf C)$ such that $g_0 = {\rm Id}$ and $g_t(P(x_t)) = P(x)$ for any $t \in  (\mathbf C, 0)$. Then the push-forward
germs of webs $(g_t \circ P)_*(\boldsymbol{\mathcal L\hspace{-0.05cm}\mathcal W}_{X_t,x_t})$  form a genuine local holomorphic deformation of
$P_*\big( \boldsymbol{\mathcal L\hspace{-0.05cm}\mathcal W}_{X,x}\big) $ 
 on $(\mathbf P^3, P(x)) \simeq  (\mathbf C^3, 0)$.}
 From this, we deduce that
the germs of maps $\ell_{i,t} :  (X_t, x_t) \rightarrow   F_t = F_1(X_t)$ which are the first integrals defining $\boldsymbol{\mathcal L\hspace{-0.05cm}\mathcal W}_{X_t,x_t}$  for
any $t \in  (\mathbf C, 0)$, depend holomorphically on it.\sk

For proving that
$\sum_{i=1}^6 \ell_i^*(\eta)=0$,  it suffices to show that the considered $\eta$  fits into an analytic
family $\{\eta_t\}$ where $\eta_t$  is a global holomorphic 2-form on $X_t$ for any $t$  sufficiently close to 0. Indeed,
in this case ${\rm Tr}(\eta_t) =
\sum_{i=1}^6 \ell_i^*(\eta_t)$ 
is holomorphic with respect to t from one hand, but is such
that ${\rm Tr}(\eta_t) \equiv  0$ for every $t \in  (\mathbf C, 0)$ distinct from 0 (thanks to Proposition \ref{P:LW-X}, since all the $X_t$'s
are smooth for $t \neq  0$. Specializing at $t = 0$ gives us the wanted relation. But that such an
$\eta$  can be (locally) holomorphically deformed to the smooth deformations $X_t$ of $X$  follows from standard facts of analytic geometry: the Fano surfaces $F_t = F_1(X_t) \subset  G_1(\mathbf P^4)$ for $t$ sufficiently
close to 0, organize themselves into an analytic family of surfaces $\pi : \mathcal F \rightarrow  (\mathbf C, 0)$ ({\it cf.}\,\cite{AltmanKleiman} or \cite[\S3]{Casalaina-Martin}). The direct
image by $\pi$ of the relative dualizing sheaf 
$\omega^2_{\mathcal F/ (\mathbf C,0)}
$  has fiber 
$\boldsymbol{H}^0(F_t,\omega^2_{F_t})$  
at any $t$. Since all these
fibers are 10-dimensional (according to \eqref{Eq:X-singular-Properties}), it follows that 
$\pi_*\big(\omega^2_{\mathcal F/ (\mathbf C,0)} \big) $ 
is locally-free. Thus any
$\eta\in \boldsymbol{H}^0(F_0,\omega^2_{F_0}) $ 
can be extended into an analytic section 
$t \rightarrow  \eta_t \in \boldsymbol{H}^0(F_t,\omega^2_{F_t})$. As explained just
above, this is sufficient to ensure that the following map is well-defined:
\begin{equation}
\label{Al:Tr}
{\rm Tr}  : \boldsymbol{H}^0(F,\omega^2_{F}) \longrightarrow 
 \boldsymbol{AR}\big(    \boldsymbol{\mathcal L\hspace{-0.05cm}\mathcal W}_{X,x}\big)
\, , \quad 
\eta\longmapsto 
 \big( \ell_i^*(\eta)\big)_{i=1}^6\,  .
\end{equation}
There is a last subtlety to deal with which concerns the injectivity of this map, which is not
completely obvious. Let us briefly discuss this matter. For $i = 1,\ldots, 6$, let $F'_i$  be the irreducible
component of $F$ which contains ${\rm Im}(\ell_i)$. Some of the $F'_i$'s may coincide, hence these components
can be univocally labeled $F_j$ for $j$ in some set $J$ of cardinality at most 6. Let $F^{gen}$  be the union $\cup_{j\in J}  F_j$  and denote by ${\widetilde F}^{gen}$ its image in $\mathbf P^9$ by the Pl\"ucker embedding $G_1(\mathbf P^4) \subset  \mathbf P^9$ (the notation
with the upper {\it gen} comes from the fact that $F^{gen}$ can equivalently be defined as the union of the
irreducible components of $F$ whose generic element meets the generic point of $X$).\sk 

The following statements are easily seen to be equivalent:
\begin{equation}
\label{Eq:3-equivalent-properties}
\begin{tabular}{l}
 \hspace{-0.3cm} $\bullet$ \,the map \eqref{Al:Tr} is injective;\sk   \\ 
 \hspace{-0.3cm} $\bullet$ \,the restriction map $\boldsymbol{H}^0(F,\omega_F^2)\rightarrow \, \oplus_{j\in J} \Omega^2_{\mathbf C(F_j)}$, $\omega\mapsto \big( \omega\lvert_{F_j}\big)_{j\in J}$ is injective; \sk\\ 
 \hspace{-0.3cm} $\bullet$ \,the surface $\widetilde F^{gen}$ is non degenerate in $\mathbf P^9$, {\it i.e.} $\big\langle \widetilde F^{gen}\big\rangle= \mathbf P^9$. 
\end{tabular}
\end{equation}

That these equivalent properties hold true is immediate for instance when $F$ is irreducible, but needs to be justified for an arbitrary cubic with finitely many singular points. 
This is the content of the lemma just below. The proposition is fully proved. \end{proof}

\begin{lem}
\label{L:X-6lines-PG}
For $X$ as above (that is, $X$ is an irreducible cubic threefold with finitely many singular
points and such that the lines it contains form a genuine 6-web whose leaves are generically
in general position), then $\widetilde F^{gen}$ is non degenerate in $\mathbf P^9$.
\end{lem}

\begin{proof} 
The proof is easy and goes by reduction ad absurdum. Recall that the Pl\"ucker embedding
$P : G_1(\mathbf P^4) \subset \mathbf P^9$ has for coordinates the Pl\"ucker coordinates
$\Delta_{ij}$  for $i, j$ such that $1 \leq  i < j \leq  5$:
if $L$ is the line passing through $[p]$ and $[q]$ for two distinct points $p, q \in \mathbf  C^5$  with coordinates $p_i$ 
and $q_i$  respectively, then by definition $ \Delta_{ij}(L)= \Delta_{ij}(p, q) = p_iq_j-p_jq_i$ 
  for any $i,j$ as above and the
 point 
 $ \big[ \Delta_{ij}(L)
 \big]_{1\leq i < j \leq 5} \in \mathbf P^9$ is well-defined (that is depends only on $L$) and corresponds to $P(L)$.\mk

That 
 ${\widetilde{F}}^{gen}= P(F^{gen}) \subset  \mathbf P^9$  
 is degenerate means that there exist scalars 
defining a hyperplane 
for $[p] \in  X$ generic and any $[q]$ such that the line $L_{p,q} = \langle [p], [q]\rangle $  joining $[p]$ to $[q]$ is contained
in $X$, then the following relation is satisfied
\begin{equation}
\label{Eq:HP(L)}
\Big\langle 
H, P(L_{p,q})\Big\rangle 
=\sum_{1\leq i < j \leq 5} 
h_{ij} \, \Delta_{i j}(p, q) = 0 .
\end{equation}

Let $F$ be a cubic homogeneous equation in five variables such that $F = 0$ be an equation of the affine cone over $X$ in $\mathbf C^5$ and denote by 
$F( \cdot , \cdot , \cdot)$ the associated symmetric trilinearization. For $[p] \in  X$  generic, the $q$'s such that the line $L_{p,q}$ is included in $X$ are those satisfying the following set of homogeneous equations (of degree 1,2 and 3 in q respectively): 
$$ F(p, p, q) = 2dF_p(q) = 0\, , \qquad F(p, q, q) = 0 \quad \mbox{ and } \quad  F(q, q, q) = 6F(q) = 0\, .$$ 
 Assuming that \eqref{Eq:HP(L)} is satisfied as soon as these preceding equations hold true leads to two possibilities, depending whether the  two $q$-linear equations $dF_p(q) = 0$ and \eqref{Eq:HP(L)}  are colinear or not. If it is the case, then there exists $\Lambda(p) \in \mathbf  C^*$  such that the equality between linear forms $ \Lambda(p)\cdot 
 \big\langle 
H, P(L_{p,\cdot })\big\rangle = dF_p(\cdot )$  holds true. Since this identity holds for $[p]$ generic in $X$ which is non-degenerate in $\mathbf P^4$  and because $dF_p(\cdot)$ is homogeneous of degree 2 with respect to $p$, it follows that the $\Lambda(p)$'s can be taken to be such that $p \mapsto \Lambda(p)$  be a linear form. But then this would imply
that $dF = 0$ on the hyperplane section $ X \cap  \{\Lambda=0 \}$, contradicting the assumption that $X$ has
isolated singularities. \sk 
 
Thus necessarily \eqref{Eq:HP(L)}  cuts a genuine line in the projectivized tangent space $\mathbf P(T_pX)$ which
obviously contains all the tangent directions at p of the lines 
$\ell_i(p)$'s. But this would imply in particular
that these six tangent directions span a 2-plane in $T_pX$, contradicting the general position
assumption made about the leaves of $ \boldsymbol{\mathcal L\hspace{-0.05cm}\mathcal W}_{X}$ at a generic point. The lemma follows. \end{proof}

\begin{rem}
The proof above actually shows that the conclusion of Lemma 3.4 is still satisfied
under the weaker assumptions that $\dim\big(X_{sing}\big) \leq  1$, $\dim\big(F_1(X)\big)=2$  and that among the lines
contained in $X$ through a generic point, at least three are in general position. 
For instance, the chordal cubic $\mathscr C$  satisfies these hypotheses (but in this case $F^{gen}$ coincides with $F_1(\mathscr C )$ and it was
known that its Pl\"ucker image in $\mathbf P^9$ is non degenerate according to the discussion in \S\ref{SS:je-sais-pas-quoi}).
\end{rem}
\mk

\paragraph{\bf On Abel's theorem for abelian 1-forms on the Fano surface.} 
\label{Abel's-1-form-Fano-Surface}
We believe that  some meaningful parts of 
Proposition  \ref{P:LW-X} generalize in a straightforward way to any cubic threefold $X$ with isolated singular points, at least as soon as the lines included in it define a genuine 6-web on it.\footnote{Note that for $X$ a smooth cubic with isolated singularities, it may happen that 
$ \boldsymbol{\mathcal L\hspace{-0.05cm}\mathcal W}_{X}$ be a genuine 6-web but not a skew one (this precisely happens for Segre's cubic $\boldsymbol{S}$ for instance). In such a situation, 
the space of 1-ARs of $ \boldsymbol{\mathcal L\hspace{-0.05cm}\mathcal W}_{X}$ may be of infinite dimension hence 
one cannot expect the whole  Proposition  \ref{P:LW-X} to be generalized as it is to the case of such cubic threefolds.}  However, as is already the case for smooth cubics ({\it cf.}\,\cite[p.\,338]{ClemensGriffiths}), everything that concerns 
the links between some differential $k$-forms on the Fano surface $F=F_1(X)$ and the $k$-ARs of   $ \boldsymbol{\mathcal L\hspace{-0.05cm}\mathcal W}_{X}$ is much more subtle when $k=1$ than when $k=2$. We discuss below several approaches which could give a proof of the extension for cubics with isolated singularities of the fact that the first trace map may induce an isomorphism between a space of 1-differential forms on the associated Fano surface and 
the  space of 1-ARs of the corresponding web of lines on the cubic.  We warn the reader that nothing is proved below,  rigorous investigations on this are left for a hypothetical future work.  \mk 

In what follows, $X$ stands for a fixed cubic threefold in $\mathbf P^4$ with isolated singularities.
\begin{center}
$\star$
\end{center}

As a first attempt, one might naively want to extend in a straightforward manner 
the approach discussed in the preceding paragraph to the case when $X$ has isolated singularities: 
considering a smoothing $\pi: \mathcal X\rightarrow B$ of $X=X_0=\pi^{-1}(b_0)$ over a smooth 1-dimensional base $B$,  one defines a coherent sheaf of `relative abelian 1-forms'  $\omega^1_{\mathcal F/B}$ on the total space of the associated family of Fano surfaces 
$\pi: \mathcal F\rightarrow B$  by requiring that  the following sequence of sheaves is exact 
$0\rightarrow \pi^*(\Omega_B^1)\rightarrow \omega_{\mathcal F}^1\rightarrow \omega^1_{\mathcal F/B}\rightarrow 0$.\footnote{Here  $\omega_{\mathcal F}^1$ 
stands for the sheaf of `{\it Abelian} (or {\it Barlet} \cite{Barlet}) {\it differential 1-forms on $\mathcal F$}', which 
can be characterized as the 
 meromorphic 1-forms on (analytic open subsets of) $\mathcal F$ inducing $\overline{\partial}$-closed currents on their definition domains.} 
  Then  setting $F_0=F_1(X_0)=F_1(X)$, the same arguments as in the preceding paragraph 
 would give that the trace induces an injective  linear map  $
\boldsymbol{H}^0(F_0,\omega_{F_0}^1) \rightarrow 
\boldsymbol{AR}^{(1)}\big( \boldsymbol{\mathcal L\hspace{-0.05cm}\mathcal W}_X\big)$, 
provided that the two following facts hold true: 
$$(i).\, \mbox{ one has }\,   \big(\omega^1_{\mathcal F/B}\big)\big\lvert_{F_0}=\omega^1_{F_0}\, ; 
\qquad \mbox{ 
and  } \qquad (ii).\hspace{0.4cm} 
h^0\big(F_0,\omega^1_{F_0}\big)=5\, .$$
 If one expects that these two facts are indeed satisfied, 
 verifying it is not obvious. Abelian differentials are not known to behave particularly well under pull-back hence $(i).$ needs to be justified. As for $(ii).$, it leads to the following interesting questions: 
according to \cite[(1.10).(ii)]{AltmanKleiman}, 
denoting by $F_0^*$ the smooth locus of $F_0$, 
one has $\Omega^1_{F_0^*}\simeq \big( \mathcal T^\vee\big)\lvert_{F_0^*}$ (compare with 
\eqref{Eq:TBT} in the smooth case). Does this identification extend through the singularities of $F_0$ into an isomorphism $\omega^1_{F_0}\simeq \big( \mathcal T^\vee\big)\lvert_{F_0}$? And if so, 
is  the restriction map $V^\vee \simeq \boldsymbol{H}^0\big(G_1(\mathbf P(V)), \mathcal T^\vee\big)\rightarrow 
\boldsymbol{H}^0(F_0, \omega^1_{F_0})$  injective or even  an 
isomorphism?  It does not seem easy to answer these questions, a reason for that being that 
$F_0$ has rather involved singularities (in particular, its singular locus $F_0\setminus F_0^*$ is 1-dimensional).  It is known that $h^1(\mathcal O_{F_0})=5$ ({\it cf.}\,\cite[Prop.\,(1.15)]{AltmanKleiman}) but one cannot deduce directly from this that $h^0(\omega^1_{F_0})=5$ 
since it is not clear if/how the pure Hodge structures on the first cohomology spaces 
$\boldsymbol{H}^1(F_b,\mathbf C)= \boldsymbol{H}^{1,0}(F_b)\oplus \boldsymbol{H}^{0,1}(F_b)$ for $b\in B$ with  $F_b$ smooth degenerate when $b$ tends to $b_0$. 
\sk

It seems to us that one might be able to deduce from the (rather technical) Hodge-theoretic considerations in \cite{Collino1} that everything discussed above indeed holds true when considering the chordal cubic $\mathscr C$ and its non irreducible hence singular Fano surface $F_1(\mathscr C)$ (this independently of the fact that the lines in $\mathscr C$ do not define a 6-web in it, but only a 3-web). This makes us confident that everything go as sketched above in the simpler case of cubics with only finitely many  singular points.
\begin{center}
\vspace{-0.35cm}$\star$
\end{center}\sk

Constructing the 1-ARs of $ \boldsymbol{\mathcal L\hspace{-0.05cm}\mathcal W}_{X}$ by means of the abelian 1-forms of the Fano surface $F$ is certainly interesting conceptually,  but not really necessary if mainly interested in the 1-rank of this web. 
Considering a smoothing $\pi: \mathcal X\rightarrow B$ as above, 
it would already be satisfying to be able to describe an analytic family of spaces $\boldsymbol{H}^1_b\subset \Omega_{\mathbf C(F_b)}^1 $ of rational 1-forms on $F_b=F_1(X_b)$ with 
$X_b=\pi^{-1}(b)$ for any $b\in B$, with $\boldsymbol{H}^1_b\simeq \boldsymbol{H}^0(F_b,\Omega_{F_b}^1)$ for any $b$ such that $X_b$ is smooth and  such that 
 the trace induces an injective map $\boldsymbol{H}^1_b\rightarrow 
\boldsymbol{AR}^{(1)}\big( \boldsymbol{\mathcal L\hspace{-0.05cm}\mathcal W}_{X_b}\big)$ for every $b \in B$, in particular for $b=b_0$.  We see several ways to  construct  such an analytic family of vector spaces $\{ \boldsymbol{H}^1_b\}_{b \in B}$: 
\begin{itemize}
\item  If $X_b$ is smooth, the map $V^\vee\rightarrow \boldsymbol{H}^0(F_b,\Omega_{F_b}^1)$ of {\bf v} page \pageref{lala} above can be made rather explicit (see around (12.8) in \cite{ClemensGriffiths}). If by chance the construction by means of residues might be extended to $X_0$, it would possibly give a map $V^\vee\rightarrow \boldsymbol{H}^0\big(F_0^*,\Omega^1_{F_0^*}\big)$ whose elements   in the  image extend as rational 1-forms on $F_0$ and  deform into holomorphic 1-forms on the smooth nearby fibers $X_b$ for $b\in (B,b_0)$. Denoting by $\boldsymbol{H}^1_{b_0}\subset  \Omega_{\mathbf C(F_0)}^1$ their span, we would thus obtain a linear map  $\boldsymbol{H}^1_{b_0}\rightarrow \boldsymbol{AR}^{(1)}\big( \boldsymbol{\mathcal L\hspace{-0.05cm}\mathcal W}_{X}\big)
$ whose image is expected to be 5-dimensional. \mk 
\item  Instead of trying to deal with 1-forms on the singular Fano surface $F_0$, one can rather consider working with the global holomorphic forms on a semi-abelian (hence smooth) variety 
naturally associated with it.  Indeed, {\it 1.} for $b\in B$ with $X_b$ smooth, the Albanese map ${\bf alb}: F_b\rightarrow {\rm Alb}(F_b)$  induces an isomorphism ${\bf alb}^*: \boldsymbol{H}^0\big(\Omega^1_{{\rm Alb}(F_b)}\big)
\rightarrow \boldsymbol{H}^0(F_b,\Omega_{F_b}^1)$; {\it 2.} an Albanese variety with an associated Albanese map has been constructed by several authors for cubics with singularities and this construction has been proven to behave well for suitable degenerating families (see {\it e.g.}\,\cite{GwenaThesis} for cubic with ordinary nodes only,  or \cite{Casalaina-Martin} and the reference therein for more involved isolated singularities).\footnote{Actually, in most papers dealing with singular cubics, what is constructed is a `generalized intermediate Jacobian' $J(X_0)$ associated to a suitable singular cubic threefold $X_0$, which is a semi-abelian variety of dimension 5 (actually, the semi-abelian intermediate Jacobian $J(X_0)$ is constructed by relating it to a Prym variety 
of a non ramified 2-1 cover of a plane quintic curve associated (in a non canonical way) to $X_0$, see \cite[\S3.2]{Casalaina-Martin}). For a smooth cubic $X$, there is a well-understood 
isomorphism $J(X)\simeq {\rm Alb}(F_1(X))$ and it is expected (and proved in many cases) that this extends to singular cubics, or at least to some of them.  The fact that most of these constructions for singular cubics are in terms of `intermediate Jacobians' and not of `Albanese varieties' actually is not important for our purpose.}. In the case under scrutiny, $ \{ {\rm Alb}(F_1(X_b))\}_{b\in B}$  would be an analytic family of 
5-dimensional semi-abelian varieties from which it should be possible to construct an analytic flat family of spaces $\boldsymbol{H}^1_b$ with the required properties.  We believe that such an approach should work for cubics with ordinary nodes \cite{GwenaThesis} or for those with `{\it allowable singularities}' according to the terminology used in \cite{Casalaina-Martin}.
\mk 
\item  In the case when $X_0$ has isolated singularities of the simplest type, namely ordinary nodes, an approach might be to find a nice desingularization $\tilde \pi : \widetilde{\mathcal F}\rightarrow B$ of the relative Fano surface $ \pi : \mathcal F\rightarrow B$ such that the central fiber $\tilde F_0=\tilde \pi^{-1}(0)$ be a reduced surface with normal crossing.  In such a situation, the $(1,0)$-component of the Hodge structure on $\boldsymbol{H}^1(\tilde F_0,\mathbf C)$ obtained as the limit of the standard one of the smooth Fano surfaces $\tilde  F_b$ for $b\in (B,b_0)$ coincides  with the space of global sections of a certain sheaf $\Lambda_{\tilde F_0}^1$
of meromorphic 1-forms on 
 $\tilde F_0$ (see \cite[\S3]{Friedman}).  The arguments sketched at the beginning of this paragraph ({\it cf.}\,our `first attempt' above), if applied not to 
$\mathcal F$ but  to  the desingularized family $\tilde \pi : \widetilde{\mathcal F}\rightarrow B$, 
should provide the sought-after injective morphism $\boldsymbol{H}^0(\tilde F_0,\Lambda_{\tilde F_0}^1\big)\rightarrow 
\boldsymbol{AR}^{(1)}\big( \boldsymbol{\mathcal L\hspace{-0.05cm}\mathcal W}_{X}\big)$.

 \label{Colino-Desingularization}
 In  \cite{Collino2}, Collino describes quite explicitly the desingularization $\tilde \pi : \widetilde{\mathcal F}\rightarrow B$ when the initial non smooth $X_0$ is Segre's cubic primal and the limit Hodge structure on $\tilde F_0$ (see Appendix B at the end of this text for more details in the case of Segre's cubic).  We believe that most of the results in this paper may be adapted to follow 
 the strategy sketched just above when $X_0$ is any cubic threefold with ordinary nodes. 
\end{itemize}
\begin{center}
\vspace{-0.05cm}$\star$
\end{center}\sk 

Finally, as a last  possible way to construct the 1-ARs of 
$\boldsymbol{\mathcal L\hspace{-0.05cm}\mathcal W}_{X}$, we would like to mention the old and not well-known but very interesting paper \cite{BW}.   
In it (in \S6 more precisely),  Blaschke and Walberer  construct quite directly from any cubic homogeneous form  defining a given cubic threefold 
$X\subset \mathbf P^4$, the 1-ARs not of the web by lines on $X$ but of an associated 3-web, namely the curvilinear 3-web $\boldsymbol{\mathcal T\hspace{-0.07cm}\mathcal W}_{X}$ naturally defined on the variety $\boldsymbol{\mathcal T}_{\hspace{-0.05cm}X}$ of `triangles included in $X$'.\footnote{We will define this variety and discuss the  3-web 
naturally defined on it 
in the case of the chordal cubic $\mathscr C$  in Appendix B to which we refer the reader for more details.} 
It would be interesting to check whether their approach can be adapted to build by hand an explicit family of dimension 5 of 1-RAs, but for $\boldsymbol{\mathcal L\hspace{-0.05cm}\mathcal W}_{X}$ this time.  Note that in  \cite{BW}, essentially no assumption  about the singularities of $X$ are made; it  seems to be  only required  that $\boldsymbol{\mathcal T\hspace{-0.07cm}\mathcal W}_{X}$ is a skew curvilinear 3-web whose leaves are in general position generically. 
  It would be interesting to revisit \cite{BW} with a  modern and rigorous approach, for instance to determine precisely under which hypothesis the results contained in it hold true.

\subsection{Segre's cubic (case $n=3$).}
\label{SS:Segre's-cubic}
We now focus on the case of Segre's cubic. Since it is a very well-known and studied threefold, 
no proof of its properties stated below will be given, the interested reader will refer to the specific references indicated in the following subsection.







\subsubsection{\bf References.}
Segre's cubic has been discussed in many papers since its discovery by Segre in 1887. 
As classical references, in addition to Segre's original paper
\cite{Segre} (\S24 to \S27 therein), one can mention  \S19 to \S31 in `Capitolo $8^{\rm o}$' of Bertini's book \cite{Bertini}, \cite{Snyder}, \cite[Chap.\,VIII, \S2.32]{SempleRoth}. 
 As for other recent references, the following ones are interesting considering our purpose but there may exist many others we are not aware of:  
Segre's cubic primal is discussed in several papers by Dolgachev, 
such as \cite{DolgachevSegre} (in particular \S2 and \S4 therein) or
\cite[\S5]{Dolgachev15}.  Other modern references are:  \cite[\S3.2]{Hunt}, 
\cite{GwenaThesis,Gwena}, 
\cite{Collino1},   \cite{Collino2} and  \cite[\S2]{DolgachevFarbLooijenga}.
Among the previous references, \cite{Collino2} is perhaps the most relevant one  since  Collino describes there the Fano surface of lines of Segre's cubic quite explicitly and studies some of its features. It is  unfortunate that such an  interesting text has not be published in proper form.


\subsubsection{\bf Segre's cubic}
We recall well-known classical facts about Segre's cubic. 

First, we introduce some notation: 
Let $[U_1,U_2,U_3,U_4]$ be some fixed homogeneous coordinates on $\mathbf P^3$ and 
  $x,y,z$ stand for  the affine coordinates on $\mathbf C^3= \mathbf P^3\setminus \{ U_4=0\}$ corresponding to the 
  embedding $(x,y,z)\mapsto [x:y:z:1]$.  
  As above,  the $p_i$'s (for $i=1,\ldots,5$) stand for the following points  in 
  $\mathbf P^3$:  $p_i=[\delta^i_j]_{j=1}^4$ for $i=1,\ldots,4$ and $p_5=[1,\ldots,1]$.  
  We set $P=\{ p_i\}_{i=1}^5$.
  \sk 

 In what follows, whitout any supplementary precision, $i,j,k,l,m$ will stand for pairwise distinct  elements of $\{1,\ldots,5\}$ (hence such that $\{i,j,k,l,m\}=\{1,\ldots,5\}$).  For such indices, 
 we denote by 
\begin{itemize}
\item $L_{ij}$  the line   passing through $p_i$ and $p_j$: 
{\it i.e.} $L_{ij}=\big\langle p_i,p_j  \big\rangle\subset \mathbf P^3$;
\item $P_{ijk}$ or $P^{lm}$  the 2-plane  containing $p_i,p_j$ and $p_k$: 
{\it i.e.}
$P_{ijk}=P^{lm}=\big\langle  p_i,p_j,p_k\big\rangle \subset \mathbf P^3$.
\end{itemize}
\mk 

Recall ({\it cf.}\,\S\ref{SSS:Ln-varphin} above) that 
$\boldsymbol{\mathcal L}_3$ is the linear system of quadric surfaces in $\mathbf P^3$ passing through all the $p_i$'s: 
$ \boldsymbol{\mathcal L}_3=
 \lvert 2H-\sum_{i=1}^5 p_i \lvert$.  Its elements are exactly the surfaces in $\mathbf P^3$ cut out by homogeneous equations $\sum_{i<j} a_{ij} U_iU_j=0$ for some coefficients $a_{ij}$'s whose total sum is zero. 

 In the coordinates $x,y,z$, the rational map associated to  $
\boldsymbol{\mathcal L}_3$ can be taken to be 
\begin{equation}
\label{Eq:varphi-3-xyz}
\varphi_3: (x,y,z)\dashrightarrow \big[ \, x(z-y)\, :\, 
x(1-y) \, :\,
y(z-x) \, :\,
y(1-x) \, :\,
z-xy\,
\big]\, . 
\end{equation}
Its base-points in $\mathbf C^3$ are origin $\boldsymbol{0}=(0,0,0)$ and the point 
$\boldsymbol{1}=(1,1,1)$ which correspond to $p_4$ and $p_5$ respectively.  We will again denote by $\varphi_3$ the extension to $\mathbf P^3$ of the above map (as a rational map). 
 By definition, the variety $V_3$ is the closure of $ \varphi_3\big( \mathbf P^3\setminus P\big)$ in $ \mathbf P^4$.  It can be verified that it is the cubic hypersurface cut out by 
the following equation: 
$$ X_1X_2X_3 - X_1X_2X_4 - X_1X_3X_4 + X_1X_4X_5 + X_2X_3X_4 - X_2X_3X_5=0 \, .$$

For any distinct $p_i$ and $p_j$, the line $L_{ij}=\langle p_i,p_j\rangle $ is contracted onto a singular point of $V_3$, which will be denoted by 
$p_{ij}$.  The $p_{ij}$'s are easy to determine: they are the six vertices of the canonical simplex in $\mathbf P^4$ (namely the points $\big[\delta_i^j\big]_{j=1}^5$ for $i=1,\ldots,5$ and  $[  1  :1  :  1 : 1: 1]$) and the following four other points 
$[  0:  0   :  1   :   1  :  1 ]$, $[  0 :  1   :  0   :   1  :  1 ]$, $[ 1  : 1  :   0  :  0  :  1 ]$ and $[ 1  : 0  :   1 :  0   : 1 ]$.  
Hence the $p_{ij}$'s are 10 pairwise distinct singular points of $V_3$. On another hand, the map $\varphi_3$ has rank 3 outside the union of the 10 lines $L_{ij}$'s, hence $V_3$ is smooth outside the $p_{ij}$'s. Because these singular points  all are ordinary nodes (as an easy analysis 
shows),  one deduces that $V_3$ coincides with Segre's 10-nodal cubic hypersurface. For that reason,  we will denote this threefold by $\boldsymbol{S}$ from now on. We will write $N_{\boldsymbol{S}}=\{\, p_{ij} \, \lvert \, 1\leq i<j\leq 5\, \}$ for the set of its nodes.
\sk 

Let us describe the 2-planes  and some lines contained in $\boldsymbol{S}$ and their incidence with  the nodes.  For any $i,j,k$ (pairwise distinct), the quadric of $\boldsymbol{\mathcal L}_3$ passing through a generic point $p\in P_{ijk}$ is the union of $P_{ijk}$ itself with the 2-plane passing through $p_l$, $p_m$ and $p$. It follows that $\varphi_3$ maps $P_{ijk}$ onto a plane in $\boldsymbol{S}$, that we will denote $P_{ijk}$ too. 
These  ten 2-planes $P_{ijk}$'s are pairwise distinct.  In addition to them, $\boldsymbol{S}$ contains five other planes, which can be described as follows: let $\pi: \widetilde{\boldsymbol{S}}={\rm Bl}_P \mathbf P^3\rightarrow \mathbf P^3$ be the blow-up at the five base-points of $\lvert  \boldsymbol{\mathcal L}_3\lvert$. Then $\widetilde{\varphi}_3=\varphi_3\circ\pi : \widetilde{\boldsymbol{S}} \rightarrow \boldsymbol{S}$ is a morphism and if $E_i$ stands for the exceptional divisor corresponding to $p_i$ ({\it i.e.}\,$E_i=\pi^{-1}(p_i)$) for any $i$, then 
$P_i=\widetilde{\varphi}_3(E_i)$ is a 2-plane in Segre's cubic.   In this way, we get five other planes included in $\boldsymbol{S}$ which with the 15 described above form the whole set of 2-planes contained in Segre's cubic, which therefore has cardinality 16.  
In order to have an uniform notation, we set $P^{ij}=P_{klm}$ and $P^{i6}=P_i$, here 
for any $i,j$ such that $1\leq i\leq j\leq 5$. 
\begin{rem}
There is a nice uniform interpretation of the planes $P^{u,v}$ for any $1\leq u<v\leq 6$ 
in terms of Knudsen-Mumford modular compactification $\overline{\mathcal M}_{0,6}$. 
According to a well-known result of Kapranov, the blow-up of 
$\widetilde{\boldsymbol{S}}$ along the union of the strict transforms  by $\pi$ of the lines $L_{i,j}$ is isomorphic to $\overline{\mathcal M}_{0,6}$ and 
any  plane $P^{uv}$ in Segre's cubic is the image by the morphism $\overline{\mathcal M}_{0,6}\rightarrow \boldsymbol{S}$ associated to $\varphi_3$, of the boundary divisor 
$\Delta_{uv} $ of $ \overline{\mathcal M}_{0,6}$ (naturally isomorphic to $\overline{\mathcal M}_{0,5}$) formed by limits of 
configurations of 6 points on the projective line obtained when 
making the $u$-th and $v$-points of the configurations collapse. 
A generic point of $\Delta^{uv}$ is a rational curve $C_1\cup C_2$ with two irreducible components $C_1,C_2$ intersecting at a node, with three marked points on $C_1$, namely the $u$-th,  the $v$-th and the node.
%
%
%
\end{rem}

One easily determines the incidence relations between the nodes of $\boldsymbol{S}$ and the 2-planes included in it: for any 
indices $i,j,k,l,m$ such that their set coincides with $\{1,\ldots,5\}$, one has 
\begin{itemize}
\item $P^{ij}\cap N_{\boldsymbol{S}}=\{\, p_{ij}, p_{kl}, p_{km}, p_{lm} \,\}$ and 
$P^{i6}\cap N_{\boldsymbol{S}}=\{\, p_{ij}, p_{ik}, p_{il}, p_{im} \,\}$; moreover, the four nodes in a plane included in $\boldsymbol{S}$ are in general position in it.
\sk 
\item The set of 2-planes in $\boldsymbol{S}$ adjacent to $p_{ij}$ is $ \{\, P^{ij}, P^{kl}, P^{km}, P^{lm}, P^{i6}, P^{j6}\, \}$.
\end{itemize}
\sk

Finally, for any distinct pairs $(i,j)$ and $(k,l)$ (with $1\leq i<j\leq 5$ and 
similarly for $k$ and $l$), the line  $\ell_{ij,kl}=\langle  p_{ij}, p_{kl} \rangle\subset \mathbf P^4$ is included in $\boldsymbol{S}$ as well. From the preceding point, it follows that these lines are  pairwise distinct  hence are 45 in number.  These lines serve to describe the intersections of two distinct  planes of Segre's cubic. For $i,j,k,l$ in $\{1,\ldots,5\}$ such that $\{i,j\}\neq \{k,l\}$, the intersection of $P^{ij}$ with $P^{kl}$ is either a point or a line: 
\begin{equation}
\label{Eq:takolo}
\begin{tabular}{l}
 $P^{ij}\cap P^{i6}=\{ \, p_{ij}\, \}$, $P^{i6}\cap P^{j6}=\{ \, p_{ij}\, \}$ and $P^{ij}\cap P^{ik}=\{ \, p_{lm}\, \}$ \hspace{0.27cm} (intersection of dim. 0);
 \vspace{0.1cm}
 \\
$P^{ij}\cap P^{kl}=\ell_{ij,kl}= \langle  \, p_{ij}, p_{kl} \rangle$, 
 $P^{ij}\cap P^{k6}=\ell_{kl,km}=\langle  \, p_{kl}, p_{km} \rangle$ 
 \hspace{0.2cm} (intersection of dim. 1).
\end{tabular}
\end{equation}


\begin{rem}
 It may be helpful to indicate how the lines $\ell_{ij,kl}$'s can be realized `at the source' (that is, in terms of the $\mathbf P^3$ serving as source of the rational parametrization $\varphi_3$ of $\boldsymbol{S}$). When $\{i,j\}\cap \{k,l\}$ is empty, $\ell_{ij,kl}$ is just the image of the line $P^{ij}\cap P^{kl}\subset \mathbf P^3$ by $\varphi_3$.  Giving an interpretation of $\ell_{ij,ik}$ for $i,j,k$ pairwise distinct is a bit more involved: the closure of the set of tangent directions at $p_i$ of the lines joining $p_i$ to 
 a generic point of $\langle p_j,p_k\rangle$ is a line in $\mathbf P(T_{p_i}\mathbf P^3)$. We denote by $l_{i,jk}$ the corresponding line in the exceptional divisor $E_i=\pi^{-1}(p_i)\subset \widetilde{\boldsymbol{S}}$. Then 
 an easy computation shows that, by restriction,  
 $\widetilde{\varphi}_3$ induces a linear isomorphism from $l_{i,jk}$ onto $\ell_{ij,ik}\subset \boldsymbol{S}$.\sk
 
Beware that our notations above for the lines through two nodes of $\boldsymbol{S}$ 
are lightly different from the ones, namely $L[(ij),(kl)]$, used  in \cite{Collino2}. In Collino's text, 
$i,j,k,l$ stand for indices ranging from 1 to 6 (with $\{i,j\}\neq \{k,l\}$ of course). 
When all fourth are less than or equal to 5, then Collino's $L[(ij),(kl)]$ coincides with the line we denote by $\ell_{ij,kl}$.  But if one among these indices is equal to 6, for example $l$, then one has instead 
$L[(ij),(k6)]=\langle p_{ik},p_{jk}\rangle=\ell_{ik,jk}$.
%
\end{rem}

\subsubsection{\bf Fano surface of Segre's cubic} 
\label{SSS:Fano-Segre-cubic}
Everything here is taken from \cite[\S2.2.8]{Collino2}.  
For a line $L$ in $\mathbf P^4$, we write $[L]$ for the point of the Grassmann manifold $G_1(\mathbf P^4)$ to which it corresponds.  From the preceding description of the nodes and of the planes in $\boldsymbol{S}$, one can deduce a combinatorial description of its Fano surface denoted by $\Sigma$, that we will often identify with its image by the  
Pl\"ucker embedding
\begin{equation}
\label{Eq:Plucker-F1(Segre)}
\Sigma=F_1(\boldsymbol{S})\subset G_1(\mathbf PV)
\hookrightarrow  \mathbf P \big(\wedge^2V\big) \simeq \mathbf P^{9}\, .
\end{equation}

It is a non irreducible surface with 21 irreducible components, all of which are rational and can be  described easily: 
\begin{itemize}
\item for $1\leq i< j\leq 6$, the set of lines of  $P^{ij}$ forms an irreducible component of $\Sigma$, denoted by  $\Sigma^{ij}$. Any such component is isomorphic to $\mathbf P^2$ and  there are 15 components of this type; 
\mk
\item  for $i=1,\ldots,6$,  the images by $\varphi_3$  of the lines through 
 $p_i$ in $\mathbf P^3$ if $i\leq 5$, or of the twisted cubics through all the $p_k$'s when $i=6$, give rise to an irreducible component of $\Sigma$, denoted by  $\Sigma(i)$.  For $i\leq 5$, it  identifies with the blow up at the $[L_{ij}]$'s (with $j\in \{1,\ldots,5\}\setminus \{i\}$) of the 
 projective plane parametrizing the lines of $\mathbf P^3$ passing through $p_i$.  Therefore  $\Sigma(i)$ is isomorphic  to $\overline{\mathcal M}_{0,5}$. It can be verified that the same holds true for $\Sigma(6)$ too;
\mk 
\item the images 
 of the components of $\Sigma$ by the Pl\"ucker embedding are not difficult to describe: 
 \sk
\begin{itemize}
\item 
\vspace{-0.3cm}
for any $i,j$ distinct, $\Sigma^{ij}$ is  the set of lines of a 2-plane  in $\mathbf PV\simeq \mathbf P^4$ hence its image in $\mathbf P^9$ is a 2-plane as well; 
\sk
\item  for each $i$,  the restriction of  \eqref{Eq:Plucker-F1(Segre)} to 
$\Sigma(i)$ coincides with the anticanonical embedding hence  
its image is a quintic del Pezzo surface (which spans a 5-plane in $\mathbf P^9$).
\end{itemize}
\end{itemize}
\sk

Our notations are unfortunately not well suited to deal with the birational $\mathfrak S_{6}$-symmetries of $\boldsymbol{S}$ and of its Fano surfaces. In order to describe the  incidence relations between the 21 components of $\Sigma$,  
it is much more convenient to use Collino's notations  $L[(ij),(kl)]$ 
(recalled above) 
for the lines joigning two  nodes of $\boldsymbol{S}$. 
Then the intersections between the different components of $\Sigma$ 
can be 
seen to be as follows where now $i,j,k,l,m,n$ stand for elements ranging in $\{1,\ldots,6\}$ 
(with the pairs $\{i,j\}$ and $\{k,l\}$ distinct, etc) : 
\begin{itemize}
\item[$-$] the relations  \eqref{Eq:takolo} can be stated in equivalent terms for the $\Sigma^{ij}$'s: the two planes  $\Sigma^{ij}$ and $ \Sigma^{kl}$ do not meet if 
 $\{i,j\}\cap \{k,l\}\neq \emptyset $ whereas one has 
$\Sigma^{ij}\cap \Sigma^{kl}=\{ \,L[(ij),(kl)] \, \}$ otherwise;
\sk  
\item[$-$] as for the intersection of two del Pezzo components,  
when $i$ and $j$ are distinct then $\Sigma(i)\cap\Sigma(j)$ is the set of three points of $\Sigma$ 
corresponding to the lines $L[(kl),(mn)]$ for all $k,l,m,n$ such that $\{i,j,k,l,m,n\}=\{1,\ldots,6\}$;
\mk  
\item[$-$] there are two possibilities for the intersection of a plane $\Sigma^{ij}$ with a del Pezzo component $\Sigma(k)$:  first,  one obviously has $\Sigma^{ij}\cap \Sigma(k)=\emptyset$
if $k\in \{i,j\}$.   At the opposite, when $i,j$ and $k$ are pairwise distinct, then  $\Sigma(k)$ intersects 
$\Sigma^{ij}$ along a line denoted by $R(k,ij)$, which is isomorphic to $\overline{\mathcal{M}}_{0,4}$.  In our non symmetric presentation,  for $k<6$ the line  $R(ij,k)$ corresponds to the (images by $\varphi_3$ of the) lines joining $p_k$ to a point of $L_{lm}=\langle p_l,p_m\rangle \subset \mathbf P^3$ and the three boundary points  correspond to the lines
passing through $p_l$, $p_m$ and the point of intersection  in $\mathbf P^3$ of the line $L_{ij}=\langle p_i,p_j\rangle$ with the plane $P^{ij}=\langle p_k,p_l,p_m\rangle$. 
\mk  
\end{itemize}
A noteworthy last point concerns the incidence between the lines joining two nodes of $\boldsymbol{S}$ but considered as points of the Fano surface $\Sigma$, and the 2-dimensional components of the latter:  using Collino's notation, 
when $i,j,k,l,m,n$ stand for the elements of $\{1,\ldots,6\}$ then the two del Pezzo's
$\Sigma(m)$, $\Sigma(n)$ and the two planes $\Sigma^{ij}$ and $\Sigma^{kl}$ are exactly 
 the components of $\Sigma$ to which $L[(ij),(kl)]\in \Sigma$ belongs.  Essentially  quoting Collino here, one concludes by saying that 
{\it `this shows that $\Sigma=F(\boldsymbol{S})$ is not a normal crossing surface: at any point $L[(ij),(kl)]$, four irreducible components of $\Sigma$ meet, and moreover these components split in two couples which intersect locally only at
this point'}.

\subsubsection{\bf Abelian relations and rank of $\mathcal W_{{0,6}}$}
\label{SSS:2-ARs-W-06}
We have seen above that $\boldsymbol{\mathcal W}_{0,6} $ identifies with the 6-web by lines 
$\boldsymbol{\mathcal L\hspace{-0.05cm} \mathcal W}_{\boldsymbol{S}}$ on Segre's cubic. 
Our goal here is to establish that Proposition C of the Introduction, which is stated for a `sufficiently generic' hypercubic of $\mathbf P^4$, holds true for $\boldsymbol{S}$ 
as well: 
\begin{prop} 
\label{P:2-rank-W06}
1. The trace induces an isomorphism ${\rm Tr}: \boldsymbol{H}^0\big(\Sigma,\omega_\Sigma^2\big)\simeq 
\boldsymbol{AR}\big( 
\boldsymbol{\mathcal L\hspace{-0.05cm} \mathcal W}_{\boldsymbol{S}}
\big)$. 
\vspace{0.05cm}

2. Therefore $\boldsymbol{\mathcal W}_{0,6}$
is an algebraizable 6-web of maximal 2-rank $
{\bf rk}\big(\boldsymbol{\mathcal W}_{0,6}\big)=
h^0\big(\Sigma,\omega_\Sigma^2\big)=10$. 
\end{prop}



\paragraph{\bf An elementary and explicit proof of Proposition \ref{P:2-rank-W06}}
The key fact behind the elementary effective approach  discussed below is   
point 2.\,in \eqref{Eq:X-singular-Properties}, namely that  the dualizing sheaf $\omega_F^2$ of the Fano 
surface $F$ of a cubic threefold $X$ with finitely many ordinary singular points 
 coincides with the pull-back of $\mathcal O_{\mathbf P^9}(1)$ under the
   Pl\"ucker 
embedding $F\subset G_1(\mathbf P^4)\hookrightarrow \mathbf P^9$. 
Since the dualizing sheaf $\omega_F$ naturally identifies with the sheaf  $\omega^2_F$ of abelian differential 2-forms on $F$ ({\it cf.}\,\cite{Barlet}), the aforementioned  fact allows to get 
explicit expressions for a basis of $\boldsymbol{H}^0(F,\omega^2_{F})$ 
as soon as the embedding $F\hookrightarrow \mathbf P^9$ is explicitly known.  One is reduced to verify whether the elements of this basis give rise to abelian relations for 
the associated linear web $\boldsymbol{\mathcal L\hspace{-0.05cm} \mathcal W}_{{X}}$, which can be done through direct computations.  In the sequel, we apply this approach to 
Segre's cubic $\boldsymbol{S}$ in order to describe 
the ARs then to compute  the rank of $\boldsymbol{\mathcal W}_{0,6}\simeq \boldsymbol{\mathcal L\hspace{-0.05cm} \mathcal W}_{{\boldsymbol{S}}}$. 

Recall that to a rank 2 matrix $M\in {\rm Mat}_{2,5}(\mathbf C)$, one can associate the line in $\mathbf P^4$, denoted by $\langle M\rangle $, passing through the two points whose homogeneous coordinates are the two lines of $M$. hereafter, we will use the well-known fact that    
$
{\rm Mat}_{2,5}(\mathbf C)\dashrightarrow G_1(\mathbf P^4)$, $
M\mapsto \langle M\rangle $ is a birational chart.
\sk


Let $(x,y,z)$ stand for the coordinates of a generic point $p$ of $\mathbf P^3$;  here  are the rational first integrals for $\boldsymbol{\mathcal W}_{0,6}$ we will work with: 
$U_1=(y,z)$, $U_2=(x,z)$, $U_3=(x,y)$, $U_4=(x/z,y/z)$, $U_5=((x-1)/(z-1),(y-1)/(z-1))$ and  
$U_6=(x(z-1)/(z(x-1)) ,y(z-1)/(z(y-1)))$. \sk 

 For each $i=1,\dots,6$, solving $U_i(x,y,z)=(u_i,v_i)$  with $u_i,v_i\in \mathbf Q(x,y,z)$, one gets easily a parametrization $\xi_{i,p}: \mathbf P^1\rightarrow \mathbf P^3 $ of the $i$-th leaf of $\boldsymbol{\mathcal W}_{0,6}$ through $p$, whose coefficients are rational in $u_i$ and $v_i$. Then evaluating $\varphi\circ \xi_{i,p}: \mathbf P^1\rightarrow \mathbf P^4$ at two generic points, for instance $0$ and $\infty$, one easily computes a $2\times 5 $ matrix $M_i\in {\rm Mat}_{2,5}\big(  \mathbf Q(u_i,v_i)
 \big)$ such that the $i$-th leaf of $\boldsymbol{\mathcal L\hspace{-0.05cm} \mathcal W}_{{\boldsymbol{S}}}$ through $\varphi_3(p)$ is the line passing through the two points whose homogeneous coordinates are the two lines of $M_i$.   Of course, the $M_i$'s are not unique but 
some elementary computations give that they 
 can be taken to be the following ones: 
\begin{align*}
M_1=& \,
 \begin{bmatrix}
0 & 0& u_1v_1& u_1& v_1 \\
u_1-v_1& u_1-1& u_1& u_1& u_1
\end{bmatrix}  \\ 
M_2=& \, \begin{bmatrix}
u_2v_2& u_2& 0& 0& v_2 \\
u_2& u_2& u_2-v_2 & u_2-1 & u_2
\end{bmatrix}
\\
M_3=&\, \begin{bmatrix}
u_3v_3 &  u_3( v_3-1) &  u_3v_3 & v_3(u_3-1) &  v_3 \\
u_3 & 0 & v_3 & 0 & 1
\end{bmatrix}
\\
 M_4=&\, \begin{bmatrix}
0 & u_4& 0 & v_4& 1 \\
u_4(1-v_4)& -u_4v_4& v_4(1-u_4)& -u_4v_4& - u_4v_4
\end{bmatrix}
 \\
M_5=&\, 
\begin{bmatrix}
u_5(v_5 - 1)&  u_5v_5& v_5(u_5 - 1)& u_5v_5
& u_5v_5 \\ 
(v_5 - 1)(1 - u_5)& v_5(1 - u_5)& (v_5 - 1)(1 - u_5)
&   u_5(1-v_5 )& (v_5 - 1)(1 - u_5)
\end{bmatrix}
%
\\
\mbox{and } \quad M_6=&\, 
\begin{bmatrix}
 0 &  u_6 & 0& v_6 & 1 \\
 u_6(v_6 - 1)&  0& v_6(u_6-1) &  0&  (v_6 - 1)(u_6-1) 
\end{bmatrix}\, . 
\end{align*}

The map $\mathbf C^2\dashrightarrow \Sigma(i)$, $(u_i,v_i)\mapsto \langle M_i\rangle$ is easily seen to be birational, from which it follows that $(u_i,v_i)$ can be taken as global rational coordinates on $\Sigma(i)$ for any $i=1,\ldots,6$.  It is then immediate to compute a parametrization of the image $\widetilde \Sigma(i)$  of 
the component $\Sigma(i)\subset \Sigma$ by the Pl\"ucker embedding $\rho : \Sigma\subset G_1(\mathbf P^4)\hookrightarrow \mathbf P^9$. Indeed,  denoting by $\Delta_{a,b}(M_i)$ 
the $2\times 2$ minor of $M_i$ corresponding to its $a$-th and  $b$-th columns (in this order) for  any $(a,b)$ such that $1\leq a<b \leq 5$, the aforementioned parametrization is given by 
\begin{align*} 
\kappa_i : \big(u_i,v_i\big)    \dashrightarrow  \Big[ \Delta_{a,b}(M_i) \Big]_{1\leq a < b \leq 5}\in \mathbf P^9\, .
\end{align*}

Since $\kappa_i$ corresponds to the expression, in the rational coordinates 
$u_i,v_i$, of 
the restriction on $\Sigma(i)$ of the canonical map $\kappa=\kappa_{\lvert \omega^2_\Sigma
\lvert 
}: \Sigma \dashrightarrow \mathbf PH^0(\Sigma,\omega^2_\Sigma)^\vee$, there exists a rational function $\sigma_i=\sigma_i(u_i,v_i)$ (to be made explicit further below) such that  
$$
\Big\{\, \Delta_{a,b}(M_i)  \, \big( \sigma_i \, du_i\wedge dv_i\big) \, \big\lvert \,  1\leq a <b\leq b\, \Big\}
$$
spans of the space of rational 2-forms $\iota_i^* \big(\boldsymbol{H}^0(\Sigma,\omega_\Sigma^2)\big)$ on $\Sigma(i)$.\footnote{More prosaically, $\iota_i^* \big(\boldsymbol{H}^0(\Sigma,\omega_\Sigma^2)\big)$ is the image of the restriction map 
$\boldsymbol{H}^0\big(\Sigma,\omega_\Sigma^2\big)\rightarrow \Omega^2_{\mathbf C(\Sigma(i))}$, $\omega\mapsto \omega\lvert_{\Sigma(i)}$.} Actually, working a bit further, one verifies that 
\begin{enumerate}
\item for any $i$, the composition $\rho \circ \iota_i : \Sigma(i)\subset \Sigma \rightarrow \mathbf P^9$ is given by the complete anticanonical linear system $\lvert
K_{\Sigma(i)}^{-1}
\lvert $, hence $\widetilde \Sigma(i)=\rho(\Sigma(i))$ 
is a quintic del Pezzo surface which spans 
 a $\mathbf P^5$ in the ambiant   $\mathbf P^9$;\sk 
\item the union of the $\widetilde \Sigma(i)$'s spans the whole ambiant projective space, that is  $\big\langle 
\cup_{i=1}^6 \widetilde \Sigma(i)\big\rangle= \mathbf P^9$. It follows that  the linear map 
$\boldsymbol{H}^0\big(\Sigma,\omega_\Sigma^2\big)  \longrightarrow \oplus_{i=1}^6 \Omega^2_{\mathbf C(\Sigma(i))}$, $\omega  \longmapsto  \big( \omega\lvert_{\Sigma(i)}\big)_{i=1}^6$
is injective and consequently, the set of 6-tuples of rational 2-forms 
\begin{equation}
\label{Eq:basis}
\bigg\{\,  \Big( \Delta_{a,b}(M_i)  \, \big( \sigma_i \, du_i\wedge dv_i\big) \Big)_{i=1}^6 
\in \oplus_{i=1}^6 \Omega^2_{\mathbf C(\Sigma(i))}
\, \big\lvert \,  1\leq a <b\leq b\, \bigg\}
\end{equation}
 can be identified with a basis of the space of global abelian differential 2-forms on $\Sigma$. 
\end{enumerate}\mk 
At this point, it is straightforward to give an explicit expression in the rational coordinates $x,y,z$,  for the trace of the 2-form $\omega_{a,b}\in \boldsymbol{H}^0\big(\Sigma,\omega_\Sigma^2\big)$ corresponding to the  6-tuple in \eqref{Eq:basis} associated to the $(a,b)$-minor: 
for any $a,b$ such that  $1\leq a<b\leq 5$, one has 
\begin{equation}
\label{Eq:TrOmega-ab}
{\rm Tr}\big( \omega_{a,b}\big)=\sum_{i=1}^6 \Delta_{a,b}(M_i)  \, \big( \sigma_i \, du_i\wedge dv_i\big)
\end{equation}
where all the $u_i$'s and $v_i$'s are now  viewed as rational functions of $x,y,z$, hence 
any term of the right-hand side sum has to be considered as a rational 2-form  in the variables $x,y$ and $z$. \mk 

An easy computation gives that the only 6-tuple of functions $(\sigma_i)_{i=1}^6 $ with $\sigma_i=\sigma_i(u_i,v_i)$ such that 
the RHS of any expression \eqref{Eq:TrOmega-ab} vanishes identically is a nonzero constant multiple of 
$$
\left( \frac{1}{m_5(u_i,v_i)}\right)_{i=1}^6 
$$
where $m_5$ stands for the `$\mathcal M_{0,5}$-polynomial' defined by $m_5(u,v)=uv(u-1)(v-1)(u-v)$. 
\mk 

We thus have obtained the following result which has to be seen as an explicit version of 
Proposition \ref{P:2-rank-W06}: 
\begin{prop}
In the coordinates $x,y,z$ and relatively to the first integrals $U_i=(u_i,v_i)$, $i=1,\ldots,6$, the ARs of $\boldsymbol{\mathcal W}_{0,6}$ are all rational and their space admits as a basis the set of $6$-tuples of rational 2-forms (for all $a,b$ such that $1\leq a < b \leq 5$)
\begin{equation}
\label{AR:DDab}
\left(  \frac{ \Delta_{a,b}(M_i) }{m_5(u_i,v_i)} \Big( du_i\wedge dv_i\Big)
\right)_{i=1}^6 
\end{equation}
which all satisfy the  relation $
\sum_{i=1}^6 
\big({ \Delta_{a,b}(M_i) }/{m_5(u_i,v_i)}\big) \big( du_i\wedge dv_i\big)=0  
$ identically.
\end{prop}

Since the $U_i$'s as well as the $M_i$'s are all explicit, is it just an elementary computational matter to make the ARs \eqref{AR:DDab} explicit. For instance, the one corresponding to $(a,b)=(1,2)$ is  
\begin{align*}
\frac{ u_2\, du_2\wedge dv_2}
 {
 v_2(u_2-1)(u_2-v_2)}
- \frac{ 
  u_3 \, du_3\wedge dv_3} 
  {v_3(u_3-1)(u_3 - v_3)}+  
\frac{   u_4\, du_4\wedge dv_4}{v_4(u_4-1)(u_4 - v_4)} 
-\frac{u_6\, du_6\wedge dv_6}{v_6(u_6-1)(u_6 - v_6)}
=0
\end{align*}
and that associated to $(a,b)=(4,5)$ is 
\begin{align}
\label{Eq:RA-Delta-45}
\frac{du_1\wedge dv_1}{v_1(u_1-1)(v_1 - 1)}
 -& \frac{du_2\wedge dv_2}{ u_2(v_2 - 1)(u_2-v_2)}
+
  \frac{du_3\wedge dv_3}{ u_3(v_3-1)(u_3 - v_3)} \\
   - & \frac{du_4\wedge dv_4}{(u_4 - 1)(u_4 - v_4)}
+
 \frac{du_5\wedge dv_5}{(u_5-1)(u_5 - v_5)}
+
 \frac{du_6\wedge dv_6}{u_6(u_6 - v_6)}
 =0 \, .\nonumber 
 \end{align} 
 \mk



Realizing that all the ARs of $\boldsymbol{\mathcal W}_{0,6}$ were rational in some natural coordinates led us to doubt about the validity of the decomposition 
\eqref{Eq:Direct-Sum} of $\boldsymbol{AR}(\boldsymbol{\mathcal W}_{0,6})$ as a $\mathfrak S_6$-module claimed in \cite{D}.  In order to clear this up,  we have investigated 
the birational action of $\mathfrak S_6$ on the ARs \eqref{AR:DDab}
 to come to the conclusion that there was indeed an error on this point.


\subsubsection{\bf The representation of $\mathfrak S_{6}$ on ${AR}_C(\mathcal W_{{0,6}})$}
\label{SS:ARW06-as-a-S6-module}
We use below notations 
almost identical to those used in the preceding subsection.  
Here $x,y,z$ stand for the rational coordinates corresponding to the map $\mathbf C^3\dashrightarrow \boldsymbol{\mathcal M}_{0,6}$, $(x,y,z)\mapsto [\infty,0,1,x,y,z]$. 
The first integrals we work with here for $\boldsymbol{\mathcal W}_{{0,6}}$
are the  six rational first integrals $U_i=(u_i,v_i): \mathbf C^3 \dashrightarrow \mathbf C^2$ 
of the previous subsection, with the only difference that we have exchanged $U_1$ and $U_3$ (namely, one takes $U_1= (x,y)$ and $U_3=(y,z)$ below while the others $U_i$'s are left unchanged).
\mk 

The associated normals $\Omega_i=du_i\wedge dv_i$  for $i=1,\ldots,6$ are easy to compute: the sheaf of 2-forms on $\mathbf C^2$ is globally spanned by 
$\Omega_1= dx\wedge dy$, $\Omega_2=dx\wedge dz$ and $\Omega_3=dy\wedge dz$ 
 and as global rational 2-forms,  the others $\Omega_i$'s are determined by the following matricial relation: 
\begin{equation}
\label{Eq:RelationsOmega-i}
\begin{bmatrix}
\Omega_4\\
\Omega_5\\
\Omega_6
\end{bmatrix}
=
\begin{bmatrix}
 \frac{1}{z^2}& -\frac{y}{z^3}& \frac{x}{z^3} \\
 \frac{1}{(z-1)^2}& -\frac{y-1}{(z-1)^3}&
\frac{x-1}{(z-1)^3} \\ 
\frac{(z-1)^2}{(x-1)^2(y-1)^2z^2}& -\frac{y(z-1)}{(x-1)^2(y-1)z^3}
&
\frac{x(z-1)}{(x-1)(y-1)^2 z^3}
\end{bmatrix} \, \begin{bmatrix}
\Omega_1\\
\Omega_2\\
\Omega_3
\end{bmatrix}
\end{equation}

The transformations  
$$
\mathscr T: \, 
[p_i]_{i=1}^6 \longmapsto [p_1,p_2,p_3,p_5,p_4,p_6]
\qquad \mbox{ and } 
\qquad 
\mathscr C:\,  [p_i]_{i=1}^6 
\longmapsto [p_6,p_1,\ldots,p_5]
$$
are 
two automorphisms of $\boldsymbol{\mathcal M}_{0,6}$ 
which correspond to the transposition  $\tau=(45)$ and to the cycle $c=(1,\ldots,6)$ of $\mathfrak S_6$ respectively. The corresponding rational maps in the affine coordinates $(x,y,z)$ are given respectively by  
\begin{equation}
\label{Eq:TC}
\mathcal T: (x,y,z)\mapsto (y,x,z) 
\qquad \mbox{ and } 
\qquad 
\mathcal C: (x,y,z)\mapsto \left(  \frac{1}{x}\, , \, \frac{x-z}{(x-1)z}\, , \, \frac{y-z}{(y-1)z}
\right)\, .
\end{equation}

Clearly,  $\mathcal T$ and $\mathcal C$ generate the image  of 
${\rm Aut}(\boldsymbol{\mathcal M}_{0,6}) \simeq \mathfrak S_6$ in ${\bf Bir}_3$.
Their actions under pull-backs  on the $\Omega_i$'s can be easily computed. One verifies that they are given by the following relations:
\begin{align}
\label{Eq:Pull-Backs}
\Big( \mathcal T^*\big(\Omega_i\big)\Big)_{i=1}^6= & \, 
\Big( \, -\Omega_1\, , \, \Omega_3\, , \, \Omega_2\, , \, 
-\Omega_4\, , \,
-\Omega_5\, , \,
-\Omega_6\, 
\Big) \\
\mbox{and }\quad 
\Big( \mathcal C^*\big(\Omega_i\big)\Big)_{i=1}^6= & \, 
\Big( \, 
\eta_2\,\Omega_2\, , \, \eta_3\, \Omega_3\, , \, \Omega_6\, ,\,    
\widetilde \eta_5\, 
\Omega_5
\, , \,  
\eta_1\,
\Omega_1\, , \, 
\widetilde \eta_4\, 
\Omega_4
\Big)\, ,  \nonumber
\end{align}
with $\eta_i=\eta(U_i)
 =\eta(u_i,v_i)$ for any $i$ (and similarly for $\widetilde \eta_i$), 
 where $\eta$ and $\widetilde \eta$ are standing for the rational functions  
 $\eta(u,v)=(v-1)/(v^3(u-1)^2)$ and  $\widetilde \eta(u,v)=\eta(1+u,1-v)
=v/(u^2(v-1)^3)$. 
%
\sk 

From \eqref{Eq:Pull-Backs},  it follows  that the action 
on $\boldsymbol{\mathcal W}_{0,n+3}=(\mathcal F_i)_{i=1}^6$
induced by $\mathcal T$ and $\mathcal C$ 
are given respectively by the transposition $\tau=(23)$ and the 6-cycle $c=(123645)$: for $i=1,\ldots,6$, one has 
\begin{equation}
\label{Eq:tau-and-c}
\mathcal T^*\big(\mathcal F_i\big)=\mathcal F_{\tau(i)} 
\qquad \mbox{ and }\qquad \mathcal C^*\big(\mathcal F_i\big)=\mathcal F_{c(i)}\, .
\end{equation}


\paragraph{\bf An explicit combinatorial basis of  $\boldsymbol{AR}(\boldsymbol{\mathcal W}_{0,6})$.}  
Some of the ARs \eqref{AR:DDab} are combinatorial but not all of them ({\it cf.}\,\eqref{Eq:RA-Delta-45}).  But since these abelian relations form a basis of $\boldsymbol{AR}(\boldsymbol{\mathcal W}_{0,6})$, it is just a matter of computing to construct all the combinatorial ARs of 
$\boldsymbol{\mathcal W}_{0,6}$. We want to do that explicitly here.\sk

We set $I=\{1,\ldots,6\}$ and $I^2_\neq$ stands for the set of all pairs $(i,j)\in I^2$ such that $i<j$.  For any $(i,j)$ in this set, one denotes by $AR_{ij}$ the non trivial AR for the 4-subweb $\boldsymbol{\mathcal W}_{\widehat{\imath \jmath}}=\boldsymbol{\mathcal W}(U_k \, \big\lvert \, k\in I\setminus \{i,j\}\,\big)$ of $\boldsymbol{\mathcal W}_{0,6}$. Actually, $AR_{ij}$ is only defined 
up to multiplication by a non zero scalar but this ambiguity 
will  be irrelevant with regard to what we are going to  discuss in the sequel. \mk 

As it can easily be verified, each row of the matrix page \pageref{Page:AR-matrix}, denoted by $M_{AR}$ in what follows, encodes an abelian relation $AR_{ij}$ (where therefore $i$ and $j$ correspond to the zero coefficients):  denoting by $m_{ab}$ (for $a$ and $b$ ranging from 1 to 10 and 6 respectively) its coefficients, then $m_{ab}$ is a function of $U_b$ for any $b$ (hence one can write $m_{ab}=m_{ab}(U_b)$) and for any $a=1,\ldots,10$, the following relation  $ \sum_{b=1}^{10} m_{ab}(U_b)\, \Omega_b= 0 $ holds true identically. 

One denotes by $AR(a)$ the combinatorial AR corresponding to 
the previous relation
(hence $AR(1)=AR_{13}$, $\ldots$, $AR(10)=AR_{56}$). By elementary linear algebra,  one verifies that the $AR(a)$'s for $a=1,\ldots,10$ are linearly independent abelian relations, from which we get the 
\begin{prop} 
\label{Prop:tololo}
1. The space $\boldsymbol{AR}_C(\boldsymbol{\mathcal W}_{0,6})$ has dimension 10.\sk 

2. Consequently $\boldsymbol{AR}(\boldsymbol{\mathcal W}_{0,6})=\boldsymbol{AR}_C(\boldsymbol{\mathcal W}_{0,6})$ and $\boldsymbol{\mathcal W}_{0,6}$ has maximal rank.
\end{prop}

We remark that the five combinatorial abelian relations $AR_{ij}$ which do not correspond to one of the lines  of 
$M_{AR}$ 
correspond to the pairs (12), (23), (36), (46) and (45) hence one deduces a more invariant way 
to describe the basis of $\boldsymbol{AR}(\mathcal W_{0,6})$ formed by the $AR(a)$'s for $a=1,\ldots, 10$: 
\begin{fact} 
The family of $AR_{ij}$'s for all pairs $(i,j)\in I^2_\neq $  
minus the five abelian relations $(\mathcal C^\ell)^*(AR_{12})$ for $\ell=0,\ldots,4$, 
form a basis of $\boldsymbol{AR}(\mathcal W_{0,6})$.
\end{fact}
The interest of formulating things that way is that it naturally  asks whether  the following straightforward
 generalization holds true or not: 
\begin{question} For any odd integer $n\geq 2$, does the family of $AR_{ij}$'s for all pairs $(i,j)\in I^2_\neq $  
minus the $n+2$ abelian relations $(\mathcal C^\ell)^*(AR_{12})$ for $\ell=0,\ldots,n+1$, 
form a basis of $\boldsymbol{AR}(\boldsymbol{\mathcal W}_{0,n+3})$?
\end{question}

\paragraph{\bf The representation of $\mathfrak S_6$ on $\boldsymbol{AR}(\boldsymbol{\mathcal W}_{0,6})$.} 
Looking at  \eqref{Eq:tau-and-c} and because the space of ARs of $\boldsymbol{\mathcal W}_{\widehat{\imath \jmath}}$ is 1-dimensional for any $(i,j)\in I^2_\neq$,  it is clear that 
$$\mathcal T^*(AR_{ij})=\lambda_{ij} \,AR_{\tau(i)\tau(j)}
\qquad \mbox{ 
 and } \qquad 
\mathcal C^*(AR_{ij})=\gamma_{ij} \,AR_{c(i)c(j)}$$ 
for some non zero scalars 
$\lambda_{ij}$ and $\gamma_{ij}$ which are easy to make explicit using formulas \eqref{Eq:Pull-Backs}. 
One gets that relatively to the basis $(AR(a))_{a=1}^{10}$, the actions  
on $\boldsymbol{AR}(\boldsymbol{\mathcal W}_{0,6})$
induced by the pull-backs under $\mathcal T$ and $\mathcal C$ (defined in \eqref{Eq:TC}) are given by the two  following matrices: \sk
$$
T= \scalebox{0.7}{
$\begin{bmatrix}
{}^{}\hspace{0.1cm}1&1&-1&-1&0&0&0&0&0&0\hspace{0.1cm}{}^{}\\ 
{}^{}\hspace{0.1cm}0&-1&0&0&0&0&0&0&0&0\hspace{0.1cm}{}^{}\\ 
{}^{}\hspace{0.1cm}0&0&-1&0&0&0&0&0&0&0\hspace{0.1cm}{}^{}\\
{}^{}\hspace{0.1cm}0&0&0&-1&0&0&0&0&0&0\hspace{0.1cm}{}^{}\\
{}^{}\hspace{0.1cm}0&0&0&0&0&0&0&1&0&0\hspace{0.1cm}{}^{}\\
{}^{}\hspace{0.1cm}0&0&0&0&0&0&0&0&1&0\hspace{0.1cm}{}^{}\\
{}^{}\hspace{0.1cm}0&1&-1&-1&-1&1&1&1&-1&0\hspace{0.1cm}{}^{}\\
{}^{}\hspace{0.1cm}0&0&0&0&1&0&0&0&0&0\hspace{0.1cm}{}^{}\\
{}^{}\hspace{0.1cm}0&0&0&0&0&1&0&0&0&0\hspace{0.1cm}{}^{}\\
{}^{}\hspace{0.1cm}0&0&0&0&0&0&0&0&0&-1\hspace{0.1cm}{}^{}
\end{bmatrix}$}
\hspace{0.6cm}
\mbox{ and }\hspace{0.4cm}
C=  \scalebox{0.7}{
$
\begin{bmatrix}
{}^{}\hspace{0.1cm}0& 0& 1& 0& 0& 1& 0& 0& 0& 0\hspace{0.1cm}{}^{}\\
{}^{}\hspace{0.1cm}0& 0& 1& 0& 0& 0& 0& 0& 0& -1\hspace{0.1cm}{}^{}\\
{}^{}\hspace{0.1cm}0& 0& -1& 0& 0& 0& 0& 0& 0& 0\hspace{0.1cm}{}^{}\\
{}^{}\hspace{0.1cm}0& 0& -1& 0& 0& 0& 0& 0& 1& 0\hspace{0.1cm}{}^{}\\
{}^{}\hspace{0.1cm}0& 0& 0& -1& 0& 0& 0& 0& 0& 0\hspace{0.1cm}{}^{}\\
{}^{}\hspace{0.1cm}0& -1& 0& 0& 0& 0& 0& 0& 0& 0\hspace{0.1cm}{}^{}\\
{}^{}\hspace{0.1cm}-1& 0& 0& 0& 0& 0& 0& 0& 0& 0\hspace{0.1cm}{}^{}\\
{}^{}\hspace{0.1cm}0& 0& 0& 0& 0& 0& -1& 0& 0& 0\hspace{0.1cm}{}^{}\\
{}^{}\hspace{0.1cm}0& 0& 0& 0& -1& 0& 0& 0& 0& 0\hspace{0.1cm}{}^{}\\
{}^{}\hspace{0.1cm}0& 0& 0& 0& 0& 0& 0& 1& 0& 0\hspace{0.1cm}{}^{}
\end{bmatrix}$}\, .\bk 
$$

Their traces are given by 
$$
{\rm Tr}(T)=-2 \qquad \mbox{ and } \qquad 
{\rm Tr}(C)=-1 \, . 
$$

The representation of $\langle \mathcal T,\mathcal C
\rangle\simeq \mathfrak S_6$ on  the  vector space 
$\boldsymbol{AR}_C(\boldsymbol{\mathcal W}_{0,6})=\boldsymbol{AR}(\boldsymbol{\mathcal W}_{0,6})$ 
of dimension 10 
 is isomorphic to the one defined by the group morphism $\rho: \mathfrak S_6\rightarrow {\rm GL}_{10}$ charaterized by 
 $$\rho(\tau)=T\qquad \mbox{ and } \qquad \rho(c)=C\,.$$

 \begin{minipage}{100mm}
  \centering
\rotatebox{90}{\scalebox{1.4}{
$\begin{bmatrix}
 {}^{}\quad  0 & {\frac {v_2-u_2}{u_2 v_2   \left( u_2-1 \right)\left( v_2-1 \right)}}&0&{\frac {{ u_4}-1}{{ u_4}\,  \left( {u_4}-{v_4} \right) \left( {\it v_4}-1 \right) }}&{\frac {1-{ u_5}}{{ u_5}\, \left( {u_5}-{ v_5} \right)  \left( { v_5}-1 \right) }}& {\frac {1-{ u_6}}{{ u_6}\, \left( {u_6}-{v_6} \right) \left( {v_6}-1 \right) }} \quad 
\bk \\
 {}^{}\quad 0&{\frac {v_2-1}{ \left( u_2-1 \right)  \left( u_2-v_2 \right) v_2}}&{\frac {1-v_3}{ \left( u_3-1 \right)  \left( u_3-v_3 \right) v_3}}&0&{\frac {-1}{ \left( {\it u_5}-1 \right) \left( {\it v_5}-1 \right) {\it u_5}\,{\it v_5}}} &{\frac {1}{ \left( {\it v_6}-1 \right)  \left( {\it u_6}-1 \right) }}   \quad 
\bk \\
 {}^{}\quad  0&{\frac {v_2}{ u_2 \left( u_2-v_2 \right)  \left( v_2-1 \right) }}& {\frac {-v_3}{u_3 
\left( u_3-v_3 \right) 
\left( v_3-1 \right)   }}&{\frac {-1}{ u_{{4}}v_{{4}} \left( u_{{4}}-1 \right) \left( v_{{4}}-1 \right) }}&0&{\frac {1}{ u_{{6}} v_{{6}} \left( u_{{6}}-1 \right) \left( v_{{6}}-1 \right) }}  \quad 
\bk \\
 {}^{}\quad  0&{\frac {-1}{v_2  \left( u_2-v_2 \right)  \left( v_2-1 \right) }} & {\frac {1}{ v_3 \left( u_3-v_3 \right)  \left( v_3-1 \right)   }} & {\frac {1}{ \left( u_{{4}}-1 \right)  \left( v_{{4}}-1 \right)  }} & {\frac {-1}{ \left( u_{{5}}-1 \right)  \left( v_{{5}}-1 \right) }} &0  \quad 
\bk \\
 {}^{}\quad  
{\frac {v_1-1}{v_1 \left( u_1-1 \right)  \left( u_1-v_1 \right) }} & 0& {\frac {1-u_3}{u_3   \left( u_3-v_3 \right) \left( v_3-1 \right)}}& 0& {\frac {-v_{{5}}}{u_{{5}} \left( u_{{5}}-v_{{5}} \right) 
\left( v_{{5}}-1 \right) }} & {\frac {1}{v_{{6}} \left( u_{{6}}-v_{{6}} \right)  \left( v_{{6}}-1 \right) }}  \quad 
\bk \\
 {}^{}\quad  {\frac {v_1}{u_1 \left( v_1-1 \right)  \left( u_1-v_1 \right) }}&0&{\frac {-u_3}{ v_3\left( u_3-1 \right)  \left( u_3-v_3 \right) }}&{\frac {-v_{{4}}}{u_{{4}} \left( u_{{4}}-v_{{4}} \right) \left( v_{{4}}-1 \right) }} &0&{\frac {v_{{6}}}{u_{{6}}\left( u_{{6}}-v_{{6}} \right)  \left( v_{{6}}-1 \right)   }}  \quad 
\bk \\
 {}^{}\quad  {\frac {-1}{ v_1 \left( u_1-v_1 \right)  \left( v_1-1 \right)  }}&0&{\frac {1}{ u_3 \left( u_3-1 \right)   \left( u_3-v_3 \right)}}&{\frac {1}{v_{{4}}  \left( u_{{4}}-v_{{4}} \right) \left( v_{{4}}-1 \right)  }}&{\frac {-1}{v_{{5}}\left( u_{{5}}-v_{{5}} \right)   \left( v_{{5}}-1 \right)   }}&0   \quad 
\bk \\
 {}^{}\quad  {\frac {u_1-1}{u_1  \left( u_1-v_1 \right) \left( v_1-1 \right)  }}&{\frac {1-u_2}{u_2 \left( u_2-v_2 \right) \left( v_2-1 \right)   }}&0&0&{\frac {-u_{{5}}}{ \left( u_{{5}}-1 \right)  \left( u_{{5}}-v_{{5}} \right) v_{{5}}}}&{\frac {1}{u_{{6}} \left( u_{{6}}-v_{{6}} \right)  \left( u_{{6}}-1 \right) }}  \quad 
\bk \\
 {}^{}\quad  {\frac {u_1}{v_1 \left( u_1-1 \right)  \left( u_1-v_1 \right) }}&{\frac {-u_2}{v_2 \left( u_2-1 \right)  \left( u_2-v_2 \right) }}&0&{\frac {-u_{{4}}}{v_{{4}} \left( u_{{4}}-1 \right)  \left( u_{{4}}-v_{{4}} \right) }}&0&{\frac {u_{{6}}}{v_{{6}} \left( u_{{6}}-1 \right)  \left( u_{{6}}-v_{{6}} \right) }}  \quad 
\bk \\
 {}^{}\quad  {\frac {-1}{u_1v_1}}& {\frac {1}{u_2v_2}}& {\frac {-1}{u_3v_3}}& {\frac {1}{u_{{4}}v_{{4}}}}&0& 0  \quad 
\end{bmatrix}$}} 
\end{minipage}%

\label{Page:AR-matrix}

From the explicit expression for $T$ and $C$ above and considering basic facts of the 
 \href{https://groupprops.subwiki.org/wiki/Linear_representation_theory_of_symmetric_group:S6}{\it linear repre-}
  \href{https://groupprops.subwiki.org/wiki/Linear_representation_theory_of_symmetric_group:S6}{\it
 -sentation theory of $\mathfrak S_6$}, 
it is not difficult to  recognize  to which representation  $\rho$ corresponds.  
 First, elementary computations show that $T$ and $C$ do not have any  eigenvector in common. Hence if the representation under scrutiny is not  irreducible, then 
 it is necessarily the sum of two irreducible $\mathfrak S_6$-modules both of dimension 5.  Under this assumption, looking at the character table of the $\mathfrak S_6$-representations, one deduces from ${\rm Tr}(C)=-1$ that  
 the representation is the direct sum of the standard representation with 
 another representation with Young diagram either $[3^2]$ or $[2^3]$.  \sk

 In these cases, the trace of the matrix corresponding to a transposition would be $3-1=2$ and $3+1=4$ respectively, which would contradict ${\rm Tr}(T)=-2$. This implies that the representation $\rho$ is irreducible.  From ${\rm Tr}(T)=-2$, one deduces that it is the third exterior power of the standard representation of $\mathfrak S_{6}$. We have proved the
\begin{prop} 
As a $\mathfrak S_6$-module, the 10-dimensional space 
$\boldsymbol{AR}_C(\boldsymbol{\mathcal W}_{0,6})=\boldsymbol{AR}(\boldsymbol{\mathcal W}_{0,6})$ is irreducible with associated Young diagram $[31^3]$.
\end{prop}


Proposition \ref{Prop:tololo} and the one just above contradict some of the claims in \cite{D}. The first thing one may think of for explaining this would be that there is a problem with the construction given by Damiano of Euler's abelian relation $\boldsymbol{\mathcal E}_3$ and that this AR actually does not exist.  It turns out that this is not the case, Damiano's construction of Euler's abelian relation is valid  for any $n\geq 2$: we will discuss this in depth in Section  \S\ref{S:Euler-AR}. According to us, he problem lies in the fact that, unaware of the dichotomy according to the parity of $n$ in what concerns the properties of web $\boldsymbol{\mathcal W}_{0,n+3}$, Damiano missed an elementary but crucial fact when he studied the $\mathfrak S_{n+3}$-module structure of the space of  combinatorial abelian relations $\boldsymbol{AR}\big( \boldsymbol{\mathcal W}_{0,n+3} \big)$. It is what we are going to discuss in the next section.

\newpage

\section{\bf The space of combinatorial abelian relations as a $\mathfrak S_{n+3}$-representation} 
\label{S:Conjectures-ARs}	
\label{SS:Conjectures-AR(W0n+3-as-a-Sn+3-module}	
We would like to get a better understanding of the space of combinatorial ARs of 
 $\boldsymbol{\mathcal W}_{0,n+3}$.  Damiano  used a powerful combinatorial and representation theoretic approach to do so. Using it,  he obtained  very interesting and almost complete results about the structure of $\boldsymbol{AR} (\boldsymbol{\mathcal W}_{0,n+3})$ as a $\mathfrak S_{n+3}$-module.  However, since he was not aware of the dichotomy between the cases when $n$ is even or odd 
  in what concerns the  $\mathfrak S_n$-module structure of $\boldsymbol{AR}\big(\boldsymbol{\mathcal W}_{0,n+3}\big)$, 
 he did not go all the way and missed a complete description of it.  On another hand, the proofs in Damiano's article \cite{D}  are very concise and/or a bit elliptical in some places.  More details are given in his thesis \cite{DThesis} but this one is less easily available.   For these reasons and especially for  the sake of completeness, we thought it  relevant  to take up  Damiano's approach with as much details as possible and it is the purpose of the current section. \sk

Most of the material below, the statements as the proofs, are due to Damiano and are taken with essentially no change either from the fourth chapter of Damiano's thesis \cite{DThesis} or from Section \S3 of the corresponding  paper \cite{D}. We only have added a few details 
in  some places. 
 Proofs have been included below mainly for the sake of completeness but also because these results are particularly important regarding the main purpose of this text and  are quite nice, as well as their proofs.  The unique genuine  novelty in this section is Lemma \ref{L:n-odd} which together with Damiano's result and although it is elementary, allows us to show that actually all ARs of $\boldsymbol{\mathcal W}_{0,n+3}$ are combinatorial when $n$ is odd.  
\sk


 The plan of the sequel is as follows: after having introduced some notations in \S\ref{SS:Some-Notations},  we state 
Theorem \ref{THM:AR-C(W)-S-n+3-module} which is   the main result of this section. Proofs  are given in the following subsections \S\ref{SS:lalala} and \S\ref{SS:Proof-of-THM:AR-C(W)-S-n+3-module}. 
Finally in \S\ref{SS:Components-of-ARs}, we discuss explicit formulas for the components of the combinatorial ARs of $\boldsymbol{\mathcal W}_{0,n+3}$, which will be used later on to make explicit the components of Euler's abelian relation in Section \S5 (see Proposition \ref{P:5.20} more specifically).




\subsection{Some Notations.}
\label{SS:Some-Notations}
Here, $n$ is a fixed integer bigger than or equal to 2 and 
 $\boldsymbol{I}_n$  stands for the set  of pairs $(i,j)$ with $1\leq i < j\leq n+3$. 
  We will take $\mathbf C$ as base field. 
Since we are only dealing with rational quantities, this choice actually is not relevant and we could have chosen to work over $\mathbf R$ instead (at the cost of some cosmetic changes
 that we will not mention).  \mk 
 
We use the coordinates system $(x_1,\ldots,x_n)$ associated to the birational map  $\psi : \mathbf C^n\dashrightarrow \mathcal M_{0,n+3},\, x=(x_i)_{i=1}^n\mapsto [0,1,\infty,x_1,\ldots,x_n]$ and we will work with the following cross-ratio
$$
{\rm cr}( z_1,z_2,z_3,z_4)=-\frac{(z_1-z_4)(z_2-z_3)}{(z_1-z_2)(z_3-z_4)}
$$
(chosen because of the normalization ${\rm cr}( 0,1 ,\infty , z)=z$ for any $z\in \mathbf C\setminus\{0,1\}$).
\sk 

We identify $\boldsymbol{\mathcal W}_{0,n+3}$ with its pull-back under $\psi$.  As first integrals in the coordinates $x_i$'s for the foliations of this web, we take the following rational maps (for $i$ ranging from 1 to $n$): 
\begin{align}
\label{Eq:U-i}
U_i(x)=&\,\,  \Big(x_j\Big)_{j\leq n, j\neq i} 
 && 
U_{n+1}(x)= \left(\frac{x_j-1}{x_n-1}\right)_{j\leq n-1}  \bk \\
U_{n+2}(x)=&\, \left(\frac{x_j}{x_n}\right)_{j\leq n-1}
&& U_{n+3}(x)= \left(\frac{x_j(x_n-1)}{x_n(x_j-1)}\right)_{j\leq n-1}\,. 
\nonumber
\end{align}

We will denote by $\mathcal F_i$ the foliation admitting $U_i$ as a first integral. For  $i=1,\ldots,n+3$, one sets 
\begin{equation}
\label{Eq:Omega-i}
\Omega_i=dU_{i,1}\wedge \ldots \wedge dU_{i,n-1}
\end{equation}
 where $U_{i,s}$ stands for the $s$-th component of $U_i$ for any $s\leq n-1$.  Then one has 
$\Omega_k=\wedge_{l=1,\ldots,n,\, l\neq k} dx_k$ for $k=1,\ldots,n$ hence these forms 
form a basis of the module of $(n-1)$-differential forms on any open domain of $\mathbf C^n$. 
\sk

We will use the following automorphisms of $ \mathcal M_{0,n+3}$: \sk
\begin{itemize}
\item 
\vspace{-0.2cm}
the  cyclic shift $\mathcal C: 
 \, [0,1,\infty,x_1,\ldots,x_n]\mapsto [1,\infty,x_1,\ldots,x_n,0]$; \sk 
 \item the transposition ${{\mathcal T}}: 
 \, [0,1,\infty,x_1,\ldots,x_n]\mapsto [0,\infty,1,x_1,\ldots,x_n]$.
\end{itemize}
\sk 

We denote by $C$ and $T$ the two birational maps corresponding to  $\mathcal C$ and $\mathcal T$ but expressed in the coordinates $x_i$'s (formally $C=\psi^{-1}\circ \mathcal C\circ \psi$ and similarly for $T$ with respect to $\mathcal T$).  One verifies easily that $C$ and $T$ are given by the following formulas: 
 \begin{equation}
 \label{Eq:R-C-1}
C(x)=\bigg( \, \frac{1-x_2}{x_1-x_2},  \frac{1-x_3}{x_1-x_3},\ldots, \frac{1-x_n}{x_1-x_n}  ,   \frac{1}{x_1}\, \bigg)
\qquad \mbox{and}\qquad 
T(x)=\bigg( \, \frac{x_1}{x_1-1}\, , \, \ldots \, , \,  \frac{x_n}{x_n-1}\, \bigg)\,\,  .
\end{equation}

\begin{lem}
 The maps $C$ and $T$ are automorphisms of $\boldsymbol{\mathcal W}_{0,n+3}$ and their actions on the foliations $\mathcal F_i$ are as follow: 
for any $i=1,\ldots,n+3$, one has 
$$
C^*\big(\mathcal F_i\big)=\mathcal F_{c(i)}
\qquad \mbox{ and } \qquad 
T^*\big(\mathcal F_i\big)=\mathcal F_{\tau(i)}\, 
$$
where $c$ stands for the $(n+3)$-cycle $(12\cdots (n+3))$ and $\tau$ for the transposition $((n+2)(n+3))$. 
\end{lem}
The group  generated by $C$ and $T$, denoted by $S_{n+3}$,  is isomorphic to $\mathfrak S_{n+3}$.
More precisely, for any $G$ in this group, let $\delta_G$ be the permutation of $\{1,\ldots,n+3\}$ such that $G^*(\mathcal F_i)=\mathcal F_{\delta_G(i)}$ for $i=1,\ldots,n+3$.  Then $\delta:S_{n+3}
\rightarrow \mathfrak S_{n+3}$, $G\mapsto \delta_G$ is an isomorphism.\footnote{For   $G,G'\in  S_{n+3}$, one has $\delta_{G'\circ G}=\delta_{G'}\cdot \delta_G$ where $\cdot$ denotes the product on $\mathfrak S_{n+3}$ given by applying the permutations according to their order of appearance from the right to the left, 
{\it i.e.}\,one has $(\delta_{G'}\cdot \delta_G)(k)=\delta_G(\delta_{G'}(k))$ for any $k$.}
We will denote by $G_\sigma$ the birational map of 
${\mathcal M}_{0,n+3}$ corresponding to $\sigma\in \mathfrak S_{n+3}$ up to this isomorphism, that is $$G_\sigma=\delta^{-1}(\sigma)\in {\bf Bir}\big({\mathcal M}_{0,n+3}\big)\, .$$
Note that $S_{n+3}$  identifies with the group of  birational automorphisms of 
$\boldsymbol{\mathcal W}_{0,n+3}$ (which coincides with the group of automorphisms of  ${\mathcal M}_{0,n+3}$). \sk 

The $\mathfrak S_{n+3}$-module structure of $\boldsymbol{AR}\big( 
\boldsymbol{\mathcal W}_{0,n+3} \big)$ is as follows: for any permutation 
$\sigma\in \mathfrak S_{n+3}$ and any abelian relation $A=(A^i\cdot \Omega_i)_{i=1}^{n+3}\in \boldsymbol{AR}\big( 
\boldsymbol{\mathcal W}_{0,n+3} \big)$, one has 
\begin{equation}
\label{Eq:Sigma.A}
\sigma\cdot A=G_\sigma^*\big( A\big) = \Big(  G_\sigma^*\big( 
A^i\Omega_i
\big)
\Big)_{i=1}^{n+3}\, . 
\end{equation}

 For any $k=1,\ldots,n+3$ and any  $(i,j)\in \boldsymbol{I}_n$, we denote by  
$\boldsymbol{\mathcal W}_{\widehat{k }}$
and  
$\boldsymbol{\mathcal W}_{\widehat{\imath \jmath}}$ 
the subwebs of $\boldsymbol{\mathcal W}_{0,n+3}$ defined by the first integrals $U_l$ for $l\in \{1,\ldots,n+3\}\setminus \{k\} $ and 
$l\in \{1,\ldots,n+3\}\setminus \{ i,j\} $ respectively. The corresponding spaces of abelian relations will be denoted by $\boldsymbol{A}_{\widehat{k}}$ and $\boldsymbol{A}_{\widehat{\imath \jmath}}$ respectively.  These vector spaces can naturally be seen as subspaces of 
$\boldsymbol{AR}( \boldsymbol{\mathcal W}_{0,n+3})$.


\subsection{The main theorem.}
\label{SS:lalala}
Here we aim to give an almost complete proof of the theorem below, which describes  
$\boldsymbol{AR}( \boldsymbol{\mathcal W}_{0,n+3})$ as a $\mathfrak S_{n+3}$-module. 
This result contradicts,   when $n$ is odd, one of the main results claimed by Damiano. 
 However the proof, whether $n$ is odd or not, 
  is essentially Damiano's one! The single novelties here are the few details we have added and also a  final argument, elementary but missed by Damiano, 
showing that all ARs of $  \boldsymbol{\mathcal W}_{0,n+3}$ are combinatorial when $n$ is odd.  \sk


The approach we will follow is  Damiano's and can be summarised as such:
\begin{itemize}
\item[$-$] 
\vspace{-0.2cm}
we start by considering the $(n+1)$-subwebs  $\boldsymbol{\mathcal W}_{\widehat{\imath \jmath}}$ and their abelian relations;\sk
\item[$-$] then we turn to the  $(n+1)$-subwebs $\boldsymbol{\mathcal W}_{\widehat{k}}$
whose structure of the space of its abelian relations will be determined as an $\mathfrak S_{n+2}$-module; 
\sk
\item[$-$] eventually we consider $\boldsymbol{AR}
\big( \boldsymbol{\mathcal W}_{0,n+3} 
\big) $  and establish Theorem \ref{THM:AR-C(W)-S-n+3-module} below. 
\end{itemize}
\mk 

For any pair $(i,j)$ in  $\boldsymbol{I}_n$, the space 
$\boldsymbol{A}_{\widehat{\imath \jmath}}=\boldsymbol{AR}
\big( \boldsymbol{\mathcal W}_{\widehat{\imath \jmath}} 
\big) $  is a subspace of dimension at most 1 of $
\boldsymbol{AR}( \boldsymbol{\mathcal W}_{0,n+3})$. Our goal here is first to give an effective construction of a non trivial element in 
$\boldsymbol{A}_{\widehat{\imath \jmath}}$. We will thus get a explicit family  of ARs generating $
\boldsymbol{AR}_C( \boldsymbol{\mathcal W}_{0,n+3})$ which will make it possible to study this $\mathfrak S_{n+3}$-module  more in depth.
\mk

We first consider some explicit scalar quantities $A_{n+2,n+3}^i$'s for $i=1,\ldots,n+3$: 
one has $A_{n+2,n+3}^{k}=0$ for $k=n+2,n+3$
and the other $A_{n+2,n+3}^i$'s are given  by the following formulas:  
\begin{align*}
A_{n+2,n+3}^{i}=   \frac{(-1)^{i-1}}{\prod_{k=1}^{n-1} \big(U_{i,k}-1\big)}  \quad \mbox{for }\; i=1,\ldots,n
\quad \qquad \mbox{ and } \qquad  \quad 
A_{n+2,n+3}^{n+1}=   \frac{(-1)^{n}}{\prod_{k=1}^{n-1} U_{n+1,k}}\, .
\end{align*}


Then it can be verified that the following $(n+3)$-tuple of differential forms 
$$AR_{n+2,n+3}=\left(A_{n+2,n+3}^i \cdot \Omega_i  \right)_{i=1}^{n+3}$$ 
is an abelian relation for 
$\boldsymbol{\mathcal W}_{\widehat{n+2,n+3}} $ ({\it i.e.}\, one has $\sum_{i=1}^{n+1} A_{n+2,n+3}^i\cdot \Omega_i=0$ as a rational $(n-1)$-form on $\mathbf C^n$). 
 Our approach to build the ARs of the other subwebs $\boldsymbol{\mathcal W}_{\widehat{\imath \jmath}}$ consists in taking the pull-backs of $AR_{n+2,n+3}$ under some birational isomorphisms the constructions of which are not difficult to show.
\sk 

One sets
\begin{itemize}
\item $T_{k-1,k}= \big(C\big)^{-k}\circ  T\circ  \big(C\big)^{k}$ for $k=2,\ldots,n+4$;    and $ T_{1,n+3}= T_{n+3,n+4}$;\sk
\item $ C_1={\bf I}{\rm d}$ and 
$ C_{l}= T_{l-1,l}\circ  T_{l-2,l-1}
\cdots \circ  T_{1,2}$ for $l=2,\ldots,n+3$;
\item $ T_{1,\ell}= \big(  C_{\ell-1}\big)^{-1}\circ  T_{\ell-1,\ell}\circ  C_{\ell-1}$ 
and $ T_{2,\ell}= 
 \big(  T_{1,2}\big)^{-1}\circ  T_{1,\ell}\circ  T_{1,2}$
for $\ell=3,\ldots,n+3$;\sk
\item $ T_{i,j}=  T_{2,j}\circ  T_{1,i}$ for any pair $(i,j)$ such that $3\leq i<j\leq n+3$. \mk 
\end{itemize}

It follows from elementary standard computations in $\mathfrak S_{n+3}$ that 
for any $(i,j)\in \boldsymbol{I}_n$, $ T_{i,j}$ is an involution such that 
$\delta(T_{i,j})=(1i)(2j)$. Consequently, for any such pair, the pull-back 
$$
AR_{i,j}=\Big(T_{i,j}\circ T_{n+2,n+3}\Big)^*\left( 
AR_{n+2,n+3}
\right) 
$$
is a non trivial abelian relation of $\boldsymbol{\mathcal W}_{\widehat{\imath \jmath}}$.  Thus any space $\boldsymbol{A}_{\widehat{\imath \jmath}}$ 
is spanned by $AR_{i,j}$ hence 
has dimension 1 and   $\{ AR_{{i,j}}\, \lvert \,  (i,j)\in \boldsymbol{I}_n\,\}$ is a generating family for the space of combinatorial ARs. Because the $AR_{{i,j}}$'s have been constructed effectively, this gives a way to investigate $\boldsymbol{AR}_C( \boldsymbol{\mathcal W}_{0,n+3})=
\big\langle  \, {AR}_{i,j}
\, \lvert \,   (i,j)\in \boldsymbol{I}_n
\big\rangle$ as a $\mathfrak S_{n+3}$-module. 
 \sk

 It will be useful to set  ${AR}_{k,l}={AR}_{l,k}$ for any $k,l=1,\ldots,n+3$ with $n+3\geq k>l\geq 1$. 
 \begin{center}
 $\star$
 \end{center}
 \sk

For $k\leq n+1$, one denotes by $\boldsymbol{I}_{n}^k$ the set of $(i,j)$'s in $\boldsymbol{I}_n$ such that $i=1,\ldots,k$ and $j=i+1,\ldots,n+2$, and one defines
 $$\boldsymbol{J}_n=\begin{cases} {}^{} \, \boldsymbol{I}_{n}^n
  \hspace{0.7cm} \mbox{when } n \mbox{ is even }; \\
{}^{} \, \boldsymbol{I}_{n}^{n+1}\hspace{0.4cm} \mbox{when } n \mbox{ is odd }.
 \end{cases}
$$
Remark that  one has  
$ \boldsymbol{I}_{n}^{n+1}=\boldsymbol{I}_{n}^n\cup \{ (n+1,n+2) \}$ and 
${\rm Card}\big(\boldsymbol{I}_{n}^{n+1}\big)= {\rm Card}\big(\boldsymbol{I}_{n}^n\big)+1=
(n+1)(n+2)/2$. \sk 

\vspace{-0.4cm}
{Our goal below is to prove the following}
\begin{thm} 
\label{THM:AR-C(W)-S-n+3-module}
The following  assertions hold true for any integer $n\geq 2$: 
\begin{enumerate}
\item[1.] If $n$ is odd, then:\sk
\begin{enumerate}
\item[--] the abelian relations $AR_{ij}$'s for $(i,j)\in \boldsymbol{J}_n$ 
form a basis of $\boldsymbol{AR}_C( \boldsymbol{\mathcal W}_{0,n+3})$ which  has dimension $(n+1)(n+2)/2$. Consequently 
$\boldsymbol{AR}( \boldsymbol{\mathcal W}_{0,n+3})=\boldsymbol{AR}_C( \boldsymbol{\mathcal W}_{0,n+3})$
 and  the web $\boldsymbol{\mathcal W}_{0,n+3}$ has maximal rank;
\mk 
\item[--] as a $\mathfrak S_{n+3}$-representation, the space $\boldsymbol{AR}( \boldsymbol{\mathcal W}_{0,n+3})=\boldsymbol{AR}_C( \boldsymbol{\mathcal W}_{0,n+3})$ is irreducible and isomorphic to the $n$-th wedge product of the standard $\mathfrak S_{n+3}$-representation (with associated Young symbol $[3,1^n]${\rm )}.\footnote{Remember that for $m\geq 3$ and any $p\in \{1,\ldots,m\}$, 
the $p$-th wedge product of the standard $\mathfrak S_m$-representation 
with Young symbol $[ m-1,1]$ is irreducible with  Young symbol 
$[m-p,1^p]$ (see Proposition 3.12 and Exercise 4.6 in \cite{FultonHarris}).}
\end{enumerate}
\mk 
\item[2.] When $n$ is even:  \sk
\begin{enumerate}
\item[--] the abelian relations $AR_{ij}$'s for $(i,j)\in \boldsymbol{J}_n$  
form a basis of $\boldsymbol{AR}_C( \boldsymbol{\mathcal W}_{0,n+3})$.  Thus this space has dimension $n(n+3)/2=(n+1)(n+2)/2-1$. 
\mk 
\item[--] 
as a $\mathfrak S_{n+3}$-representation, the space 
$\boldsymbol{AR}_C( \boldsymbol{\mathcal W}_{0,n+3})$
 of combinatorial ARs  
is irreducible  with associated Young symbol $\big[221^{n-1}\big]$.
\end{enumerate}
\end{enumerate}
\end{thm}

\subsection{The subwebs  ${\mathcal W}_{\widehat{\imath}} $'s 
and their abelian relations}
\label{SS:lalala}
  Since $\mathfrak S_{n+3}$ acts transitively on the foliations of $ \boldsymbol{\mathcal W}_{0,n+3}$, any $(n+2)$-web  $\boldsymbol{\mathcal W}_{\widehat{\imath}} $ is isomorphic to 
  $\boldsymbol{\mathcal W}_{\widehat{n+3}} $. The latter is the linear web formed by the $n+2$ linear families of lines passing through $n+2$ points in general position in $\mathbf P^n$ 
  (which can be taken to be the points $p_k$ of \S\ref{SSS:W0n+3}, {\it cf.}\,\eqref{Eq:points-p-i}). \sk 
  
For any $i$,  the subgroup $\mathfrak S_{n+2}\simeq {\rm Fix}(i)< \mathfrak S_{n+3} $ acts on the space 
of combinatorial ARs 
$$ \boldsymbol{A}_C(i)
=\Big\langle  \, AR_{i,j}\, \big\lvert \, j=1,\ldots,n+3,\, j\neq i\, 
\Big\rangle 
\subset   
\boldsymbol{AR}\big(\boldsymbol{\mathcal W}_{\widehat{\imath}}\big)=
\boldsymbol{A}(i)\, . $$

 From the action 
  \eqref{Eq:Sigma.A} 
  of $\mathfrak S_{n+3}$ on $\boldsymbol{AR}_C\big(\boldsymbol{\mathcal W}_{0,n+3}\big)$, 
 one gets a  $\mathfrak S_{n+2}$-module structure  on $ \boldsymbol{A}_C(i)$, whose structure is made clear thanks to the following 
\begin{prop}[Damiano]  ${}^{} $ 
\label{P:Damiani-4.1.4}
\begin{enumerate}
\vspace{-0.15cm}
\item[1.]  The space $\boldsymbol{A}_C(i)$ has dimension $n+1$ and 
the abelian relations $AR_{i,j}$'s for $j=1,\ldots,n+3$ with  $j\neq i$  are in general position 
in it:  any set of $n+1$ of them form a basis of this space.
\mk
\item[2.]  Consequently, one has  $ \boldsymbol{A}_C(i)= \boldsymbol{A}(i)$ and 
$\boldsymbol{\mathcal W}_{\widehat{\imath}} $ has maximal rank $n+1$.
\mk 
\item[3.] As a $\mathfrak S_{n+2}$-representation, $\boldsymbol{A}(i)$ is irreducible with Young symbol $[2,1^n]$. 
\end{enumerate}
\end{prop} 

 To prove this result, we will need the following lemma which will be used again at the end: 
 \begin{lem}
 \label{L:(i,j)-(k,l)}
 For any elements $(i,j)$ and $(k,j)$ of $ \boldsymbol{I}_n$ with $\{i,j\}\cap \{k,l\}=\emptyset$, one has 
 \begin{equation}
  (i,j)\cdot AR_{i,j}=(-1)^{n-1}\,AR_{i,j}
 \qquad  \quad 
 \mbox{ and } 
 \qquad   \quad 
  (k,l)\cdot AR_{i,j}=-AR_{i,j}\, .
 \end{equation}
 \end{lem}
 \begin{proof}
Considering the action of $\mathfrak S_{n+3}$ on the foliations of $
\boldsymbol{\mathcal W}_{0,n+3}$, one can restrict oneself to the case when $(i,j)=(n+2,n+3)$ and 
$(k,l)=(1,2)$. 

In the rational coordinates $x_1,\ldots,x_n$ we are working with,
 the birational maps associated to the two corresponding transpositions are  
given by 
$$G_{n+2,n+3}(x)=T(x)=\Big( {x_i}/({x_i-1}) \Big)_{i=1}^n 
\qquad \mbox{ and } \qquad 
G_{1,2}(x)=\big( x_2,x_1,x_3,\ldots,x_n\big)\, .$$ 

Let $\tau$ stand for $(n+2,n+3)$ or $(1,2)$.  Since $G_\tau$ leaves 
$\boldsymbol{\mathcal W}_{\widehat{\imath\jmath}} $ invariant and because $\tau$ is an involution, one has $\tau\cdot AR_{n+2,n+3}=G_{\tau}^*\big( AR_{n+2,n+3}\big)= \epsilon_\tau\, AR_{n+2,n+3}$ for  $\epsilon_\tau\in \{\pm 1\}$ to be determined. \sk 

To do so, it suffices to compare $G_\tau^*(A_{n+2,n+3}^3\, \Omega_3)$ with $A_{n+2,n+3}^3\Omega_3$, which is easy to do. 
Indeed, one verifies easily that $A_{n+2,n+3}^3\Omega_3$ is equal to the wedge product of the logarithmic 1-forms $d \hspace{0.03cm} {\rm Log}( x_\ell-1)$ for $\ell=1,2,\ldots,n$ distinct from 3: 
$$
A_{n+2,n+3}^3 \hspace{0.06cm} \Omega_3=d \hspace{0.03cm}{\rm Log}\big( x_1-1\big)\wedge d \hspace{0.03cm} {\rm Log}\big( x_2-1\big)\wedge 
\left( \, \bigwedge_{\ell=4}^n d  \hspace{0.03cm} {\rm Log}\big( x_\ell-1\big)
\,  \right) \, .
$$

 Since $G_{1,2}$ is just exchanging $x_1$ and $x_2$, one gets $(1,2)\cdot A_{n+2,n+3}^3\, \Omega_3=-A_{n+2,n+3}^3\, \Omega_3$ hence $\epsilon_{1,2}=-1$.  
Because  taking the pull-bak under $T$ consists in replacing $x_i$ by $x_i/(x_i-1)$, 
one has 
$$  {\rm Log}\left(\frac{x_i}{x_i-1}-1\right)={\rm Log}\left(\frac{1}{x_i-1}\right)=-{\rm Log}\big(x_i-1\big)$$
for any $i=1,\ldots,n$,  
which immediately gives us  that 
$T^*\big( A_{n+2,n+3}^3\, \Omega_3\big)=(-1)^{n-1}A_{n+2,n+3}^3\,  \Omega_3$. Hence one has 
 $\epsilon_{n+2,n+3}=(-1)^{n-1}$, which  finishes  the proof of the lemma.
 \end{proof}

 We now turn to the proof of the preceding proposition.   The one below is entirely taken from \cite{DThesis} ({\it cf.}\,\S{4.1.4} in it) and is reproduced here for the sake of completeness.  Next, in \S\ref{SSS:Alternative-Proof-of-Proposition- P:Damiani-4.1.4}, we will describe (without proof) a more geometric approach of Proposition \ref{P:Damiani-4.1.4} which is much more in the spirit of classical web geometry. 
  \begin{proof}[Proof of Proposition \ref{P:Damiani-4.1.4}]
  It suffices to prove the proposition for $i=n+3$.  To simplify the notation, we set  
  $\boldsymbol{A}=\boldsymbol{A}(n+3)=\boldsymbol{AR}\big( \boldsymbol{\mathcal W}_{\widehat{n+3}} \big)$  and 
  $\boldsymbol{A}_C=\boldsymbol{A}_C(n+3)=\boldsymbol{AR}_C\big( \boldsymbol{\mathcal W}_{\widehat{n+3}} \big)$.
For 
  $l=1,\ldots,n+2$, we denote by $q_l$ the point of 
  $\mathbf P(\boldsymbol{A})
 $  corresponding to $AR_{l,n+3}$. 
By definition,  $\mathbf P(\boldsymbol{A}_C)$ is the subspace of 
 $\mathbf P(\boldsymbol{A})$ spanned by all the $q_l$'s and 
   from Proposition \ref{Prop:Damiano'sBound}, we know that $\dim\big( \mathbf P(\boldsymbol{A}) \big)\leq n$.

 Given $\sigma\in \mathfrak S_{n+2}$ and $q=[ Q]\in  \mathbf P(\boldsymbol{A})
 $ with $Q\in \boldsymbol{AR}(\boldsymbol{\mathcal W})\setminus \{0 \}$, we set $\sigma\cdot q=[ \sigma \cdot Q]$. Then for any $i,j$ such that  $1\leq i<j\leq n+2$, 
 the transposition $\sigma=(i,j)$ exchanges $q_i$ and $q_j$ and admits all the other  
$q_l$'s (for $l\in \{1,\ldots,n+2\}$ distinct from $i$ and $j$)  as fixed points. 
 \sk 
 
 Let $\kappa \in \{1,\ldots,n\}$ be the largest integer satisfying the property
 that  there exists 
 a subset 
 $J=\{ {j_1},\ldots, j_{\kappa+2} \}\subset \{1,\ldots,n+2\}$ of cardinality $\kappa+2$ 
 such that 
the points  
$q_{j_1},  \ldots,q_{j_{\kappa+2}}$ are in general position in their span 
$\mathbf P(J)=\langle  q_{j}\, ,\, j\in J\rangle $ which is of dimension $\kappa$. If $\kappa=n$ then 1.\,is proved hence assume that it is not the case. 
  Then for any $i\in \{1,\ldots,n+2\}\setminus J$, the transposition $\sigma_i=(i,j_1)$ 
 lets $q_{j_s}$ fixed for  any $s=2,\ldots,\kappa+2$. Since these points form a projective frame of $\mathbf P(J)$ according to the definition of $J$, we have that $\sigma_i\cdot \mathbf P(J)=\mathbf P(J)$. On the other hand, $\sigma_i$ exchanges $q_i$ with $q_{j_1}$ hence we get that $q_i$ belongs to $\mathbf P(J)$. Because this applies to any $i$ not in $J$,  we obtain that $\mathbf P(J)
 =\mathbf P(\boldsymbol{A}_C)$.
Assume that $\kappa<n-1$. Then there exist $i,i'\in \{1,\ldots,n+2\}\setminus J$ distinct. 
On the one hand, the permutation $(i,i')$ lets  $q_j$ invariant for all $j\in J$. Since these points are in general position in $\mathbf P(J)
 $, it follows that $(i,i')$ acts as the identity on this space. 
 On the other hand,  $(i,i')$ exchanges $q_i$ and $q_{i'}$, which is a contradiction. Thus necessarily, one has $\kappa= n-1$ or $\kappa=n$.  \sk
 
 Since we are done in the latter case, let us assume than the former holds true: one has 
 $\kappa= n-1$.  We denote by $i$ the element of $\{1,\ldots,n+2\}$ not in $J$. Then $\mathfrak S_{n+1}\simeq {\rm Fix}_{\mathfrak S_{n+2}}(i)$ acts on $\mathbf P(J)
 =\mathbf P(\boldsymbol{A}_C)\simeq \mathbf P^{n-1}$ inducing the standard permutation action on the $n+1$ points $q_1,\ldots,\widehat{q_i},\ldots,q_{n+2}$. It is easily verified that 
 necessarily, the action on $\mathbf P^{n-1}$ comes from a linear action on a vector space of dimension $n$ which is irreducible. But this action also lets $q_i$ invariant, a contradiction. \sk 
 
We thus have proved 1. The second point follows immediately. \sk 

Let us turn to the structure of the space $ \boldsymbol{A}$ as a $\mathfrak S_{n+2}$-module. Picking a projective basis for $\mathbf P(\boldsymbol{A})$, 
say the $(n+1)$-tuple $(q_1,\ldots,q_{n+1})$,  we identify it with $\mathbf P^n$. From the arguments above, it follows that the image of 
$\mathfrak S_{n+2}$ in ${\rm PGL}_{n+1}(\mathbf C)$ corresponding to this identification can be described as the set of projective transformations leaving $\{q_i\}_{i=1}^{n+1}$ invariant and inducing a permutation of it.  From this, one gets easily that 
$ \boldsymbol{A}_{\widehat{n+3}}$  necessarily is an irreducible 
$\mathfrak S_{n+2}$-representation. Since moreover 
 it has dimension $n+1$, there are only two possibilities: either it is the standard representation or its dual (with respective Young symbols $[n+1,1]$ and $[2,1^{n}]$). 

 From above, we know  that the abelian relations $AR_{1,n+3}$ for $i=1,\ldots,n+1$ form a basis of 
$\boldsymbol{A}=\boldsymbol{A}_C$ which has  dimension $n+1$. In order to recognize it as a $\mathfrak S_{n+2}$-representation, we consider the matrix in this basis of the linear involution associated to the transposition $(1,2)$, denoted by $M_{(1,2)}$. From 
 Lemma \ref{L:(i,j)-(k,l)}, we have $(1,2)\cdot AR_{j,n+3}=-AR_{j,n+3}$ for any $j=3,\ldots,n+1$ on the one hand, whereas one has $(1,2)\cdot AR_{k,n+3}=\epsilon_k\,AR_{l,n+3}$ for $k,l$ such that $\{k,l\}=\{1,2\}$, with $\epsilon_k,\epsilon_l\in \{\pm 1\}$ on the other hand. 
It follows that ${\rm Tr}\big(M_{(1,2)} \big)=-(n-1)$.  Since the character map $\chi_{[n+1,1]}$ of the standard representation is such that $\chi_{[n+1,1]}(\sigma)=\lvert {\rm Fix}(\sigma)\lvert-1$ for any $\sigma\in \mathfrak S_{n+2}$ where ${\rm Fix}(\sigma)\subset \{1,\ldots,n+2\}$ stands for the set of fixed points of $\sigma$, we deduce that the space 
$\boldsymbol{A}(n+3)$ is not the standard representation but its dual, which concludes the proof of the proposition. 
 \end{proof}
 
\subsubsection{\bf An alternative view on Proposition \ref{P:Damiani-4.1.4}}
\label{SSS:Alternative-Proof-of-Proposition- P:Damiani-4.1.4}
 Here we would like to explain another way to establish Proposition \ref{P:Damiani-4.1.4} much more in the spirit of classical web geometry and which suggests an interesting possible development (see \S\ref{SS:Algebraization-(n+2)-webs} further). This approach is a rather straightforward generalization 
 for curvilinear $(n+2)$-web in dimension $n$ 
 of the classical linearization/algebraization of maximal rank planar 4-webs which goes back to works by Lie, Poincar\'e and Blaschke (for which we refer to \cite[\S4.3]{Coloquio}). \mk
 
We denote by $V$ the complex vector space $\mathbf C^{n+1}$ the projectivization of which is the ambiant projective space $\mathbf P^n$ we are working in:  $\mathbf P^n=\mathbf P(V)$. If $v_1,\ldots,v_{n+1}\in V$ stand for the vectors of the canonical basis of $V$ and $
v_{n+2}=(1,\ldots,1)\in V$, the. one has $p_i=[v_i]$ for $i=1,\ldots,n+2$ where $[\cdot]: V\setminus \{0\}\rightarrow \mathbf P^n$ stands for the standard projectivization map. 
We denote by $P$ the set of the $p_i$'s and by $\boldsymbol{\mathcal L \hspace{-0.05cm}\mathcal W}_{P}$ the $(n+2)$-linear web on $\mathbf P^n$ formed by the families of lines through the $p_i$'s.  It is nothing else but the subweb $\boldsymbol{\mathcal W}_{\widehat{n+3}} $ of $\boldsymbol{\mathcal W}_{0,n+3}$.  
\sk 

Here we continue to work with the notations of \S\ref{SS:Some-Notations}. In particular, 
 the rational coordinates $x_i$'s are those associated to the birational map $\mathbf P^n\dashrightarrow \overline{\mathcal M}_{0,n+3},\, (x_1,\ldots,x_n)\rightarrow [0,1, \infty,x_1,\ldots,x_n]$.  The subgroup ${\rm Fix}(3)\subset \mathfrak S_{n+3}$ is isomorphic to $\mathfrak S_{n+2}$ hence the latter group acts on $\mathbf P^n$ by birational maps which  (1) preserve the foliation by RNC of degree $n$ of $\boldsymbol{\mathcal W}_{0,n+3}$ (2) are birational automorphisms of $\boldsymbol{\mathcal L \hspace{-0.05cm}\mathcal W}_{P}$.  We thus have a representation 
\begin{equation}
\label{Eq:REP}
\mathfrak S_{n+2}\simeq {\rm Fix}(3)\rightarrow {\rm Bir}(\mathbf P^n)\, . 
\end{equation}

The transpositions $(i,j)$ with $3<i<j\leq n+3$, $(2,4)$ and $(1,2)$ generate ${\rm Fix}(3)$ 
and the corresponding Cremona transformations $G_\sigma$ for such any  transpositions $\sigma$ are easily seen to be given by 
$$
G_\sigma(x)= 
\begin{cases}\hspace{0.1cm} 
\big( \, x_{\sigma(k+3)-3}\, \big)_{k=1}^{n}  
\, \quad \mbox{ for }  \sigma=(i,j) \\ 
 \hspace{0.1cm} 
\big( 1-x_{k}\big)_{k=1}^{n}  
 \hspace{1cm}  \mbox{ for }  \sigma=(1,2) \\
  \hspace{0.1cm} 
\Big( \frac{1}{x_1}, \frac{x_2}{x_1},\ldots, 
\frac{x_n}{x_1}\Big) 
 \quad \mbox{ for }  \sigma=(2,4) \, .
\end{cases}
$$
These birational maps are projectivizations of linear automorphisms 
$\widehat{G}_\sigma$ 
of  $V$ which it is straightforward to make explicit.  Using their explicit form, it is not difficult to verify that 
\eqref{Eq:REP} is induced by a linear representation which is irreducible. Since moreover one has ${\rm Tr}(\widehat{G}_\sigma)=n-1$, we obtain that this representation is equivalent to the standard one. We thus have established the 
\mk

\noindent{\bf Fact:} {\it The representation \eqref{Eq:REP}  is linear. More precisely, it is the projectivization of a $\mathfrak S_{n+2}$-representation on $V$ which is irreducible and isomorphic to the standard representation}. 
\mk 

As in the preceding section, we denote by $\boldsymbol{A}$ the space of abelian relations $\boldsymbol{AR}\big(\boldsymbol{\mathcal L \hspace{-0.05cm}\mathcal W}_{P}\big) =\boldsymbol{AR}\big(\boldsymbol{\mathcal W}_{\widehat{n+3}}\big)$.
For any $i=1,\ldots,n+2$,  given $x\in \mathbf P^n$ generic, the map ${\rm ev}_i: \boldsymbol{A}\rightarrow \mathbf C, ( A_k\,\Omega_k)_{k=1}^{n+2}\mapsto A_i(x)$ is a non trivial linear form hence ${\rm ker}( {\rm ev}_i)$ is a hyperplane in $\boldsymbol{A}$. After projectivization, we get the {\it $i$-th canonical map} of the web under scrutiny, which by definition is the map 
$$
\xymatrix@R=0.1cm@C=0.4cm{
\kappa_i \, : \hspace{-0.5cm}  & 
 \mathbf P^n   \ar@{-->}[r] & \check{\mathbf P}^n
=\mathbf P\big( V^\vee\big) \\
& x \ar@{->}[r] & 
  \kappa_i(x)= \big[{\rm ker}( {\rm ev}_i) \big]\, . 
}
$$
Clearly, each $\kappa_i$ is constant along the leaves of the $i$-th foliation (which is the bundle $\boldsymbol{\mathcal L}_{p_i}$ 
of lines through $p_i$) and,  as a rational map, it factorizes through the map 
$U_i: \mathbf P^n\rightarrow \mathbf P^{n-1}$. The map $\kappa_i$ has maximal rank $n-1$ (generically), it is a canonical first integral for $\boldsymbol{\mathcal L}_{p_i}$ and (the closure of) its image is the hyperplane $H_i\subset \check{\mathbf P}^n$ which is dual to $p_i$: considering the $v_i$'s as the canonical coordinates on $V^{\vee}$, $H_i$ is cut out by $v_i=0$ for $i=1,\ldots,n+1$ and  by 
$v_{n+2}=\sum_{i=1}^{n+1}v_i=0$.
%
The union of the $H_i$'s, denoted by $H$, is a hypersurface of $\check{\mathbf P}^n$, of degree $n+2$, cut out by $v=v_1\cdots v_{n+2}=0$.

Any AR $(A_i\,\Omega_i)_{i=1}^{n+2}$ is such that $\sum_{i=1}^{n+2}A_i\,\Omega_i=0$ as a rational $(n-1)$-form. This relation is equivalent to $n$ linearly scalar identities, from which it can be deduced that generically, the $n+2$ points $\kappa_i(x)$'s actually span a line, that we denote by $\boldsymbol{\mathfrak L}(x)$: one has 
$$
\boldsymbol{\mathfrak L}(x)=\Big\langle \, \kappa_i(x)\, \big\lvert \, i=1,\ldots,n+2\Big\rangle \in G_1\big(\check{\mathbf P}^n\big)\, . 
$$
We then obtain a map $\boldsymbol{\mathfrak L}: \mathbf P^n\dashrightarrow G_1\big(\check{\mathbf P}^n\big)$ which allows to get an interesting geometric model of 
$
\boldsymbol{\mathcal L \hspace{-0.05cm}\mathcal W}_{P}$. Indeed, the closure $Z_n=\overline{\boldsymbol{\mathfrak L}(\mathbf P^n)}$ is a $n$-dimensional submanifold of the grassmannian variety of lines in $\check{\mathbf P}^n$ and the push-forward onto it 
of $\boldsymbol{\mathcal L \hspace{-0.05cm}\mathcal W}_{P}$
by $\boldsymbol{\mathfrak L}$ is defined by the incidence in 
$\check{\mathbf P}^n$
between the lines belonging to  $Z_n$ and the points of the hyperplanes $H_i$'s. In other terms, the following holds true: \mk \\
 \noindent{\bf Fact.} {\it The push-forward 
 of  $\boldsymbol{\mathcal L \hspace{-0.05cm}\mathcal W}_{P}$
 by $\boldsymbol{\mathfrak L}$ is the trace along $Z_n$ of the algebraic web $\boldsymbol{\mathcal W}_H$ {\rm :} one has}
 $$
 \boldsymbol{\mathfrak L}_*\Big( \boldsymbol{\mathcal L \hspace{-0.05cm}\mathcal W}_{P}\Big) =\big( 
 \boldsymbol{\mathcal W}_H\big) \big\lvert_{Z_n}\, .
 $$

Now, as it classically follows from Abel's theorem and its converse,  the ARs of  $\boldsymbol{\mathcal W}_H$ can be  described in terms of the abelian $(n-1)$-differential forms on $H$ ({\it cf.}\,\cite{Barlet,HenkinPassare}). 
 The maps $\rho_i:  G_1\big(\check{\mathbf P}^n\big) \dashrightarrow H_i,\, \ell \, 
 \mapsto \rho_i(\ell)$ where
$\rho_i(\ell)$ stands for the intersection point of a generic line $\ell$ in $\check{\mathbf P}^n$ with the hyperplane $H_i$  for each $i$, 
are natural rational first integrals of 
$\boldsymbol{\mathcal W}_H$. Moreover, the trace gives rise to a linear isomorphism 
$
\boldsymbol{H}^0\big( H,\omega_H^{n-1}\big)\longrightarrow  
\boldsymbol{AR}\big( \boldsymbol{\mathcal W}_H\big)$, $ \omega\longmapsto \big( \rho_i^*(\omega)\big)_{i=1}^{n+2}
$. Furthermore, it can be verified that the natural map $\boldsymbol{AR}\big( \boldsymbol{\mathcal W}_H\big)\rightarrow \boldsymbol{AR}\big( ( \boldsymbol{\mathcal W}_H)\lvert_{Z_n}\big)$ is injective, hence up to pull-back under $\boldsymbol{\mathfrak L}$,  we obtain a linear isomorphism of complex vector spaces
\begin{equation}
\label{Eq:SOS}
\boldsymbol{AR}\big( \boldsymbol{\mathcal L \hspace{-0.05cm}\mathcal W}_{P}\big)
\simeq   \boldsymbol{AR}\big( \boldsymbol{\mathcal W}_{H}\big)
\simeq  \boldsymbol{H}^0\big( H,\omega_H^{n-1}\big)
\, .
\end{equation}

On the other hand, we have the following short exact sequence of sheaves
\begin{equation}
\label{Eq:SES}
0 \rightarrow \Omega^n_{\check{\mathbf P}^n}\longrightarrow 
\Omega^n_{\check{\mathbf P}^n}\otimes \mathcal O_{\check{\mathbf P}^n}(H)
\stackrel{R}{\longrightarrow} \omega^{n-1}_{H} \rightarrow 0
\end{equation}
(where $\omega^{n-1}_{H}$ denotes the sheaf of abelian differentials 
of top degree on $H$ and $R$ is the morphism of sheaves induced by Poincar\'e residue, see \cite{HenkinPassare}). Since $\deg(H)=n+2$ and because $
\Omega^n_{\check{\mathbf P}^n}\simeq \mathcal O_{\check{\mathbf P}^n}(-n-1)$, we 
have $
\Omega^n_{\check{\mathbf P}^n}\otimes \mathcal O_{\check{\mathbf P}^n}(H)
\simeq \mathcal O_{\check{\mathbf P}^n}(1)
$. Since $\check{\mathbf P}^n=\mathbf P\big( V^\vee)$, $\boldsymbol{H}^0\big(\check{\mathbf P}^n, \mathcal O_{\check{\mathbf P}^n}(1)\big)$ naturally identifies 
with the dual of $V^{\vee}$. Injecting this into the first non trivial part of the long cohomology sequence associated to 
 \eqref{Eq:SES} 
gives rise to a sequence of natural linear isomorphisms
\begin{equation}
\label{Eq:QQQ}
V= \big(V^{\vee}\big)^\vee
= \boldsymbol{H}^0\Big(\check{\mathbf P}^n, \mathcal O_{\check{\mathbf P}^n}(1)\Big) 
\simeq
\boldsymbol{H}^0\Big(H , \omega^{n-1}_{H} \Big)\, . 
\footnotemark
\end{equation}
\footnotetext{The map $V\rightarrow \boldsymbol{H}^0\big(H , \omega^{n-1}_{H} \big)$ can be made explicit rather easily: the abelian differential on $H$ 
associated to $\overline{v}\in V$ is the Poincar\'e residue along $H$ of the global section 
$(\overline{v}/v)\sum_{i=1}^{n+1} (-1)^ i v_i dv_1\wedge \cdots \wedge \widehat{dv_i}\wedge \cdots \wedge dv_{n+1}$ of $\Omega^n_{\check{\mathbf P}^n}\otimes \mathcal O_{\check{\mathbf P}^n}(H)$.}
One verifies that the preceding isomorphisms actually are isomorphisms of 
$\mathfrak S_{n+2}$-modules, when the action on  $\boldsymbol{H}^0\big(H , \omega^{n-1}_{H} \big)$ is induced by push-forward by the birational maps $G_\sigma$ for $\sigma\in \mathfrak S_{n+2}$. Considering \eqref{Eq:SOS}, and because the action \eqref{Eq:Sigma.A} is induced by pull-backs under the $G_\sigma$'s, we get the 
\begin{prop}
The action of $\mathfrak S_{n+2}$ by push-forward on 
$\boldsymbol{AR}\big( \boldsymbol{\mathcal L \hspace{-0.05cm}\mathcal W}_{P}\big)$
is isomorphic to the action of 
$\mathfrak S_{n+2}$ on $V$ (hence is isomorphic to the standard representation). Consequently, the $\mathfrak S_{n+2}$-action \eqref{Eq:Sigma.A} on 
this space of ARs is the dual action (irreducible module with Young symbol $[2,1^n]${\rm )}.
\end{prop}

In addition to giving a more conceptual explanation for  Proposition \ref{P:Damiani-4.1.4}.3, the material discussed above is interesting since it asks possibly fruitful questions, in two distinct directions at least: 
\begin{enumerate}
\item[$\bullet$] 
One may think that 
the construction of canonical maps and of a canonical geometric model considered above 
for $\boldsymbol{\mathcal L \hspace{-0.05cm}\mathcal W}_{P}$ 
could be generalized to any $(n+2)$-web in dimension $n$ with maximal $(n-1)$-rank. This suggests a possible approach for algebraizing such webs.  We will say a few words about this problem further in \S\ref{SS:Algebraization-(n+2)-webs}. 
\mk 
\item[$\bullet$] 
Another interesting question that emerges from the above considerations concerns the varieties $Z_n$. Given a generic point $p\in \mathbf P^n$, let $C_p$ be the rational normal curve in $\mathbf P^n$ passing through all the $p_i$'s and $p$.
 Denoting by $\mathscr L_p$ the (embedded) tangent line to $C_p$ at $p$, one gets a rational map $\mathscr L: \mathbf P^n\dashrightarrow G_1(\mathbf P^n),\, p\longmapsto \mathscr L_p$ whose image $\mathscr Z_n=\overline{\mathscr L(\mathbf P^n)}$ is a $n$-dimensional submanifold of the grassmannian variety of lines in $\mathbf P^n$.  For $n=3$ and up to an identification of $\mathbf P^n$ with its dual $\check{\mathbf P}^n$, 
  we have verified that $\mathscr Z_3$ coincides with $Z_3$ and that it is a degree six complex of lines in $\mathbf P^3$. Contrarily to what we expected,  
 we have not found any mention of it in the classical literature on line complexes. Is it really new? And in higher dimension, do $Z_n$ and $\mathscr Z_n$ coincide? In any case, we think that they are interesting subvarieties of grassmannian varieties of lines that deserve further study. 
\end{enumerate}


\subsection{Proof of Theorem \ref{THM:AR-C(W)-S-n+3-module}}
\label{SS:Proof-of-THM:AR-C(W)-S-n+3-module}
This subsection is devoted to proving Theorem \ref{THM:AR-C(W)-S-n+3-module}. 
 We have reproduced Damiano's arguments for completeness, with only minor modifications added  in order to make them clearer.  Our only significant (although elementary) contribution is Lemma \ref{L:n-odd} which definitively settles the case when $n$ is odd. \sk

We start with two technical lemmata taken from \cite[\S4]{DThesis}.   To state the first, one considers $n+2$ points $q_1,\ldots,q_{n+2}$  in general position in $\mathbf P^n$ 
together with  an irreducible linear $\mathfrak S_{n+2}$-action on this projective space which induces the natural action by permutations on the $q_i$'s ({\it i.e.}\,one has $\sigma\cdot q_i=q_{\sigma(i)}$ for every $i$ and every $\sigma\in \mathfrak S_{n+2}$).  Let $m\leq n+1$ be fixed. The subgroup $\cap_{s=m+2}^{n+2} {\rm Fix}(s)$
is naturally isomorphic to $\mathfrak S_{m}$ and will be denoted in the same way.
\begin{lem} 
\label{Lem:TECH0} 
For any $q\in \mathbf P^n\setminus \langle q_{m+1},\ldots,q_{n+2}\rangle$, one has 
$
\big\langle \, \mathfrak S_{m}\cdot q\, , \, q_{m+1}
\, , \, \ldots \, , \,  q_{n+2}\, 
\big\rangle=\mathbf P^n
$. 
\end{lem}
\begin{proof}
By  induction on $m$, one can restrict oneself to only deal with the case when $m=n+1$. 

Let $V$ be the $\mathfrak S_{n+2}$-module of dimension $n+1$ whose projectivization is the considered $\mathfrak S_{n+2}$-action on $\mathbf P^n$. Since $V$ is irreducible and has dimension $n+1$, its Young symbol $\lambda$ is either $[n+1,1]$ or $[2,1^n]$. 
It then follows immediately from the branching rule (see \cite[Theorem 2.8.3]{Sagan}) that 
the restriction 
$V\hspace{-0.1cm}\downarrow_{\mathfrak S_{n+1}}$ is the sum 
of two irreducible $\mathfrak S_{n+1}$-representations, one 
with Young symbol $\lambda_-$ of dimension $n$, denoted by $V_-$, the other of dimension 1 denoted by $W$. Since $q_{n+2}$ is fixed by $\mathfrak S_{n+1}={\rm Fix}(n+2)$, this point is the projectivization $\mathbf P(W)$ of $W$. Any  $q\in \mathbf P^n\setminus \{q_{n+2}\}$  has a lift  in $V$ which can be written $\hat{q}_-+w$ with $\hat{q}_-\in V_-\setminus \{0\}$ and $w\in W$.  
Denote by ${q}_-$the image of $\hat{q}_-$ in $\mathbf P^n$. 

Clearly, one has $
\big\langle \, \mathfrak S_{n+1}\cdot q\, , \, q_{m+2}
\, 
\big\rangle=\big\langle \, \mathfrak S_{n+1}\cdot q_-\, , \, q_{m+2}
\, 
\big\rangle
$.  On the other hand, since $V_-$ is an irreducible $\mathfrak S_{n+1}$-module, 
the orbit $\mathfrak S_{n+1}\cdot q_-$ spans the projectivization  $\mathbf P(V_-)$ of $V_-$.   
Thus the linear span of 
$ \mathfrak S_{n+1}\cdot q$ and $q_{m+2}$ contains both $\mathbf P(V_-)$ and $\mathbf P(W)$ hence the lemma follows.  \end{proof}

The previous lemma will be used in the proof of the second one, which is the following: 
\begin{lem} 
\label{Lem:TECH}
Let $\{ \, q_{i,j}=q_{j,i}\, \lvert \,  (i,j)\in \boldsymbol{I}_n\, \}$ be a family of 
$(n+3)(n+2)/2$ 
points in a projective space $\mathbf P^N$. Assume that these points satisfy the following hypotheses: 
\begin{enumerate}
\item[$(H_1)$.] For any $i=1,\ldots,n+3$, the $n+2$ points $q_{i,j}$'s for $j\neq i$ 
are in general position in their span, which is a 
 linear subspace  of dimension $n$ of\, $\mathbf P^N$, denoted by  $\mathbf P^n_i$. 
\sk
\item[$(H_2)$.] For any ${(i,j)}\in \boldsymbol{I}_n$, the intersection of\, $\mathbf P^n_i$ with $ \mathbf P^n_j$ 
consists of the point $ q_{i,j}$;
\sk 
\item[$(H_3)$.]
There is a linear $\mathfrak S_{n+3}$-representation on 
$\mathbf P^N$  such that : \sk \\
{\rm (a).} for every $\sigma \in \mathfrak S_{n+3}$ and  any $(i,j)\in \boldsymbol{I}_n$, one has  $\sigma \, q_{i,j}=q_{\sigma i,\sigma j}$ ; \sk \\
{\rm (b).} the restricted action of $\mathfrak S_{n+2}\simeq {\rm Fix}_{\mathfrak S_{n+3}}(i)$ on $\mathbf P^n_i$ is irreducible for any $i\leq n+3$.\sk 
 \end{enumerate}
 
 Then one necessarily  has  $N\geq (n+2)(n+1)/2-2$.  
\end{lem}
\begin{proof}
There is no loss of generality by assuming that the ambiant projective space $\mathbf P^N$  
is spanned by the $q_{i,j}$'s, which we do in what follows. 

For $\ell\in \{ 1,\ldots,n+3\}$, we set 
$P_\ell=\big\langle \, \mathbf P^n_1,\ldots,\mathbf P^n_\ell\, \big\rangle\subset \mathbf P^N
$  and one denotes (a bit abusively) by $\mathfrak S_{n+3-\ell}$ the subgroup $\cap_{i=1}^\ell {\rm Fix}_{\mathfrak S_{n+3}}(i)=
\mathfrak S_{\{\ell+1,\ldots,n+3\}}
$ of $\mathfrak S_{n+3}$.  The $P_\ell$'s  are linear subspaces of $\mathbf P^N=P_{n+3}$
which form an increasing sequence (for the inclusion).   First remark that for any $\ell=1,\ldots,n+2$, since $q_{i,\ell+1}\in 
\mathbf P^n_i\subset 
 P_\ell$ for $i=1,\ldots,\ell$ and because these points are in general position in $\mathbf P^n_{\ell+1}$ (by $(H_1)$), 
 one has $\dim\,  \big( P_\ell\cap \mathbf P^n_{\ell+1}\big) \geq  \max(\ell-1,n) $.

One has  $\dim\,  \big( P_1\cap \mathbf P^n_{2}\big) = 0$  thanks to $(H_2)$ 
hence the set of indices $\ell$ such that  $\dim\,  \big( P_\ell\cap \mathbf P^n_{\ell+1}\big)= \ell-1 $ is not empty. Its maximal element, denoted by $L$, is a well defined element of $\{1,\ldots,n+1\}$.  We are going to  
prove that $P_{L+1}=\mathbf P^N$ with $L\geq n-1$, from which
it will be easy to get the looked sought-after minoration. \sk

For $\ell$ such that  $1\leq \ell\leq L$, the points $q_{1,\ell+1}, 
\ldots,q_{\ell,\ell+1}$  all belong to $P_{\ell+1}$. Since $\dim\big(P_{\ell}\cap \mathbf P^n_{\ell+1}\big)$ has dimension $\ell-1$  and because of $(H_1)$, we get that 
the points $ q_{i,\ell+1}$ for $i=\ell+1,\ldots,n+2$ are projectively linearly independent modulo $P_\ell$. By induction, we deduce that 
\begin{enumerate}
\item[1.] the points 
$q_{i,j}$'s for $i=1,\ldots,\ell$ and $j=i+1,\ldots,n+2$ form   a projective frame for $P_\ell$;
\sk 
\item[2.] hence one has 
$\dim(P_{\ell+1})=N(\ell)$ with 
$ N(\ell)=
   \Big[  \, 
 \big(n+1\big) +\big( n+1-1\big)+\cdots + \big(n+1-\ell\big)\, 
  \Big]-1$.
\end{enumerate}

It follows that in order to prove the proposition, one has to establish that $L\geq n-1$.
Assuming the contrary and setting $k=L+1$, one has $k\leq n-1$ and 
$\dim\,  \big( P_k\cap \mathbf P^n_{k+1}\big) >
 k-1$ hence there exists a point 
 $q\in P_{k}\cap \mathbf P^n_{k+1} \setminus \big\langle 
 q_{1,k+1},\ldots,q_{k,k+1}
 \big\rangle$. 
The group $\cap_{i=1}^{k+1} {\rm Fix}(i)\subset \mathfrak S_{n+2}$, denoted 
a bit abusively
 by $\mathfrak S_{n+3-(k+1)}$ here, acts on $\mathbf P^{n}_{k+1}$. 
Thanks to hypothesis $(H_3.b)$, Lemma \ref{Lem:TECH0} applies and gives us that 
 $\mathbf P^n_{k+1}$ is spanned by the union of $\{  q_{1,k+1},\ldots,q_{k,k+1}\}$ with the
 $\mathfrak S_{n+3-(k+1)}$-orbit of $q$: one has
  \begin{equation}
  \label{Eq:Pn-k+1}
  \Big\langle \, 
   \mathfrak S_{n+3-(k+1)}\cdot q \, 
 \, , \,   q_{1,k+1}
 \, , \,  
  \ldots
  \, , \, 
  q_{k,k+1}\, 
  \Big\rangle  =\mathbf P^n_{k+1}\, . 
\end{equation}
%

 Since $q$ belongs to $P_k$ which  is invariant under the action of $\mathfrak S_{n+3-(k+1)}
 =\cap_{i=1}^{k} {\rm Fix}(i)
 $, one 
 has that $\mathfrak S_{n+3-k}\cdot q\subset P_k$. 
 Combined with the fact that  $q_{i,k+1}\in \mathbf P^n_i\subset P_k$ for $i=1,\ldots,k$, 
it then  
follows from 
 \eqref{Eq:Pn-k+1} that  $\mathbf P^n_{k+1}\subset P_k$. 
On the other hand, for any  $l>k+1$, the permutation $(k+1,l)\in \mathfrak S_{n+3}$
 interchanges $\mathbf P^n_{k+1}$ and $\mathbf P^n_{l}$ while leaving $P_k$ fixed. 
 One deduces that $\mathbf P^n_l \subset P_{k}$ for any $l= k+1,\ldots,n+3$ hence $P_k=P_{n+3}=\mathbf P^N$.


 Let $F=\big\{ \, p\in \mathbf P^N
  \, \big\lvert \,\tau \cdot p=p\, \big\}\subset \mathbf P^N$ be the fixed point set
 of $\tau=(n+2,n+3)\in \mathfrak S_{n+3}$. 
 From $( H_1)$, it comes that 
  the set of $q_{i,j}$'s with $i=1,\ldots,k$ and $j=i+1,\ldots,n+2$ spans $P_k$. 
  But since any such point $q_{i,j}$  belongs to $F$ if $j\leq n+1$, it comes that 
\begin{equation}
\label{Eq:mopo} 
 \dim \big( P_k \big)\leq \dim \big( F \big)+
 {\rm Card}\Big(\, \big\{ \, 
 q_{1,n+2},\ldots,  q_{k,n+2} \, 
\big\}\,  \Big)
=  \dim \big( F \big)+k 
  \, .
  \end{equation}
  
 On the other hand,   since 
 $\tau\big(F\big)=F$ and 
 $\tau\big( \mathbf P^{n}_{n+3}\big)= \mathbf P^{n}_{n+2}$, it follows that $F\cap  \mathbf P^{n}_{n+3}
 =  \{ \,q_{n+2,n+3}\, \}$.  Hence $\dim\big( F  \cap \mathbf  P^{n}_{n+3}\big)=0$ so 
  ${\rm codim}_{\mathbf P^N}(F)\geq n$, which can also be written
\begin{equation}
\label{Eq:mopa}
\dim \big( F \big)+n \leq  \dim \big( P_k \big) \, .
  \end{equation}

  From   \eqref{Eq:mopo} and \eqref{Eq:mopa} together, one gets $n\leq k$ contradicting the assumption $L<n-1$. Thus necessarily $L\geq n$ which gives us the proposition. 
\end{proof}

This lemma has an  interesting consequences when applied to  $\mathbf P^N= \mathbf P\big(\boldsymbol{AR}_C(\boldsymbol{\mathcal W}_{0,n+3})\big) $ with 
$q_{i,j}=\mathbf P\big(\langle AR_{i,j}\rangle\big)$  for any $i, j=1,\ldots,n+3$ with $i\neq j$. 
For any subset $J\subset \{1,\ldots,n+3\}$, one sets $\boldsymbol{A}_J=\langle 
\hspace{0.15cm}  \boldsymbol{A}_{\widehat{\jmath}}\hspace{0.15cm} \lvert \, j\in J\rangle\subset \boldsymbol{AR}\big( \boldsymbol{\mathcal W}_{0,n+3}\big)$ 
and 
\begin{equation}
\label{Eq:M}
\eta=\min_{
\scalebox{0.7}{
$J\subset 
\{ 1,\ldots,n+3\} $ }}\Big\{ \hspace{0.15cm} \lvert \, J\, \lvert \hspace{0.15cm}\big\lvert 
\hspace{0.15cm}
\boldsymbol{A}_J=\boldsymbol{AR}_C\big(\boldsymbol{\mathcal W}_{0,n+3}\big)
\hspace{0.15cm}\Big\}\, . 
\end{equation}

Then, from the previous lemma as well as from  its proof, one deduces the 
\begin{cor} 
\label{C:M=n-or-M=n+1}
\begin{enumerate}
\item[]  ${}^{}$\hspace{-1.2cm} {\rm 1.}\hspace{0.13cm}Either $\eta=n$ or $\eta=n+1$. 
\mk
\item[2.]
In both cases, one has $\boldsymbol{A}_J=\boldsymbol{AR}_C\big(\boldsymbol{\mathcal  W}_{0,n+3}\big)$ for any subset $J$ of cardinality $\eta$ and the $ AR_{i,j}$'s for $ i=1,\ldots,\eta$ and $j=i+1,\ldots,n+2$ form a basis of 
$\boldsymbol{AR}_C\big(\boldsymbol{\mathcal  W}_{0,n+3}\big)$.
\mk
\item[3.]  Consequently,  $\boldsymbol{AR}_C\big(\boldsymbol{\mathcal  W}_{0,n+3}\big)$ has 
dimension $(n+1)(n+2)-1$ if $\eta=n$ and dimension one more when $\eta=n+1$. 
In the latter case, $\boldsymbol{\mathcal  W}_{0,n+3}$ has maximal rank hence necessarily 
$ \boldsymbol{AR}_C\big(\boldsymbol{\mathcal W}_{0,n+3}\big)=\boldsymbol{AR}\big(\boldsymbol{\mathcal  W}_{0,n+3}\big)$. 
\end{enumerate}
\sk 
\end{cor}

Another nice feature of Lemma \ref{Lem:TECH} is that it is rather easy to deduce from it a 
description 
of $ \boldsymbol{AR}_C\big(\boldsymbol{\mathcal W}_{0,n+3}\big)$ as a 
$\mathfrak S_{n+3}$-module: 
\begin{cor} 
\label{P:AR-C-as-a-S(n+3)-module}
\begin{enumerate}
\item[]  ${}^{}$\hspace{-1.3cm} {\rm 1.}\hspace{0.14cm}As a $\mathfrak S_{n+3}$-representation, $ \boldsymbol{AR}_C\big(\boldsymbol{\mathcal W}_{0,n+3}\big)$ is irreducible.
\mk 
\item[2.] Its Young symbol is $[221^{n-1}]$ if $\eta  =n$ and $[31^n]$  when $\eta =n+1$. 
\end{enumerate}
\end{cor} 
\begin{proof}  Assume that $  \boldsymbol{A}_C=\boldsymbol{AR}_C\big(\boldsymbol{\mathcal W}_{0,n+3}\big)$
is not irreducible. Pick $F\subset  \boldsymbol{A}_C$, a non trivial and proper stable sub-$\mathfrak S_{n+3}$-module of maximal dimension. 
Since the $\mathfrak S_{n+3}$-orbit of any $\boldsymbol{A}_{\widehat{\imath}}$ coincides with the whole space $ \boldsymbol{AR}_C\big(\boldsymbol{\mathcal W}_{0,n+3}\big)$, 
we have that $F\cap \boldsymbol{A}_{\widehat{\imath}}=\{0\}$ for any $i=1,\ldots,n+3$, hence 
${\rm codim}(F)\geq n+1$. 

Let $\pi: \mathbf P \boldsymbol{A}_C
\setminus 
\mathbf PF\rightarrow \mathbf P(\boldsymbol{A}_C /F\big)=\mathbf P^M$ with $M=\dim \big(\boldsymbol{A}_C\big)-\dim(F)-1$. For any $i=1,\ldots,n+3$, since $\mathbf PF\cap \mathbf P \boldsymbol{A}_{\widehat{\imath}}=\emptyset$, it follows that 
$\pi\big( \mathbf P \boldsymbol{A}_{\widehat{\imath}}\big)$ has dimension $n$, 
which justifies denoting this subspace by $\mathbf P_{i}^n$. In particular, one has $M\geq n$. 
\sk

For any $i,j$ distinct, let us prove that $\mathbf P^n_{i}\cap \mathbf P^n_{j}=q_{i,j}=\pi\big(
p_{i,j}
\big)$ with $p_{i,j}=[AR_{i,j}]\in \mathbf P \boldsymbol{A}_C$.  Otherwise, there would exist another point $q$ in this intersection, distinct from 
$q_{i,j}$. 
We denote here  by $\mathfrak S_{n+2}$ (resp.\,$\mathfrak S_{n+1}$) the subgroup of $\mathfrak S_{n+3}$ formed by permutations fixing $i$ (resp.\,fixing both $i$ and $j$): one has 
$\mathfrak S_{n+1}={\rm Fix}(i)\cap {\rm Fix}(j)< \mathfrak S_{n+2}={\rm Fix}(i)<\mathfrak S_{n+3}$.  Note that the action of $\mathfrak S_{n+2}$ on $\mathbf P^n_{i}$ is irreducible with Young symbol $[2,1^n]$ according to Proposition \ref{P:Damiani-4.1.4}.3.  
On the other hand
$$\big[2,1^n\big]\, \big\lvert_{{\mathfrak S}_{n+1}}=\big[1^{n+1}\big]\oplus \big[2,1^{n-1}\big]$$
as it follows immediately from the classical branching rules for the representations of  symmetric groups (see \cite[\S2.8]{Sagan}). 
 This  means that the representation of $\mathfrak S_{n+1}$ on $\mathbf P^n_{i}$
induced by restriction  
  admits a unique fixed point, which is $q_{i,j}$ of course, and acts irreducibly on a supplementary hyperplane.  It follows that  
the orbit $\mathfrak S_{n+1}\cdot q$ spans a hyperplane in $\mathbf P^n_i$ which does not contain $q_{i,j}$ hence  one has 
$
\big\langle  q_{i,j}\, , \, \mathfrak S_{n+1}\cdot q
\big\rangle =\mathbf P^n_i
$ ({\it cf.}\,also Lemma \ref{Lem:TECH0}). Since all these arguments also apply when exchanging $i$ and $j$, one has 
$
\big\langle  q_{i,j}\, , \, \mathfrak S_{n+1}\cdot q
\big\rangle =\mathbf P^n_j 
$ as well,  from which it comes that $\mathbf P^n_i=\mathbf P^n_j$.  Since all $\mathbf P^n_i$'s coincide and because their union spans $\mathbf P^M$, we get that $M=n$, hence that ${\rm codim}(F)=n+1$. But this is impossible: indeed, from our assumption of the maximality of $F$ (as a proper submodule of $\boldsymbol{A}_C$), it comes that $\boldsymbol{A}_C/F$ is an irreducible $\mathfrak S_{n+3}$-representation of dimension $n+1$ and such a thing does not exist when $n\geq 2$. \sk 

From the above, it comes that $\mathbf P^M$ and the points $q_{i,j}$ with $(i,j)\in \boldsymbol{I}_n$ satisfy all the assumptions of Lemma \ref{Lem:TECH} which therefore applies and gives us that $M\geq (n+1)(n+2)-2$. Since $F$ has been assumed to be non trivial and proper, the single possibility is that $\boldsymbol{A}_C$ and $F$ have dimension $(n+1)(n+2)/2$ and 1 respectively. But since any $\mathfrak S_{n+3}$-representation is completely reducible, $F$ must admit a complementary invariant subspace of dimension $(n+1)(n+2)-1>1=\dim(F)$, contradicting the maximality of $\dim(F)$ among the non trivial proper sub-representations of $\boldsymbol{A}_C$. Thus  there is no such a sub-representation, which proves the first point.\mk 

Let $[\lambda]$ be the Young diagram of 
$\boldsymbol{AR}_C\big(\boldsymbol{\mathcal W}_{0,n+3}\big)$ as a $\mathfrak S_{n+3}$-module, for a partition $\lambda$ of  $ n+3$.  Since  $\mathfrak S_{n+2}={\rm Fix}(n+3)$ acts irreducibly on 
$\boldsymbol{A}_{\widehat{n+3}}$ with Young symbol $[2,1^n]$ according to the third point of Proposition \ref{P:Damiani-4.1.4}, it comes that $[\lambda]\, \big\lvert_{\mathfrak S_{n+2}}$ admits 
$[2,1^n]$ as one of its irreducible factors. 
Since $[\lambda]$ is obtained by adding one box to $[2,1^n]$, 
we deduce that there are only two possibilities for the former Young symbol: it is either 
$[3,1^n]$ which has dimension $(n+1)(n+2)/2$, or $[2,2,1^n]$, which is of dimension one less. This finishes the proof of the proposition. 
\end{proof}

\begin{lem} 
\label{L:n-odd}
When $n$ is odd,  $AR_{n+2,n+3}$ does not belong 
to $\Big\langle \, AR_{i,j} \hspace{0.1cm} \big\lvert 
\, \scalebox{0.85}{\begin{tabular}{l}
$i=1,\ldots,n$  \vspace{-0.09cm}
\\
${}^{}$  \hspace{-0.3cm} $j=i+1,\ldots,n+2$
\end{tabular}} \hspace{-0.05cm}\Big\rangle$.
%
\end{lem} 

\begin{proof}
Assume that there exist scalars $c_{i,j}$
for $(i,j)\in \boldsymbol{I}_n^n$ such that 
\begin{equation}
\label{Eq:An+2,n+3}
A_{n+2,n+3}=\sum_{(i,j)\in \boldsymbol{J}_n} c_{i,j}\, AR_{i,j} 
 \, .
\end{equation}
From the second point of Corollary \ref{C:M=n-or-M=n+1}, it follows that this  relation  is necessarily unique. 

On the other hand,  for any $(k,l)$ with $k=1,\ldots,n$ and $l=k+1,\ldots,n+1$, it follows from 
Lemma \ref{L:(i,j)-(k,l)}  that  
$(k,l)\cdot A_{n+2,n+3}=-A_{n+2,n+3}$ and $(k,l)\cdot A_{k,l}=(-1)^{n-1} A_{k,l}=A_{k,l}$ (since $n-1$ is even).  Applying $(k,l)$ to \eqref{Eq:An+2,n+3} would give another such identity   if $c_{k,l}$ were not zero, and this is not possible by the argument given just above. We thus get that all the constants $c_{i,j}$ are zero for $i=1,\ldots,n$ and $j=1,\ldots,n+1$. 

Hence  \eqref{Eq:An+2,n+3} actually is  written 
$$
A_{n+2,n+3}=\sum_{i=1}^n c_{i,n+2}\, AR_{i,n+2} \, . 
$$
and such a relation does not exist since, as it follows from 
the first point of Proposition \ref{P:Damiani-4.1.4}, the combinatorial 
ARs appearing in it form a free family.  Thus there is no identity of the form 
\eqref{Eq:An+2,n+3}, which gives us the lemma. 
\end{proof}

With the previous results at hand, it is then easy to establish Theorem \ref{THM:AR-C(W)-S-n+3-module}: 
\begin{proof}[Proof of Theorem \ref{THM:AR-C(W)-S-n+3-module}.]
When $n$ is odd, it follows from Lemma \ref{L:n-odd} 
that  the  integer $\eta$ defined in \eqref{Eq:M} is equal to $n+1$.  Hence 
$ \boldsymbol{AR}_C\big(\boldsymbol{\mathcal W}_{0,n+3}\big)$ has dimension $(n+1)(n+2)/2$  
from which it follows that  $\boldsymbol{\mathcal W}_{0,n+3}$ has maximal rank with all its ARs combinatorial.  Then Corollary \ref{P:AR-C-as-a-S(n+3)-module}.2 allows us to conclude.\sk 

To prove the theorem in case when $n$ is even, it suffices to prove that $\eta=n$, which would follow from exhibiting a non trivial abelian relation which is not combinatorial. We will show that the so-called Euler's abelian relation does the job, hence the theorem follows. 
 \end{proof}


\subsection{} 
\label{SS:Components-of-ARs}
Leaving aside the question of the structure of $ \boldsymbol{AR}_C( \boldsymbol{\mathcal W}_{0,n+3})$ as a $\mathfrak S_{n+3}$-module, it is not difficult to describe another rather efficient computational approach for building the combinatorial abelian relations. To this end, we first need to set some notation which will be also be used later on. \sk 

Given $x=(x_1,\ldots,x_n)$, we use the shorthand $x'=(x_2,\ldots,x_n)=U_1(x)$; we 
introduce two new but fixed variables $x_{n+1}=0$ and $x_{n+2}=1$; we set $\tilde x=(x_1,\ldots,x_n,x_{n+1},x_{n+2})$; and for $i,j$ such that $i=1,\ldots,n+1$ and $j=i+1,\ldots,n+2$, we set $\tilde x_{\widehat{\imath\jmath}}=(x_s)_{s=1, s\neq i,j}^{n+2}$. 
We allow ourselves to combine all these notation:  for instance, 
for $i,j$ as above  $\tilde x'_{\widehat{\imath \jmath}}$ stands for the $(n-1)$-tuple obtained by removing the first, the $i$th and the $j$-th entries to $\tilde x$.
Since we will always deal with pairs of indices $(i,j)$ with $2\leq i\leq n$, 
there will be no ambiguity with this notation in what follows.  We denote by $B_{n}$ the set of pairs we will work with: 
$$
B_{n}=\Big\{ \, \big( i,j\big) \hspace{0.1cm} \big\lvert 
 \hspace{0.1cm} i=2,\ldots,n,\, j=i+1,\ldots,n+2\, 
\Big\}\, .
$$

For any $N>0$, we consider the following homogeneous polynomial 
of $z\in \mathbf C^N$: 
\begin{align*}
M_{0,N}(z)= & \, \prod_{1\leq p<q\leq N}\big( z_p-z_q\big) 
\, .
\end{align*}

\begin{rem}
For $N=n+3$, one has $M_{0,N}(\tilde x)=- \prod_{1\leq i\leq n} x_i(x_i-1) \prod_{1\leq i<j\leq n}(x_i-x_j)$ for any $x\in \mathbf C^n$ hence the equation $M_{0,N}(\tilde x)=0$ 
 cuts out the braid arrangement $A_{n+3}$ in $\mathbf C^n$, 
whose complement $ \mathbf C^n\setminus A_{n+3}$ 
  is isomorphic to  the moduli space $\mathcal M_{0,n+3}$  (see \eqref{Eq:M0n-Cn-An}).  This explains the notation 
  $M_{0,N}=M_{0,n+3}$ for the polynomial above.
\end{rem}
 \sk
 
We define rational functions $F_0$ and $F_{ij}$ for   $(i,j)\in B_{n}$, of $x'$, by setting 
\begin{align}
\label{Eq:F0-Fij}
F_0(x')=\frac{1}{x_2\ldots x_n}
\qquad   \mbox{ and } \qquad  
F_{ij}(x')= \Big( \tilde x_i-\tilde x_j\Big)^{n-1} \frac{
M_{0,n-1}\Big( \tilde x'_{\widehat{\imath \jmath}} \Big)
}{M_{0,n+1}\big(\tilde x'\big)}\quad \mbox{for any }\, (i,j)\in B_n\, .
\end{align}
 
 Then we set ${B}_{n}(x')=\{\, F_{ij}\,\}_{(i,j)\in B_n}$ and 
\begin{align}
\label{Eq:mathfrak-Bn}
\boldsymbol{\mathfrak B}_n(x')=
\begin{cases}
 \hspace{0.1cm} {B}_n(x') 
  \hspace{1.4cm} \mbox{ if } n \mbox{ is even}\,; \vspace{0.1cm}\\
    \hspace{0.1cm} {B}_n(x')  \cup 
    \big\{ \, 
 F_0
\, 
\big\}
  \hspace{0.2cm} \mbox{ if } n \mbox{ is odd}\,.
\end{cases}
\end{align}

The cardinal of $\boldsymbol{\mathfrak B}_n(x')$ is 
$(n-1)(n+2)/2$ for $n$ even and 
$n(n+1)/2=(n-1)(n+2)/2+1$ otherwise. In view of 
describing the components  of the ARs of $\boldsymbol{\mathcal W}_{0,n+3}$, 
we set 
\begin{equation}
\label{Eq:yolop}
\begin{tabular}{l}
${}^{}$  \hspace{-1.4cm} $\bullet$ $V_1(x)=U_1(x)=x'=(x_2,\ldots,x_n)$  and $\Pi_1=\Omega_1=dx_2\wedge \cdots\wedge dx_n$; 
\vspace{0.15cm}\\
${}^{}$  \hspace{-1.4cm} $\bullet$ $V_i(x)=\big(C^{\circ (i-1)}\big)^*\big(V_1\big)$ and 
$\Pi_i=\big(C^{\circ (i-1)}\big)^*\big(\Pi_1\big)$
for $i=2,\ldots,n+3$;
\vspace{0.15cm}\\
${}^{}$  \hspace{-1.4cm} $\bullet$ $\boldsymbol{\mathfrak B}_n(V_i)=
\big(C^{\circ (i-1)}\big)^*\Big(\boldsymbol{\mathfrak B}_n(x')  \Big)
=\big\{ \,F\circ V_i\,\big\}_{ F\in \boldsymbol{\mathfrak B}_n(x')}$ 
for $i=1,\ldots,n+3$.
\end{tabular}
\end{equation}

By considering their respective poles, one gets easily that the elements of 
$\boldsymbol{\mathfrak B}_n(x')$ are linearly independent.  Then, by direct elementary but computational checks in what concerns the second point, we prove the 
\begin{prop} 
\label{P:La-Base}
{\rm 1.}  For any $n\geq 2$, the elements of 
$\boldsymbol{\mathfrak B}_n(x')$ are linearly independent over $\mathbf C$; 
\sk 

\noindent {\rm 2.} The following holds true for any $n\leq 12$: for any $i=1,\ldots,n+3$, the space 
$ \boldsymbol{AR}_C( \boldsymbol{\mathcal W}_{0,n+3})[i]$
 of $V_i$-th components of combinatorial abelian relations 
admits $\big\{ \, F_i\, \Pi_i\, \lvert \, F_i\in \boldsymbol{\mathfrak B}_n(V_i)\, \big\}$
 as a basis. 
%
\end{prop}

%
%
%
%
%

The interest of this result is threefold: first and of course, we conjecture that it holds true in full generality. Second, it indicates that $ \boldsymbol{AR}_C( \boldsymbol{\mathcal W}_{0,n+3})$ can be determined by pure linear algebra in the finite dimensional direct sum $\oplus_{i=1}^{n+3} 
\boldsymbol{\mathfrak B}_n(V_i)$. This elementary fact opens the door for the possibility to arrive at an explicit closed formula for (a multiple of) the non trivial abelian relation $AR_{i,j}  $ of 
$\boldsymbol{AR}_{\widehat{\imath \jmath}}$ discussed just above, this for any $(i,j)\in 
\boldsymbol{I}_n$. Finally, as we will see  in the next section, the proposition above and more generally the notation and formulas of the current subsection are useful  to explicitly describe Euler's abelian relation. 

\begin{rem}
\label{Eq:-OK-with-other-C}
%
One can verify that, if taking instead of the map $C$ in \eqref{Eq:R-C-1} 
the cyclic birational automorphism of  $\boldsymbol{\mathcal W}_{0,n+3}$ given in the coordinates $x_i$'s by 
$$\tilde C(x)=\left( \, \frac{x_n}{x_n-x_{i-1}} \,\right)_{i=1}^n$$ 
(with the convention that $x_{0}=1$), then 
all the preceding definitions make sense  and all the subsequent results (especially the preceding proposition) hold true mutatis mutandis.  This will be used later on in \S\ref{SSSS:ici}.
\end{rem}

%
%

\newpage
\section{\bf On Euler's abelian relation.}
\label{S:Euler-AR}
In this section, we discuss Euler's AR. 
Since every AR of $\boldsymbol{\mathcal W}_{0,n+3}$ is combinatorial hence rational when $n$ is odd, it is certain that there is something wrong in Damiano's claims \eqref{Eq:Direct-Sum} about this particular abelian relation for this parity of $n$. In order to make everything as clear as possible, in this section we present  a thorough study of this AR and before that, of its construction. \sk 

Following Damiano but giving more details, we first describe the 
construction of a new abelian relation  $\boldsymbol{\mathcal E}_n$ for $\boldsymbol{\mathcal W}_{0,n+3}$ using Gelfand-MacPherson's theory applied to a characteristic class on 
 the {\it oriented} Grassmannian variety $G_2^{or}(\mathbf R^{n+3})$ 
  of oriented 2-planes in $\mathbf R^{n+3}$.  
In \S\ref{SS:En-integral-representation}, we give a rather concise integral formula for the components of $\boldsymbol{\mathcal E}_n$ which we use to determine the fourth-order jet of  $\boldsymbol{\mathcal E}_3$ at a specific point.  Then in \S\ref{SSS:Dihedral-Invariance-Properties}, we discuss some invariance properties satisfied by $\boldsymbol{\mathcal E}_n$ for a certain dihedral action preserving a particular connected component of  $G_2^{or}(\mathbf R^{n+3})$ (its {\it positive part $G_2^{or}(\mathbf R^{n+3})^{>0}$}) and explain how two non trivial functional identities satisfied by the components of $\boldsymbol{\mathcal E}_n$ can be deduced from them. 

After having completely explicited $\boldsymbol{\mathcal E}_3$ in \S\ref{SSS:tilde-e2}, we turn to 
the invariant properties of $\boldsymbol{\mathcal E}_n$ with respect to the birational action of the whole symmetric group $\mathfrak S_{n+3}$ on $\mathcal M_{0,n+3}$ in \S\ref{SSS:Sn+3-Invariance-Properties}. We show that things actually are more subtile than as described by Damiano, the main and crucial fact being that, contrarily to the case when $
n$ is even, when $n$ is odd Euler's abelian relation is actually  not canonically defined on each connected component of ${\mathcal M}_{0,n+3}$, but only up to sign. 
We explain that this fact, that does not seem really significant at first glance, actually has the consequence that Damiano's construction of a 1-dimensional $\mathfrak S_{n+3}$-representation associated to Euler's abelian relation is irremediably flawed when $n$ is odd. 

Finally, in \S\ref{SSS:En-n-odd} and \S\ref{SS:En-n-even}, which concern the cases when $n$ is odd and even respectively, we give two closed formulas for the components of Euler's abelian relation. Both formulas are conjectural in full generality but are proven to  hold true indeed for $n$ less than or equal to 12 (by means of formal computations).

\subsection{On Gelfand-MacPherson's theory of generalized dilogarithm forms.} 
\label{S:Euler-AR}
Euler's abelian relation is discussed in \cite[\S5]{D} (see also \cite[Chap.\,6]{DThesis}). Its construction relies on the general theory of 'generalized dilogarithm forms' exposed in \cite{GelfandMacPherson}.  Gelfand and MacPherson worked with usual grassmannian varieties $G_N(\mathbf R^{N+M})$ whereas Damiano's chose to do so with oriented grassmannians 
$ G_N^{or}(\mathbf R^{N+M})$, moreover in the specific case when $N=2$ and $M=n+1$. 
In order to show more clearly  where  the problem in Damiano's construction of a polylogarithmic AR for $\boldsymbol{\mathcal W}_{0,n+3}$ 
lies, we give below an overview of the corresponding constructions.   For  proofs, we refer the reader to the two previously cited articles. 

\subsubsection{} 
\label{SSS:GrassmannianStuff}
In ordre to consider both cases together, namely the standard and the oriented grassmannians and associated configuration spaces,  we agree that $*$ either stands for nothing or for ${\rm `or'}$, the former notation referring of course to the standard case and 
the latter to the oriented one. Considering this, we set ${\rm O}^{or}_m(\mathbf R)={\rm SO}_m(\mathbf R)$ for any $m\geq 1$.  \sk 

Below, {\it `plane'} means a subvector space of dimension 2, and when $*={\rm or}$, a {\it `$*$-plane'} is a 2-plane provided with an orientation . We will denote a $*$-plane by $\xi^*$ 
whereas $\xi$ will stand for the associated (non oriented) plane. 

\begin{itemize}
\item[$-$] $G^*_2(\mathbf R^{n+3})={\rm O}^{*}_{n+3}(\mathbf R)/\big( {\rm O}^{*}_{2}(\mathbf R)
\times  
{\rm O}^{*}_{n+1}(\mathbf R)
\big) $ is the grassmannian of $*$-planes in $\mathbf R^{n+3}$; 
\sk 
\item[$-$] ${\rm GL}_{n+3}(\mathbf R)$ acts on $\mathbf R^{n+3}$ hence on 
$G^*_2(\mathbf R^{n+3})$ as well. This action is not effective: since the center  $\mathbf R^* {\rm Id}$ acts trivially, this action factors through $\mathbf P{\rm GL}_{n+3}(\mathbf R)= {\rm GL}_{n+3}(\mathbf R)/
( \mathbf R^* {\rm Id})$;
\sk
\item[$-$]  $\widehat G^*_2(\mathbf R^{n+3})$ stands for the dense open subset  of {\it generic} $*$-planes  in $\mathbf R^{n+3}$, which are those which 
do not contain any coordinate line and which are  not contained in any coordinate hyperplane;
\sk
\item[$-$]  $\widetilde H
\subset {\rm GL}_{n+3}(\mathbf R)$  denotes the abelian subgroup of diagonal matrices and $H=H/ ( \mathbf R^* {\rm Id})$ stands for its image in $\mathbf P{\rm GL}_{n+3}(\mathbf R)$. The latter group is a maximal torus 
isomorphic to $( \mathbf R^*)^{n+2}$ (in natural but non unique  ways);
\sk
\item[$-$] $\widetilde H_0$ and $H_0$ are the connected components of the identity of $\widetilde H$  and 
$H$ respectively (thus $H_0\simeq ( \mathbf R_{>0})^{n+2}$ but  again not in a canonical way);
\sk 
\item[$-$] $\widetilde K$ is the subgroup of $\widetilde H$ formed by the  matrices with diagonal coefficients in $\{\pm 1\}\simeq \mathbf Z_2$ and $K$ stands for its  image in $H$.  One has 
$K\simeq (\mathbf Z_2)^{n+2}$ and since $H$ is commutative, we have that $H\simeq H_0\times K\simeq (\mathbf R_{>0})^{n+2}\times (\mathbf Z_2)^{n+2}$;
\sk 
\item[$-$] we define $C_2^{*}(n+3)=\widehat G^*_2(\mathbf R^{n+3})/ H$ 
(resp.\,$EC_2^{*}(n+3)=\widehat G^*_2(\mathbf R^{n+3})/ H_0$) as the space of 
{\it $*$-configurations} (resp. {\it enhanced $*$-configurations});
\sk 
\item[$-$] given an oriented 2-plane $\xi^{or}$, we denote by $\check{\xi}^{or}$ the same plane but endowed with the opposite orientation, and by $\xi$ the standard plane (disregarding the orientations); 
\sk 
\item[$-$]
 The map $D:\xi^{or}\rightarrow D(\xi^{or})=\check{\xi}^{or}$ is an involutive isomorphism  which is the deck transformation of the 2-1 (universal) covering $ \nu : 
G^{or}_2(\mathbf R^{n+3})\rightarrow  G_2(\mathbf R^{n+3})$, $\xi^{or}\mapsto \xi$.
\end{itemize}

The spaces, the quotients maps by $H$ and $H_0$ and the coverings induced by $\nu$ discussed above  
all  fit into the following commutative diagram : 
\begin{equation}
\label{Eq:Diagrammm}
\begin{tabular}{c}
  \xymatrix@R=0.4cm@C=1.5cm{ 
   \ar@/_4pc/@{->}[dddd]_{\pi_H}  \widehat {G}_{2}^{or}\big(\mathbf R^{n+3}\big) 
\ar@{->}[dd]_{\pi_{0}}  \ar@{->}[r]^{\nu
 } & \widehat {G}_{2}\big(\mathbf R^{n+2}\big)\, 
\ar@{->}[dd]^{\pi_{0}}
 \ar@/^4pc/@{->}[dddd]^{\pi_H}
\\ & \\
EC_2^{or}(n+3) \ar@{->}[dd]_{\kappa} 
 \ar@{->}[r]^{ \hspace{0.1cm} 
\nu
 }  & 
 EC_2(n+3)  
 \ar@{->}[dd]^{\kappa} 
\\ & \\
C^{or}_2(n+3)  \ar@{->}[r]_{ \hspace{0.1cm}\nu} 
& C_2(n+3) \, . 
  }
  \end{tabular}
\end{equation}

In this diagram, he horizontal maps all are  non ramified 2-1 coverings induced by $\nu$ (hence are denoted by the same notation),  the maps $\pi_0$ and $\kappa$ correspond to quotienting by $H_0$ and  $K$ respectively, hence the compositions $\pi_H=\kappa\circ \pi_0$ are the maps induced by quotienting by $H$. 
\sk 

As is well-known,  $C^*_2(n+3)$  naturally identifies with  the space $(\mathbf P^1_{\mathbf R}\big)^{n+3}- \Delta$ of $(n+3)$-tuple of pairwise distinct points on the real projective line\footnote{Here $\Delta$ stands for the union of all the small diagonals 
in 
$\big(\mathbf P^1_{\mathbf R}\big)^{n+3}$.}, quotiented by ${\rm GL}_2^*(\mathbf R)/(\mathbf R_{>0} {\bf I}_2)$ 
where we agree that ${\rm GL}_2^{or}(\mathbf R)={\rm GL}_2^{+}(\mathbf R)=\{ g\in {\rm GL}_2(\mathbf R)\, \lvert \, {\rm det}(g)>0\,\}$. 
In the standard case when $*$ is empty, one recovers $\mathcal M_{0,n+3}(\mathbf R)$. When $*=or$,  the quotient $C^{or}_2(n+3)=\widehat G^{or}_2(\mathbf R^{n+3})/ H$ 
can be seen 
 as the set of $(n+3)$-tuples of pairwise distinct points on $ \mathbf P^1_{\mathbf R}\simeq \mathbf S^1=\{ z \in \mathbf C \, , \, \lvert z\lvert=1 \} \subset \mathbf C$ modulo the action of the subgroup ${\rm Mob}^+(\mathbf S^1)$ of Mo\"ebius transformations preserving the standard orientation of $\mathbf S^1$.  Hence it appears  natural to denote this space of oriented configurations by $\mathcal M_{0,n+3}^{\,or}(\mathbf R)$.  The map 
 $$\nu : 
 \mathcal M_{0,n+3}^{\,or}(\mathbf R)\rightarrow \mathcal M_{0,n+3}(\mathbf R)$$
  is simply given by identifying configurations when reversing the orientation of $ \mathbf P^1_{\mathbf R}\simeq \mathbf S^1$. \sk 

To explain what an enhanced configuration is, we denote by ${\mathbb S}^1\subset \mathbf R^2$ the circle from which $\mathbf P^1_{\mathbf R}=\mathbf S^1$ is obtained by identifying diametrically opposite points.  Then an `enhanced configuration' (that is, an element of 
$EC_2(n+3)$) 
 is a projective equivalence class of a $(n+3)$-tuple $(c_i)_{i=1}^{n+3}\in \mathbf P^1_{\mathbf R}-\Delta$ enhanced by a choice of a lift $\tilde c_i\in{\mathbb S}^1$ of $c_i$ for each $i=1,\ldots,n+3$.  Of course, there is a similar description of what  an  `enhanced oriented configuration'  is, 
 we leave it to the interested reader to elaborate on this. 
   \sk
 
 Recall (see \S\ref{SSS:Real-Locus} above)  that $\mathcal M_{0,n+3}(\mathbf R)$ is not connected and that its connected components $\mathcal M(\boldsymbol{\sigma})$ are first, biunivocally labeled by classes $\boldsymbol{\sigma}\in \mathfrak K_{n+3}=\mathfrak S_{n+3}/D_{0,n+3}$ and second, all isomorphic to  $\mathcal M(\boldsymbol{1})=\mathcal M_{0,n+3}^{>0}(\mathbf R)$.  Our purpose below is to explain Damiano's construction of an AR on any component $\mathcal M(\boldsymbol{\sigma})$ and to describe some of its properties. We will denote by $\mathcal E_n(\boldsymbol{\sigma})$ this AR which is an element of $\boldsymbol{AR}(\boldsymbol{\sigma})$. The question of the nature of the 
$\mathcal E_n(\boldsymbol{\sigma})$'s and how these ARs are related is crucial and will be discussed. But in a first step, we will only consider the case of the positive part 
$\mathcal M_{0,n+3}^{>0}(\mathbf R)$ and of the AR $\mathcal E_n(\boldsymbol{1})$ it carries, that we will denote by $\mathcal E_n^{>0}$ to simplify.

\subsubsection{\bf Euler's differential form}
When $n$ is even, one could work with usual grassmannians of 2-planes, which 
are precisely the manifolds used by Gelfand and MacPherson 
to develop their theory of higher polylogarithmic forms and of the differential relations they satisfy. However this is not possible when $n$ is odd (see Remark \ref{Rem:n-odd-vs-n-even} below for an explanation of this) and it is then necessary to work at the level of  the oriented  grassmannian. In order to give a presentation independent of the parity of $n$, we will 
 place ourselves within this `oriented framework'.\sk

We assume below that $\mathbf R^{n+3}$ is endowed with its standard euclidean structure: $e_1,\ldots,e_{n+3}$ are the elements of the canonical basis, that is $e_i=(\delta_{ij})_{j=1}^{n+3}$ for $i=1,\ldots,n+3$. One denotes by $x_1,\ldots,x_{n+3}$  the corresponding standard coordinates, and one has 
$$
\big(x,y\big)=\sum_{k=1}^{n+3} x_ky_k
\qquad \mbox{ and } \qquad \lvert x\lvert=\sqrt{(x,x)}
=\sqrt{ \, (x_1)^2+\cdots+ (x_{n+3})^2}
$$
for any two elements $x=(x_k)_{k=1}^{n+3}$ and $y=(y_k)_{k=1}^{n+3}$ of $ \mathbf R^{n+3}$.  Finally, we will denote by $\mathbf R^{n+2}_i$ the coordinate hyperplane in $
\mathbf R^{n+3}$ cut out by the equation $x_i=0$. 

\mk 
 
 Let $\mathcal T^{or}$ be the tautological bundle over $G_2^{or}(\mathbf R^{n+3})$. Its Euler's class ${\rm E}_n^{or}={\rm E}(\mathcal T^{or})$ is a non trivial element in $\boldsymbol{H}^2( G_2^{or}(\mathbf R^{n+3}),\mathbf R)$.
%
%
%
 Since the oriented grassmannian is a symmetric space, there exists a unique ${\rm SO}_{n+2}(\mathbf R)$-invariant global differential  2-form $E_n^{or}$ on it such that $[E_n^{or}]={\rm E}(\mathcal T^{or})$.  To simplify, we will no longer write the 
 superscript ${}^{or}$ and just write $E_n$ in what follows.\sk

 We consider the Stiefel manifold $\boldsymbol{\mathbb S}_2^{or}(\mathbf R^{n+3})$ of oriented 2-frame in $\mathbf R^{n+3}$:  it is the manifold of dimension $2n+3$ 
 formed by orthogonal 2-frames in $\mathbf R^{n+3}$: 
$$
\boldsymbol{\mathbb S}_2^{or}\big(\mathbf R^{n+3}\big)=\left\{\, 
\big( e_1,e_2  \big)\in \big(\mathbf R^{n+3}\big)^2\hspace{0.1cm} 
\Big\lvert 
\hspace{0.1cm} \lvert e_1\lvert = \lvert e_2\lvert=1\mbox{ and } \, \langle e_1,e_2\rangle=0
\, \right\}\, .$$
 It comes with a natural map $\mathcal S: \boldsymbol{\mathbb S}_2^{or}(\mathbf R^{n+3})\rightarrow G_2^{or}(\mathbf R^{n+3})$, $(e_1,e_2)\mapsto e_1\wedge e_2$  
 which makes of the Stiefel variety a ${\rm SO}_2(\mathbf R)$-bundle over the oriented grassmannian.   The interest of considering the Stiefel manifold is that there is 
 a simple explicit formula for 
the   pull-back of $E_n$ under $\mathcal S$. 

Indeed, for $i=1,2$, let $de_i$ be the $\mathbf R^{n+3}$-valued 1-form whose components are the exterior derivatives of the scalar components of $e_i$: if $e_i=(e_i^k)_{k=1}^{n+3}$, then $de_i=(de_i^k)_{k=1}^{n+3}$.  Then the componentwise wedge product of $de_1$ with $de_2$ is a multiple of the pull-back of 
$E_n$ under $\mathcal S$ since
\begin{equation}
\label{Eq:de1-wedge-de2}
\mathscr E_n = de_1\wedge de_2=\sum_{k=1}^{n+3} de_1^k \wedge de_2^k= \frac{-1\,}{2\pi}\, \mathcal S^*\Big( E_n \Big)
\end{equation}
 ({\it cf.}\,\cite[\S6.2]{DThesis} or \cite[\S6]{D} for more details and references). 
 Since we are only interested in $ E_n $ up to a non zero multiple, it will be more convenient to work with $\mathscr E_n
 \in  \Gamma\Big( \Omega^2(\boldsymbol{\mathbb S}_2^{or}\big(\mathbf R^{n+3}\big)\big)\Big)
$ in what follows.\sk 

  Formula  \eqref{Eq:de1-wedge-de2} above gives us an effective way 
to compute $\mathscr E_n$. Indeed,   let $\gamma: U\rightarrow G_2^{or}(\mathbf R^{n+3})$ be a smooth map where $U$ stands for an open domain in an affine space of dimension $N$.  
Given affine coordinates $u_1,\ldots,u_N$ on $U$, the question is to give an explicit formula for (a multiple of) $\gamma^*(E_n^{or})$ in the $u_i$'s. We assume that 
$\gamma$ is given by $u\mapsto \gamma_1(u)\wedge \gamma_2(u)$ where both $\gamma_1$ and $\gamma_2$ are smooth maps from $U$ to $\mathbf R^{n+3}$ such that 
$\gamma_1(u)$ and $\gamma_2(u)$ are not colinear for any $u\in U$.  
Denote by $GS$ the Gram-Schmidt orthogonalization process, defined here as the map
\begin{align*}
GS: \, \Omega^{n+3} & \longrightarrow  \boldsymbol{\mathbb S}_2^{or}\big(\mathbf R^{n+3}\big)\\
\big( x,y\big) & \longmapsto  \left( \frac{x}{\lvert x\lvert}\, , \, 
 \frac{\tilde y}{\lvert \tilde y\lvert}
\right) 
\end{align*}
where $\Omega^{n+3}$ stands for the set of pairs  of two linearly independent vectors of $\mathbf R^{n+3}$ and where for such a pair $(x,y)$, we have set $\tilde y=y-\big((x,y)/\lvert x\lvert^2\big)\, x$. \sk 

Then $  \Gamma= GS\circ \gamma$ is a lift of $\gamma$ to the Stiefel variety, {\it i.e.} the following diagram is commutative
\begin{equation*}
\label{Diag:GS-circ-gamma}
\xymatrix@R=0.8cm@C=2cm{
&   \ar@{->}[d]^{\mathcal S}  \boldsymbol{\mathbb S}_2^{or}\big(\mathbf R^{n+3}\big)  \\
 U 
  \ar@/^/@{->}[ru]^{\Gamma}
   \ar@{->}[r]_{\gamma} &  G_2^{or}(\mathbf R^{n+3})
}
\end{equation*}
hence we obtain that $\Gamma^* (\mathscr E_n)$ 
coincides with $-1/(2\pi)$ times the pull-back of $E_n$ under $\gamma$: one has
\begin{equation}
\Gamma^* (\mathscr E_n)=\frac{-1\,}{2\pi}\, \gamma^*\Big( E_n \Big)\, . 
\end{equation}

We use the material of the previous paragraph in the somehow `most general map $\gamma$' in order to deduce some formulas for (pull-backs of) Euler's form which will be used later on. The map we are referring to here, denoted by $\xi$,  is just  the map associating to a $2\times (n+3)$ matrix $M$ the oriented 2-plane  $\xi(M)$ spanned by the first and the second rows of $M$, denoted  by $x(M)$ and $y(M)$ respectively: one has 
$\xi(M)=x(M)\wedge y(M)$.  For $x,y\in \mathbf R^{n+3}$ we denote by $M_{x,y}$ the matrix such that $x(M_{x,y})=x$ and $y(M_{x,y})=y$ and we set $\xi(x,y)=\xi(M_{x,y})\in G_2^{or}(\mathbf R^{n+3})$.   The map $\xi$ under scrutiny here is defined on the subset $\Omega_2\subset {\rm M}_{2\times (n+3)}(\mathbf R)$ of rank 2 matrices: 
\begin{align}
\label{Eq:XI}
\xi : \, \Omega_2 & \longrightarrow  G_2^{or}(\mathbf R^{n+3})\\
M=\Big[ \hspace{-0.15cm}
\begin{tabular}{c}
$x$
\vspace{-0.15cm}
\\
$y$
\end{tabular}
\hspace{-0.15cm}
 \Big] & \longmapsto \xi(M)=\xi(x,y)=x\wedge y\, . 
 \nonumber
\end{align}

From $\gamma : \Omega_2\rightarrow G_2^{or}(\mathbf R^{n+3})$, we build other such maps by setting for any $M_{x,y}\in \Omega_2$
$$\check{\gamma}(M_{x,y})={\gamma}(M_{x,y}) 
\qquad \mbox{ and }\qquad 
\gamma_G(M_{x,y})=\gamma(M_{x,y}\cdot G)\, $$ 
 where $G$ stands for a previously given constant
element of the linear group ${\rm GL}_{n+3}(\mathbf R)$. For $\epsilon=(\epsilon_k)_{k=1}^{n+3}\in \{ \pm 1\}$, we denote by $D_\epsilon$ the diagonal matrix with $\epsilon_k$ for its $k$-th coefficient, that is $D_\epsilon={\rm diag}(\epsilon_1,\ldots,\epsilon_{n+3})
\in {\rm GL}_{n+3}(\mathbf R) $ and we set 
$\gamma_\epsilon= \gamma_{D_\epsilon}$. 
With the preceding notation, one can state some transformation formulas for (the pull-backs of) Euler's form  that will prove to be important in the sequel.

\begin{lem} 
\label{Lem:Properties-En}
 The 2-form $E_n$ satisfies the following transformation formulas: 
\begin{enumerate}
\item[{\rm 1.}] $
\check{\gamma}^*\big( E_n\big)=-  {\gamma}^*\big( E_n\big)$; 
\sk 
\item[{\rm 2.}] 
$ \gamma_\epsilon^*\big( E_n\big)={\gamma}^*\big( E_n\big)$ 
for any $\epsilon\in \{ \pm 1\}^{n+3}$;
\sk 
\item[{\rm 3.}] 
$ \gamma_G^*\big( E_n\big)={\gamma}^*\big( E_n\big)$ 
for any $G\in {\rm SO}_{n+3}(\mathbf R)$.
\end{enumerate}
\end{lem}
\begin{proof} As the notation suggests, one has $\check{\gamma}=D\circ \gamma$ where 
$D$ is the change of  orientation. Since $D$ has for $\mathcal S$-equivariant lift the exchange map $(e_1,e_2)\rightarrow (e_2,e_1)$ on the Stiefel manifold, {\rm 1.} follows from $de_2\wedge de_1=-de_1\wedge de_2$ together with \eqref{Eq:de1-wedge-de2}. The second point is a direct consequence of this formula as well.  As for {\rm 3.}, it follows immediatly from the ${\rm SO}_{n+3}(\mathbf R)$-invariance of $E_n$.
\end{proof}

For $i\in \{1,\ldots,n+3\}$, the natural inclusion $\mathbf R^{n+2}_i\subset \mathbf R^{n+3}$ induces a natural embedding  $\iota_i : G^{or}_2(\mathbf R^{n+2}_i)
\hookrightarrow G^{or}_2(\mathbf R^{n+3})$ which is such that the following lemma holds true: 
\begin{lem}  
\label{L:Pull-backs}
For any $k\geq 1$, denote by $E_n^k$ the $k$-th wedge product of $E_n$. 
\vspace{-0.1cm}
\begin{enumerate}
\item[{\rm 1.}]  The pull-back $\iota_i^*\big( E_n)$ coincides with the Euler form of 
the tautological bundle of $G^{or}_2(\mathbf R^{n+2}_i)$.

\sk 
\item[{\rm 2.}]  
The pull-back $\iota_i^*\big( E_n^n)$ is a ${\rm SO}_{n+2}$-invariant volume form on $G^{or}_2(\mathbf R^{n+2}_i)$.
\end{enumerate}
\end{lem}
\begin{proof} Since $\iota_i^*(\mathcal T^{or}) $ coincides with the tautological bundle on $G^{or}_2(\mathbf R^{n+2}_i)$, the first point 
follows from the naturality of the Euler's class and of the ${\rm SO}_{n+3}$-invariance of the representative $E_n$.

The second point follows from the first combined with all the following facts: 
$i.$ $G^{or}_2(\mathbf R^{n+2})$ is orientable and compact hence admits a necessarily non exact volume form; and $ii.$ the top degree cohomology space $\boldsymbol{H}^{2n}\big( 
G^{or}_2(\mathbf R^{n+2}), \mathbf R
\big)\simeq \mathbf R$ is generated over $\mathbf R$ by (the class of) the
${\rm SO}_{n+2}$-invariant form $E_{n-1}^{n}$.
\end{proof}

With the constructions and results discussed above at hand, one can now follow Damiano's construction of an AR from a power of the Euler class $E_n$.  It is the subject of the next subsection.

\subsubsection{\bf The positive Eulerian abelian relation} 
\label{SSS:positive-Euler-AR}
Here, we review Damiano's construction (which itself heavily relies on  Gelfand-MacPherson theory \cite{GelfandMacPherson}) of the Euler abelian relation for $\boldsymbol{\mathcal W}_{0,n+3}$ on the positive part $\mathcal M_{0,n+3}^{>0}$, which is the privileged component of the moduli space $\mathcal M_{0,n+3}$ we will work on. We will discuss later the Euler's abelian relations on the other components $\mathcal M(\boldsymbol{\sigma})$. \mk

Let $U\subset  \mathbf R^{n}$ be the set of $n$-tuples $u=(u_i)_{i=1}^n\in \mathbf R^n$ such that $1<u_1<u_2<\cdots <u_n$. The map 
\begin{align}
\label{Eq:Varphi-U}
\varphi : U\,  &\longrightarrow \mathcal M_{0,n+3}^{\, or,\, >0} \\
u & \longmapsto \Big[\,  \infty\, ,\,0\, ,\,-1\, ,\,-u_1\, ,\,-u_2\, ,\, \ldots \, ,\,  -u_n\, \Big]
\nonumber
\end{align}
can be seen as a global isomorphism hence the $u_i$'s form a  global coordinate system on $\mathcal M_{0,n+3}^{>0}$,  which is the one we will work with.  \sk 

Recall that a matrix $M\in {\rm M}_{2,n+3}(\mathbf R)$ is said {\it `positive'} if all its $2\times 2$ minors are positive. We then define the {\it `positive oriented grassmannian'} 
$G^{or}_2(\mathbf R^{n+2})^{>0}$
as the image by the map \eqref{Eq:XI} of the open subset  $\Omega_2^{>0}
\subset \Omega_2$  of positive matrices: one has  
\begin{equation}
\label{Eq:PositiveOrientedGrassmannian}
G^{or}_2(\mathbf R^{n+2})^{>0}=\xi\Big( \Omega_2^{>0} \Big)\, .
\end{equation}
One verifies that the positive grassmannian $G^{or}_2(\mathbf R^{n+2})^{>0}$ 
is
 \begin{itemize}
\item[$-$] 
 included in 
$\widehat G^{or}_2(\mathbf R^{n+2})$;
\sk
\item [$-$]  
 stable under the action of the 
`positive part'  $H_0$ of the Cartan torus $H\subset {\rm SL}_{n+3}(\mathbf R)$;
\sk 
\item [$-$] 
such that 
the quotient of it by $H$ is $ \mathcal M_{0,n+3}^{>0}$, {\it i.e.}\,one has 
({\it cf.}\,diagram \eqref{Eq:Diagrammm}):
$$  \pi_H \Big( G^{or}_2\big(\mathbf R^{n+2})^{>0}\Big)=\mathcal M_{0,n+3}^{>0}\, . $$
 \end{itemize}

We consider the following map 
\begin{align}
\label{Eq:Map-M}
M \hspace{0.15cm} : \hspace{0.15cm}U\hspace{0.15cm}  &\longrightarrow  \hspace{0.15cm} \Omega_2^{>0}  \\
u & \longmapsto M_u=
\begin{bmatrix}
\, 1 &  0  &  -1  & -u_1  &   -u_2 & \cdots & -u_n \, \\
\, 0 &  1  &  {}^{} \, 1  & 1  &  1 & \cdots & 1 \, 
\end{bmatrix}
\nonumber
\end{align}
which is such that  (1) the composition $\gamma=\xi\circ M: U\rightarrow G^{or}_2(\mathbf R^{n+2})$ takes values into the positive grassmannian;  and (2) makes the following diagram commutative: 
%
\begin{equation}
\label{Diag:gamma-triangle}
\begin{tabular}{c}
\xymatrix@R=0.8cm@C=2cm{
  &   \ar@{->}[d]^{ \pi_H}  
G_2^{or}\big(\mathbf R^{n+3}\big)^{>0}  \\
 U   
   \ar@{->}[ru]^{\rotatebox{24}{\scalebox{0.7}{\begin{tabular}{c} 
   $\gamma=\xi\circ M$
   \vspace{-0.8cm}
  \end{tabular} 
   }}}
   \ar@{->}[r]_{\varphi} &  
   \mathcal M_{0,n+3}^{>0}\, .
}
\end{tabular}
\end{equation}

It is not difficult to give an explicit formula for 
$\pi= \varphi^{-1}\circ\pi_H : G_2^{or}\big(\mathbf R^{n+3}\big)^{>0}\rightarrow U$. 
  Given $\xi^{or}\in G_2^{or}\big(\mathbf R^{n+3}\big)^{>0}$, one has $\xi^{or}=\xi_{x,y} =\xi\big(\big[ \hspace{-0.15cm}
\scalebox{0.9}{\begin{tabular}{c}
 $x$
\vspace{-0.15cm}
\\
$y$
\end{tabular}}
\hspace{-0.15cm}
 \big]\big)$ for $x,y\in \mathbf R^{n+3}$ such that $x_iy_j-y_ix_j>0$ for any $i,j$ such that $1\leq i<j \leq n+3$. In particular, the submatrix 
 $\big[ \hspace{-0.15cm}
\scalebox{0.8}{
\begin{tabular}{c}
 $x_1$  $x_2$
\vspace{-0.15cm}
\\
$y_1$  $y_2$
\end{tabular}}
\hspace{-0.15cm}
 \big]$ is invertible, hence one has 
   $\big[ \hspace{-0.15cm}
\scalebox{0.8}{
\begin{tabular}{c}
 $x_1$  $x_2$
\vspace{-0.15cm}
\\
$y_1$  $y_2$
\end{tabular}}
\hspace{-0.15cm}
 \big]\cdot  \big[ \hspace{-0.15cm}
\scalebox{0.8}{
\begin{tabular}{c}
 $x$
\vspace{-0.15cm}
\\
${}^{}$\hspace{-0.2cm} $y$
\end{tabular}}
\hspace{-0.15cm}
 \big]=
  \big[ \hspace{-0.15cm}
\scalebox{0.8}{
\begin{tabular}{c}
 $1$  $0$ $\tilde x_3$ $\cdots$ $\tilde x_{n+3}$
\vspace{-0.1cm}
\\
${}^{}$\hspace{-0.2cm}  $0$  $1$ $\tilde y_3$ $\cdots$ $\tilde y_{n+3}$
\end{tabular}}
\hspace{-0.15cm}
 \big]$ for some $\tilde x_k$'s and $\tilde y_k$'s with $k=3,\ldots,n+3$, which are explicit rational expressions in the original $x_i$'s and $y_i$'s.  Notice that, since all $2\times 2$ minors are positive, 
 one necessarily has $\tilde x_k<0$ and $\tilde y_k>0$ for any $k=3,\ldots,n+3$. Then setting 
$h_1=-\tilde x_3$, $h_2=\tilde y_3$ and $h_k=\tilde y_3/\tilde y_k$ for all such $k$, one has 
$$
\begin{bmatrix}
h_1^{-1} & 0\\
0 & h_2^{-1}
\end{bmatrix}\cdot 
\begin{bmatrix}
1 & 0 & \tilde x_3 & \cdots & \tilde x_{n+3} \\
0 &1 & \tilde y_3 & \cdots & \tilde y_{n+3} 
\end{bmatrix}\cdot 
\begin{bmatrix}
{h_1} & &  0\\
 & \ddots & \\
 0& & {h_{n+3}}
\end{bmatrix}
= 
\begin{bmatrix}
1 & 0 &-1 & -\frac{\tilde x_4\tilde y_3}{\tilde y_4\tilde x_3}   & \cdots & -\frac{\tilde x_{n+3}\tilde y_3}{\tilde y_{n+3} \tilde x_3} \\
0 &1 & 1 &  1 & \cdots & 1
\end{bmatrix}
$$
from which it follows that the map $\pi$ under scrutiny is given in coordinates by 
\begin{align}
\label{Eq:Pi}
\pi \, : \, G_2^{or}\big(\mathbf R^{n+3}\big)^{>0}& \longrightarrow U \\
\xi_{x,y} & \longmapsto  
\left( \, \frac{\tilde x_4\tilde y_3}{\tilde y_4\tilde x_3}\, ,\, \ldots\, ,\,   
\frac{\tilde x_{n+3}\tilde y_3}{\tilde y_{n+3}\tilde x_3}\right)
\nonumber 
\end{align}
(it is easy and left to the reader to show that this is well-defined and is the right formula). 
\begin{center}
$\star$
\end{center}


Recall the way of parametrizing the $H_0$-orbit of any generic oriented 2-plane $\xi^{or}$ 
by means of the hypersimplex  $\Delta_{2}^{n+3}=\big\{ \, 
(t_i)_{i=1}^{n+3}\in ]0,1[^{n+3}\, \big\lvert \, \sum_{i=1}^{n+3} t_i=2\, 
\big\}$ (see \cite{GelfandMacPherson}): the map 
$$h:  \Delta_{2}^{n+3}\longrightarrow (\mathbf R_{>0})^{n+3},\, t=\big(t_i\big)_{i=1}^{n+3} \longmapsto h(t)= \Big( t_1/(1-t_i)\Big)_{i=1}^{n+3}$$ 
is such that the image of 
\begin{align*}
 Dh={\rm Diag}\circ h\, :\hspace{0.15cm}   \Delta_{2}^{n+3}& \longrightarrow {\bf P}{\rm SL}_{n+3}(\mathbf R)\\ 
t& \longmapsto {\rm Diag}( h(t)\big)
=\scalebox{0.8}{$\begin{bmatrix}
\frac{t_1}{1-t_1} & &  0\\
 & \ddots & \\
 0& & \frac{t_{n+3}}{1-t_{n+3}}
\end{bmatrix}$}
\end{align*}
 is equal to the positive component $H_0$ of the Cartan torus 
 $H\subset  {\bf P}{\rm SL}_{n+3}(\mathbf R)$.

Using  $Dh$ and the maps  $\xi$ and $M$ defined in \ref{Eq:XI} and \eqref{Eq:Map-M}, one constructs
\begin{equation}
\label{Diag:PhiU}
\xymatrix@R=0.2cm@C=1.65cm{
\Phi_U = \xi\circ\big( M\cdot Dh\big)  \hspace{0.15cm} : \hspace{0.2cm} U\times \Delta_2^{n+3}   \hspace{0.15cm}
   \ar@{->}[r]
    &  
     \hspace{0.15cm} G_2^{or}\big(\mathbf R^{n+3}\big)^{>0}
     \hspace{2cm} {}^{ } 
     \\
{}^{} \hspace{4.4cm}   \big( u,t\big)   \hspace{0.05cm}   \ar@{->}[r]  &  \hspace{0.15cm}  \xi\Big( M_u\cdot 
Dh(t)\Big) \hspace{2cm} {}^{ }
}
\end{equation}
which can easily be proven to be an isomorphism (from the direct product $U\times \Delta_2^{n+3}   $ onto 
the whole positive oriented grassmannian $G_2^{or}\big(\mathbf R^{n+3}\big)^{>0}$). 
Let 
\begin{equation}
\label{Eq:delta-projection}
\delta:  \, U\times \Delta_2^{n+3}  \longrightarrow U\, 
\end{equation}
be the projection onto the first factor.  
 Since $E_n$ is a global smooth 2-form on 
the oriented grassmannian,  for any $k\geq 1$, its $k$-th wedge product $E_n^k$ is integrable along the fibers of the $H_0$-action. Consequently, for any such $k$ the corresponding $2k$-form 
$\widetilde E_n^k=\big(\Phi_U\big)^*(E_n^k)$ on $U\times \Delta_2^{n+3}  $ admits a push-forward by $\delta$ which we will denote by 
$$
 \mathscr E_n^k=\delta_*\big(\, \widetilde E_n^k\,  \big)\, .
$$
It is a smooth $(2k-{n+2})$-form on $U$, where we use the convention that a $\ell$-differential form  
is 0 when $\ell$ is negative.  Two important properties of these forms are given in the following: 
\begin{prop}[\cite{D}] 
\label{P:Properties-Enk}
 1. The $(n-2)$-form $ \mathscr E_n^n$ is trivial: one has $\mathscr E_n^n=0$; \sk 

2. The $n$-form $ \mathscr E_{n}^{n+1}$  is non zero and can be written 
\begin{equation}
\label{Eq:e-m}
\mathscr E_{n}^{n+1}(u)=
e_{n}(u) \, du_1\wedge \ldots \wedge du_n
\end{equation}
 for a non vanishing global analytic function $e_n$ on $U$.
\end{prop}
\begin{proof}
The first point is proved through a direct computation by Damiano (see \cite[\S6.3]{DThesis}). 
The proof of 2. relies on the very same arguments used to establish the second point of
Lemma \ref{L:Pull-backs}. 
\end{proof}

In order to build the Eulerian AR discussed here, it is necessary to extend $\Phi_U$ to the product $U\times \overline{\Delta}_2^{n+3}$ where $\overline{\Delta}_2^{n+3}$ stands for the closure of the hypersimplex:
$$
\overline{\Delta}_2^{n+3}=\Big\{ \, \big(t_1,\ldots,t_{n+3}\big)\in [0,1]^{n+3} \hspace{0.15cm}
\Big\lvert \hspace{0.15cm} 
\begin{tabular}{l}
 $t_1+\cdots +t_{n+3}=2$
\end{tabular}
\Big\}\, .
$$

In the classical case,  that is when the target is the usual grassmannian, 
the extension property of the corresponding map $\nu\circ \Phi_U: U\times \Delta_2^{n+3}  
\rightarrow G_2\big(\mathbf R^{n+3}\big)$ 
 is well-known ({\it cf.}\,\cite[Prop.\,2.3.2]{GelfandMacPherson}) but since the map takes values into the oriented grassmannian in the case under scrutiny, we believe that it has to be justified, which is easy to do.
\begin{prop}
\label{Prop:bar-Phi-U}
\begin{enumerate}
\item[{\rm 1.}] The map $\Phi_U$ in \eqref{Diag:PhiU} admits a unique continuous extension 
\begin{equation}
\label{Diag:PhiUbar}
\xymatrix@R=0.2cm@C=1.65cm{
\overline{\Phi}_U   \hspace{0.15cm} : \hspace{0.2cm} U\times \overline{\Delta}_2^{n+3}   \hspace{0.15cm}
   \ar@{->}[r]
    &  
     \hspace{0.15cm} G_2^{or}\big(\mathbf R^{n+3}\big)\,.
     }
\end{equation}

\item[{\rm 2.}]  Moreover, this extension $\overline{\Phi}_U$ takes values into the totally non negative oriented grassmannian $G_2^{or}\big(\mathbf R^{n+3}\big)^{\geq 0}$ and is actually 
smooth as a map between manifolds with corners.
\end{enumerate}
\end{prop}

This proposition is rather direct consequence of the following lemma: 
\begin{lem}
\begin{enumerate}
\item[{\rm 1.}]
 The preimage of $G_2(\mathbf R^{n+3})^{>0}$ by the covering $\nu: 
 G_2^{or}(\mathbf R^{n+3})\rightarrow 
 G_2(\mathbf R^{n+3})$ is the disjoint union of the positive and the negative oriented grassmannians: one has 
 $$
 \nu^{-1}\left( 
 G_2\big(\mathbf R^{n+3}\big)^{>0}
 \right)=G_2^{or}\big(\mathbf R^{n+3}\big)^{>0}
 \sqcup 
 G_2^{or}\big(\mathbf R^{n+3}\big)^{<0}\, .
 $$
 
 \item[{\rm 2.}]
 The map \eqref{Diag:PhiU} is a smooth diffeomorphism (in particular, it is surjective);\mk 
 \item[{\rm 3.}] The isomorphism between $G_2^{or}\big(\mathbf R^{n+3}\big)^{>0}$ and 
 its non oriented counterpart $G_2\big(\mathbf R^{n+3}\big)^{>0}$ induced by the restriction of $\nu$ 
  (again denoted by $\nu$ to simplify) extends to an isomorphism of smooth varieties with corner
 $$ 
 \nu : G_2^{or}\big(\mathbf R^{n+3}\big)^{\geq 0} 
\longrightarrow  
 G_2\big(\mathbf R^{n+3}\big)^{\geq 0}\, . $$ 
\end{enumerate}
\end{lem}
\begin{proof}
The first point is obvious and the second is hardly more difficult to prove.  The third point 
follows from the first two together with \cite[Prop.\,2.3.2]{GelfandMacPherson}. 
\end{proof}

\begin{rem} 
That 
$\Phi_U$ is an isomorphism has the following interesting corollary: 
since $U\simeq 
 \mathcal M_{0,n+3}^{>0}$ is known to be isomorphic to (the interior of) a polyhedron, the so-called $n$-th associahedron ${\rm Ass}_n$, 
 \eqref{Diag:PhiU}
 provides a description  of 
 the positive grassmannian of 2-planes in $\mathbf R^{n+3}$ 
as the product of (the interiors of) two polyhedra: one has 
 $$G_2\big(\mathbf R^{n+3}\big)^{>0} \simeq {\rm Ass}_n\times \Delta_2^{n+3}\, .$$
 As far as we are aware of, such a polyhedral description of  $G_2\big(\mathbf R^{n+3}\big)^{>0}$ is not mentioned in the literature yet, although it is quite likely that it is known by specialists.
\end{rem}

From the second point of Proposition \ref{Prop:bar-Phi-U}, one deduces that 
$\widetilde E_n^k$ admits a unique smooth extension $\overline{ E}_n^k$ to  $U\times \overline{\Delta}_{2}^{n+3}$, which is such that 
$$
\overline{ E}_n^k=\big( \overline{\Phi}_U\big)^*\big( E_n^k\big)\, .
$$

Considering the boundary $\partial {\Delta}_{2}^{n+3}=\overline{\Delta}_{2}^{n+3}\setminus {\Delta}_{2}^{n+3}$, 
the natural inclusion $\partial {\Delta}_{2}^{n+3}\subset \overline{\Delta}_{2}^{n+3}$ gives rise to an 
injective morphism of bundles over $U$ denoted by $\tau: U\times \partial {\Delta}_{2}^{n+3}\hookrightarrow  U\times \overline{\Delta}_{2}^{n+3}$.  The 
map $\delta^\partial = \delta\circ \tau$ is nothing else than 
the  projection map 
$ U\times \partial {\Delta}_{2}^{n+3}\rightarrow U$. Then, 
from the Stokes-type formula for integration along fibers with boundary, 
we get that the following identity (which is an equality between differential $(n-1)$-forms on $U$) is satisfied: 
\begin{equation}
\label{Eq:dE-n-n}
(-1)^{n+2} d \mathscr E_n^n= \big(\delta^{\partial}\big)_*\Big( \tau^*( \mathscr E_n^n\big)\Big)\, .
\end{equation}

Let us now explain how this relation gives rise to an AR for the web $\boldsymbol{\mathcal W}_{0,n+3}$ on $\mathcal M_{0,n+3}^{>0}\simeq U$.  First, it follows from Proposition \ref{P:Properties-Enk}.1 that the  RHS in \eqref{Eq:dE-n-n} is 0.    Secondly, 
for $i=1,\ldots,n+3$ and $\upsilon=0,1$, let  
$d_i^\upsilon: \mathbf R^{n+2} \hookrightarrow  \mathbf R^{n+3}$ 
be the affine map  associating 
$ (t_1 ,\ldots,t_{i-1},,\upsilon,t_i,\ldots,t_{n+2})$ 
to any 
$(n+2)$-tuple $(t_s)_{s=1}^{n+2}$ (`$\upsilon$ is inserted at the $i$-th position').  
Then the restriction of this map 
to $$\overline{\Delta}_{2-\upsilon}^{n+2}= \left\{\, (s_k)_{k=1}^{n+2}
\in [0,1]^{n+2} 
\, \big\lvert \, 
\sum_k s_k=2-\upsilon \, 
\right\}$$ induces a linear inclusion 
whose image, denoted by 
$\overline{\Delta}_{2-\upsilon}^{n+2}(i)$,  is the facet of $ \overline\Delta_{2}^{n+3}$ obtained by intersecting it with the affine hyperplane  cut out by $x_i=\upsilon$:
one has  
$$ 
\overline{\Delta}_{2-\upsilon}^{n+2} \stackrel{\sim}{\longrightarrow} 
\overline{\Delta}_{2-\upsilon}^{n+2}(i)=
d_i^\upsilon\Big(\overline{\Delta}_{2-\upsilon}^{n+2} \Big) =  \overline\Delta_{2}^{n+3} 
 \cap \big\{ \, x_i=\upsilon\,  \big\}\subset \overline\Delta_{2}^{n+3} \subset \mathbf R^{n+3}\, .  $$ 

The codimension 1 boundary of $\overline\Delta_{2}^{n+3}$ is exactly the disjoint union of 
the interiors ${\Delta}_{2-\upsilon}^{n+2}(i)$ of the (hyper)simplices $\overline{\Delta}_{2-\upsilon}^{n+2}(i)$ for $i=1,\ldots,n+3$ and $\upsilon=0,1$. 

An important fact 
from which a crucial property satisfied by $\mathcal E_n^n$ will be deduced  below 
is the following homological relation  
$$
\partial \bigg[ \, \overline\Delta_{2}^{n+3} \bigg]
=\sum_{i=1}^{n+3} (-1)^i \, \big(d_i^1\big)_*
 \bigg[ \, \overline\Delta_{1}^{n+2}\bigg]-
 \sum_{i=1}^{n+3} (-1)^i \, \big(d_i^0\big)_*
 \bigg[ \, \overline\Delta_{2}^{n+2}\bigg]
$$
which holds true in the relative homology group $\boldsymbol{H}_{n+3}\big( \overline\Delta_{2}^{n+3}, \partial \Delta_{2}^{n+3},\mathbf Z\big)$ ({\it cf.}\,\cite[Prop.\,2.1.3]{GelfandMacPherson}).\sk 

Since the facets $\overline{\Delta}_{1}^{n+2}(i)$ for $i=1,\ldots,n+3$  (all isomorphic to the $(n+1)$-dimensional simplex $\Delta_1^{n+2}$) do not play any role in building Euler's abelian relation, we will no longer consider them in what follows and we will write $d_i=d_i^0$ from now on.

Combining the above homological relation with \eqref{Eq:dE-n-n} and Proposition \ref{P:Properties-Enk}.1, we obtain that 
%
%
%
\begin{equation}
\label{Eq:dE-n-n-2}
0 = \sum_{i=1}^{n+3} (-1)^i \big( \delta_i\big)_*\Big( \tau_i^*\big(\mathscr E_n^n\big) \Big) 
\end{equation}
where 
$\tau_i: U\times {\Delta}_{2}^{n+2}(i) \subset U\times \partial {\Delta}_{2}^{n+3}$ is the morphism of trivial bundles over $U$ induced by the inclusion ${\Delta}_{2}^{n+2}(i) \subset \partial {\Delta}_{2}^{n+3}$ and where $\delta_i=\delta\circ \tau_i: U\times {\Delta}_{2}^{n+2}(i)\rightarrow  U$  is the natural projection.\sk 

In order of interpreting \eqref{Eq:dE-n-n-2} as an abelian relation, one has now to express each term 
$\big( \delta_i\big)_*\big( \tau_i^*\big(\mathscr E_n^n\big) \big) $ in a different way.  The key point to do this is Lemma \ref{L:Pull-backs}.1, which says that  
up to the natural identification of $G_2^{or}\big( \mathbf R^{n+2}_i\big)$ with
$G_2^{or}\big( \mathbf R^{n+2}\big)$, the restriction of $E_n$ along $G_2^{or}\big( \mathbf R^{n+2}_i\big)$ coincides with the invariant representative $E_{n-1}$ of the Euler class of the tautological bundle on $G_2^{or}\big( \mathbf R^{n+2}\big)$. 

To build Euler's AR, we then consider the following diagram
\begin{equation}
\label{Diag:Gros}
\begin{tabular}{c}
\xymatrix@R=0.9cm@C=1.3cm{
U\times \overline{\Delta}_2^{n+3}   \ar@{->}[r]^{
\overline{\Phi}_U 
} & G_2^{or}\big( \mathbf R^{n+3}\big)^{\geq 0} &  \\
 U\times \overline{\Delta}_2^{n+2}(i)  
 \ar@{_{(}->}[u]_{\tau_i}
 \ar@{->}[d]_{\delta_i}
  \ar@{->}[r]^{
\overline{\Phi}_U 
} & 
\incl[u]
G_2^{or}\big( \mathbf R^{n+2}_i\big)^{\geq 0} 
 & G_2^{or}\big( \mathbf R^{n+2}\big)^{\geq 0} 
  \ar@{->}[d]_{\pi'} 
      \ar@{>}[l]_{\mu_i}^{\sim}
 &  
U'\times \overline{\Delta}_2^{n+2} 
 \ar@{->}[d]^{\delta'}  \ar@{->}[l]_{\hspace{0.4cm}\overline{\Phi}_{U'}}
\\
U  \ar@{->}[rr]^{\psi_i} & 
&  U'\ar@{=}[r]
& U'
  \ar@{->}[lu]_{\gamma'}
}
\end{tabular}
\end{equation}
where
\begin{itemize}
\item $U'$ stands for the set of $(n-1)$-tuples $u'=(u'_k)_{k=1}^n\in \mathbf R^{n-1}$ such that $1<u'_1<\cdots < u'_{n-1}$;
\mk 
\item  $\gamma'$, $\pi'$, $\overline{\Phi}_{U'}$  and $\delta'$ correspond  respectively to the maps 
  \eqref{Diag:gamma-triangle},  
\eqref{Eq:Pi}, 
\eqref{Diag:PhiUbar}
and \eqref{Eq:delta-projection}   above 
%
but in the case when the dimension has been taken to be $n-1$ instead of $n$;
\mk  
\item the map $\mu_i$ is the restriction of the natural identification between 
$G_2^{or}\big( \mathbf R^{n+2}\big)$
and 
 $ G_2^{or}\big( \mathbf R^{n+2}_i\big)$ induced by the linear map 
$d_i: \mathbf R^{n+2}\stackrel{\sim}{\rightarrow} {\rm Im}(d_i)= \mathbf R^{n+2}_i\subset \mathbf R^{n+3}$;
\mk  
\item the maps $\psi_i: U\rightarrow U'$ are given by  
$$
\psi_{1}(u)=\left( \, \frac{u_1(u_j-1)}{u_j(u_1-1)} \, 
\right)_{j=2}^n\, ,\qquad 
\psi_{2}(u)=\left( \, \frac{u_j-1}{u_1-1} \, 
\right)_{j=2}^n
\, ,\qquad 
\psi_{3}(u)=\left(\,  \frac{u_j}{u_1}\, 
\right)_{j=2}^n
$$
and  $\psi_{3+i}(u)=\big(u_1,\ldots, \widehat{u_i},\ldots,u_n\big)$ for $i=1,\ldots,n$. 
\mk 
\item  for $i=1,\ldots,n+3$ and $u'\in U'$, $\gamma_i(u')$ is  the oriented 2-plane in $\mathbf R^{n+2}_i\subset \mathbf R^{n+3}$ associated to the  $2\times (n+3)$ matrix obtained from $M_{u'}\in {\rm Mat}_{2,n+2}(\mathbf R)$ (see \eqref{Eq:Map-M}) 
by inserting in it the zero column at the $i$-th place.
\end{itemize}

By  straightforward computations (left to the reader), it can be verified that 
\eqref{Diag:Gros} is commutative. Combining this with the fact recalled in the paragraph just before this diagram, one deduces the 
\begin{prop}
\label{P:RA-EULER-a}
\begin{enumerate}
\item[]  
${}^{}$\hspace{-1cm}{\rm 1.} 
One has 
$
\big( \delta_i\big)_*\big( \tau_i^*\big(\mathscr E_n^n\big) \big)
= \psi_i^* \big( 
\mathscr E_{n-1}^n
\big)
$\, 
 for  $i=1,\ldots,n+3$.
\mk 
\item[{\rm 2.}] Consequently 
 \eqref{Eq:dE-n-n-2}
can be written 
$
0 = \sum_{i=1}^{n+3} (-1)^i \psi_i^* \big( 
\mathscr E_{n-1}^n
\big)
$. 
\mk 
\item[{\rm 3.}] Since $\mathscr E_{n-1}^n$ never vanishes (according to Proposition \ref{P:Properties-Enk}.2)
it follows that 
\begin{equation}
\label{Eq:E_n->0}
\boldsymbol{\mathcal E}_n^{>0}=\bigg( 
\, (-1)^{i-1}\, \psi_i^* \left( 
\mathscr E_{n-1}^n \right)
\, \bigg)_{i=1}^{n+3}
\end{equation}
is a complete hence in particular non-trivial  abelian relation for $\boldsymbol{\mathcal W}^{\, >0}_{0,n+3}$.
\end{enumerate}
\end{prop}

The construction above gives indeed an AR on $U\simeq \mathcal M_{0,n+3}^{\, >0}$ but several questions about it remain to be answered: 
\begin{equation}
\label{Eq:Remain-2-be-answered}
\begin{tabular}{l}
$i.$ {\it Does the construction above of $\boldsymbol{\mathcal E}_n^{>0}$ depend on some choices? If yes, on which ones} \\ 
${}^{ }$ \hspace{0.1cm} {\it and to what extent?}\mk \\
$ii.$  {\it Following a completely similar approach, one can construct an Eulerian 
abelian  re-}\\
${}^{ }$ \hspace{0.1cm} 
 {\it lation
$\boldsymbol{\mathcal E}_n(\boldsymbol{\sigma})$ for  $ \boldsymbol{\mathcal W}_{0,n+3}$ on the  component $\mathcal M(\boldsymbol{\sigma})$ of  $ \mathcal M_{0,n+3}(\mathbf R)$ for any 
$\boldsymbol{\sigma}\in \mathfrak K_{n+3}$  ({\it cf.}} \\
 ${}^{ }$ \hspace{0.1cm} 
  {\it 
 \S\ref{SSS:Real-Locus}). How are the 
$\boldsymbol{\mathcal E}_n(\boldsymbol{\sigma})$'s related, especially with respect 
to the birational action}\\
${}^{ }$ \hspace{0.1cm} 
 {\it  of $\mathfrak S_{n+3}$ on the  real moduli space  $ \mathcal M_{0,n+3}(\mathbf R)$?}
\end{tabular}
\end{equation}
%

We will come back to these important points further on but 
we will first discuss how to make 
$\boldsymbol{\mathcal E}_n^{>0}$ explicit by giving formulas for its components. 
Since  all are pull-backs of the form 
$\mathcal E_{n-1}^n$  on $U'\simeq \mathcal M_{0,n+2}^{\, >0}$, this amounts to studying the 
function $e_{n-1}$ defined in Proposition \ref{P:Properties-Enk}.2.

\subsubsection{\bf An integral formula for $e_n$}
\label{SS:En-integral-representation}
A multivariable integration scheme for computing $e_{n-1}$ has already been described by Damiano in \cite[\S6]{D} (more details are given in \cite[\S6]{DThesis}). But it is  not quite explicit and the author was able to use it only in the case when $n=2$, recovering the well-known \href{https://mathworld.wolfram.com/RogersL-Function.html}{\it Rogers' dilogarithm}. 
Our goal below is to go a bit further by describing a slight simplification of Damiano's integration scheme, which 
gives rise to a more explicit integral formula for  $e_{n-1}$ that we will use to get effective informations about $e_{n-1}$ 
for higher but still small $n$ (namely $n=2,3$ and $4$). 
\sk 

In what follows,  $m$ is an integer bigger than or equal to $1$. 
We do not make any assumption about the parity of $m$.  The integer $m$ has to be thought of as $n$ shifted by 1, namely $m=n-1$, this just in order to simplify the writing.  We use below the same notation as above, but in which $n$ has been replaced by $m$. For instance, $U$ here stands for 
the set of $m$-tuples $u=(u_1,\ldots,u_m)\in \mathbf R^m$ such that $1<u_1<\cdots u_{m-1}<u_m$, etc.
\sk

Our main goal here is to give a nice closed formula for the function $e_m(u)$  defined 
 by equality \eqref{Eq:e-m}.  The main facts we use for this purpose are first that $E_m^{m+1}$ coincides (up to sign) with an invariant volume form on $G_2^{or}(\mathbf R^{m+3})$ ({\it cf.}\,Lemma \ref{L:Pull-backs}.2) and secondly that there is a simple  formula for such a volume form when working with  a natural parametrization. 
\mk 

We use the notation introduced circa  \eqref{Eq:XI}. For $M\in \Omega_2$ (that is, $M$ is a $2\times (m+3)$-matrix of rank 2), we write $M=[M_{12},N]$ where $M_{12}$ is the square $2\times 2$ matrix obtained by considering the first 2 columns of $M$ while $N$ is the one formed by the $n+1$ others.  
Then 
$ \Omega_2'=\{ \, M\, \lvert \, {\rm det}(M_{12})> 0\}$ is an open subset of $\Omega_2$ 
and the map $A : \Omega_2'\rightarrow M_{2,m+1}(\mathbf R),\, M\mapsto M_{12}^{-1}\cdot N$ is a surjective ${\rm GL}_2^{>0}(\mathbf R)$-fibration.  The injection 
$B: {M}_{2,m+1}(\mathbf R) \hookrightarrow \Omega_2',\, N\mapsto \big[ {\rm Id}_2,N\big]$ is such that $A\circ B$ is the identity. 
Denoting again by  $\xi$ the restriction 
of the map \eqref{Eq:XI} to $\Omega_2'$, we define a  map $ \Xi: {M}_{2, m+1}(\mathbf R) \rightarrow G_{2}^{or}(\mathbf R^{m+3})$ by requiring that the following diagram commutes: 
\begin{equation}
\label{Diag:gogogo}
\begin{tabular}{c}
\xymatrix@R=1.3cm@C=1.3cm{
\Omega_2'  \ar@{->}[rr]^{\xi}   
\ar@{->}[d]^{A}
&  &G_2^{or}\big( \mathbf R^{m+3}\big) \\
 {M}_{2,m+1}(\mathbf R) 
   \ar@/_1pc/@{->}[rru]_{\quad \Xi} 
   \ar@/^2pc/@{->}[u]^{B}
  & &
}
\end{tabular}
\end{equation}
More prosaically, $\Xi$ is noting else but the map that associates to the $2\times (m+1)$ matrix whose 
two lines correspond to the vectors $(x_s)_{s=0}^m,(y_s)_{s=0}^m\in \mathbf R^{m+1}$, the oriented 2-plane  of $\mathbf R^{m+3}$ directly spanned by $(1,0,x_0,\ldots,x_m)$ and $(0,1,y_0,\ldots,y_m)$, taken in this order. It is well-known that $\Xi$ induces a coordinates chart on $G_2^{or}\big( \mathbf R^{m+3}\big)$, the corresponding coordinates being the $x_s$ and $y_s$ for $s=0,\ldots,m$. 
We will denote by ${\rm dVol}(N)$ the standard Euclidean volume in these coordinates, {\it i.e.}
%
%
$${\rm dVol}(N)=\big( dx_0\wedge dy_0\big)  \wedge 
\big( dx_1\wedge dy_1\big)
\cdots 
\wedge \big( dx_m\wedge dy_m\big)\, .$$ 
The formula in the coordinates $x_s,y_s$ 
for the pull-back under 
$\Xi$ 
of an invariant volume form on $G_2^{or}\big( \mathbf R^{m+3}\big)$ 
 is well known:\footnote{This formula is a particular case of a more general one holding true 
 for some  affine parametrizations (in terms of Jordan structures) of a vast class of homogeneous varieties, see Proposition X.6.3 in \cite{Bertram}. We 
 have not been able to locate a classical reference for this formula in the case of real grassmannians.}
\begin{lem} 
There exists a non zero constant $C_m$ such that\footnote{One has  $C_m=(m+1)!$  but knowing the exact value of this constant is actually not relevant for our purpose hence we will not give a proof of this.} 
\begin{equation}
\label{Eq:Xi^*Em-m+1}
\Xi^*\Big( E_m^{m+1}\Big)=C_m \cdot  \frac{{\rm dVol}(N)}{
{\rm det}\Big( \, 
{\rm Id}_2+N \,{}^t\hspace{-0.03cm}N \, \Big)^{\frac{m+3}{2}}
}
\, .
\end{equation}
\end{lem}

It is easy to get a nice integral formula for $e_m$ from  \eqref{Eq:Xi^*Em-m+1}. 
We list below the facts and notation we need to establish the sought-after formula: 
\begin{itemize} 
\item one denotes by $k=(k_1,k_2)$ (resp.\,by $h=(h_1,\ldots,h_m)$) an element of 
$\big(\mathbf R_{>0}\big)^{2}$ (resp.\,of $\big(\mathbf R_{>0}\big)^{m}$). 
By means of the map $(k,h)\mapsto (k_1,k_2,1,h_1,\ldots,h_m)$,  we 
naturally 
identify $\big(\mathbf R_{>0}\big)^{2}\times \big(\mathbf R_{>0}\big)^{m}=\big(\mathbf R_{>0}\big)^{2+m}$ 
with the open subset of $(m+3)$-tuples in $\mathbf R^{m+3}$ 
 with positive coordinates and such that the third coordinate is 1;\footnote{The fact that we have chosen to normalize the third coordinate is not important at all. We could have decided to normalize any other coordinate with essentially no change.}
\sk 
\item we will consider the following map which is an isomorphism:
\begin{equation}
\label{Eq:U-R-isom}
\Delta_2^{m+3} \longrightarrow \big(\mathbf R_{>0}\big)^{2+m},\, 
 \big(t_i\big)_{i=1}^{m+3} \longmapsto 
\left(  
\frac{t_s(t_3-1)}{t_3(t_s-1)}
\right)_{s=1}^{m+3}\, ;
\end{equation}
\item we set $dk=dk_1\wedge dk_2$, $dh=
 dh_1\wedge  \cdots \wedge dh_{m}$
and $du=
 du_1\wedge \cdots \wedge du_m$;
 \sk 
\item  the notation $k,h>0$ will be used to mean that the pair $(k,h)$ varies in $\big(\mathbf R_{>0}\big)^{2+m}$;\sk 
\item  the following map  is an isomorphism onto its image which is an open subset in 
$\Omega_2$: 
\begin{align}
F: \, 
U\times 
\big(\mathbf R_{>0}\big)^{2+m} 
 & \longrightarrow M_{2,n+3}(\mathbf R)\\
\big(u,k, {h}) & \longmapsto 
\begin{bmatrix}
\sqrt{k_1} & 0 & -1 & -u_1 \sqrt{h_1} & \cdots &   - u_m  \sqrt{h_{m}}\, \\
0 & \sqrt{k_2} & 1 & \sqrt{h_1}  & \cdots &  \sqrt{h_{m}}
\end{bmatrix}
\nonumber 
\end{align}
\item $X$ and $Y$ stand for the line vectors of the matrix $F(u,k,h)$; one has:
\begin{align*}
X=& \, \Big(\sqrt{k_1}, 0 , -1 ,  -u_1 \sqrt{h_1},  \ldots ,   - u_m  \sqrt{h_{m}}\, \Big)\\
\mbox{ and  } \quad 
Y=& \, \Big(\, 0 ,  \sqrt{k_2} , 1,  \sqrt{h_1}  ,  \ldots ,   \sqrt{h_{m}}\, \Big)\, ; 
\end{align*}
\item for $(u,k,h)\in U\times \big(\mathbf R_{>0}\big)^{2+m}$, we denote by 
$S_{u,k,h}$ the 
 matrix $F(u,k,h)\cdot F(u,k,h)^t$; 
 \sk 
 \item   one has 
$$ S_{u,k,h} =\begin{bmatrix}
X^2 & XY \\
XY & Y^2
\end{bmatrix}
$$
with 
\begin{equation}
\label{Eq:X2-Y2-XY}
X^2= 1+k_1+\sum_{i=1}^m h_i\, {u_i}^2, \qquad 
XY= -1-\sum_{i=1}^m h_i\, u_i
\qquad \mbox{ and } \qquad 
Y^2=1+k_2+\sum_{i=1}^m h_i\, .
\end{equation}
\end{itemize}

\begin{prop}
\label{P:Prop-em-integral-formula}
\begin{enumerate}
\item[{\rm 1.}] 
Up to the isomorphism \eqref{Eq:U-R-isom} between $\Delta_2^{m+3}$ and $(\mathbf R_{>0})^{2+m} $,  the map $$\xi\circ F : 
U\times 
(\mathbf R_{>0})^{2+m}   \rightarrow G_2^{or}(\mathbf R^{m+3})$$ identifies with  the 
 parametrization \eqref{Diag:PhiU} of $G_2^{or}(\mathbf R^{m+3})^{>0}$. \mk 
\item[{\rm 2.}]  In particular, the projection $(u,k,h)\rightarrow u$ corresponds to quotienting by the   torus action. \mk
\item[{\rm 3.}] There exists a non zero scalar $\lambda_m$ such that the pull-back  of $E_m^{m+1}$ under $\xi\circ F$  is  given by 
$$
\big( \xi\circ F\big)^*\Big( E_m^{m+1}\Big)= \lambda_m\, 
{\rm det}\left(S_{u,k,h}  \right) ^{-\frac{m+3}{2}}
\,{dk\wedge dh\wedge du}\, .
$$
\item[{\rm 4.}] Consequently, up to multiplication by a non zero constant,  
one has 
\begin{equation}
\label{Eq:e-m(u)}
e_m(u)=\int_{ k,h>0 }
{\Big(\,  X^2\cdot Y^2-(XY)^2 \, \Big)^{-\frac{m+3}{2}}} {dk\wedge dh}
\end{equation}
for any $u\in U$,   where $X^2,Y^2$ and $XY$ are the expressions given in \eqref{Eq:X2-Y2-XY} above.
\end{enumerate}
\end{prop}

\begin{proof} Both the two first points are easy to establish (this is left to the reader). Since 4.\,follows immediatly from 3., we will only discuss the latter.  \sk 

For $(u,k,h)\in U\times (\mathbf R_{>0})^{2+m}$, we denote by $\widetilde F(u,k,h)$  the $2\times (m+1)$ submatrix of $F(u,k,h)$ 
given by its last $m+1$ columns and we 
consider the map 
\begin{align*}
f : U\times (\mathbf R_{>0})^{2+m}& \rightarrow M_{2,m+1}(\mathbf R)\\
(u,k,h)& \longmapsto \Big[ 
\scalebox{0.7}{
\begin{tabular}{cc}
$1/\sqrt{k_1}$ & \\
0 & $1/\sqrt{k_2}$
\end{tabular}}
\Big]\cdot \widetilde F(u,k,h)=
\begin{bmatrix}
{-1}/{\sqrt{k_1}} &  -u _1\sqrt{h_1/k_1} &  \cdots &  -u _m\sqrt{h_m/k_1} \\
\hspace{0.3cm} {1}/{\sqrt{k_2}} & \hspace{0.3cm}\sqrt{h_1/k_2} & \cdots  & \hspace{0.3cm} \sqrt{h_m/k_2}
\end{bmatrix}\, . 
\end{align*}

The two following identities hold true: 
\begin{itemize}
\item   $f^*\Big({\rm det}\big( 
{I}_2+N \,{}^t\hspace{-0.03cm}N \big)\Big)
={\rm det}\big( S_{u,k,h} \big)/(k_1k_2)$; 
\item $f^*\big({\rm dVol}(N)\big)=
\lambda'_m \cdot(k_1k_2)^{-(m+3)/2} dk\wedge dh\wedge du$ for a certain non zero constant $\lambda'_m$. 
\end{itemize}
(The first identity is elementary and the second can be obtained by an easy recurrence on $m$).  \sk 

Since $\xi\circ F=\Xi\circ f$, point 3.\,follows 
from \eqref{Eq:Xi^*Em-m+1} combined with the two preceding identities. \end{proof}

\begin{rem} 
There is no point in justifying the convergence of the integral in 
\eqref{Eq:e-m(u)}: it follows immediately from the fact that $E_m^{m+1}$ is a global smooth form on $G_2^{or}(\mathbf R^{m+3})$ which is compact. 
\end{rem}

The interest of \eqref{Eq:e-m(u)} lies in the fact that it is an explicit and closed formula. This offers a way to study Euler's abelian relation more concretely. We use this formula below to describe explicitly Gelfand-MacPherson's computation  leading to Rogers dilogarithm in the case $n=2$. Then we discuss the case when $n=3$ and give the first terms of the Taylor series  of $e_2$ that can be computed from \eqref{Eq:e-m(u)}.

\paragraph{\bf Computation of $e_1$: Rogers dilogarithm} 
In this case, the preceding proposition gives us that for any $u\in ]1,+\infty[$, one has  
\begin{align*}
e_1(u)= & \, \int
\hspace{-0.2cm}
\int
\hspace{-0.2cm}
\int_{k_1,k_2,h_1=0}^{+\infty}
\frac{dk_1dk_2dh_1}{\Big( \alpha_0 -2 \alpha_1 u +\alpha_2 u^2 \Big)^{2}}
\end{align*}
with $\alpha_0=  k_1k_2+h_1(k_1+1)+k_2$, $
\alpha_1= h_1$ and 
$ \alpha_2= h_1(1+k_2)$. 
\mk 

By successive direct computations, we get that 
\begin{align*}
e_1(u)= & \, 
\int \hspace{-0.2cm}
\int_{k_1,k_2=0}^{+\infty}
\frac{dk_1dk_2}{
\Big(k_1k_2 + k_1 + k_2\Big)\Big( \big(k_2 +1\big)  u^2- 2u + (k_1 + 1)\Big)
}
\\
= & \, 
\int \hspace{-0.2cm}
\int_{k_1,k_2=1}^{+\infty}
\frac{dk_1dk_2}{
\Big( k_1k_2-1\Big)\Big( k_2   u^2- 2u + k_1\Big)}
\\
= & \, 
\int_{k_2=1}^{+\infty}
\frac{ \log \Big( 
k_2/(k_2-1)
\Big) 
+\log \Big( 
k_2 u^2 - 2u + 1
\Big) 
}{ \big(u\, k_2-1 \big)^2} 
{dk_2}\\
= & \hspace{0.25cm}
2\left( \frac{\log(u)}{u-1}-
\frac{\log(u-1)}{u}\right)= 4\,{\mathcal R}'(u)
\end{align*}
with ${\mathcal R}(u)=R\left(  \frac{u-1}{u}\right)$ for any $u>1$, where $R$ is Rogers dilogarithm \eqref{Eq:R}.
\mk 

It would be interesting to study \eqref{Eq:e-m(u)} further and in particular to see  whether the explicit computation above can be generalized for $n>2$. Our few preliminary attempts in this direction have not been successful and it might be the case that this is something difficult. However, again by means of explicit computations (performed on a computer algebra system), we have been able to extract interesting information about $e_{n-1}$ for $n=3,4$. We present some of the results we have obtained in the following two paragraphs.

\paragraph{\bf On the Taylor expansion of $e_2$} 
We consider here the case $n=3$ which corresponds to the function $e_2$. 
This function takes as arguments elements of $U\subset ]1,+\infty[^2$, that we will denote  here by $(u_2,u_3)$ (hence $u_2$ and $u_3$  are such that $1<u_2<u_3$).

For such a pair $(u_2,u_3)$, 
modulo an elementary and unimportant change of the integration scheme, 
\eqref{Eq:e-m(u)} reads in explicit form
\begin{equation}
\label{Eq:e-2}
e_2(u_1,u_2)=\int_{h_2,\ldots,h_5>0}
B(u,h)
^{-5/2} dh
\end{equation}
with $dh=dh_2\wedge \ldots \wedge dh_5$ and 
where $B(u,h)$ is the following polynomial of degree 2 in $u_1 $ and $u_2$
$$
B(u,h)=\Big( \beta_{11} \,u_1^2-2 h_4h_5u_1u_2+\beta_{22} \,u_2^2\Big) -2 h_3
\Big( 
h_4 \,u_1+h_5 \,
u_2\Big) +\beta_0
$$
with  $\beta_{11}=  h_4(h_2 + h_3 + h_5)$, 
$\beta_{22}= h_5(h_2 + h_3 + h_4) $
 and $\beta_0=   (1+h_3)(1+h_2 + h_4 + h_5)-1$.

Even with the help of a symbolic integration software, we have not been able to compute the quadruple multiple integral above for $u_1,u_2)$ arbitrary in $U$.  However, 
it is not the same if working with jets of finite order at a suitable base-point in $U$, 
such as 
$$u^*=
 (2,3)\, .$$

Let 
$\varepsilon_2(u,h)$ be the integrand in \eqref{Eq:e-2}, which is obviously analytic at $u^*$. We denote by 
$\varepsilon_2(u,h)=\sum_{ w=0}^{+\infty} \sum_{\lvert \ell\lvert=w}
\varepsilon_2^{\ell}(h) (u-u^*)^{\ell}$
  its Taylor expansion at $u^*$, where we use the notation  $(u-u^*)^{\ell}=(u_1-2)^{\ell_1}(u_2-3)^{\ell_2}$ for any $\ell=(\ell_1,\ell_2)\in \mathbf N^2$.  Setting $B(h)=B(u^*,h)$, one can verify that there are polynomials $ P^{\ell}(h)$  for any $\ell$\footnote{One can give an explicit formula for $ P^{\ell}(k,h)$ but since it is a bit involved and not relevant for the discussion here we will not elaborate on this.} such that $
\varepsilon_2^{\ell}(h)={ P^{\ell}(h)}/{ B(h)^{\frac{5}{2}+w}}
$ for any $w\geq 0$ and any  $\ell$ such that $\lvert \ell\lvert=\ell_1+\ell_2=w$. 
With some work (left to the reader), one can justify the exchange of the integration and the summation below
$$
e_2(u_1,u_2)
=
\int_{h>0}   \Bigg[ \sum_{\substack{w\geq 0\\ \lvert \ell\lvert=w} }
\varepsilon_2^{\ell}(h) (u-u^*)^{\ell}\Bigg] \,  dh
=\sum_{\substack{w\geq 0\\ \lvert \ell\lvert=w} }
 \Bigg[
\int_{h>0}  
\varepsilon_2^{\ell}(h) dh\Bigg]  \, (u-u^*)^{\ell}\,. 
$$

The interest of doing this lies in the fact that we have been able to compute in closed form 
the integrals $\int_{h>0}  
\varepsilon_2^{\ell}(h)\,  dh$, at least for the pairs $\ell$ of weight $\lvert \ell\lvert$ small enough and we have been able to get that way the first terms of the Taylor expansion of $e_2$ at $u^*$: 
\begin{align*}
\frac{4}{\pi}\,  e_2(u)=
\frac{2}{9}-& \left( \frac{\mathit{(u_1-2)}}{9}+\frac{5\,  (\mathit{u_2}-3)}{27}
\right)
+  \frac{(\mathit{u_1} -2)^{2}}{18}+\frac{5\,  (\mathit{u_1} -2) (\mathit{u_2} -3)}{54}+\frac{19\,  (\mathit{u_2} -3)^{2}}{162}
 \\
-& \left( \frac{(\mathit{u_1} -2)^{3}}{36}+\frac{5 \, (\mathit{u_1} -2)^{2} (\mathit{u_2} -3)}{108}+\frac{19\,  (\mathit{u_1} -2) (\mathit{u_2} -3)^{2}}{324}+\frac{65 \, (\mathit{u_2} -3)^{3}}{972}
\right) \\
 +&  \, \frac{(\mathit{u_1} -2)^{4}}{72}+\frac{5\,  (\mathit{u_1} -2)^{3} (\mathit{u_2} -3)}{216}+\frac{19\,  (\mathit{u_1} -2)^{2} (\mathit{u_2} -3)^{2}}{648}+\frac{65\,  (\mathit{u_1} -2) (\mathit{u_2} -3)^{3}}{1944}
 \\
 +&\,  \frac{211\,  (\mathit{u_2} -3)^{4}}{5832}
+ O\Big( \big(u-u^*\big)^5
\Big) \, .
\end{align*}

We observe that, up to the fourth order,  
the 
RHS of this equality coincides with the Taylor series at $u^*$ of 
$$
\frac{8}{3\,u_1u_2(u_2-1)}=
\frac{2}{9}\sum_{w =0}^{+\infty} \sum_{l=0}^w (-1)^w 
\frac{\big( 
3^{l + 1} - 2^{l + 1}
\big) 
}{2^{w}\,3^{l}}
\big(u_1 - 2\big)^{w - l}\big(u_2 - 3\big)^l\,.
$$
This coincidence could leads us to think that we have identically 
\begin{equation}
\label{Eq:e2=1/truc}
e_2(u_1,u_2)=
\frac{(2\pi /3)}{ u_1u_2(u_2-1)}
\end{equation}
and we will see a bit further that this is indeed the case.

\subsubsection{\bf Invariance properties of  the function $e_n$: some transformation formulas} 
\label{SSS:Dihedral-Invariance-Properties}
The positive grassmannian $G_2^{or}(\mathbf R^{n+3})^{>0}$ admits dihedral symmetries. 
Our goal here is to explain how one can deduce from them some quite strong invariance properties for the function $e_{n-1}$, this for any $n\geq 2$. \sk

\vspace{-0.3cm}
Euler's abelian relation $ \boldsymbol{\mathcal E}_n^{>0}$ (see \eqref{Eq:E_n->0}) is defined on the 
positive part $\mathcal M_{0,n+3}^{\, >0}$.  As explained in \S\ref{SSS:Real-Locus}, this space admits a dihedral group of symmetries and a natural question is about the behavior of 
$ \boldsymbol{\mathcal E}_n^{>0}$ with respect to them and what can be deduced from that for $e_{n-1}$. \mk

The subgroup $D_{0,n+3}$ of $\mathfrak S_{n+3}$ letting $\mathcal M_{0,n+3}^{\, >0}$ invariant is generated by the maps $C$ and $R$, respectively induced by 
\begin{itemize}
\item the shift $\tilde C :  (\mathbf P^1)^{n+3}\rightarrow (\mathbf P^1)^{n+3},\, 
(z_i)_{i=1}^{n+3}\mapsto (z_{n+3},z_1,\ldots,z_{n+2})$ which is of order $n+3$; \sk 
\item  the involutive transformation $\tilde R : (\mathbf P^1)^{n+3}\rightarrow (\mathbf P^1)^{n+3},\, (z_i)_{i=1}^{n+3}\mapsto (z_1,z_{n+3},z_{n+2},\ldots,z_3,z_2)$ which, as an automorphism of a regular $(n+3)$-gon,   corresponds to the  reflection with respect to the dihedral axis joining the center on the $(n+3)$-gon to that of its  vertices labeled by 1. 
\end{itemize}

We set $n'= \lceil (n+3)/2\rceil$. 
The two permutations of $\{1,\dots,n+3\}$ corresponding to  $C$ and $R$ 
 are the $(n+3)$-cycle $c=(1 \ldots n+3)$ and the product of $n'-1$ transpositions given by  $r=(2,n+3)(3,n+2)\cdots (n',n'+1) $ when $n$ is even and $r=(2,n+3)(3,n+2)\cdots (n',n'+2)$ otherwise (see Figure \ref{Fig:n+3-gon} below).  In particular, if ${\bf sgn}: \mathfrak S_{n+3}\rightarrow \{ \pm 1\, \}$ denotes the signature, then one has
 \begin{equation}
\label{Eq:Signature-C-r} 
 {\bf sgn}(c)=(-1)^n \qquad 
 \mbox{ and } 
 \qquad  
  {\bf sgn}(c)=(-1)^{n'-1}\, .
 \end{equation}
 
\begin{figure}[h!]
\centering
\includegraphics[width=90mm]{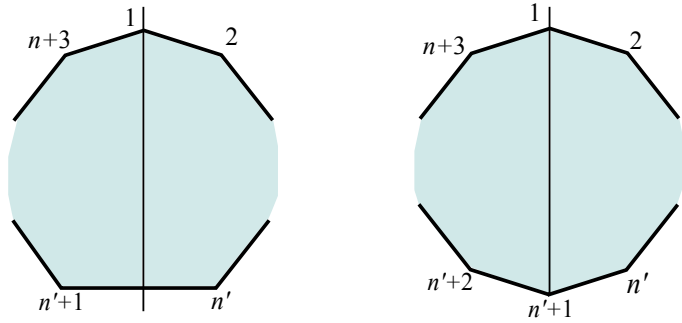}
\vspace{-0.3cm}
\caption{Axial symmetry of the $(n+3)$-gon when $n$ is even (left) and odd (right).}
\label{Fig:n+3-gon}
\end{figure}

\noindent{\bf Warning:} {\it The transformation denoted by $C$ here corresponds to the inverse of the one unfortunately denoted the same in \S\ref{SS:Some-Notations} which moreover is expressed there with respect to another coordinates system on $\mathcal M_{0,n+3}$ than the one used here.\footnote{How the $x_i$'s of  \S\ref{SS:Some-Notations} and the $u_k$'s of this section are related  can be summarized  by the  equality $[0,1,\infty,x_1,\ldots,x_n]=[\infty,0,-1,-u_1,\ldots,-u_n]$ (in $\mathcal M_{0,n+3}$) which has to be thought of as identically satisfied.} We apologize to the reader for this inconvenience.}
\mk

The map \eqref{Eq:Varphi-U} gives rise to an isomorphism $U\rightarrow \mathcal M_{0,n+3}^{\, >0}$ 
again denoted by $\varphi$.  It is straightforward (and left to the reader) to express  $C$ and $R$ 
in the affine coordinates $u_1,\ldots,u_n$ on $U\subset ]1,+\infty[^n$ explicitly. Denoting (a bit abusively) by the same notation the automorphisms of $U$ corresponding to 
$R$ and $C$  (given by conjugating them by $\varphi$) 
and agreeing that $u_{0}=1$ and $u_{-1}=0$, 
one has 
$$C(u)=\left( \, \frac{u_n}{u_n-u_{i-1}} \,\right)_{i=1}^n
\qquad \quad \mbox{ and } 
\qquad  \quad 
R(u)=\left( \, \frac{u_n-u_{n-i-1}}{u_n-u_{n-1}} \,\right)_{i=1}^n\, . 
$$ 

Actually, for our purpose here, we do not have to only consider $C$ and $R$, which are automorphisms of $\mathcal M_{0,n+3}^{\, >0}$, but some lifts of these to the oriented 
grassmannian $G_2^{or}({\mathbf R}^{n+3})$ as well, and there is a subtlety about this. 
This subtle point has not be considered in \cite{D} although it is quite relevant regarding the invariant properties of 
$ \boldsymbol{\mathcal E}_n^{>0}$ with respect to the action of $\mathfrak S_{n+3}$.
\sk 

Actually 
the lifts to the oriented grassmannian of automorphisms  of 
$\mathcal M_{0,n+3}$ considered by Damiano (either in \cite{D} or in \cite{DThesis}) are the simplest/most naive ones: 
given a permutation $\sigma\in \mathfrak S_{n+3}\simeq {\rm Aut}(\mathcal M_{0,n+3})$, 
the lift over it that Damiano considers is the automorphism $\widetilde \sigma \in {\rm Aut}\big(G_2^{or}({\mathbf R}^{n+3})\big)$ induced by the natural linear action of $\sigma$ on $\mathbf R^{n+3}$, namely 
$$\sigma\cdot x=\sigma \cdot  \big(x_i\big)_{i=1}^{n+3}=\big( x_{\sigma(i)}\big)_{i=1}^{n+3}\, .$$

The nice feature of the naive lifting map $\sigma\mapsto \tilde \sigma$ considered by Damiano is that it gives rise to  a group  monomorphism $\mathfrak S_{n+3}={\rm Aut}\big(\mathcal M_{0,n+3}\big) \rightarrow {\rm Aut}\big( G_2^{or}(\mathbf R^{n+3})\big)$. 
But the problem with it is  clearly shown when considering the case of $C$: using the notations \eqref{Eq:Map-M} and \eqref{Diag:gamma-triangle},  
the associated permutation 
$c=(1\ldots n+3)$ is such that Damiano's naive lift $\tilde c$ satisfies 
$$\tilde c\cdot \gamma(u)=\xi\Big(\,  \tilde c\cdot M_u\, \Big)=
\xi\left( \, 
\begin{bmatrix}
- u_n &  1 &  0  &  -1  & -u_1  &   \cdots  & -u_{n-1} \\
1 & 0 &  1  &  {}^{} \, 1  & 1  &   \cdots   & 1 
\end{bmatrix}
\, 
\right)
$$  
for any $u\in U$.  Clearly, some of the $2\times 2$ minors of $\tilde c\cdot M_u$ are positive, whereas others (namely all those involving the first column of $\tilde c\cdot M_u$) are negative. This has the consequence that the matrix $\tilde c\cdot M_u$ is not merely positive, which translates into the fact that with Damiano's choice for the lift $\tilde c$, the image of $\tilde c\cdot \gamma$ into the oriented grassmannian  of 2-planes in $\mathbf R^{n+3}$ does not coincide with  the positive part
$ G_2^{or}(\mathbf R^{n+3})^{>0}$.  Actually, the image 
$\tilde c\cdot \gamma(U)$, which is $\tilde c\cdot G_2^{or}(\mathbf R^{n+3})^{>0}$,   is disjoint from the positive part of the oriented grassmannian: one has 
$$ 
\tilde c \left(G_2^{or}\big(\mathbf R^{n+3}\big)^{>0}\, \right)  \cap 
G_2^{or}\big(\mathbf R^{n+3}\big)^{>0} =\emptyset \, . $$ 

In order to study the invariant properties of $ \boldsymbol{\mathcal E}_n^{>0}$ with respect to the action of $C$, one has to look for a lift $\hat c$  making the following diagram 
commutative: 
\begin{equation}
\label{Eq:hat-c}
\begin{tabular}{c}
\xymatrix@R=1.1cm@C=1.1cm{
G_2^{or}\big(\mathbf R^{n+3}\big)^{>0}  \ar@{->}[r]^{\hat c} & G_2^{or}\big(\mathbf R^{n+3}\big)^{>0}
 \\
 \ar@{->}[u]^{\gamma}
U \ar@{->}[r]^{C} & U
\ar@{->}[u]_{\gamma}
}
\end{tabular}
\end{equation} 
In particular,  the lift $\hat c$ must let invariant the positive grassmannian, 
which precisely Damiano's lift $\tilde c$ does not do. \sk

\paragraph{\bf Invariance properties of Euler's abelian relation $\mathcal E_n^{>0}$ with respect to 
pull-back under $C$.}
\label{Par:CCCCC*}
The relevant lift of $c$ to deal when considering the notion of positivity on the oriented grassmannian is well-known: it is the cyclic automorphism $\hat c$ of $G_2^{or}\big(\mathbf R^{n+3}\big)$ 
induced by the linear isomorphism of $\mathbf R^{n+3}$ given by 
\begin{equation}
\label{Eq:hat-cc}
 \big(x_1,\ldots,x_{n+3}\big)\mapsto \big(-x_{n+3}, x_1,\ldots,x_{n+2}\,\big)\,
\end{equation}
whose matrix  (written by blocs) with respect to the standard basis of $\mathbf R^{n+3}$ is 
$$
\begin{bmatrix}
0 & {\rm I}_{n+2}\\
-1 & 0  
\end{bmatrix}\, .
$$
This matrix belongs to  
${\rm O}_{n+3}(\mathbf R)$ and has determinant equal to $(-1)^{n-1}$.   
Hence $\hat c$ is not always in the special orthogonal group but it is easily verified (using \eqref{Eq:de1-wedge-de2} for instance) that  one has 
$\hat c^*(E_n)=E_n$  in any case.  It follows that $E_n^n$ is left invariant by the lift 
$\hat c$ as well. 
Consequently, it follows that the pull-back of  Euler's abelian relation 
$ \boldsymbol{\mathcal E}_n^{>0}$ under $C$  
  coincides with $\boldsymbol{\mathcal E}_n^{>0}$ or with its opposite, according 
whether   $\hat c$ preserves the orientation of a generic $H_0$-orbit in 
$G_2^{or}(\mathbf R^{n+3})^{>0}$.\footnote{We recall that $H_0\simeq (\mathbf R_{>0})^{n+2}$ stands for the positive part of the diagonal Cartan torus $H\subset {\rm SL}(\mathbf R^{n+3})$.} 

In order to determine the action of $\hat c$ on the orientation of the $H_0$-orbits of elements in 
$G_2^{or}(\mathbf R^{n+3})^{>0}$, let us consider the following parametrization of the  positive part of this oriented grassmannian
\begin{align}
\label{Al:lala}
U\times   (\mathbf R_{>0}\big)^{2}\times (\mathbf R_{>0}\big)^{n} & \longrightarrow  G_2^{or}(\mathbf R^{n+3})^{>0}\\
(u,k,h)
=\Big( \,(u_i)_{i=1}^n , \big(k_1,k_2\big), \big( h_i\big)_{i=1}^n 
\,\Big) 
& \longmapsto 
\begin{bmatrix}
k_1 & 0 &  -1 & -u_1  h_1 & \ldots & -u_n h_n
\\ 
0 & k_2 & {}^{} \hspace{0.16cm} 1 & {}^{} \hspace{0.16cm} h_1 & \ldots & {}^{} \hspace{0.16cm} h_n
\end{bmatrix}\,
\nonumber
\end{align}
(where $U$ is the set of $n$-tuples $(u_i)_{i=1}^n\in \mathbf R^n$ such that $1<u_1<\ldots u_{n-1}<u_n$).  This map is an isomorphism hence the $u_i$, $k_s$ and $h_i$ for $i=1,\ldots,n$ and $s=1,2$ form a system of global coordinates on $G_2^{or}(\mathbf R^{n+3})^{>0}$, which is particularly useful for our purpose since relatively to these coordinates, quotienting by the $H_0$-action corresponds exactly to the standard projection $(u,k,h)\mapsto u$ ({\it cf.}\,Proposition \ref{P:Prop-em-integral-formula} above).  Let $\hat C$ be the automorphism of $U\times   (\mathbf R_{>0}\big)^{2}\times (\mathbf R_{>0}\big)^{n}$ obtained from the automorphism  $\hat c$ of $G_2^{or}(\mathbf R^{n+3})^{>0}$ 
by means of the identification induced by \eqref{Al:lala}. 
It is not difficult to give an explicit formula for $\hat C$. 
  Indeed setting $u_0=1$ 
 and denoting by $c_i(u)=u_n/(u_n-u_{i-1})$ the $i$-th component of $C(u)$  for $i=1,\ldots,n$, 
straightforward computations (the details of which are left to the reader) give us that 
$$\hat C (u,k,h)=\Big( \hat u ,\hat k,  \hat h\Big)$$
 with 
\begin{align*}
\hat u = & \, C(u)= \big(c_i(u)\big)_{i=1}^n
=\left( 
\frac{u_n}{u_n-u_{i-1}}\right)_{i=1}^n \,, 
\\
\hat k=& \,\Big( \, \hat k_1\, ,\,  \hat k_2\, \Big) = \left( \frac{h_n}{k_2}, \frac{k_1}{k_2u_n}\right)\\
 \mbox{and }\quad \hat h=&\, 
\Big(  \, \hat h_i\,
\Big)_{i=1}^n 
 =
 \left( \,  \frac{ 1}{k_2\,c_1(u)} \,  ,\,  
 \frac{h_1}{k_2\, c_2(u) }\, , \, \ldots \, ,\,  
 \frac{h_{n-1}}{k_2\, c_n(u)} \, \right)\, .
\end{align*} 
The implicitly chosen (because natural) orientation on the $H_0$-orbits of elements of $G_2^{or}(\mathbf R^{n+3})^{>0}$
corresponds to the $(n+2)$-form $dk\wedge dh=dk_1\wedge d k_2\wedge dh_1\wedge \cdots \wedge dh_n$ in the coordinates $u_i, k_s$ and $h_i$ ($i=1,\ldots,n$, $s=1,2$). 
Using the expressions above for the components of $\hat k$ and $\hat h$, one easily computes 
what is the $dk\wedge dh$-component of the pull-back of this form under $\hat C$: setting 
$\partial k\wedge \partial h=\partial_{k_1}\wedge \partial_{k_2}\wedge \partial_{h_1}\wedge \cdots \wedge \partial_{h_n}$, one obtains that 
$$
\big( \partial k\wedge \partial h \big) 
\,\scalebox{1.4}{$\lrcorner$}\, 
\hat C^* \Big( \, dk \wedge dh\, \Big)= 
(-1)^{n}
\frac{\prod_{i=0}^{n-1}(u_n-u_i) }{k_2^{n+3}u_n^{n+1}}
\, .
$$

Since $u\in U$, one has $u_n-u_i>0$ for any $i=0,\ldots,n-1$ hence the sign of the above expression is $(-1)^n$.  On the other hand, 
 $\hat c$ lets $E_n^n$ invariant as we have seen above, hence we deduce the 
\begin{prop}
\label{P:C^*Euler-AR}
The positive Euler's abelian relation is stable up to multiplication by $(-1)^n$ under pull-back by $C$, {\it i.e.} identically on $U$, one has
$$
C^*\Big(  \boldsymbol{\mathcal E}_n^{>0} \Big) =(-1)^n\,   \boldsymbol{\mathcal E}_n^{>0}\, .
$$
\end{prop}

When $n$ is even, we will state a more general result whose proof is  much more conceptual 
and which admits  the above as a particular case (see Proposition \ref{P=sign(sigma)} further below). \sk

In any case, the preceding proposition has two interesting outcomes.  The first is that one recovers the fact that 
$\boldsymbol{\mathcal E}_n^{>0}$ is entirely known as soon as one, say the first,  of its components is known. Indeed, the map $U_{n+3}(u)=(u_1,\ldots,u_{n-1})$ is a first integral of the $(n+3)$-th foliation of 
$\boldsymbol{\mathcal W}_{0,n+3}^{\, >0}$ on $U$ (that foliation corresponding to forgetting the $(n+3)$-th point on ${\mathcal M}_{0,n+3}^{\, >0}$. 

For $i=0,\ldots,n+2$, the rational map 
 $$ U_{n+3-i}=\big(C^{\circ i} \big)^*\big(U_{n+3}\big)  $$
 (which is just given by taking the first $n-1$ components of
 the map obtained by composing $C$ $i$-th times) 
 is a rational first integral of the $(n+3-i)$-th foliation of $\boldsymbol{\mathcal W}_{0,n+3}^{\, >0}$ (corresponding to forgetting the $(n+3-i)$-th point on $\mathcal M_{0,n+3}^{>0}$).  
 For any $i$, denote by $U_i^k$ for $k=1,\ldots,n-1$ the components of $U_i$ 
 and set $dU_i=dU_i^1\wedge \cdots \wedge dU_i^{n-1}$. 
  Since the 
 $U_{n+3}$-component of $\boldsymbol{\mathcal E}_n^{>0}$ is 
 $$e_{n-1}\big(U_{n+3}\big) 
 \, dU_{n+3}
 =
 e_{n-1}(u_1,\ldots,u_{n-1}) du_1\wedge \cdots \wedge du_{n-1}\,,$$  
 it follows from 
 the above proposition  that $\boldsymbol{\mathcal E}_n^{>0}$ 
 is equivalent to the fact that the differential relation 
 \begin{equation}
\label{Eq:AR-Euler-cyclique}
0= 
\sum_{i=1}^{n+3} {(-1)^{i\, n}} \, e_{n-1}(U_i)\,  dU_i^1\wedge \cdots \wedge dU^{n-1}_i
%
%
\end{equation} 
holds true identically for any 
$n$-tuple of real numbers $(u_1,\ldots,u_n)\in U$.
\begin{exm} 
\label{Ex:n=2}
When $n=2$, the maps $U_i$ are given by 
$$
U_5=u_1
\, , \quad  
U_4=\frac{u_2}{u_2 - 1}
\, , \quad 
U_3=\frac{u_2}{u_1} 
\, , \quad  
U_2= \frac{u_2-1}{u_2 - u_1}
   \quad  \mbox{ and } \quad 
U_1=   \frac{u_1(u_2-1)}{u_2( u_1-1)}\, .
$$

Up to integration, 
\eqref{Eq:AR-Euler-cyclique} is  equivalent to the fact that the relation
\begin{equation}
\label{Eq:rm R}
{\rm R}(u_1)
+{\rm R}\left( \frac{u_2}{u_2 - 1} \right) 
+{\rm R}\left(
\frac{u_2}{u_1} \right) 
+{\rm R}\left(
 \frac{u_2-1}{u_2 - u_1} \right) 
+  {\rm R}\left(\frac{u_1(u_2-1)}{u_2( u_1-1)} \right) =0
\end{equation}
holds true for any $u_1,u_2$ such that $1<u_1<u_2$, where 
${\rm R}$ stands for the function 
defined by $$
{}^{} \qquad 
{\rm R}(u)=\frac{1}{2}\int_u^{+\infty}  \left(
\frac{\log(u-1)}{u}-\frac{\log(u)}{u-1}
 \right) du +\frac{{}^{}\hspace{0.15cm} \pi^2}{10} 
 \qquad \quad \big(\, u>1\, \big)\,.\footnotemark
 $$
 \footnotetext{The constant $\pi^2/10$ in the definition of ${\rm R}$  is just a normalization added in order that the 
  RHS of \eqref{Eq:rm R}
be zero.}

 It can be verified that given $u>1$, one has 
${\rm R}(u)=R_1\big(  (u-1)/u
\big)$ where $R_1$ stands for the following version of Rogers dilogarithm: 
$R_1(x)={\bf L}i_2(x) + \log(x)\log(1 - x)/2 - \pi^2/15$ for $x\in ]0,1[$. 

One checks easily  that \eqref{Eq:rm R} is indeed left invariant by the 
5-cyclic  birational map 
 $$
 C(u_1,u_2)=\left( \frac{u_2}{u_2-1}, \frac{u_2}{u_2-u_1}\right)
 \, .$$
\end{exm}
The second interesting consequence of Proposition \ref{P:C^*Euler-AR} is a nice transformation formula for the function $e_{n-1}$ appearing in the components of $\boldsymbol{\mathcal E}_n^{>0}$.   Recall the rational functions $\psi_i$ introduced just before 
Proposition \ref{P:RA-EULER-a}, which   are first integrals for $\boldsymbol{\mathcal W}_{0,n+3}^{>0}$ as well. With respect to these, the 
last two components of $\boldsymbol{\mathcal E}_n^{>0}$ in 
\eqref{Eq:dE-n-n-2} are written  respectively 
\begin{align}
\label{Eq:f-n+2}
 &f_{n+2}(u_1,\ldots,u_{n-2},u_n\big) \cdot du_1\wedge \cdots \wedge d u_{n-2}\wedge du_n
\\
 \mbox{ and } \qquad 
& f_{n+3}(u_1,\ldots, u_{n-1}\big) \cdot du_1\wedge \cdots \wedge d u_{n-2}\wedge d u_{n-1}
\nonumber
\end{align}
for two functions $f_{n+2}$ and $f_{n+3}$ globally defined (and analytic) on $U$.

On the other hand, for any $i=1,\ldots,n+3$,  the $i$-th component
$(-1)^{i-1}\psi_i^*\big( \mathcal E_{n-1}^n\big)$ 
of $\boldsymbol{\mathcal E}_n^{>0}$ can also be written $\big( \delta_i\big)_*\big( \tau_i^*\big(\mathscr E_n^n\big) \big)$. 
 Then following carefully the computations of \S\ref{SS:En-integral-representation} 
for these expressions when $i$ is $n+2$ or $n+3$ (with $m=n-1$), we obtain that the following holds true:
\begin{fact}  As functions of $n-1$ variables,   one has (up to multiplication by a non zero constant)
\begin{equation}
\label{Eq:f-n+2-f-n+3-e-n-1}
f_{n+2}=f_{n+3}=e_{n-1}
\end{equation}
where   $e_{n-1}$ stands of course for 
the function defined in 
\eqref{Eq:e-m(u)}  (in the case when $m=n-1${\rm )}. 
\end{fact}

On the other hand, one has:
\begin{align}
\label{Eq:C^*}
 C^*\Big(  \, du_1\wedge \cdots \wedge d u_{n-1} \, 
\Big) =   &\, 
\, d\left(\frac{u_n}{u_n-1}\right) \wedge 
d\left(
 \frac{u_n}{u_n-u_1} \right) \wedge \cdots \wedge 
d \left(\frac{u_n}{u_n-u_{n-2}}\right)\\
 = &\,  (-1)^{n-1} \frac{  \big( u_n\big)^{n-2} }{ 
 \prod_{i=0}^{n-2}\big(u_n-u_i\big)^2 }
 \cdot \Big(  du_1\wedge \cdots \wedge d u_{n-2} \wedge d u_{n} \Big) \, .
 \nonumber
\end{align}

Combining \eqref{Eq:f-n+2}, \eqref{Eq:f-n+2-f-n+3-e-n-1} and \eqref{Eq:C^*}, we deduce the 
\begin{prop}
\label{Prop:C*E}
The function $e_{n-1}$   satisfies the following relation:  
\begin{equation}
\label{Formula:C*E}
e_{n-1}\left(\frac{u_n}{u_n-u_0}, \frac{u_n}{u_n-u_1},\ldots,\frac{u_n}{u_n-u_{n-2}}\right)
= 
\frac{
 \prod_{i=0}^{n-2}\big(u_n-u_i\big)^2 }{
(u_n)^{n-2} }\cdot 
e_{n-1}\Big( u_1,\ldots, u_{n-2}, 
 u_n\Big)
\end{equation}
for any real numbers $u_0,u_1,\ldots,u_{n-2}, u_n$ such that $1=u_0<u_1<\ldots <u_{n-2}<u_n$.
\end{prop}

\noindent {\bf Example \ref{Ex:n=2} (continued).} 
{\it Up to multiplication by a non zero constant, one has $e_1(u)=\log(u-1)/u-\log(u)/(u-1)$ 
 for any $u>1$.  Given such a $u$, 
one has  
\begin{align}
\nonumber
e_1\left( \frac{u}{u-1}\right)=&\,  -\frac{(u-1)\log(u-1)}{u}-(u - 1)\big( 
\log\left(u\right)-
\log\left({u-1}\right) \big)
\\
\label{EQ:n=22}
=&\, \left(-\frac{u-1}{u}  +(u-1) \right)  \log(u-1)
-(u - 1){
\log\left(u\right)}\\ \nonumber
=&\, (u-1)^2 \frac{ \log(u-1) }{u} 
-(u - 1)^2
\frac{
\log\left(u\right)}{u-1}
= 
(u-1)^2 e_1(u)\, .
\end{align}}

\paragraph{\bf Invariance properties of Euler's abelian relation $\mathcal E_n^{>0}$ with respect to 
pull-back under $R$.}
The case of the automorphism $R$ of $\mathcal M_{0,n+3}^{>0}$ is handled by similar but 
somewhat 
 more subtile 
arguments that we are going to discuss below. The lift $\hat r$ of $R$  that we consider here is the  automorphism of the oriented grassmannian induced by the linear map 
$(x_i)_{i=1}^{n+3}\mapsto \big(-x_1, x_{n+3}, x_{n+2},\ldots,x_2,x_1\big)$
whose matrix  (written by blocs) with respect to the standard basis of $\mathbf R^{n+3}$ is 
$$
N_n=
\begin{bmatrix}
-1 & 0\\
0 & {\rm J}_{n+2}   
\end{bmatrix}\,
$$
where ${\rm J}_{n+2}$ stands for the anti-diagonal $(n+2)\times (n+2)$ matrix. The matrix $N_n$  is orthogonal with determinant $\delta_n=(-1)^{\lfloor  n/2\rfloor}$ thus does not always belong to ${\rm SO}_{n+3}(\mathbf R)$. Nevertheless, in any case (using for instance \eqref{Eq:de1-wedge-de2}) 
it can be proved that 
$\hat r$ lets the Euler form $E_n$ invariant: one has $\hat r^*(E_n)=E_n$.  But contrarily to $\hat c$, the map $\hat r$ does not give rise to an automorphism of the positive part of the oriented grassmannian since, 
as it can be verified straightforwardly, 
 one has 
$$ \hat r \left(  G_2^{or}\big(\mathbf R^{n+3}\big)^{>0}
\right) =G_2^{or}\big(\mathbf R^{n+3}\big)^{<0}\, .
$$
In order to land into $G_2^{or}\big(\mathbf R^{n+3}\big)^{>0}
$, it is necessary to post compose $\hat r$ with the change of orientation map $D$ defined in 
\S\ref{SSS:GrassmannianStuff}. One then obtains a commutative diagram
\begin{equation}
\label{Eq:hat-c}
\begin{tabular}{c}
\xymatrix@R=1cm@C=1.1cm{
G_2^{or}\big(\mathbf R^{n+3}\big)^{>0}  \ar@{->}[r]^{D\circ \hat r} & G_2^{or}\big(\mathbf R^{n+3}\big)^{>0}
 \\
 \ar@{->}[u]^{\gamma}
U \ar@{->}[r]^{R} & U
\ar@{->}[u]_{\gamma}
}
\end{tabular}
\end{equation}

Because $D^*(E_n)=-E_n$  (by Lemma \ref{Lem:Properties-En}.1), it follows that 
one has 
\begin{equation}
\label{Eq:Dr*(En)}
(D\circ \hat r)^*(E_n^n)=(-1)^n E_n^n\, . 
\end{equation} 

Now one has to determine the action of $D\circ \hat r$ on the orientation of the $H_0$-orbits in 
$G_2^{or}\big(\mathbf R^{n+3}\big)^{>0}$.  We will proceed as in the case of $C$, by 
considering the map $\tilde R$ from $U\times (\mathbf R_{>0})^2\times 
 (\mathbf R_{>0})^n$ into itself corresponding to $D\circ \hat r$. 
In the coordinates $u,k$ and $h$ on $U\times (\mathbf R_{>0})^2\times 
 (\mathbf R_{>0})^n$ considered above, one has
 $$
 \tilde R(u,k,h)=\Big( \, \tilde u \, , \, \tilde k  \, , \, \tilde h\, \Big) 
 $$
with 
\begin{align*}
 \tilde u = R(u)
=\left( 
\frac{u_n-u_{n-1-i}}{u_n-u_{n-1}}\right)_{i=1}^n \, , 
\quad 
 \tilde k= 
 \left( \frac{k_1}{h_{ n-1}(u_n-u_{n-1})} \, , \, \frac{h_n}{h_{n-1}}\right) 
 \quad 
 \mbox{and }\quad \tilde  h=
 \left( \,  \frac{ h_{n-1-i}}{h_{n-1}} \, 
\, \right)_{i=1}^n
\end{align*} 
 where we use the following notation: $u_{-1}=0$, $u_0=1$, $h_0=1$  and $h_{-1}=k_2$.



Using the expressions above for the components of $\tilde k$ and $\tilde h$, one easily computes   
what is the $dk\wedge dh$-component of the pull-back of this form under $\tilde R$: 
it is given by  
$$
\big( \partial k\wedge \partial h \big) 
\,\scalebox{1.4}{$\lrcorner$}\, 
\tilde R^* \Big( \, dk \wedge dh\, \Big)= 
\frac{(-1)^{\lfloor n/2\rfloor+1} }{(h_{n-1})^{n+3}(u_n-u_{n-1})}
\, .
$$
The sign of the above expression on 
$U\times (\mathbf R_{>0})^{n+2}$ is that of  $(-1)^{\lfloor n/2\rfloor+1} $. Together with 
\eqref{Eq:Dr*(En)} and because $n+\lfloor n/2\rfloor+1-\lfloor (n-1)/2\rfloor=2\lfloor n/2\rfloor+2$ is even 
for any $n\geq 2$, this gives us the 
\begin{prop}
\label{P:R^*Euler-AR}
Identically on $U$, one has 
\begin{equation}
\label{Eq:R^*Euler-AR-odd}
R^*\Big(  \boldsymbol{\mathcal E}_n^{>0} \Big) =(-1)^{\lfloor (n-1)/2\rfloor} \, \,   \boldsymbol{\mathcal E}_n^{>0}\, .
\end{equation}
\end{prop}

\noindent {\bf Example \ref{Ex:n=2} (finished).} 
{\it When $n=2$, the birational involution $R$ is given by 
$$R(u_1,u_2)=\left( \frac{u_2-1}{u_2-u_1}, \frac{u_2}{u_2-u_1}\right)\, .$$ 
A straightforward verification gives that it lets 
the relation \eqref{Eq:rm R} entirely invariant.}

The ingredients needed in order to get a transformation formula for $e_{n-1}$ from 
\eqref{Eq:R^*Euler-AR-odd} are similar but not completely the same according to the parity of $n$.  We now set $m=\lfloor n/2\rfloor$. 

Let us first discuss the case when $n$ is odd.  In this situation, $R$ lets invariant the $((n+5)/2)$-th foliation of $\boldsymbol{\mathcal W}_{0,n+3}^{>0}$. This foliation admits the map $\psi_{(n+5)/2}(u)
$ as first integral whose components are the $u_i$'s for $i=1,\ldots,n$ distinct from $(n-1)/2$. 
The component of $\boldsymbol{\mathcal E}_n^{>0}$ with respect to this first integral can be seen to be a multiple of 
\begin{equation}
\label{Eq:momo}
e_{n-1}\Big( u_1,\ldots,u_{(n-3)/2},
\reallywidehat{ { u_{(n-1)/2}}}, u_{(n+1)/2}, \ldots,u_n\Big)
\, \Omega_{{(n+5)/2}}
\end{equation}
with 
$ \Omega_{{(n+5)/2}}=\wedge_{i\neq 
(n-1)/2} 
 du_i=du_1\wedge \cdots \wedge du_{(n-3)/2} \wedge du_{(n+1)/2}\wedge \cdots \wedge du_n$. 

By a direct computation, one gets that 
\begin{equation}
\label{Eq:R*}
R^*\Big( 
\Omega_{{(n+5)/2}}
\Big) =\frac{(-1)^{ \lfloor n/2 \rfloor}}{\big(u_n-u_{n-1} \big)^n} \hspace{0.15cm} \Omega_{{(n+5)/2}}\, .
\end{equation}


Let us now deal with the case when $n$ is even. Setting 
$m=n/2$, one has  $n'= \lceil (n+3)/2\rceil=m+2$ and 
the birational involution  $R$ exchanges the $n'$-th and the $(n'+1)$-th foliations of $\boldsymbol{\mathcal W}_{0,n+3}^{>0}$ (see Figure \ref{Fig:n+3-gon}). The two corresponding first integrals $\psi_{n'}$ and $\psi_{n'+1}$ have for components the $u_i$'s for 
$i=1,\ldots,n$   distinct from $m-1$ and $m$ respectively.  Accordingly, we set $\Omega_{n'}=   \wedge_{i\neq 
m-1} 
 du_i$  and $\Omega_{n'+1}=  \wedge_{i\neq 
m} 
 du_i$. Then, up to a common multiple, 
the components of $\boldsymbol{\mathcal E}_n^{>0}$ with respect to these first integrals can be seen to be 
\begin{align}
\label{Eq:momo-even}
e_{n-1}&\Big( u_1,\ldots,
\reallywidehat{  {\, u_{m-1}\, }},  \ldots,u_n\Big)
\, \Omega_{n'} \\
\mbox{and} \quad 
  -e_{n-1}&\Big( u_1,\ldots, 
\reallywidehat{ \,u_{m}\,} , \ldots,u_n\Big)
\, \Omega_{n'+1}\, .
\nonumber
\end{align}
 On the other hand, a direct computation gives us that when $
 n$ is (even and) strictly bigger than 2, one has 
\begin{equation}
\label{Eq:R*II}
R^*\Big( 
\Omega_{n'}
\Big) =\frac{(-1)^{ \lfloor n/2 \rfloor}}{\big(u_n-u_{n-1} \big)^n} \hspace{0.15cm} \Omega_{n'+1}
\end{equation}
whereas when $n=2$, 
one has $n'=3$ and 
the following relation holds true: 
\begin{equation}
\label{Eq:R*2}
R^*\Big( 
\Omega_{3}
\Big) =\frac{-du_2}{\big(u_2-1 \big)^2} \, 
=\frac{-1}{\big(u_2-1 \big)^2} \,
\Omega_{4}\, .
\end{equation}

We recall that  $m=\lfloor n/2\rfloor\geq 1$. 
We now have everything at hand to get another transformation formula satisfied by $e_{n-1}$. Indeed, from \eqref{Eq:R^*Euler-AR-odd} together with 
\eqref{Eq:momo} and \eqref{Eq:R*} when $n$ is odd, and with 
\eqref{Eq:momo-even} and \eqref{Eq:R*II} (or \eqref{Eq:R*2} if $n=2$) when $n$ is even, we deduce the 
\begin{prop} 
\label{P:Formula-R}
 {\rm 1.} The function $e_{1}$   satisfies the following relation for any $x_2>1$
\begin{equation}
\label{Eq:tokolo}
\frac{e_1\left( 
\frac{x_2}{x_2-1}
\right)}{(x_2-1)^2}=e_1(x_2)\, .
\end{equation}
{\rm 2.}
For any $n$ strictly bigger than 2, the function $e_{n-1}$   satisfies the  relation:  
\begin{equation}
\label{Formula:R*E-even-odd}
 \frac{e_{n-1}\left(  
\frac{u_n-u_{n-2}}{u_n-u_{n-1}},\ldots, 
\reallywidehat{ 
\frac{u_n-u_{m}}{u_n-u_{n-1}}
}, \ldots, 
\frac{u_n-1}{u_n-u_{n-1}}, \frac{u_n}{u_n-u_{n-1}}
\right)}{\big(u_n-u_{n-1} \big)^n}=  
e_{n-1}\Big( u_1,\ldots,
\reallywidehat{u_{m}}, \ldots,u_n\Big)
\end{equation}
for any $u_1,\ldots, u_{m-1} , \reallywidehat{u_{m}},u_{m+1}, \ldots , u_{n}$ such that $1<u_1<\ldots <
u_{m-1}<u_{m+1}<\ldots
<u_n$.
\end{prop}

Remark that in the case when $n=2$, the two functional relations  
given by Proposition \ref{Prop:C*E} and Proposition \ref{P:Formula-R}
 that $e_1$ satisfies  actually coincide 
(equations \eqref{EQ:n=22} and \eqref{Eq:tokolo} are the same!).


\subsubsection{\bf Euler's abelian relation ${\mathcal E}_3^{>0}$.}
\label{SSS:tilde-e2}
Let us write down \eqref{Formula:C*E} and \eqref{Formula:R*E-even-odd} explicitly  when $n=3$: the two corresponding formulas are respectively 
\begin{equation}
\label{Eq:RC-n=3}
\frac{ 
e_2\left( \frac{u_3}{u_3-1 }, 
 \frac{u_3}{u_3-u_1 }
\right)\, u_3 }{(u_3-1)^2(u_3-u_1)^2} =  e_2(u_1,u_3)
\qquad \mbox{ and }\qquad 
\frac{e_2\left(  \frac{u_3-1}{u_3-u_2} , 
 \frac{u_3}{u_3-u_2}
\right)
}{(u_3 - u_2)^3}=e_2(u_2,u_3)\, , 
\end{equation}
and these two equalities  are satisfied for any $u_1,u_2,u_3$ such that $1<u_1<u_2<u_3$. \sk

We have determined above the first terms of the Taylor expansion of $e_2$ at a given base point $u^*$ and noticed that, up to multiplication by a non zero constant,  these terms were the same as those of the Taylor expansion at $u^*$ of the rational function 
$$
\tilde e_2: 
(u_1,u_2)\mapsto \frac{1}{u_1u_2(u_2-1)}\, . 
$$ As elementary verifications show, the function $\tilde e_2$ satisfies both identities 
\eqref{Eq:RC-n=3}
 as well.  That is not a coincidence. \mk 
 
 For $i=1,\ldots,n+3$, denote by $\boldsymbol{AR}( \boldsymbol{\mathcal W}_{0,n+3}^{>0})[i]$ the subspace of $\psi_i^*\Omega^{n-1}$ spanned by the $i$-th components 
of the ARs of  $\boldsymbol{\mathcal W}_{0,n+3}^{>0}$, with respect to the first integrals $\psi_i$. From the results of \S\ref {SS:Components-of-ARs}, one gets that 
 $\boldsymbol{AR}( \boldsymbol{\mathcal W}_{0,n+3}^{>0})[n+3]$ is exactly the 
 space of 2-forms $F(u_1,u_2) du_1\wedge du_2$  where $F$ ranges in the vector space 
of rational functions  spanned by the set 
 $$\mathfrak B_3(u_1,u_2)=
 \left\{\, F_0(u_1,u_2)=\frac{1}{u_1u_2}\, , \, F_{ij}(u_1,u_2)\, \Big\lvert  \begin{tabular}{l}
 $i=2,3$, \, $j=i+1,\ldots,5$
 \end{tabular}
 \right\}  \subset \mathbf Q(u_1,u_2)$$
defined in  \eqref{Eq:mathfrak-Bn}.  It is straightforward to verify that the function 
$\tilde e_2$ defined above  first   belongs to $\langle \mathfrak B_3(u_1,u_2)\rangle $; and secondly and most importantly, is the unique element  of this space (up to multiplication by a non zero constant) satisfying the same identities as 
\eqref{Eq:RC-n=3}.  Then taking into account the determination of the order three jet of $e_2$ at $u^*$ computed above, we deduce that  equality 
\eqref{Eq:e2=1/truc} is indeed indentically satisfied on $U\subset  ]1,+\infty[^2$. 

It is  just a computational matter to write down \eqref{Eq:AR-Euler-cyclique}
in explicit form: Euler's abelian relation  $\boldsymbol{\mathcal E}_3^{>0}$ corresponds to the following identity
\begin{equation}
\label{Eq:E3-explicit}
 \sum_{i=1}^6 (-1)^ i \mathcal E_{3,i}=0
\end{equation}
where the six 2-forms $\mathcal E_{3,i}=\big(C^{\circ 6-i}\big)^*\big(\mathcal E_{3,6}\big)$'s are given by 
$$
\mathcal E_{3,6}=
\frac{ du_1\wedge du_2}{u_1u_2(u_2-1)}\, , \qquad 
\mathcal E_{3,5}= \frac{ du_1\wedge du_3}{u_1u_3(u_3-1)}
\, , \qquad  
\mathcal E_{3,4}=
\frac{ du_2\wedge du_3}{u_2u_3(u_3-1)}
$$
and 
 \begin{align*}
 \mathcal E_{3,3}= & \, 
 \frac{-du_1\wedge du_2}{u_2(u_3 + u_1)}+ \frac{du_1\wedge du_3}{u_3(u_3-u_1)}-\frac{du_2\wedge du_3}{u_2u_3(u_3-u_1)}
\\
\mathcal E_{3,2}= &\, 
\frac{du_1\wedge du_2}{(u_2 - 1)(u_3  u_1)}-\frac{du_1\wedge du_3}{
(u_3 - 1)(u_3 - u_1) }+\frac{(u_1-1)\, du_2\wedge du_3}{
(u_2-1)(u_3-1)(u_3-u_1)
}
\\
 \mathcal E_{3,1}= &\, 
\frac{-u_3\, du_1\wedge du_2}{u_1u_2(u_2-1)(u_3-u_1)}+\frac{du_2\wedge du_3}{u_1(u_3-1)(u_3-u_1)}
- \frac{(u_1-1)\, du_2\wedge du_3}{u_2(u_2-1)(u_3-1)(u_3-u_1)}\,.
\end{align*}

Obviously, $\boldsymbol{\mathcal E}_3^{>0}$ is rational.  Actually, 
since all the ARs of $\boldsymbol{\mathcal W}_{0,6}$ are combinatorial, one can state the following striking result: 
\begin{prop}
Euler's abelian relation $\boldsymbol{\mathcal E}_3^{>0}$ is combinatorial:   one has
$$\boldsymbol{\mathcal E}_3^{>0} \in \boldsymbol{AR}_C\big( 
\boldsymbol{\mathcal W}_{0,n+3}
\big)\, .$$
\end{prop}

\subsubsection{\bf  Euler's abelian relations and their behavior with respect to the 
action of $\mathfrak S_{n+3}$} 
\label{SSS:Sn+3-Invariance-Properties}
We now want discuss the  questions 
\eqref{Eq:Remain-2-be-answered} raised above. 
\sk

Let us start by considering the first one:  {\it `Does the construction given above of $\boldsymbol{\mathcal E}_n^{>0}$ depend on some choices? If yes, on which ones?'} 
The main object we used 
to construct $\boldsymbol{\mathcal E}_n^{>0}$  on 
${\mathcal M}_{0,n+3}^{>0}$ in section \S\ref{SSS:positive-Euler-AR} is the  map 
$\gamma=\xi\circ M: U\rightarrow G_2^{or}(\mathbf R^{n+3})^{>0}$ defined in diagram 
\eqref{Diag:gamma-triangle}.  Up to the identification \eqref{Eq:Varphi-U} between $U$ and $\mathcal M_{0,n+3}^{>0}$, $\gamma$  corresponds to a section over 
$\mathcal M_{0,n+3}^{>0}$ of the natural map $\widehat{G}_2^{or}(\mathbf R^{n+3})\rightarrow \mathcal M_{0,n+3}$ obtained by composing the arrows from the top left corner to the bottom right corner in diagram \eqref{Eq:Diagrammm}.  Considering this diagram is helpful to better understand $\boldsymbol{\mathcal E}_n^{>0}$: 
it shows that from an intrinsic point of view,  $\boldsymbol{\mathcal E}_n^{>0}$ 
is not defined on $\mathcal M_{0,n+3}^{>0}$ but 
rather on one of the two components 
over $\mathcal M_{0,n+3}^{>0}$ in $\mathcal M_{0,n+3}^{\, or}(\mathbf R)$.

Let us write 
$$
\nu^{-1}\Big( \mathcal M_{0,n+3}^{>0}\Big)=
\mathcal M_{0,n+3}^{\, or,\, >0} \sqcup \mathcal M_{0,n+3}^{\, or\, , <0}\, 
$$
this notation being justified by the fact that one has $ 
\mathcal M_{0,n+3}^{\, or,\, > 0}(\mathbf R)= G_2^{or}(\mathbf R^{n+3})^{>0} / H
\subset \mathcal M_{0,n+3}^{\, or} (\mathbf R)
$ and similarly for the negative setting (corresponding to formally replacing $>$ by $<$). 
Actually \eqref{Eq:Varphi-U} is an identification between $U$ and 
$\mathcal M_{0,n+3}^{\, or,\, >0}$ and $\gamma$ corresponds to a section of 
$$\pi_H: 
\widehat{G}_2^{or}(\mathbf R^{n+3})\longrightarrow \mathcal M_{0,n+3}^{\, or}(\mathbf R)
$$ 
over the latter component.  Since $H\simeq (\mathbf R_{>0})^{n+2}\times (\mathbf Z_2)^{n+2}$, there is no canonical choice for such a section: the section $\gamma$ we have considered is the one landing in $G_2^{or}(\mathbf R^{n+3})^{>0}$. Any other section is of the form $\gamma_{\underline{\varepsilon}}=\xi\circ M_{\underline{\varepsilon}}$ 
for $\underline{\varepsilon}=(\varepsilon_i)_{i=1}^{n+2}\in \{ \pm 1\, \}^{n+3}$
 where $M_{\underline{\varepsilon}}: U\rightarrow \Omega_2$ is the map 
 $u\mapsto M_u\cdot {\rm Diag}(\varepsilon_1,\ldots,\varepsilon_{n+3})$, 
{\it i.e.}\,for $u\in U$,  $M_{\underline{\varepsilon}}(u)$ is the matrix whose  $i$-th column is $\varepsilon_i$ times the corresponding column of  $M_u$  ({\it cf.}\,\eqref{Eq:Map-M}).  For each section $\gamma_{\underline{\varepsilon}}$, similar arguments  to those in \S\ref{SSS:positive-Euler-AR} apply. But in view of  Lemma \ref{Lem:Properties-En}.2, the Eulerian AR one constructs using $\gamma_{\underline{\varepsilon}}$ actually coincides with the one obtained 
by means of $\gamma$, which is $\boldsymbol{\mathcal E}_{n}^{>0}$.  This shows that this AR is well-defined on  $\mathcal M_{0,n+3}^{\, or,\, >0}$. 

But a totally similar construction can be made on the other component $\mathcal M_{0,n+3}^{\, or,\, <0}$ of  the inverse image of 
$\mathcal M_{0,n+3}^{\, >0}$ by $\nu$.
 One obtains a well-defined Eulerian AR on this component, that we will denote accordingly by 
$\boldsymbol{\mathcal E}_{n}^{<0}$.  In order to construct it in coordinates, that is on $U$, one just has to consider the map $\check{\gamma}=D\circ \gamma : U\rightarrow 
G_2^{or}(\mathbf R^{n+3})^{<0}$. Since the pull-back under $D$ of Euler's 2-form $E_n$ is $-E_n$, we obtain that $D^*(E_n^n)=(-1)^n E_n^n$, which gives us the following 
\begin{lem} 
\label{Lem:+-}
Let $\nu^{>0}$ and $\nu^{<0}$ be the restrictions of $\nu$ to $\mathcal M_{0,n+3}^{\, or,\, >0}$ and $\mathcal M_{0,n+3}^{\, or,\, <0}$ respectively.

1.  One has $\big( \nu^{>0}\big)_*\big( \boldsymbol{\mathcal E}_{n}^{>0}\big)=(-1)^n \big(\nu^{<0}\big)_*\big( \boldsymbol{\mathcal E}_{n}^{<0}\big)$. 

2. Consequently, Euler's abelian relation $\boldsymbol{\mathcal E}_{n}^{>0}$   constructed on 
$\mathcal M_{0,n+3}^{>0}$ in \S\ref{SSS:positive-Euler-AR} is 
\begin{itemize}
\item[$a.$] 
\vspace{-0.2cm}
well-defined when $n$ is even;\sk
\item[$b.$]  only defined up to sign if $n$ is odd. 
\end{itemize}
\end{lem}

Albeit rather elementary, this result is of crucial importance regarding the birational 
action of $\mathfrak S_{n+3}$ on the abelian relations of 
$\boldsymbol{\mathcal W}_{0,n+3}$. \sk 

A first remark, is that when $n$ is odd, the notation $\boldsymbol{\mathcal E}_n^{>0}$ is a bit misleading since it may let the reader think that 
there is a more natural/canonical Eulerian AR among $\boldsymbol{\mathcal E}_n^{>0}$
 and $-\boldsymbol{\mathcal E}_n^{>0}$. As it can be verified easily (using one of the two formulas \eqref{P:C^*Euler-AR} or \eqref{P:R^*Euler-AR} for instance) this is not the case: there is no non arbitrary way to distinguish one from the other.    As we will see below, this has non trivial consequences regarding the invariance properties 
satisfied (or not) by Euler's abelian relation on $\mathcal M_{0,n+3}^{or,>}(\mathbf R)$.  
Note however that 
in spite of this, it does not make  null and void some invariance results stated in
\S\ref{SSS:Dihedral-Invariance-Properties} (namely Proposition \ref{P:C^*Euler-AR} and Proposition \ref{P:Formula-R}) 
 when $n$ is odd: 
 it would just be necessary to make the statements of these propositions more precise in this case, which 
is easy to do and  
is left to the reader.

Secondly, the construction above is in no way specific to the positive component of $\mathcal M_{0,n+3}(\mathbf R)$. It can be generalized to any component of it: for any $\boldsymbol{\sigma}\in \mathfrak K_{n+3}$, we construct an Eulerian AR  on $\mathcal M(\boldsymbol{\sigma})$
(of the web $\boldsymbol{\mathcal W}_{0,n+3}$) which is well-defined, but only up to sign when $n$ is odd. In any case, we denote by $\boldsymbol{\mathcal E}_n^{\boldsymbol{\sigma}}$ this (when $n$ is even) or one of these two (when $n$ is odd) abelian relations, the choice of $\boldsymbol{\mathcal E}_n^{\boldsymbol{\sigma}}$ instead of its opposite 
being 
 arbitrary when $n$
 is odd. The span $\langle  \boldsymbol{\mathcal E}_n^{\boldsymbol{\sigma}}\rangle$  is a well-defined  1-dimensional vector subspace of the space $\boldsymbol{AR}(\boldsymbol{\sigma})$ of abelian relations of the web 
$\boldsymbol{\mathcal W}_{0,n+3}^{\, \boldsymbol{\sigma}}$. 
Since 
the  permutations in $ \mathfrak S_{n+3}$ induce isomorphisms between the
$\boldsymbol{AR}(\boldsymbol{\sigma})$'s and because these 
spaces are pairwise non-identical, it is necessary to have a fixed way to identify  them all to one of them, 
say $\boldsymbol{AR}(\boldsymbol{1})
=\boldsymbol{AR}\big(\boldsymbol{\mathcal W}_{0,n+3}^{\, >0}\big)$, by means of fixed 
isomorphisms $\Xi^{\boldsymbol{\sigma}}: \boldsymbol{AR}(\boldsymbol{\sigma})\rightarrow \boldsymbol{AR}(\boldsymbol{1})$, one  for each  $\boldsymbol{\sigma}$. 

Of course, it is natural to assume that these isomorphisms satisfy the following properties: first, one should have $\Xi^{\boldsymbol{1}}={\rm Id}$. Secondly,
since the space $\boldsymbol{AR}_{R}\big(\boldsymbol{\mathcal W}_{0,n+3}\big)$
of rational ARs of $\boldsymbol{\mathcal W}_{0,n+3}$ (on the whole moduli space 
${\mathcal M}_{0,n+3}(\mathbf R)$ or even on its complexification ${\mathcal M}_{0,n+3}$) naturally embeds into $\boldsymbol{AR}(\boldsymbol{\sigma})$ for any 
$\boldsymbol{\sigma}$, a natural property that the isomorphisms  $\Xi^{\boldsymbol{\sigma}}$ must all satisfy is that they coincide with the identity when taking their restriction along $\boldsymbol{AR}_{R}\big(\boldsymbol{\mathcal W}_{0,n+3}\big)$. 
Finally, if one expects the Euler ARs to give rise to a 1-dimensional representation,
then each $\Xi^{\boldsymbol{\sigma}}(  \boldsymbol{\mathcal E}_n^{\boldsymbol{\sigma}})$ must be a non trivial multiple of $ \boldsymbol{\mathcal E}_n^{>0}$ for any $\boldsymbol{\sigma}$ (which makes sense even in the case when $n$ is odd).

To summarize, in order that the action of $\mathfrak S_{n+3}$ on the space(s) of abelian relations of the web(s) under scrutiny give rise to a representation with the expected properties, it is necessary that the  $\Xi^{\boldsymbol{\sigma}}$'s satisfy the following properties:
\begin{equation}
\label{Eq:Proprio-Xi-sigma}
\begin{tabular}{l}
$(i).$
\hspace{0.1cm}
 one has \,   $  \Xi^{\boldsymbol{1}}={\rm Id}_{
\boldsymbol{AR}(\boldsymbol{1})}
$;
\sk 
\\ 
${}^{}$ \hspace{-0.3cm} $(ii).$ \, for any 
$ \boldsymbol{\sigma}\in 
  \mathfrak K_{n+3}$, one has \, 
  $\Xi^{\boldsymbol{\sigma}}\lvert_{
 \scalebox{0.6}{$
 \boldsymbol{AR}_{R}\big(\boldsymbol{\mathcal W}_{0,n+3}\big)$}}={\rm Id}
_{
 \scalebox{0.6}{$
 \boldsymbol{AR}_{R}\big(\boldsymbol{\mathcal W}_{0,n+3}\big)$}} $;\sk \\ 
 ${}^{}$ \hspace{-0.3cm}$(iii).$ \, for any 
$ \boldsymbol{\sigma}\in 
  \mathfrak K_{n+3}$, there exists $\lambda^{\boldsymbol{\sigma}}\neq 0$ such that 
  $\Xi^{\boldsymbol{\sigma}}\big( 
 \boldsymbol{\mathcal E}_n^{\boldsymbol{\sigma}} 
  \big)=
 \lambda^{\boldsymbol{\sigma}}\,  \boldsymbol{\mathcal E}_n^{>0}$.
\end{tabular}
\end{equation}

We are going to going to consider the two cases according to the parity of $n$ separately, this because there exist isomorphisms satisfying  \eqref{Eq:Proprio-Xi-sigma} when $n$ is even, whereas this does not seem to be the case when n is odd (which we prove explicitly when n is equal to 3)
\begin{center}
$\star$
\end{center}

\paragraph{\bf Case when $n$ is even.} 
When $n$ is even, essentially all the results claimed in \cite{D} are indeed satisfied. Actually in this case, it is not necessary to work with oriented grassmannians but 
everything can be obtained within the (let say) `classical' Gelfand-MacPherson theory developed in \cite{GelfandMacPherson}. For this reason, 
the treatment below of the case when $n$ is even is rather concise.  
\sk 

We then assume that $n\geq 2$ is even: hence $m=n/2$ is a positive integer.  
In this case, one has $E_n^n=\big(E_n^2\big)^m$ and Euler's AR can be constructed from 
$P_1=E_n^2$. This invariant 4-form represents the first Pontryagin class of the oriented tautological bundle  $\mathcal T^{or}$ hence is the pull-back under the 2-to-1 projection 
$G_2^{or}(\mathbf R^{n+3})\rightarrow G_2(\mathbf R^{n+3})$ of the 
 invariant  4-form on $G_2(\mathbf R^{n+3})$, denoted by the same notation, and which 
represents the first Pontryagin class of $\mathcal T$. Then the $2n$-form $P_1^m$ on 
$G_2(\mathbf R^{n+3})$ is `leading' according to the terminology of Gelfand-MacPherson's paper ({\it cf.}\,\S1.3.4 therein).   Then from 
\cite[\S1.3.2]{GelfandMacPherson}, one constructs the abelian relation 
$ \boldsymbol{\mathcal E}_n^{\boldsymbol{\sigma}} $ for any $\boldsymbol{\sigma}$, 
which is well-defined on any component $\mathcal M(\boldsymbol{\sigma})$ 
 of $\mathcal M_{0,n+3}(\mathbf R^{n+3})$.  This explains  more conceptually the point {\it 2.a.}\,of Lemma \ref{Lem:+-}.

But the results of \cite{GelfandMacPherson} can also be used to investigate the  invariant properties of the Euler ARs.  Contrarily to when dealing with oriented grassmammians, in which case the choices of lifts of permutations can be a bit subtle, here we will only consider the naive lifts $\tilde \sigma \in {\rm Aut}\big( G_2(\mathbf R^{n+3})\big)$ of permutations $\sigma\in \mathfrak S_{n+3}\simeq {\rm Aut}\big( \mathcal M_{0,n+3}\big)$, whose  definition we recall:  $\tilde \sigma$ is the automorphism of the grassmannian induced by the linear map $\mathbf R^{n+3}\rightarrow 
\mathbf R^{n+3}$, $ (x_i)_{i=1}^{n+3}\mapsto 
(x_{\sigma(i)})_{i=1}^{n+3}$, also denoted by $\tilde \sigma$.  It is well-known that 
the latter is an orthogonal transformation. On the other hand, it is proved in   \cite[Corollary 3.25]{GelfandMacPherson}  that the 4-form $P_1$  on $G_2(\mathbf R^{n+3})$ not only is left invariant under the action of ${\rm SO}_{n+3}(\mathbf R)$,  but also under that of the full orthogonal group 
${\rm O}_{n+3}(\mathbf R)$. Thus one has $\tilde \sigma_*( P_1)=P_1$ hence 
$\tilde \sigma_*( P_1^m)=P_1^m$  for any permutation $\sigma$.   
On the other hand, in the case when $\sigma$ is a transposition, one verifies easily that 
 $\tilde \sigma$ reverses the natural orientation of the $H_0$-orbits. 
 
From all the preceding considerations, one deduces the 
\begin{prop} 
\label{P=sign(sigma)}
For any  ${\sigma}\in \mathfrak S_{n+3}$: 
\begin{enumerate}  
 \item[1.]
 \vspace{-0.15cm}
 one has  $\tilde \sigma^*\big(   \boldsymbol{\mathcal E}_n^{\boldsymbol{\sigma}}
 \big)  = {\bf sgn}(\sigma) \, \boldsymbol{\mathcal E}_n^{>0}
$;
\mk 
 \item[2.] Euler's abelian relation  $ \boldsymbol{\mathcal E}_n^{\boldsymbol{\sigma}}$ does not belong to $\boldsymbol{AR}_C({\boldsymbol{\mathcal W}}_{0,n+3})$ hence 
 \begin{equation}
 \label{Eq:AR_C+E-sigma}
 \boldsymbol{AR}({\boldsymbol{\sigma}})=\boldsymbol{AR}_C\big({\boldsymbol{\mathcal W}}_{0,n+3}\big)\oplus 
 \big\langle   \boldsymbol{\mathcal E}_n^{\boldsymbol{\sigma}}\big\rangle \, .
 \end{equation}
\end{enumerate}
 \end{prop}
\begin{proof} The first point  follows from the arguments given just before the proposition so let us deal with the second. 
 If  $\boldsymbol{\mathcal E}_n^{\boldsymbol{\sigma}}$ belongs to
$\boldsymbol{AR}_C\big({\mathcal W}_{0,n+3}\big)$ for one 
permutation ${\sigma}$, this holds true for all according to 1. 
We thus have a line $\langle  \boldsymbol{\mathcal E}_n^{>0} \rangle
= \langle  \boldsymbol{\mathcal E}_n^{>0} \rangle$ in the space of combinatorial ARs, which is invariant by $\mathfrak S_{n+3}$. But this is 
impossible since, as proved by Damiano (see Corollary \ref{P:AR-C-as-a-S(n+3)-module}.1 here),  
$\boldsymbol{AR}_C\big({\boldsymbol{\mathcal W}}_{0,n+3}\big)$ is an irreducible 
$\mathfrak S_{n+3}$-representation. 
 Thus for any $\boldsymbol{\sigma}$, one has 
 $\boldsymbol{\mathcal E}_n^{\boldsymbol{\sigma}}\not \in \boldsymbol{AR}_C\big({\boldsymbol{\mathcal W}}_{0,n+3}\big)$ and the decomposition in direct sum \eqref{Eq:AR_C+E-sigma} follows from dimensional considerations. 
\end{proof}

The first point of this proposition answers 
question \eqref{Eq:Remain-2-be-answered}.$ii.$
 when $n$ is even.\mk 
\begin{rem}
\label{Rem:n-odd-vs-n-even}
{\rm 1.} When $n$ is odd, one cannot do as above and get an $(n-1)$-abelian relation for 
$\boldsymbol{\mathcal W}_{0,n+3}$ using  the standard Gelfand-MacPherson theory 
applied to an invariant $2n$-form
 on the usual ({\it i.e.}\,non oriented) grassmannian $G_2(\mathbf R^{n+3})$. 
As explained above,  when $n$ is even one can construct Euler's AR by considering 
 the $(n/2)$-wedge power of an invariant representative of the first Pontryagin class. This is not possible when $n$ is odd: in this case, $2n$ is not a multiple of 4 hence one really has to work with an invariant 2-form at some point. And if $n=2m+1$, the thing is that there is no such invariant 2-form on 
 $G_2(\mathbf R^{n+3})$ as it follows from the well known description 
$$
\boldsymbol{H}^*\Big( G_2\big(\mathbf R^{n+3}\big) , \mathbf R\Big) \simeq 
 \frac{\mathbf R\Big[\,  p_1,
\overline{p}_1, 
\ldots , 
\overline{p}_{m+1}
\, 
\Big]}{\big( \, p\, \overline{p}-1\, \big)}
$$
of its 
cohomology ring over $\mathbf R$, 
where $p_1$ stands for the first Pontryagin class of the tautological bundle $\mathcal T$ on 
$G_2(\mathbf R^{n+3})$, the $\overline{p}_k$'s 
  are the Pontryagin classes of the 
rank $n+1$ cotautological bundle $\overline{\mathcal T}$, and where $p$ and $\overline{p}$ denote the corresponding Pontryagin characters.\footnote{We recall that $\mathcal T$ and $\overline{\mathcal T}$ fit into the following short exact sequence 
$0\rightarrow \mathcal T \longrightarrow \mathbf R^{n+3}\longrightarrow  \overline{\mathcal T}\rightarrow 0$ of fiber bundles over 
$G_2(\mathbf R^{n+3})$,  where $\mathbf R^{n+3}$ stands for the trivial bundle of rank $n+3$. Also, the associated 
Pontryagin characters are the elements 
$p=1+p_1$ and $\overline{p}=1+
\overline{p}_1+ \cdots + \overline{p}_{\lfloor (n+3)/2\rfloor-1}$ 
of the cohomology ring with real coefficients of $G_2\big(\mathbf R^{n+3}\big)$.}
\sk 

At this point, we believe it is interesting to recall how the cohomology ring of the oriented grassmannian $G_2^{or}\big(\mathbf R^{n+3}\big)$ is related to that of the standard one: 
via the injection of the former cohomology ring into the latter induced by the 2-1 covering $G_2^{or}\big(\mathbf R^{n+3}\big)\rightarrow G_2(\mathbf R^{n+3})$, one has
\begin{equation}
\label{Eq:Cohomology-ring-IR-even}
\boldsymbol{H}^*\Big( G_2^{or}\big(\mathbf R^{n+3}\big) , \mathbf R\Big) \simeq  
\frac{\boldsymbol{H}^*\big( G_2\big(\mathbf R^{n+3}\big) , \mathbf R\big)\big[ \, e
\, \big]}{\big(\,  p_1-e^2\, 
\big)}
=  \frac{\mathbf R\Big[ e, p_1  , 
\overline{p}_1, 
\ldots , \overline{p}_{n/2}
\Big]}{\big(\, p\, \overline{p}-1 \, ,\,    p_1-e^2\,  \big)} 
\end{equation}
when $n$ is even, and 
\begin{equation}
\label{Eq:Cohomology-ring-IR-odd}
 \quad \boldsymbol{H}^*\Big( G_2^{or}\big(\mathbf R^{n+3}\big) , \mathbf R\Big) \simeq 
\frac{\boldsymbol{H}^*\big( G_2\big(\mathbf R^{n+3}\big) , \mathbf R\big)\big[ \, e, \overline{e} 
\, \big]}{\big(\,  e\,\overline{e} 
\, , \, p_1-e^2\, , \, \overline{p}_{m+1}-\overline{e}^2\, 
\big)}
=  \frac{\mathbf R\Big[ e, p_1, \overline{e} \, , \, 
\overline{p}_1, 
\ldots , \overline{p}_{m+1}
\Big]}{\big(\, p\, \overline{p}-1 \, ,\,   e\,\overline{e}   \, ,\,    p_1-e^2\, ,\,      \overline{p}_{m+1}-
\overline{e}^2\,  \big)} 
\end{equation} 
when $n$ is odd  and is written $n=2m+1$ for a positive integer $m$, 
where in both cases $e$ (resp.\,$\overline{e}$) stands here for the Euler class of the oriented tautological bundle $\mathcal T^{or}$ (resp.\,the cotautological bundle 
$\overline{\mathcal T}^{or}$)   on $G_2^{or}\big(\mathbf R^{n+3}\big)$, with $p_1$ and the $\overline{p}_k$'s being the associated Pontryagin classes. \footnote{The descriptions of the cohomology rings above are very classical, 
see \cite{Sadykov} or \cite{He}  and the references therein.}

When working with oriented grassmannians, one can consider (an invariant representative of) the Euler class independently of the parity of $n$, which has the convenience to allow to build Euler's abelian relation $\boldsymbol{\mathcal E}_n^{>0}$ in an uniform manner.
\sk 

\noindent {\rm 2.}   From \eqref{Eq:Cohomology-ring-IR-even}
and \eqref{Eq:Cohomology-ring-IR-odd}, it follows that as an $\mathbf R$-algebra, 
the cohomology ring of $G_2^{or}\big(\mathbf R^{n+3}\big)$ is generated by $e$ when $n$ is odd, and by  $e$ and $\overline{e}$ but not by $e$ alone when $n$ is odd ({\it cf.}\,the remark 
in \cite[p.\,267]{GMZ}{\rm )}. This shows that the claim that `the real cohomology ring of 
$G_2^{or}\big(\mathbf R^{n+3}\big)$ is generated by the Euler class' p.\,1341 of \cite{D} is incorrect when $n$ is odd.  However, this is irrelevant in what concerns the validity of the proof of \cite[Theorem 5.1]{D} given  in the sixth section of \cite{D}:  since its degree is $n+1$,  the class $\overline{e}$ 
does not belong to the set of characteristic classes considered in 
 the statement of this theorem and because its square is non zero and has degree $2n+2=
 \dim  \big(G_2^{or}(\mathbf R^{n+3})\big)$, it is a non trivial multiple of the $(n+1)$-th wedge power  of $e$. This remark, which is missing in \cite[\S6]{D} from our point of view, shows that no case has been forgotten and that 
  the proof given there is complete. 
\end{rem}


%
%
%

For each class $\boldsymbol{\sigma} \in \mathfrak K_{n+3}=\mathfrak S_{n+3}/D_{0,n+3}$, thanks to the decomposition in direct sum  \eqref{Eq:AR_C+E-sigma}, one constructs an isomorphism $\Xi^{\boldsymbol{\sigma}}:  
 \boldsymbol{AR}({\boldsymbol{\sigma}})\rightarrow  \boldsymbol{AR}(
 \mathcal W_{0,n+3}^{\, >0}
 )$  uniquely characterized by the properties  \eqref{Eq:Proprio-Xi-sigma}  with $\lambda^
{\boldsymbol{\sigma}}=1$ in $(iii)$.    We then consider the following map: 
\begin{align}
\label{Eq:Representation}
 \mathfrak S_{n+3} & \longrightarrow {\rm Aut}\Big(  \boldsymbol{AR}\big(
{\boldsymbol{\mathcal W}}_{0,n+3}^{\, >0}
 \big)\Big) \\ 
 \sigma & \longmapsto   \quad \Xi^{\boldsymbol{\sigma}}\circ  \sigma^* : \hspace{0.25cm}  
 \Psi\longmapsto \Xi^{\boldsymbol{\sigma}}\Big( \sigma^*\big( \Psi\big) 
\Big) \,. \nonumber 
\end{align}


\begin{cor}
\label{Cor:Sn+3-module-n-even}
The above map makes of 
 \begin{equation}
 \label{Eq:AR_C+E-Id}
 \boldsymbol{AR}\left(
{ \boldsymbol {\mathcal W}}_{0,n+3}^{\, >0}
 \right)=\boldsymbol{AR}_C\left(\boldsymbol{\mathcal W}_{0,n+3}\right)\oplus 
 \big\langle   \boldsymbol{\mathcal E}_n^{>0}\big\rangle \, .
 \end{equation}
 a $\mathfrak S_{n+3}$-module with
 $\boldsymbol{AR}_C\big({\boldsymbol{\mathcal W}}_{0,n+3}\big)$ and $\big\langle   \boldsymbol{\mathcal E}_n^{>0}\big\rangle$ as 
 irreducible components,  with Young symbols $[221^{n-1}]$ and
  $[1^{n+3}]$ respectively.   In particular,  $\mathfrak S_{n+3}$ acts 
 as the signature representation 
  on 
  the 1-dimensional component $\big\langle   \boldsymbol{\mathcal E}_n^{>0}\big\rangle$.
\end{cor}

To conclude this discussion of the case when $n$ is even, 
we recall that when $n=2$,  
considering the local system (with non trivial monodromy) of complex abelian relations 
of $\boldsymbol{\mathcal W}_{0,n+3}$ on the complex moduli space $\mathcal M_{0,n+3}(\mathbf C)$, we have constructed  in 
Example \ref{Ex:Gr-n=2} (via a method different from that discussed just above) a  
complex 
$\mathfrak S_{5}$-representation on the 
complexification of $\boldsymbol{AR}\big({\boldsymbol{\mathcal W}}_{0,5}^{>0}\big)$. The latter has two irreducible components: one is the complexification of the natural $\mathfrak S_{5}$-representation on $\boldsymbol{AR}_C\big({\boldsymbol{\mathcal W}}_{0,5}\big)$ and the other, 
denoted by $\langle \boldsymbol{{\mathcal A}b} \rangle$ in Example \ref{Ex:Gr-n=2}, is the trivial representation.

We believe that this generalizes to any $n$ even:  considering the general form  for the components of $\boldsymbol{\mathcal E}_n^{>0}$ we conjecture and have proved this conjecture to be correct for $n$ even and less than 12 ({\it cf.}\,\eqref{Eq:Euler-varepsilon} and Proposition \ref{P:properties--en} further),  that the complexification 
$\boldsymbol{\mathcal E}_n^{\mathbf C}$
of $\boldsymbol{\mathcal E}_n^{>0}$ extends to a global multivalued AR on the whole complexified moduli space 
${\mathcal M}_{0,n+3}(\mathbf C)$, with logarithmic hence unipotent monodromy. 
Using the same elementary arguments as in Example \ref{Ex:Gr-n=2}, this would give a 
1-step filtration on $\boldsymbol{AR}\big({\boldsymbol{\mathcal W}}_{0,n+3}^{\, \mathbf C}\big)$ inducing a decomposition into irreducible $\mathfrak S_{n+3}$-modules 
\begin{equation}
\label{Eq:Gr-AR(W-0-n+3)}
{\bf Gr}
\boldsymbol{{AR}}^\bullet\left(\boldsymbol{\mathcal W}_{0,n+3}^{\, \mathbf C}\right)=
\boldsymbol{{AR}}_{C}\left(\boldsymbol{\mathcal W}_{0,n+3}^{\, \mathbf C}\right)\oplus \Big\langle \boldsymbol{\mathcal E}_n^{\, \mathbf C} \Big\rangle 
\end{equation}
with  the 1-dimensional 
component  $\big\langle \boldsymbol{\mathcal E}_n^{\, \mathbf C}  \big\rangle $ being the trivial complex $\mathfrak S_{n+3}$-representation. 
\mk 

This might seem a bit surprising at first sight, a natural (but naive) guess would be that 
\eqref{Eq:Gr-AR(W-0-n+3)} coincides with the complexification of 
the real $\mathfrak S_{n+3}$-representation \eqref{Eq:AR_C+E-Id}. But this cannot be the case since the 1-dimensional component of the former would be the trivial representation, whereas that of the latter is the signature representation. 
Here is a (not fully rigorous) explanation of this apparent inconsistency:  
focusing only on the components of dimension 1, one can say that 
\eqref{Eq:AR_C+E-Id} is obtained by comparing the Euler's abelian relations 
$\boldsymbol{\mathcal E}_n^{>0}$ and $\boldsymbol{\mathcal E}_n^{\boldsymbol{\sigma}}$
at two fixed  points $u^{>0}$ and $u^\sigma$  of two connected components $\mathcal M_{0,n+3}^{>0}$ and $\mathcal M(\boldsymbol{\sigma})$ for any  $\sigma\in \mathfrak S_{0,n+3}$.   

It is sufficient to only consider the case when $\sigma$ is a transposition in order to 
understand the complex action on $\big\langle \boldsymbol{\mathcal E}_n^{\, \mathbf C} \big\rangle$. One identifies $\mathcal M_{0,n+3}$ with $\mathbf R^n\setminus A_n$ and accordingly $\mathcal M_{0,n+3}(\mathbf C)$ with $\mathbf C^n\setminus A_n$ and 
for $i,j$ such that  $1\leq i<j\leq n+3$, one denotes by $H_{i,j}$  the 
hyperplane of the braid arrangement corresponding 
to the limits obtained  by making  coinciding the $i$-th and $j$-th components of configurations of $n+3$  points on the Riemann sphere. 
Then the action of $\sigma=(i,j)$ on $\boldsymbol{\mathcal E}_n^{\mathbf C}$ can be understood as follows:  one can find a smooth loop $\gamma: [0,1]\rightarrow \mathbf C^n\setminus A_n\simeq \mathcal M_{0,n+3}(\mathbf C)$,  based at $\gamma(0)=\gamma(1)=u^{>0}$ and such that   $\gamma(1/2)=u^\sigma$, which is of trivial index with respect to all the hyperplanes of the braid arrangement $A_n$ at the exception of the hyperplane $H_{i,j}$, with respect to which it is of index 1. For $\epsilon=0,1$, one sets ${\gamma}_\epsilon:[0,1]\rightarrow \mathbf C^n\setminus A_n$, $t\mapsto \gamma(t/2+\epsilon/2)$. For any path $\overline \gamma : [0,1]\rightarrow \mathbf C^n\setminus A_n$
and any germ of holomorphic object $F$ at $\overline \gamma(0)$
one denotes by $\overline \gamma\cdot F$ the germ at $\overline \gamma(1)$ 
obtained after analytic continuation of $F$ along $\overline \gamma$ (in case its exists of course).  

Since  the loop $\gamma$ is the concatenation $\gamma_1\,\gamma_0$, 
 one has
$$\sigma\, \boldsymbol{\mathcal E}_n^{\, \mathbf C} =(i,j)\,\boldsymbol{\mathcal E}_n^{\, \mathbf C}=\gamma\cdot \boldsymbol{\mathcal E}_n^{\, \mathbf C} =
\gamma_1\cdot \Big( \gamma_0\cdot \boldsymbol{\mathcal E}_n^{\, \mathbf C}\Big)$$
as germs of holomorphic ARs at $u^{>0}$.  
Up to sign, the 
complex analytic 
germ of abelian relation $ \gamma_0\cdot \boldsymbol{\mathcal E}_n^{\, \mathbf C}$ coincides with 
(the germ at  $u^\sigma$ of) the complexification of 
$\boldsymbol{\mathcal E}_n^{\boldsymbol{\sigma}}$, denoted by $\boldsymbol{\mathcal E}_n^{\boldsymbol{\sigma},\mathbf C}$. Since the orientation of $H_0$-orbits does not change along 
continuous deformations, one has $ \gamma_0\cdot \boldsymbol{\mathcal E}_n^{\, \mathbf C}=-\boldsymbol{\mathcal E}_n^{\boldsymbol{\sigma},\mathbf C}$. Applying the same reasoning to $\gamma_1$, 
one gets that the following fact holds true: 
\begin{quote} 
{\it 
As $\mathfrak S_{n+3}$-representations, the  irreducible components of dimension 1 in 
\eqref{Eq:AR_C+E-Id} and in 
\eqref{Eq:Gr-AR(W-0-n+3)} are related in the following manner: the latter 
$\big\langle \boldsymbol{\mathcal E}_n^{\, \mathbf C} \big\rangle $
identifies with the 
 square of the complexification $\big\langle \boldsymbol{\mathcal E}_n^{>0} \big\rangle^{\mathbf C}$ of the former. Mathematically, one has
$$
\Big\langle \boldsymbol{\mathcal E}_n^{\, \mathbf C} \Big\rangle 
\simeq \left( \big\langle \boldsymbol{\mathcal E}_n^{>0} \big\rangle^{\mathbf C}\right)^{\otimes 2} \,.
$$
}
\end{quote}

Since it has not yet been proved that  Euler' AR $\boldsymbol{\mathcal E}_n^{\mathbf C}$ has unipotent monodromy for all $n\geq 2$ even, the existence of a decomposition 
\eqref{Eq:Gr-AR(W-0-n+3)}  
 with a trivial irreducible component of dimension 1 is not known in full generality, so the previous statement is only conjectural for the moment.
\begin{center}
$\star$
\end{center}
\paragraph{\bf Case when $n$ is odd.}  
When $n$ is odd, the situation is quite different from that in the former (even) case. 
Indeed, since the Eulerian abelian relations $\boldsymbol{\mathcal E}_n^{\boldsymbol{\sigma}}$'s for
 $\boldsymbol{\sigma}\in \mathfrak K_{n+3}$ are  only well-defined up to sign in this case, 
 looking for isomorphisms $\Xi^{\boldsymbol{\sigma}}$ satisfying 
 \eqref{Eq:Proprio-Xi-sigma} actually requires to have beforehand suitably chosen 
 one of the two representatives  $\pm \boldsymbol{\mathcal E}_n^{\boldsymbol{\sigma}}$ of Euler's AR on $\mathcal M(\boldsymbol{\sigma})$ for any $\sigma$. 
 
 Since $\boldsymbol{AR}_C\big(
 \boldsymbol{\mathcal W}_{0,n+3}
 \big)$ is of maximal dimension $(n+1)(n+2)/2$, 
 $ \boldsymbol{\mathcal E}_n^{\boldsymbol{\sigma}}$ is combinatorial for any $\sigma\in \mathfrak S_{n+3}$.  Since $\boldsymbol{AR}_C\big(
 \boldsymbol{\mathcal W}_{0,n+3}
 \big)$ is irreducible as a $\mathfrak S_{n+3}$-representation when $n$ is odd (according to Theorem \ref{THM:AR-C(W)-S-n+3-module}), one immediately gets the following negative result: 
 \begin{prop} 
 There is no set $\big\{  \, {}^\star\hspace{-0.05cm} \boldsymbol{\mathcal E}_n^{\boldsymbol{\sigma}}\, \lvert \,  \boldsymbol{\sigma}\in \mathfrak K_{n+3}\, \}$ of representatives of Euler's abelian relations such that  
 for all permutations $\sigma\in \mathfrak S_{n+3}$, 
 $\sigma_*\, \big({}^\star\hspace{-0.05cm} \boldsymbol{\mathcal E}_n^{\boldsymbol{1}}\big)$ 
 is colinear with ${}^\star\hspace{-0.05cm} \boldsymbol{\mathcal E}_n^{\boldsymbol{\sigma}}$ (as elements of 
 $\boldsymbol{AR}(\boldsymbol{\sigma})${\rm)}. 
 \end{prop}

This result implies in particular that there is no way to build a 1-dimensional $\mathfrak S_{n+3}$-representation from the Euler's abelian relations following the lines considered by Damiano in \cite{D}. \sk 

It is worth considering the simplest case when this occurs, namely when $n=3$. 
Let us discuss a concrete example of how the $\boldsymbol{\mathcal E}_3^{\boldsymbol{\sigma}}$'s 
 are related  with the abelian relation $\boldsymbol{\mathcal E}_3^{>0}$
considered in \S\ref{SSS:tilde-e2}. 
Let $S$ be the
involutive automorphism of $\boldsymbol{\mathcal W}_{0,6}$ 
 given  by $S(u_1,u_2,u_3)=(u_2,u_1,u_3)$ in the coordinates $u_1,u_2,u_3$.  
This map can be seen to be the realization as a birational automorphism of 
the transposition $s=(45)\in \mathfrak S_6$.  It induces an isomorphism 
between $\mathcal M_{0,6}^{\, >0}$ and $\mathcal M(\boldsymbol{s})\simeq \{ (u_i)_{i=1}^3 \, \big\lvert \, 1<u_2<u_1<u_3\,\}$. 
Using the explicit expression \eqref{Eq:E3-explicit}  for $\boldsymbol{\mathcal E}_3^{>0}$, one gets that its pull-back under $S$ is an AR on $\mathcal M(\boldsymbol{s})$  whose sixth component is 
$$S^*\big( \mathcal E_{3,6}\big)=S^*\left( \frac{du_1\wedge du_2}{u_1u_2(u_2-1)}\right)=- \frac{du_1\wedge du_2}{u_1u_2(u_1-1)}\, .$$ On the other hand, by direct computations, one can compute the first terms of the Taylor expansion of the Eulerian AR $\boldsymbol{\mathcal E}_3^{\boldsymbol{s}}$ and verify that one has $ S^*(\boldsymbol{\mathcal E}_3^{>0})= -\,  
\boldsymbol{\mathcal E}_3^{\boldsymbol{s}}$. 
However,  as rational abelian relations of $\boldsymbol{\mathcal W}_{0,6}$, 
 $\boldsymbol{\mathcal E}_3^{>0}$ and 
  $\boldsymbol{\mathcal E}_3^{\boldsymbol{s}}$ 
    are easily seen to be linearly independent. 
 \sk 

 By direct computation, one verifies that $\boldsymbol{AR}_C(\boldsymbol{\mathcal W}_{0,6})=\boldsymbol{AR}(\boldsymbol{\mathcal W}_{0,6})$  is spanned by the $\mathfrak S_6$-orbit of $\boldsymbol{\mathcal E}_3^{>0}$ which is  consistent with the fact that  
 the space of ARs of  $\boldsymbol{\mathcal W}_{0,6}$ 
 is an irreducible $\mathfrak S_6$-representation. 
 
  \label{SSSS:ici}

\subsubsection{\bf An explicit formula for Euler's abelian relation when $n$ is odd.}
\label{SSS:En-n-odd}
By means of direct computations, we have been able,   for $n$ odd and less than or equal to 9, to give a closed formula for Euler's abelian relation. 
Indeed, using Proposition \ref{P:La-Base} (together with Remark \ref{Eq:-OK-with-other-C}), it is just a matter of linear algebra  
to determine the functions $F$ depending on $u'=U_1(u)=(u_1,\ldots,u_{n-1})$ 
satisfying the following properties on the whole domain $U$:
\begin{equation}
\label{Eq:Alpha-Beta}
\begin{tabular}{ll}
$\bullet$  the function $F$ satisfies the same functional identities as  \eqref{Formula:C*E} and \eqref{Formula:R*E-even-odd};\\
$\bullet$ the identity $ 0=
\sum_{i=1}^{n+3} (-1)^i \left(C^{\circ (i-1)} \right)^* \Big( 
F(u')\, du_1\wedge \cdots \wedge du_{n-1}
\Big)$ 
 holds true identically.
\end{tabular}
\end{equation}
We obtained that the set of such functions is a vector space 
of dimension one, spanned by an explicit rational function. In particular, it gives us an explicit closed rational formula for the function $e_{n-1}(u')$ of Proposition 
\ref{P:Prop-em-integral-formula}.  
\begin{prop} 
\label{P:5.20}
Let $n$ be an odd integer. 

\begin{enumerate}
\item[{\rm 1.}] Euler's abelian relation $\boldsymbol{\mathcal E}_n^{>0}$ is combinatorial (hence rational): one has 
$$
\boldsymbol{\mathcal E}_n^{>0}
\in \boldsymbol{AR}_C\Big( 
\boldsymbol{\mathcal W}_{0,n+3}
\Big)\, .$$
\item[{\rm 2.}]  
If $n$ is moreover assumed to be less than or equal to 11 then as a function of $u'$, the function corresponding to the $U_1$-component 
of Euler's AR $\boldsymbol{\mathcal E}_n^{>0}$ is rational and given by 
\begin{equation}
\label{Eq:tilde-e}
\tilde e_{n-1}=F_0+\sum_{2\leq i<j\leq n}(-1)^{i+j} F_{ij}+
\sum_{i=2}^n (-1)^i  F_{i,n+2}\, .
\end{equation}
Consequently, up to a non zero multiplicative constant, the function $e_{n-1}$ defined 
via 
 the integral formula \eqref{Eq:e-m(u)}  coincides with the above rational function $\tilde e_{n-1}$. 
\end{enumerate}
\end{prop}

We conjecture that the previous statement actually is satisfied for any odd integer $n\geq 3$. 

\begin{rem} 
1.  Since it is irrelevant regarding its content,  the preceding proposition 
 has been stated without mentioning the subtlety that  Euler's abelian relation  $\boldsymbol{\mathcal E}_n^{>0}$ on $\mathcal M_{0,n+3}^{\, >0}$ actually is only defined up to sign  ({\it cf.}\,Lemma \ref{Lem:+-}). 

\noindent 
2. It turns out that both conditions  in 
\eqref{Eq:Alpha-Beta}
 have to be assumed in order to  get a space of solutions of dimension 1. Indeed, for $n$ odd less than $11$, we have  verified that the elements of the space of functions spanned by $\boldsymbol{\mathfrak B}_n(u')$ (see \eqref{Eq:mathfrak-Bn}) and satisfying 
only one of these two conditions has dimension $(n+1)/2$. We have also verified (by means of direct computations as well) that for an element $F\in \langle \boldsymbol{\mathfrak B}_n(u')\rangle$,  satisfying the functional identity 
 \eqref{Formula:R*E-even-odd} is automatic if one assumes that $F$  satisfies  \eqref{Formula:C*E}.   We conjecture that these facts hold true in full generality. 
\end{rem}

\subsubsection{\bf An explicit formula for Euler's abelian relation when $n$ is even.}
\label{SS:En-n-even}
If the results of \S\ref{SSS:Sn+3-Invariance-Properties} (especially Corollary 
\ref{Cor:Sn+3-module-n-even}) give a clear picture of $\boldsymbol{AR}\big( \boldsymbol{\mathcal W}_{0,n+3}^{\, >0}\big)$ as a 
vector space (and even as a  $\mathfrak S_{n+3}$-module) when $n$ is even, they have the defect of not being explicit. However, as in the case when $n$ is odd ({\it cf.}\,the proposition just above), it is possible to make (the components of) Euler's abelian relation explicit when $n$ is odd as well, at least for small values of $n$ (we give a conjectural closed formula in the general case).   \sk

We assume that $n$ is even in what follows. Since $\boldsymbol{\mathcal E}_n^{>0}$ does not belong to $\boldsymbol{AR}_C\big( \boldsymbol{\mathcal W}_{0,n+3}\big)$, one cannot use a similar approach to that used just above to get Proposition \ref{P:5.20}.  The path we follow below is the following.  We first deal with the $n=4$ case: 
applying a variation of Abel's method for solving functional equations, we determine in explicit form (one of) the components of $\boldsymbol{\mathcal E}_4^{>0}$. Then we deduce from it a conjectural closed formula  for the function $e_{n-1}$ in \eqref{Eq:e-m(u)} for $n\geq 2$ arbitrary. Then, by means of explicit computations,  we verify that this formula is indeed correct for all even integers less than or equal to $12$.

\paragraph{\bf The case $n=4$ via Abel's method.}
\label{SS:E4-via-Abel's-method}
It is now well-known that Abel's method for solving functional equations,  by means of successive differentiations and eliminations in order to get a (partial) differential equation eventually, gives rise to a powerful and efficient method to determine the abelian relations of 1-codimensional webs.\footnote{For instance, see \cite[\S1.5]{ClusterWebs} and the references there.}
Facing the problem of making the components of $\boldsymbol{\mathcal E}_4^{>0}$ explicit, we have realised that it could be handled by adapting Abel's method again. Since it could be useful in dealing with other (curvilinear) webs, we say below a few words on how Abel's method works  in  the case under scrutiny.\sk

%
%
%
%
%
%
%
We work on $\mathbf C^4$ and $x_1,\ldots,x_4$ stand for the standard variables on it. 
We denote by $\partial_{x_1},\ldots,\partial_{x_4}$ the associated constant vector fields. 
The first integrals we are going to work with are the $U_1,\ldots,U_7$ given in \eqref{Eq:U-i},
and $\Omega_1,\ldots,\Omega_7$ stand for the corresponding normals (see \eqref{Eq:Omega-i}).
 Our goal is to apply an appropriate adaptation of Abel's method to determine, for all $i$,  the general form of the 
 functions ${\rm e}_i$ of three variables  appearing in the following identity: 
 \begin{equation}
\label{Eq:RA-n=4}
\sum_{i=1}^7 {\rm e}_i(U_i)\,\Omega_{i}=0\, .
\end{equation}
Thanks to the action of the cyclic map $C$ defined in \eqref{Eq:R-C-1}, just dealing with 
${\rm e}_1$ is enough.

For $s=1,2,3$ and any $i$, we denote by $\delta_s({\rm e}_i)$ the partial derivative of 
${\rm e}_i$ with respect to the $s$-th variable and we also consider the following rational vector fields on $\mathbf C^4$: 
$$
X_i=\partial_{x_i}\hspace{0.3cm} (i=1,\ldots,4)\, , \quad 
X_5=\sum_{i=1}^4  (x_i-1)\, {\partial_{x_i}}
\, , \quad
X_6=  \sum_{i=1}^4  x_i\,  {\partial_{x_i}}
\quad 
\mbox{ and } \quad 
X_7=\sum_{i=1}^4  x_i(x_i-1)\, {\partial_{x_i}}\, .
$$
Each $X_i$ defines the foliation which admits $U_i$ as a primitive first integral. Consequently one has identically $X_i(U_i)=0$ whereas for $i\neq j$,  $X_i(U_j)\neq 0$ holds true (generically) 
thanks to the transversality of the leaves of the foliations of $\boldsymbol{\mathcal W}_{0,7}$.

One has $\Omega_i=\wedge_{k\neq i} dx_k$ for $i=1,\ldots,4$ hence 
$(\Omega_i)_{i=1}^4$ is a (essentially the canonical) basis of $\Omega^4(\mathbf C^4)$. 
Setting $\omega_{i}=\Omega_i(\partial_{x_2},\partial_{x_3},\partial_{x_4}\big)$ for $i=1,\ldots,7$, and 
considering only the $\Omega_1$ components  
of each term of the sum in \eqref{Eq:RA-n=4}, we deduce that the following  functional identity with scalar coefficients holds true
 \begin{equation}
\label{Eq:RA-n=4-2-scal}
 {\rm e}_1(U_1)+
{\rm e}_5(U_5)\,\omega_{5}+
{\rm e}_6(U_6)\,\omega_{6}+
{\rm e}_7(U_7)\,\omega_{7}
  =0\, .
\end{equation}

Dividing 
\eqref{Eq:RA-n=4-2-scal} by $\omega_7^1$ and applying $X_7$ makes  ${\rm e}_7(U_7)$ disappear and gives us the following `partial functional-differential equation': 
\begin{align*}
 {\rm e}_1(U_1)\,  X_7\left( \frac{1}{\omega_7}
\right) + \sum_{s=1}^3 \delta_s({\rm e}_1)(U_1)\, X_7(U_{1,s})
+\sum_{k=5,6}\left(
{\rm e}_k(U_k)\,  X_7\left( \frac{\omega_k}{\omega_7}
\right) + \sum_{s=1}^3 \delta_s({\rm e}_k)(U_k)\, \frac{\omega_k X_7(U_{k,s})
}{\omega_7}
\right)
  =0\, .
\end{align*}

Dividing the LHS
by $X_7(U_{k,3}) \omega_6/\omega_7$ 
and applying $X_6$ makes the term involving  $\delta_3({\rm e}_6)$ disappear. One can then continue this process of division-derivation-elimination following an algorithmic process quite similar to that described in  \cite[\S2.2.2]{PirioSelecta}, to eventually get that 
any function ${\rm e}_{1}={\rm e}_{1}(x_2,x_3,x_4)$ appearing in an identity of the form 
\eqref{Eq:RA-n=4} necessarily also satisfies 
\begin{align*}
 \bigg( x_{34}x_{24} x_{23} \, 
{D_{234}}
-&\, x_{23}(x_2-2\,x_4+x_3) \, {D_{23}}
+  x_{24}(x_4+2x_3+x_2)\, {D_{24}} \\
-& \, x_{34} (x_3-2x_2+x_4)\, {D_{34}}
+2x_{34}\,  {D_{2}}
-2x_{24}\,  {D_{3}}
+2x_{23} \, {D_{4}}\bigg)\cdot {\rm e}_{1}=0
\end{align*}
with $x_{ij}=x_i-x_j$ for any $i,j=1,\ldots, 4$, where we use the following notations: $D_{234}$ stands for 
the order three partial derivative $\partial x_2\partial x_3\partial x_4$, $D_{23}$ for 
$\partial x_2\partial x_3$, etc.

Actually, changing the order in which the terms involving the ${\rm e}_i$'s for $i\neq 1$ (and their partial derivatives)  are eliminated gives rise to other partial differential equations. What we get at the end is an explicit system of partial differential equations, which we denote $(S\hspace{-0.05cm}{{\rm e}_1})$  (it is formed by several PDE's similar to the one above and there is no point in making it explicit here). Contrarily to so many similar computations we did in the past to determine 
the ARs of given planar webs, we have not been able to solve $(S\hspace{-0.05cm}{{\rm e}_1})$ using a computer algebra system without making additional assumptions. Looking for solutions of a certain kind\footnote{We were looking for solutions of $(S\hspace{-0.05cm}{{\rm e}_1})$ which are 
Laurent polynomials in the expressions $x_i$, $x_i-1$, $x_{ij}$ ($i,j=2,3,4$) and their logarithms. It is not relevant to elaborate more on this here.}, we have been able to find one of these in explicit form. \sk

We consider the following function, defined on the set $X=\big\{ (x_i)_{i=2}^4 \in \mathbf R^3\, \big\lvert \, 1<x_2<x_3<x_4\, \big\}$: 
\begin{align}
\label{Eq:E4(x2x3x4)}
\varepsilon_3(x_2,x_3,x_4)=&\, \frac{\mathit{x_2}^{2} \ln | \mathit{x_2} |}{\mathit{x_3} \mathit{x_4} (\mathit{x_2} -1)  (\mathit{x_2} -\mathit{x_3})(\mathit{x_2} -\mathit{x_4})}
-
\frac{\mathit{x_3}^{2} \ln | \mathit{x_3} |}{\mathit{x_2} \mathit{x_4} (\mathit{x_2}  -\mathit{x_3})(\mathit{x_3} -1) (\mathit{x_3} -\mathit{x_4}) }
   \nonumber \\ 
&\,
+\frac{\mathit{x_4}^{2} \ln | \mathit{x_4} |}{\mathit{x_2} \mathit{x_3} (\mathit{x_4} -1) (\mathit{x_2} -\mathit{x_4}) (x_3-\mathit{x_4}) }
-\frac{(\mathit{x_2} -1)^{2} \ln | \mathit{x_2} -1|}{\mathit{x_2} (\mathit{x_2} -\mathit{x_3}) (\mathit{x_2} -\mathit{x_4}) (\mathit{x_3} -1) (\mathit{x_4} -1)    }
  \nonumber \\ 
&\,
-\frac{(\mathit{x_2} -\mathit{x_3})^{2} \ln | \mathit{x_2} -\mathit{x_3} |}{\mathit{x_2}\mathit{x_3}  (\mathit{x_2} -1) (\mathit{x_2} -\mathit{x_4})  (\mathit{x_3} -1) (x_3-\mathit{x_4})}
+\frac{(\mathit{x_2} -\mathit{x_4})^{2} \ln | \mathit{x_2} -\mathit{x_4} |}{\mathit{x_2} \mathit{x_4} (\mathit{x_2} -1) (\mathit{x_4} -1)  (\mathit{x_2} -\mathit{x_3}) (x_3-\mathit{x_4}) }
  \\ 
&\,
+\frac{(\mathit{x_3} -1)^{2} \ln | \mathit{x_3} -1|}{\mathit{x_3}  (\mathit{x_2} -1) (\mathit{x_4} -1)  (\mathit{x_2} -\mathit{x_3}) (x_3-\mathit{x_4} ) }-\frac{(x_3-\mathit{x_4})^{2} \ln | x_3-\mathit{x_4}|}{\mathit{x_3} \mathit{x_4} (\mathit{x_4} -1)  (\mathit{x_3} -1) (\mathit{x_2} -\mathit{x_3}) (\mathit{x_2} -\mathit{x_4}) } \nonumber
 \\ 
&\,
-\frac{(\mathit{x_4} -1)^{2} \ln | \mathit{x_4} -1|}{\mathit{x_4}(\mathit{x_2} -1)  (\mathit{x_3} -1) (\mathit{x_2} -\mathit{x_4}) (x_3-\mathit{x_4} )  }\, . 
\mk 
\nonumber
\end{align}
%
%
%
%


\begin{rem} 
\label{Rem:uuu}
If one restricts to $X\simeq \mathcal M_{0,6}^{>0}(\mathbf R)$, one can drop all the absolute values in  \eqref{Eq:E4(x2x3x4)}. However the presence of absolute values in the above definition is useful since 
it allows to   extend  
$\varepsilon_3$ straightforwardly  
to a function defined on the whole complement $
\mathbf R^3\setminus A_3\simeq \mathcal M_{0,6}(\mathbf R)$ where
the affine space has  $x_2,x_3,x_4$ as coordinates with $A_3$ being the braid arrangement  cut out by $$x_2x_3x_4(x_2-1)(x_3-1)(x_4-1)(x_2-x_3)(x_2-x_4)(x_3-x_4)=0.$$  
\end{rem}

By direct computations, we then get the following result: 
\begin{prop} 
\label{P:properties--e3}
 The function $\varepsilon_3$ defined just above: 
 \begin{enumerate}
 \item[]  \hspace{-0.9cm}{\rm 1.} is a solution of the system of PDEs $(S\hspace{-0.05cm}{{\rm e}_1})$;
 \mk 
  \item[]  \hspace{-0.9cm}{\rm 2.} satisfies the functional relations corresponding to 
 \eqref{Formula:C*E} and \eqref{Formula:R*E-even-odd}:   identically on $X$, one has 
%
%
 \begin{align*}
  \varepsilon_3\left(  \frac{x_4}{x_4-1},  \frac{x_4}{x_4-x_2},  \frac{x_4}{x_4-x_3}
  \right)   = &\, 
  \varepsilon_3\big(x_2,x_3,x_4\big) \, (x_4-1)^2(x_4-x_2)^2(x_4-x_3)^2{x_4}^{-2}
 \vspace{0.2cm} \\
  \mbox{ and }\quad 
  \varepsilon_3\left(  \frac{x_4-x_2}{x_4-x_3},  \frac{x_4-1}{x_4-x_3},  \frac{x_4}{x_4-x_3}
  \right) = &\, \varepsilon_3\big(x_2,x_3,x_4\big) \,  (x_4-x_3)^4 
  \, ; 
  \mk
 \end{align*}
\item[]  \hspace{-0.9cm}{\rm 3.} is such that the following identity holds true on $X$: 
\begin{equation}
\label{Eq:RA-n=4-e3}
0= 
 \sum_{i=1}^7  \Big( C^{\circ (i-1)}\Big)^* \Big( \, 
 \varepsilon_3(x_2,x_3,x_4)\, dx_2\wedge dx_3 \wedge dx_4\, \Big)
\end{equation}
\hspace{-0.9cm}{\textcolor{white}{\rm 3.}}  where $C$ stands for the cyclic birational map defined in \eqref{Eq:R-C-1}.\footnote{In the 
case under consideration, one has 
$C(x)=\Big( 
(x_2-1)/(x_2-x_1), 
(x_3-1)/(x_3-x_1), 
(x_4-1)/(x_4-x_1) , 
{1}/{x_1}\Big) 
$.}
 \end{enumerate}
\end{prop} 

The relation \eqref{Eq:RA-n=4-e3} means that the 7-tuple  
\begin{equation}
\label{Eq:7-tuple}
 \bigg( \, \Big( C^{\circ (i-1)}\Big)^* \Big( \, 
 \varepsilon_3(x_2,x_3,x_4)\, dx_2\wedge dx_3 \wedge dx_4\ \Big) \, \bigg)_{i=1}^{7}
\end{equation} 
  is an AR for $
 \boldsymbol{\mathcal W}_{0,7}^{\, >0}$.  Since its components involve logarithms, it is clear that this AR does not belong to $\boldsymbol{AR}_C( \boldsymbol{\mathcal W}_{0,7})$ and is colinear to $(\boldsymbol{\mathcal E}_4^{>0})$ modulo the combinatorial ARs.   Because we know an explicit basis for $\boldsymbol{AR}_C( \boldsymbol{\mathcal W}_{0,7})$ ({\it cf.}\,Theorem \ref{THM:AR-C(W)-S-n+3-module}), one can verify that the properties 2.\,and 3.\,of the preceding proposition characterize a subspace of dimension 1 of $\boldsymbol{AR}( \boldsymbol{\mathcal W}_{0,7})$.  We thus get the 

\begin{cor} Up to multiplication by a non zero constant: 

 {\rm 1.}  the identity \eqref{Eq:RA-n=4-e3} corresponds to Euler's abelian relation $\boldsymbol{\mathcal E}_4^{>0}$; 
\mk 

{\rm 2.} the function  $\varepsilon_3$ defined above coincides with the 
function $e_3$ of Proposition \ref{P:Prop-em-integral-formula}.
\end{cor}

An interesting feature of $\varepsilon_3$, which follows from the fact that 
the arguments of  all the  logarithms involved in its definition are  
absolute values, is that $\varepsilon_3$ extends to $\mathbf R^3\setminus A_4$ (see Remark \ref{Rem:uuu} above) and one can verify that \eqref{Eq:RA-n=4-e3} actually holds true identically on 
the whole complement $\mathbf R^4\setminus A_4$ where $A_4$ stands for the braid arrangement of type $A_4$.\footnote{$A_4$ is cut out by the polynomial equation 
$\prod_{i=1}^4 x_i(x_i-1)\cdot \prod_{1\leq i<j\leq 4}(x_i-x_j)=0$.} Thus 
up to the standard identification $\mathbf R^4\setminus A_4\simeq \mathcal M_{0,7}(\mathbf R)$, for any $\boldsymbol{\sigma}\in \mathfrak K_{7}$ the restriction of   \eqref{Eq:7-tuple} 
on any connected component $\mathcal M(\boldsymbol{\sigma}) $ of $\mathcal M_{0,7}(\mathbf R)$ is an AR which can be proved to coincide with 
Euler's one $\boldsymbol{\mathcal E}_4^{\boldsymbol{\sigma}}$.
\mk 

But $\varepsilon_3$ has another 
characteristic that is most important 
for our purpose, which relies on the fact that it is a function of several variables. 
Indeed, in clear contrast with the 
single-variable 
formula  \eqref{Eq:R} 
 for Rogers' dilogarithm, it is rather easy to guess from the quite specific type of
 formula \eqref{Eq:E4(x2x3x4)}, what might be its generalization for $n\geq 2$ even arbitrary. 
Indeed, if one sets  
$$
x_{i0}=x_i\, , \qquad x_{i1}=x_i-1 \qquad \mbox{and} \qquad  x_{ij}=x_i-x_j
$$
for all $i,j$ such that $2\leq i<j\leq 4$ 
and if $M_{0,6}(x_2,x_3,x_4)$ denotes the polynomial which cuts out $A_3$ in Remark \ref{Rem:uuu}, then $\varepsilon_3$ can be written as follows
$$
\varepsilon_3(x_2,x_3,x_4)=\frac{\sum_{i=2}^{4} 
p_{i0}
(x_i)^3\, {\rm Log}\,\lvert x_{i0} \lvert
+ \sum_{i=2}^{4} p_{i1}
(x_{i1})^3\, {\rm Log}\, \lvert x_{i1} \lvert
+
\sum_{2\leq i<j\leq 4}^{n}   p_{ij}(x_{ij})^3\,{\rm Log} \, \lvert x_{ij} \lvert}{ M_{0,6}(x_2,x_3,x_4)}
$$
for some polynomials $p_{i0}$, $p_{i1}$ and $p_{ij}$ which it is not difficult to make explicit (see below).  And the nice feature of the preceding expression is that it can be generalized to any even integer $n\geq 2$ quite straightforwardly.
\sk

\begin{rem}
\begin{enumerate}
\item Of course, being closed, the 3-form $E_4$ admits a (local) not unique primitive $\int E_4$  which can be explicitly computed and seen to carry many dilogarithmic terms. Since the expression of $\int E_4$ is quite involved and because we will not use it here, we will not elaborate on this. 
\sk 
\item A direct computation gives ${C}^*\big( \varepsilon_3(U_1)\, \Omega_1\big)= -\varepsilon_3(U_2) \, \Omega_2$ hence the first two components of $(\boldsymbol{\mathcal E}_4)$ are $\varepsilon_3(x_2,x_3,x_4) dx_2\wedge dx_3\wedge dx_4$ and $
-\varepsilon_3(x_1,x_3,x_4) dx_1\wedge dx_3\wedge dx_4$.   
It follows that the involution consisting in exchanging the variables $x_1$ and $x_2$
transforms $(\boldsymbol{\mathcal E}_4)$ into its opposite.  
This is coherent with the first point of Proposition \ref{P=sign(sigma)}.
%
%
%
%
%
%
\end{enumerate}
\end{rem}

\paragraph{\bf An explicit  formula for $e_{n-1}$ for any $n\geq 2$  even.}
\label{Par:En-even}
Here $n\geq 2$ stands for a fixed even integer.
First we set some notation that we will use to define a function $\varepsilon_{n-1}$ through an explicit formula. Then we will discuss how this function is meaningful regarding the Euler's abelian relation of $\boldsymbol{\mathcal W}_{0,n+3}^{\, >0}$. \sk 


\vspace{-0.3cm}
We denote here by $X$ the open domain 
in $\mathbf R^n$ whose elements are 
$n$-tuples 
$x=(x_{i})_{i=1}^n$ such that 
$1<x_1<\ldots <x_n$ (recall that there is a natural identification $X\simeq \mathcal M_{0,n+3}^{\, >0}(\mathbf R)$). 

Here $C$ stands for the cyclic birational map defined in \eqref{Eq:R-C-1} again.  
It is useful to work with the following first integrals for $\boldsymbol{\mathcal W}_{0,n+3}^{\, >0}$, which  behave well regarding the cyclic symmetry induced by $C$:  
one sets $V_1=U_1: \mathbf R^n \rightarrow \mathbf R^{n-1},\,  x\mapsto (x_2,\ldots,x_n)$  and 
$$V_i=(V_i^2,\ldots,V_i^n)=V_1\circ \Big( C^{\circ (i-1)}\Big) 
$$
for $i=1,\ldots,n+3$. The associated normals are denoted by $\Gamma_i$:  one has 
$$
\Gamma_1=\Omega_1=dx_2\wedge \ldots  \wedge dx_n\quad \mbox{ and } \quad 
 \Gamma_i=\Big(C^{\circ (i-1)}\Big)^*\big(\Omega_1\big)=
dV_i^2\wedge \ldots  \wedge dV_i^n 
\quad \big(
 i=2,\ldots,n+3\big)\, .$$

Looking for an explicit function $\varepsilon_{n-1}(V_1)=\varepsilon_{n-1}(x_2,\ldots,x_n)$ such that the following relation 
\begin{equation}
\label{Eq:Euler-En-cyclic-form}
\sum_{i=1}^{n+3} \Big(C^{\circ (i-1)}\Big)^*\Big( \varepsilon_{n-1}(V_1)\, dx_2\wedge \ldots  \wedge dx_n\Big) =
\sum_{i=1}^{n+3} \varepsilon_{n-1}(V_i) \, \Gamma_i = 0
\end{equation}
holds true identically on $X$, one sets for any $x=(x_1,\ldots,x_n)\in \mathbf R^n$:
\begin{itemize}
\item $x'=V_1(x)=(x_2,\ldots,x_n)\in \mathbf R^{n-1}$;
\mk
\item $ x_{n+1}=0$ and $x_{n+2}=1$;
\mk
\item $\tilde x=(x_1,\ldots,x_n,x_{n+1},x_{n+2})\in \mathbf R^{n+2}$;
\mk
\item $\tilde x'=(x_2,\ldots,x_n,x_{n+1},x_{n+2})\in \mathbf R^{n+1}$;
\mk
\item ${\rm M}_{n}(x)=\prod_{1\leq i<j\leq n}(x_i-x_j)$;
\mk
\item $M_{0,n+3}(x)=
{\rm M}_{n+2}(\tilde x)
=\prod_{i=1}^n  x_i\, \big(x_i-1\big)\prod_{1\leq k<\ell \leq n}\big(x_k-x_\ell\big)
$; \mk 
\item  for $i,j=1,\ldots,n$, one sets $x_{ij}=x_i-x_j$ (thus $\tilde x_{i,n+1}=x_i$ and 
$\tilde x_{i,n+2}=x_i-1$ for any $i\leq n$); 
\sk
\item for $i=2,\ldots,n$ and $j=i+1,\ldots,n+2$, one denotes by 
$x_{\widehat{\imath \jmath}}$ the $(n-1)$-tuple  obtained from $\tilde x$ by removing from it,  in the same step: its first, $i$-th and  $j$-th coefficients.   
\end{itemize}
%
%
%
%
%
%
%
%
The notation $M_{0,n+3}$ is not absolutely needed but we use it because it is enlightening: indeed, for
$\mathbf K=\mathbf R$ or $\mathbf C$, the equation 
 $M_{0,n+3}(z)=0$ for $z\in \mathbf K^n$ cuts out the braid arrangement $A_{n}$ in $\mathbf K^n$ which is  such that there is a natural identification between $\mathcal M_{0,n+3}(\mathbf K)$ and $\mathbf K^n\setminus A_{n}$.

Then, with these notations at hand, for any $x'=(x_2,\ldots,x_n)\in \mathbf R^{n-1}\setminus A_{n-1}$, one sets:
\begin{equation}
\label{Eq:Euler-varepsilon}
\varepsilon_{n-1}(x')=\frac{1}{M_{0,n+2}\big (x' \big)}
\hspace{-0.25cm} \sum_{ \substack{i=2,\ldots,n \\ j=i+1,\ldots,n+2 }}
(-1)^{j-i} \, 
{\rm M}_{n-1}\big(x_{\widehat{\imath \jmath}}\big)\, 
\big(\tilde x_{ij} \big)^{n-1}\, 
{\rm Log}\, \lvert \, \tilde x_{ij}  \, \lvert\, .
\end{equation}
For $n=4$, one recovers exactly the function defined in \eqref{Eq:E4(x2x3x4)}. 
\mk 

%
%


By direct computations, we get the following generalization of 
Proposition \ref{P:properties--e3}: 
\begin{prop} 
\label{P:properties--en}
Let $n\geq 2$ be an even integer  less than or equal to 12. 
\begin{enumerate} 
\item[1.] The function $\varepsilon_{n-1}$ satisfies 
 the functional relations corresponding to 
\eqref{Formula:C*E} and \eqref{Formula:R*E-even-odd}:   identically on $X$, one has 
%
%
 \begin{align*}
  \varepsilon_{n-1}\left(  \frac{x_n}{x_n-1},  \frac{x_n}{x_n-x_2},  \ldots, \frac{x_n}{x_n-x_{n-1}}
  \right)   = &\, 
  \varepsilon_{n-1}\big(x'\big) \, 
{{{\rm M}_{n+1}\big( \tilde x'\big)}^2 }{(x_n)^{-n}} 
%
  \\
  \mbox{ and }\quad 
  \varepsilon_{n-1}\left(  \frac{x_n-x_{n-2}}{x_n-x_{n-1}},   \ldots,  
  \frac{x_n-x_{2}}{x_n-x_{n-1}}, 
  \frac{x_n-1}{x_n-x_{n-1}}, \frac{x_n}{x_n-x_{n-1}}
  \right) = &\, \varepsilon_{n-1}\big(x'\big) \,  \big(x_n-x_{n-1}\big)^n 
  \, ; 
  \mk
 \end{align*}
\item[2.]  Identity \eqref{Eq:Euler-En-cyclic-form} holds true identically and 
consequently it corresponds  to Euler's abelian relation $\boldsymbol{\mathcal E}_n^{>0}$ 
and $\varepsilon_{n-1}$  coincides with the 
function $e_{n-1}$ of Proposition \ref{P:Prop-em-integral-formula} (up to multiplication by a non zero constant).
\end{enumerate}
\end{prop} 

Of course, we conjecture that this proposition is actually  satisfied for any even integer  $n\geq 2$. \mk

\paragraph{\bf An Eulerian basis of the space of combinatorial abelian relations 
for any $n\geq 2$  even?}
\label{Par:En-even}
The identity \eqref{Eq:Euler-En-cyclic-form} has an interesting feature which comes from the fact that there are logarithmic terms in the definition of $\varepsilon_{n-1}$: these being transcendent over any field of rational functions, one can construct many rational ARs from Euler's one. In this paragraph, we formalise this in precise terms and we show that, at least for the first values (namely $2,4,6,8$) of $n$, this approach is relevant regarding the question of constructing a basis of $\boldsymbol{AR}_C\big(\boldsymbol{\mathcal W}_{0,n+3}\big)$. 
\sk 

We continue to use the notation introduced in the preceding paragraph.  
We also set 
$$ 
\Pi_n'=\left\{ \,  (i,j)\in \mathbf N^2\, \big\lvert 
\scalebox{0.85}{\begin{tabular}{c}
$2\leq i \leq n$ \\
 $i+1\leq j \leq n+2$
\end{tabular}}
\right\}
\qquad \mbox{and} \qquad  
\Pi_n=\left\{ \,  (i,j)\in \mathbf N^2\, \big\lvert 
\scalebox{0.85}{\begin{tabular}{c}
$1\leq i \leq n$ \\
 $i+1\leq j \leq n+2$
\end{tabular}}
\right\}
\,.
$$
(Note that  $\Pi_n'$ is precisely the set of $(i,j)$'s appearing in the sum in the definition \eqref{Eq:Euler-varepsilon} of 
$\varepsilon_{n-1}$). 

Setting also 
\begin{align*}
R_\gamma(V_k)=& \, (-1)^{j-i} \big(V_{k,ij}\big)^{n-1}
{\rm M}_{n-1}
\big(V_{k,\widehat{\imath \jmath}}\big) /{M_{0,n+2}(V_k)}
\\
\mbox{and } \, {\rm Log}\, \lvert \, \tilde V_{k,\gamma} \, \lvert \, =&\, V_k^*\big( \, 
{\rm Log}\, \lvert  \, \tilde x_{\gamma} \, \lvert\, 
\big)
\end{align*}
for any $k=1,\ldots,n+3$ and 
any $\gamma=(i,j)\in \Pi_n'$, 
 one can write 
\begin{equation}
\label{Eq:Euler-varepsilon-2}
\varepsilon_{n-1}\big(V_k\big)
=\sum_{ \gamma \in \Pi_n'} 
R_\gamma\big(V_k\big) \, {\rm Log}\, \big\lvert \, \tilde V_{k,\gamma} \, \big\lvert\, .
\end{equation}
On the other hand,  one has $V_k=(C^{\circ (k-1)}\big)^*(V_1)$. From the fact that 
 $C$ is birational and induces an automorphism of $\mathbf C^n\setminus A_n $, one 
 deduces the
\begin{lem}
\label{Lem:typo}
1. For any $k\leq n+3$ and $\gamma\in \Pi_n'$, there exist 
constants $c^\nu_{k,\gamma}$'s with $\nu \in \Pi_n$ such that 
\begin{equation}
\label{Eq:LogVk}
{\rm Log}\, \big\lvert \, \tilde V_{k,\gamma}\, \big\lvert =\sum_{\nu\in \Pi_n} c^\nu_{k,\gamma}
\, {\rm Log}\, \lvert\,  \tilde x_\nu\, \lvert \, . 
\end{equation}
2. Moreover  for  $\nu=(i,j)\in \Pi_n$, one has $c_{k,\gamma}^{\nu}=0$ if $k\in \{i,j\}$.
\mk\\ 
3. The family of functions $\big\{  \, {\rm Log}\, \lvert\,  \tilde x_\nu\, \lvert \,\big\}_{\nu\in \Pi_n}$ is free over the field of rational functions in the $x_i$'s. 
\end{lem}

Using
\eqref{Eq:Euler-varepsilon} and 
\eqref{Eq:LogVk}, it is straightforward to develop Euler's identity \eqref{Eq:Euler-En-cyclic-form} in order to get a linear expression in the  
$ {\rm Log}\, \lvert \tilde x_\nu\lvert $'s:
\begin{align*}
0=  \sum_{k=1}^{n+3} \varepsilon_{n-1}\big( 
V_k
\big) \cdot \Gamma_k
= \sum_{k=1}^{n+3} 
\sum_{ \gamma \in \Pi_n'} 
R_\gamma\big(V_k\big) \, {\rm Log}\, \big\lvert\, \tilde  V_{k,\gamma}\,  \big\lvert\, \Gamma_k
=  \sum_{\nu\in \Pi_n} \left[\,  \sum_{k=1}^{n+3} 
\sum_{ \gamma \in \Pi_n'}   c_{k,\gamma}^{\nu} R_\gamma\big(V_k\big) 
 \, \Gamma_k\, \right]
 {\rm Log}\, \lvert\,  \tilde x_\nu\, \lvert \, . 
\end{align*}
From the third point of Lemma \ref{Lem:typo}, it follows that  
Euler's identity 
\eqref{Eq:Euler-En-cyclic-form} 
 holds true 
 if and only if all the 
following relations are identically satisfied:  
\begin{equation}
\label{Eq:AR-nu}
\qquad \qquad \qquad
\sum_{k=1}^{n+3} 
\sum_{ \gamma \in \Pi_n'}   c_{k,\gamma}^{\nu} R_\gamma\big(V_k\big) 
 \, \Gamma_k=0\,  \qquad \nu \in \Pi_n\, . 
\end{equation}

For any $\nu=(i,j)\in \Pi_n$, the identity  above gives rise to an abelian relation for $\mathcal W_{0,n+3}$ and since $c_{i,\gamma}^{(i,j)}=c_{j,\gamma}^{(i,j)}=0$ according to the second point of 
Lemma \ref{Lem:typo}, it follows that \eqref{Eq:AR-nu} is an AR of the quadrilateral $(n+1)$-web $\boldsymbol{\mathcal W}_{\widehat{\imath\jmath}}$ which has rank 1. We will denote this  AR by $AR_\nu(\boldsymbol{\mathcal E}_n)$. 


\begin{prop} 
\label{Prop:Eulern246}
For $n=2,4,6,8$, both assertions below hold true. 
\begin{enumerate}
\item[1.]  For any $\nu=(i,j)\in \Pi_n$,  $AR_\nu(\boldsymbol{\mathcal E}_n)$ is non trivial  hence spans the vector space $\boldsymbol{AR}\big(\boldsymbol{\mathcal W}_{\widehat{\imath\jmath}}\big)$. 
\mk
\item[2.] The abelian relations $AR_\nu(\boldsymbol{\mathcal E}_n)$'s
for $\nu\in \Pi_n$ form a basis of $\boldsymbol{AR}_C\big(\boldsymbol{\mathcal W}_{0,n+3}\big)$. 
\end{enumerate}
\end{prop}
This result has been proved by explicit formal computations. Considering this, it is quite natural to wonder about the statements of Proposition \ref{Prop:Eulern246}: do they hold true as well for any even integer $n\geq 2$?
 We conjecture that it is indeed the case: 

\begin{conjecture} 
The statements of Proposition \ref{Prop:Eulern246} hold true for any even integer $n\geq 2$. 
\end{conjecture} 

It is worth noticing that the results of this paragraph share great similarities with some in \cite[\S6.2.3]{ClusterWebs}. Finally, let us mention that another approach to build a basis of the space of rational ARs of $\boldsymbol{\mathcal W}_{0,n+3}$ from Euler's abelian relation might be to use the `monodromy method'
\footnote{We are referring here to the generalization for curvilinear webs of the method discussed in 
\cite[\S6.2.3.2]{ClusterWebs}.}
 which relies on the fact that the complexification of $\boldsymbol{\mathcal E}_n^{>0}$ has unipotent monodromy on the complex moduli space 
 $\mathcal M_{0,n+3}(\mathbf C)$. 
   We shall not pursue on this here.

%

\newpage

\section{\bf Additional problems}
\label{S:Problems}
Many problems and conjectures have been stated previously in this text, most of them in direct relation with the webs under scrutiny here, namely the curvilinear webs $\boldsymbol{\mathcal W}_{0,n+3}$. Generally speaking, webs by curves have not been studied as much as webs by hypersurfaces, which is a pity since several of the most interesting results of web geometry precisely concern curvilinear webs. \sk 

In this last section, we state a few questions/problems about webs by curves
 we find interesting. 

\subsection{Octahedral webs}
\label{SS:Octahedral-webs}
The notion of `{\it octahedral web}' considered in \cite{D} is interesting since it appears as a generalization to curvilinear webs in arbitrary dimension of the basic notion of  `{\it hexagonal}' (aka `{\it flat}') planar 3-web. However, some quite basic questions about this notion remain to be answered, such as the following ones: 
\begin{itemize}
\item Is a curvilinear $(n+1)$-web of maximal rank ({\it i.e.}\,of rank 1) necessarily octahedral? 
\mk 
\item What is the moduli space of equivalence classes of curvilinear $(n+1)$-webs of  rank 1?  
\mk 
\item For planar 3-webs, being hexagonal is equivalent to being of rank 1, which is in  turn equivalent to being `{\it flat}', that is of having zero curvature.
 In dimension $n\geq 2$ arbitrary, is there a similar characterization of 
octahedral curvilinear $(n+1)$-webs in terms of the vanishing of certain differential invariants attached to such webs? Same question with `{maximal rank}' instead of `{octahedral}' (in case  these two notions do not coincide when $n\geq 3$). 
\sk 
\end{itemize}

\subsection{Algebraization of curvilinear $(n+2)$-webs with maximal rank in dimension $n$?}
\label{SS:Algebraization-(n+2)-webs}
It is classically known by web-geometers that Lie-Poincar\'e's approach to the linearization and algebraization of planar 4-webs with maximal rank generalizes to $2n$-webs by hypersurfaces with rank $n+1$ on $n$-dimensional manifolds.\footnote{This generalization follows from works by Lie and Wirtinger 
on the classification of double-translation hypersurfaces; see \S4.4 and especially Theorem 4.4.3 in \cite{Coloquio}.}  The discussion in \S\ref{SSS:Alternative-Proof-of-Proposition- P:Damiani-4.1.4},  aiming to give 
a  representation-theoretic description of the ARs of the curvilinear web formed by  the bundles  of lines passing through $n+2$ points in general position in $\mathbf P^n$  (for any $n\geq 2$), is of geometric nature and can be seen as well as a generalization for any $n\geq 2$,  but now for webs by curves, of Lie-Poincar\'e's approach just mentioned.   It is natural to wonder whether the arguments presented in  \S\ref{SSS:Alternative-Proof-of-Proposition- P:Damiani-4.1.4}
actually could  apply to any curvilinear $(n+2)$-web in $\mathbf C^n$ of rank $n+1$. 
\sk

More precisely, for $n\geq 2$, let $\boldsymbol{\mathcal W}=(\mathcal F_i)_{i=1}^{n+2}$ be a curvilinear of rank $n+1$ on a domain $\Omega\subset \mathbf C^n$. For any $i$ and for 
$\omega$ generic in $\Omega$, the $i$-th evaluation map associating to an AR for $\boldsymbol{\mathcal W}$ the value at $\omega$ of its $i$-th component, is a non trivial linear map the kernel of which defines an hyperplane $\kappa_i(\omega)$ in $\mathbf P\big( \boldsymbol{AR}(\boldsymbol{\mathcal W})\big)\simeq \mathbf P^n$.  Letting $\omega$ vary in $\Omega$ (or possibly in a subdomain of it), we construct that way the {\it `$i$-th canonical map'} 
$\kappa_i: \Omega\rightarrow \check{\mathbf P}^n$ of $\boldsymbol{\mathcal W}$. 
Then, inspired by the (unproven) claims  of \S\ref{SSS:Alternative-Proof-of-Proposition- P:Damiani-4.1.4}, we consider the five statements below (where for each of them 
the open domain $\Omega$ of $\mathbf C^n$ we work on is allowed to be shrinked as much as needed): 

\begin{enumerate}
\item[1.] for any $i$, the map $\kappa_i$ has rank $n-1$ hence is a canonical first integral for the 
$i$-th foliation $\mathcal F_i$ of $\boldsymbol{\mathcal W}$ and the image $V_i={\rm Im}(\kappa_i)$ is then an analytic hypersurface in $\check{\mathbf P}^n$ which is a canonical space of leaves for $\mathcal F_i$; \sk
\item[2.] for $\omega\in \Omega$, the $\kappa_i(\omega)$'s all are on a line $\boldsymbol{\mathfrak L}(\omega)\subset \check{\mathbf P}^n$ which intersects any 
$V_i$ transversely;\sk
\item[3.] the map $\boldsymbol{\mathfrak L} : \Omega\rightarrow G_1(\check{\mathbf P}^n)$, $\omega\mapsto \boldsymbol{\mathfrak L}(\omega)$ has rank $n$ hence $Z={\rm Im}(\boldsymbol{\mathfrak L})$ is a $n$-dimensional subvariety of $G_1(\check{\mathbf P}^n)$  and 
the push-forward web $\boldsymbol{\mathfrak L}_*(\boldsymbol{\mathcal W})$ is a canonical model of $\boldsymbol{\mathcal W}$ which can be described geometrically as the trace along $Z$ of the incidence web (locally) defined on $G_1(\check{\mathbf P}^n)$ by means of the incidences between the lines in $\check{\mathbf P}^n$ and the points of the $V_i$'s.\sk
\end{enumerate}
If the above facts  indeed occur then we get the following 
 Abel's type description of the ARs of the canonical model 
$ \boldsymbol{\mathcal W}^{\rm can}=
\boldsymbol{\mathfrak L}_*(\boldsymbol{\mathcal W})$ of $ \boldsymbol{\mathcal W}$: 
\begin{enumerate}
\item[4.]  
any AR of $ \boldsymbol{\mathcal W}^{\rm can}$   corresponds to a tuple $(\eta_i)_{i=1}^{n+2}$ 
of differential $(n-1)$-forms $\eta_i\in \Omega^{n-1}(V_i)$ for every $i$ such that the trace of them
with respect to the intersection of the $V_i$'s with the lines belonging to $ Z$ vanishes identically.
\sk
\end{enumerate}
At this point, 
any reader even slightly familiar with web geometry reading this line will surely think that the preceding points necessarily imply the following one: 
\begin{enumerate}
\item[5.]  
the vanishing of the traces of the tuples of differential forms corresponding to the ARs of $ \boldsymbol{\mathcal W}^{\rm can}$ ({\it cf.}\,4.\,just above) necessarily implies that everything, the $V_i$'s and the ARs of $ \boldsymbol{\mathcal W}^{\rm can}$ actually are algebraic in the following (expected since classical) sense: \sk 
\begin{enumerate}
\item[$-$] there exists a hypersurface $V\subset \check{\mathbf P}^{n}$ of degree $n+2$ such that 
$V_i\subset V$ for every $i$;\sk 
\item[$-$]  the map induced by taking their restrictions along the $V_i$'s 
 induces a linear isomorphism between 
the space  of global abelian differential forms of top degree on $V$ and the space of abelian relations of 
$ \boldsymbol{\mathcal W}^{\rm can}$:  
\begin{align*}
\boldsymbol{H}^0\big( V,\omega_V^{n-1} 
\big) & \stackrel{\sim}{\longrightarrow} \boldsymbol{AR}\big( \boldsymbol{\mathcal W}^{\rm can} \big)
\\
\omega & \longmapsto \big( \omega\lvert_{V_i}\big)_{i=1}^{n+2}\, .
\end{align*}
\end{enumerate}
\end{enumerate}

All the points above are indeed satisfied for any linear web $\boldsymbol{\mathcal L \hspace{-0.05cm}\mathcal W}_{p_1,\ldots,p_{n+1}}$ defined by $n+2$ points in general position in 
$\mathbf P^n$. One expects that this actually holds true in full generality and that the following statement holds true for any $n\geq 2$:
\begin{quote}
{\it any curvilinear $(n+2)$-web of maximal $(n-1)$-rank in dimension $n$ can be obtained by taking the restriction of an algebraic web $\boldsymbol{\mathcal W}_{V}$ associated to a reduced hypersurface $V\subset {\mathbf P}^n$ of degree $n+2$ along a (non necessarily algebraic) $n$-dimensional subvariety of the grassmannian of lines in $\check{\mathbf P}^n$. }
\end{quote}
\mk 

That the five points above are always satisfied seems very plausible to us. We believe that the one which would require the most work to be established is the fifth: an approach to prove it would be to use the classical Abel-inverse theorem but this would require to show that in the situation under scrutiny, the vanishing of the trace of the tuples 
 $(\eta_i)_{i=1}^{n+2}
\in  \prod_{i=1}^{n+2} \Omega^{n-1}(V_i)
$ corresponding to the ARs of $ \boldsymbol{\mathcal W}^{\rm can}$ holds true not only on the subvariety $Z$ but on a whole open neighborhood of it in $G_1(\check{\mathbf P}^n)$. Since the former is of dimension $n$ and the latter of dimension $2n-2$, proving this when $n\geq 3$ requires a new approach we have no idea of when writing these lines.



\subsection{A conjectural more conceptual description of the $\mathfrak S_{n+3}$-module ${AR}_C(\mathcal W_{0,n+3})$ when $n$ is odd.}
It seems to us that one of the most surprising results of this text  is the fact that the ARs of $\boldsymbol{\mathcal W}_{0,6}$ actually are all combinatorial and come from the canonical forms on the Fano surface $\Sigma=F_1(\boldsymbol{S})$ of Segre's cubic $\boldsymbol{S}\subset \mathbf P^4$.  Since all the ARs of  $\boldsymbol{\mathcal W}_{0,n+3}$ are combinatorial too when $n$ is odd, one can wonder whether a similar picture to the one described above might hold true for $n\geq 5$. 

\subsubsection{}
Let us be a bit more precise about what we have in mind here. Assume that $n$ is odd: one writes $n=2m-1$ with $m=(n+1)/2$. Hence $n+3=2m+2$.   
As we have seen in \S\ref{SSS:Ln-varphin} and \S\ref{SSS:linearizability}, $\boldsymbol{\mathcal W}_{0,n+3}$ admits a birational model $\boldsymbol{W}_{0,n+3}$ which is a web by lines on a certain $n$-dimensional projective variety $V_n\subset \mathbf P^N$. 
These varieties have been studied by several authors, in particular in the recent paper 
\cite{BM} where the authors establish many interesting properties of them.  Those which are interesting for our purpose here are the following ones, which appear as generalizations of some basic properties of Segre's cubic (which corresponds to the case $m=2$):
\begin{itemize}
\item[$-$] The variety $V_n$ has isolated singularities and  one has 
${\rm Card}\big({\rm Sing}(V_n)\big)= {  2m+2 \choose m}$. Moreover $V_n$ contains  
${  2m+1 \choose m+1}+{  2m+1 \choose m-1}$ linear subspaces of dimension $m$.
\mk 
\item[$-$] The Fano variety of lines  $\Sigma_n=F_1(V_n)\subset G_1(\mathbf P^N)$ has dimension $n-1=2m-2$. Among its components, $n+3$ are isomorphic to $\overline{\mathcal M}_{0,n+2}$ and are covering families. The others are 
formed by the lines contained in one of the $m$-planes included in $V_n$ hence all are isomorphic to $G_1(\mathbf P^m)$. 
\end{itemize}

Considering the case $n=3$, it is more than tantalizing to hope for a description of 
the ARs of $\boldsymbol{W}_{0,n+3}$ by means of global $(n-1)$-forms on 
the Fano variety $\Sigma_n$. Let $\omega_{\Sigma_n}=\omega_{\Sigma_n}^{n-1}$ stand  for the dualizing sheaf of $\Sigma_n$.  We find the following questions interesting: 
\begin{itemize}
\item[$-$] Does the trace  induce a well-defined map $\boldsymbol{ H}^0\big(\Sigma_n, \omega_{\Sigma_n}\big)\rightarrow \boldsymbol{AR}\big(\boldsymbol{\mathcal W}_{0,n+3}\big)$? In the affirmative, is this map an isomorphism? 
\sk 
\item[$-$] Does the  $\mathfrak S_{n+3}$-action on $V_n$ induce a birational action 
on $\Sigma_n$ and a linear one on $\boldsymbol{ H}^0\big(\Sigma_n, \omega_{\Sigma_n}\big)$? 
If yes, what is the latter as a $\mathfrak S_{n+3}$-module? Is it irreducible with Young symbol $[31^{n}]$?
\sk 
\item[$-$]  Is there a coherent subsheaf $\omega_{\Sigma_n}^{m-1}$ of the sheaf 
of meromorphic $(m-1)$-differential forms on $\Sigma_n$ satisfying the following properties: \sk
\begin{itemize}
\item[$\bullet$] the space $\boldsymbol{ H}^0\big(\Sigma_n, \omega_{\Sigma_n}^{m-1}\big)$ 
of its global sections 
has dimension $n+2$;\sk
\item[$\bullet$] the wedge product gives rise to an isomorphism 
$\wedge^2  \boldsymbol{ H}^0\big(\Sigma_n, \omega_{\Sigma_n}^{m-1}\big)\simeq 
\boldsymbol{ H}^0\big(\Sigma_n, \omega_{\Sigma_n}\big)$;\sk 
\item[$\bullet$] the 
action   of 
$\mathfrak S_{n+3}$  
by pull-backs 
makes of 
$\boldsymbol{ H}^0\big(\Sigma_n, \omega_{\Sigma_n}^{m-1}\big)$ an  irreducible 
$\mathfrak S_{n+3}$-module.
\sk
\end{itemize}
%
\vspace{-0.3cm}
\item[$-$]  Do the elements of $\boldsymbol{ H}^0\big(\Sigma_n, \omega_{\Sigma_n}^{m-1}\big)$ give rise to $(m-1)$-abelian relations for $\boldsymbol{W}_{0,n+3}$?
\end{itemize}

In another direction, the fact that Segre's cubic $\boldsymbol{S}$ is just a particular example of a whole family of varieties (cubic hypersurfaces in $\mathbf P^4$) all carrying a web by lines of maximal rank naturally leads to ask  the following questions for any odd integer $n\geq 3$: 
\begin{itemize}
\item[$-$]  As a subvariety of $\mathbf P^N$, 
 can $V_n$ be deformed  in such a way that its deformations carry $(n+3)$-webs by projective lines? 
  Does it exist deformations of this kind which are smooth?
 \sk 
\item[$-$] If the answer to the previous question is affirmative, what can be said about the Fano variety of lines of such a  deformation of $V_n$?  Do some rational differential forms on this Fano variety give rise to abelian relations for the $(n+3)$-web by lines on the deformation under scrutiny? 
\end{itemize}

We believe that giving answers in full generality to the  questions  asked above is difficult. 
The next case beyond the classical one of Segre's cubic is already interesting and deserves to be discussed on its own.

\subsubsection{\bf Case $n=5$.}
\label{SS:Case-n=5}
The case when $n=5$ is interesting since $V_5$ has  already  been  studied by several authors:  see  for instance \cite{Room1934} 
(in particular \S3 and \S11 therein) for a classical reference 
 and  \cite{FreitagSalvatiManni}
 or \cite{HowardAL} for recent ones. \sk

 Let $p_1,\ldots,p_7$ be seven points in general position in 
 $\mathbf P^5$.  According to \S\ref{SSS:Ln-varphin}, $V_5$ is the birational image of $\mathbf P^5$ 
 by the rational map associated to the  linear system of cubic hypersurfaces with a double point at each $p_i$. 
It can be verified that $V_5$ lives in   $\mathbf P^{13}$,  has  35 isolated double points, is of degree 40 and carries 8 covering families of lines. Many other interesting geometric properties of $V_5$ are listed in \cite[\S11]{Room1934}.   

Another construction of $V_5$ has been given by modern authors: 
from the recent references indicated above, it can be deduced that $V_5$ is the singular locus of an explicit cubic hypersurface  $\boldsymbol{C}\subset \mathbf P^{13}$, namely the hypersurface cut out by the  equation $\mathcal C=0$ where $\mathcal C$ stands for the following cubic form (in homogeneous coordinates $X_1,X_2,Y_a$ with $a=1,\ldots,4$ and $Z_b$ for $b=1,\ldots,8$): 
\begin{align*}
\mathcal C = &\, X_1X_2(X_1 + X_2) + X_1X_2(Z_1 + Z_2 + Z_3 + Z_4 + Z_5 + Z_6 + Z_7 + Z_8)
-  (X_1Y_2Y_4 + X_2Y_3Y_1) \\ &  {}^{} \quad +  (X_1Z_2Z_6 + X_2Z_3Z_7 + X_1Z_4Z_8 + X_2Z_5Z_1)
+  (Y_1Z_2Z_6 + Y_2Z_3Z_7 + Y_3Z_4Z_8 + Y_4Z_5Z_1) \\ &  {}^{} \quad -  (Z_1Z_2Z_3 + Z_2Z_3Z_4
+ Z_3Z_4Z_5 + Z_4Z_5Z_6 + Z_5Z_6Z_7 + Z_6Z_7Z_8 + Z_7Z_8Z_1 + Z_8Z_1Z_2)\,.
\end{align*}
The homogeneous ideal $I(V_5)$ of $V_5\subset \mathbf P^{13}$ is then generated by the quadrics given by the partial derivatives $\partial \mathcal C/\partial U$ where $U$ stands for any one of the fourteen homogeneous coordinates $X_1,X_2,Y_a$ or $Z_b$.  

Knowing an explicit generating set for $I(V_5)$  makes it possible to attack the questions raised in the preceding subsection by means of effective computations. We have tried to do so using the software \href{http://www2.macaulay2.com/Macaulay2/}{Macaulay2} but our attempt failed: even when working over a finite field of small cardinal to make everything simpler,  the computations to get the ideal of the 
Fano variety of lines $\Sigma_5=F_1(V_5)$ were too memory-consuming so that we 
were not able to get to the end.  A brute force approach to study $\Sigma_5$ does not seem very efficient, understanding this Fano variety may require a more conceptual approach. 
\mk 

Finally, let us mention a possibly naive idea about the varieties $V_n$ for $n$ odd bigger than 1 . Both $V_3$ and $V_5$ can be defined by means of an invariant cubic: $V_3$ itself is a cubic (namely Segre's cubic $\boldsymbol{S}$) and $V_5={\rm Sing}(\boldsymbol{C})$ where $\boldsymbol{C}\subset \mathbf P^{13} $   is the $\mathfrak S_{8}$-invariant cubic hypersurface cut out by $ \mathcal C=0$.  These two remarks lead to wonder about the general case and to ask the following questions: 
\begin{itemize}
\item[$-$]   Can $V_n$ be defined by means of a particular $\mathfrak S_{n+3}$-invariant cubic hypersurface $\boldsymbol{C}_n \subset \mathbf P^{N} $?  
\sk 
\item[$-$]  If the answer to the previous question is affirmative, as an algebraic subset of $\mathbf P^N$ and up to projective equivalence, 
does $V_n$ coincide with  the $(n-3)$-th higher singular locus 
${\rm Sing}^{(n-3)}(\boldsymbol{C}_n)$
of the cubic hypersurface $\boldsymbol{C}_n$?\footnote{For any projective variety $X\subset \mathbf P^N$,  the higher order singular loci 
${\rm Sing}^{(k)}(X)$ ($k\in \mathbf N$) are the schemes inductively defined  
${\rm Sing}^{(0)}(X)=X$ and ${\rm Sing}^{(k)}(X)={\rm Sing}
\Big( 
{\rm Sing}^{(k-1)}(X)
\Big)$  for any $k\geq 1$.}
\sk 
\end{itemize}
We confess not having any argument to support positive answers to these questions. To tell the truth, we would be amazed if they can be answered in a positive way. In any case, these questions were too beautiful and intriguing for not to be asked!

%
%
%

\subsection{About the characterization of the webs by lines on cubic threefolds.}
Here we would like to discuss briefly some interesting questions indicated to us by 
Prof.\,J.-M.\,Hwang a few years ago.  The problem under scrutiny here can be considered in the bigger realm of Fano manifolds of dimension 3 but for the sake of simplicity, we will only consider the case of cubic threefolds below.\sk

Let  
$X\subset \mathbf P^4$ be a generic (in particular smooth) cubic hypersurface and  denote  by $F=F_1(X)$ its Fano surface.   It follows from a more general result by Hwang (namely  \cite[Theorem 1.2]{Hwang}) that the local analytic class of the 6-web by lines 
$\boldsymbol{\mathcal L \hspace{-0.05cm} \mathcal W}_X$  characterizes $X$ as a projective variety: 
\begin{thm}[Hwang]
Let $X'\subset \mathbf P^4$ be another smooth cubic hypersurface. 
Assume that there exists a local biholomorphism $\Phi : (X,x)\rightarrow (X',x')$ such that 
$\boldsymbol{\mathcal L \hspace{-0.05cm} \mathcal W}_{X,x}=\Phi^*\big( \boldsymbol{\mathcal L \hspace{-0.05cm} \mathcal W}_{X',x'}\big)$. 
 Then $\Phi$ is the germ at $x$ of a global isomorphism between $X$ and $X'$, which necessarily coincides with the restriction along $X$ of a global projective automorphism of $\mathbf P^4$. 
\end{thm}

This nice result is very much in the original spirit of web geometry, namely that 
webs are quite rigid objects for which local analytic equivalence may imply (under some additional assumptions of course) algebraic/global  equivalence hence algebraization.  

We find it interesting to consider the previous result in 
 the light of the remarkable algebraization theorem of Blaschke and Walberer (about the algebraization of some maximal rank curvilinear 3-webs in dimension 3, see Appendix B for more details).  This suggests to wonder about a possible algebraization result for the 6-webs by lines in dimension 3. In view of stating it,  let us draw up a list of some nice properties  enjoyed by any curvilinear 6-web  $\boldsymbol{\mathcal W}$ equivalent to a web $\boldsymbol{\mathcal L \hspace{-0.05cm} \mathcal W}_X$ associated to a generic cubic threefold $X$: 
from the results of \S\ref{SS:Cubic-Hypersurfaces}, it is clear that 
such a web $\boldsymbol{\mathcal W}$
\begin{enumerate}
\item[1.] is linearizable; 
\sk 
\item[2.] is skew; 
\sk 
\item[3.] has maximal 2-rank (equal to $10$);
\sk 
\item[4.] has its 1-rank bigger than or equal to $5$;
\sk 
\item[5.] is such that there exists a subspace 
$\boldsymbol{A}^1(\boldsymbol{\mathcal W}) 
\subset \boldsymbol{AR}^{(1)}\big(\boldsymbol{\mathcal W}\big)$ 
of dimension 5 such that the wedge map 
$\wedge^2 \boldsymbol{A}^1(\boldsymbol{\mathcal W})\rightarrow \boldsymbol{AR}^{(2)}\big(\boldsymbol{\mathcal W}\big)$ is a well-defined isomorphism.
\sk
\end{enumerate}
Note that all these properties are stated in invariant form.  Clearly, 5.\,implies 3.\,and 4. Remark 
also that 2.\,is not satisfied by $
\boldsymbol{\mathcal W}_{0,6}\simeq 
\boldsymbol{\mathcal L \hspace{-0.05cm} \mathcal W}_{\boldsymbol{S}}$. However, for any 
smooth cubic $X\subset \mathbf P^4$, the web $\boldsymbol{\mathcal L \hspace{-0.05cm} \mathcal W}_X$ is skew indeed.  A first question that comes to mind for such $X$  is the following: 
\begin{question}
 Assume that $X$ is smooth and denote by $F=F_1(X)$  its Fano surface. 
Is  ${\rm rk}^{(1)}\big( \boldsymbol{\mathcal L \hspace{-0.05cm} \mathcal W}_X \big) $ equal to 5?
 More precisely, is the trace map $\boldsymbol{H}^0(F, \Omega_{F}^1\big)\rightarrow 
\boldsymbol{AR}^{(1)}\big(\boldsymbol{\mathcal L \hspace{-0.05cm} \mathcal W}_X\big)$ surjective or not ? 
\end{question}

Each of the five conditions listed above is a strong property and it is natural to expect that any 6-web satisfying all of them must be of a very particular type. Considering this as well as Blachke-Walberer's algebraization theorem, it is natural to make the following 
\begin{conjecture}
Any 6-web $\boldsymbol{\mathcal W}$ satisfying the properties 1.\,to 5.\,above is equivalent to an algebraic web  
$\boldsymbol{\mathcal L \hspace{-0.05cm} \mathcal W}_X$ associated to a cubic hypersurface $X\subset \mathbf P^4$. 
\end{conjecture}
Actually, we believe that the same conclusion holds true under the weaker assumption that only the first three of these conditions are satisfied. It would be interesting to even drop the linearizability assumption and to know whether a skew curvilinear 6-web in dimension 3 with maximal 2-rank is necessarily linearizable. However, answering this may be quite difficult.

\subsection{The 6-web by planes on Perazzo's cubic fourfold.}
In  \cite{Perazzo}, Perazzo considers a certain cubic hypersurface in $\mathbf P^5$, 
namely the one cut out by $$X_1X_2X_3-Y_1Y_2Y_3=0$$
with respect to  some 
homogeneous coordinates 
$X_1,X_2,X_3,Y_1,Y_2,Y_3$   on $\mathbf P^5$, named after him and denoted by $\mathcal P$ here.  This cubic enjoys several interesting  properties and has been considered in several papers such as \cite{Baker} or the more recent one \cite{Looijenga}, to which we refer the interested reader.\mk 

Here are some of the properties of $\mathcal P$ which are going to be relevant for our purpose: 
\begin{itemize}
\item Perazzo's cubic has 9 double lines $D_1,\ldots,D_ 9$ as singularities;
\mk 
\item for $p\in \mathcal P$ generic (namely $p\in \mathcal P_{\rm reg}=\mathcal P
\setminus \big( \cup_{i=1}^6 D_i\big)$), the intersection $\boldsymbol{S}_p=\mathcal P\cap T_p \mathcal P$ of Perazzo's cubic with its tangent hyperplane $T_p \mathcal P\simeq \mathbf P^4$
 at $p$ is a cubic threefold with 10 nodal points: $p$ and the nine intersection points 
  $T_p \mathcal P\cap D_i$ for $i=1,\ldots,9$. Hence  $\boldsymbol{S}_p$ is projectively equivalent to Segre's cubic threefold; \mk
\item  the Fano scheme 
$F_2(\mathcal P)$ of 2-planes included in $\mathcal P$ is of pure dimension 2 and through a general point $p\in \mathcal P$ pass six pairwise distinct 2-planes 
$\Pi_1(p),\ldots,\Pi_6(p)$ included in $\mathcal P$. 
Consequently, Perazzo's cubic fourfold carries a linear 6-web of codimension 2, that we will denote by $\boldsymbol{\mathcal L\hspace{-0.05cm}\mathcal W}_{\mathcal P}$. 
One can define six regular germs of map $\Pi_i : (\mathcal P,p)
\rightarrow G_2(\mathbf P^5)$ ($i=1,\ldots,6$) such that $\Pi_1(p'),\ldots,\Pi_6(p')$ are the six planes included in $\mathcal P$ through $p'$  for any $p'$ sufficiently close to $p$. Then  $\Pi_1,\ldots,\Pi_6$ are local first integrals for $\boldsymbol{\mathcal L\hspace{-0.05cm}\mathcal W}_{\mathcal P}$ at $p$: locally at this point, one has
$$
\boldsymbol{\mathcal L\hspace{-0.05cm}\mathcal W}_{\mathcal P}=
\boldsymbol{\mathcal W}\Big( \Pi_1,\ldots,\Pi_6\Big)\,; 
$$
\item given $p\in \mathcal P_{reg}$, for $\tilde p$ generic the restrictions along $\boldsymbol{S}_p$ of the (germs of) maps 
\begin{align*}
\Pi_{i,p}: (\mathcal P,\tilde p) & \longrightarrow  \hspace{0.3cm}
F_1({\boldsymbol{S}_p}) \, \subset \, 
G_1(T_p \mathcal P)\simeq G_1(\mathbf P^4)\\
 q & \longmapsto \hspace{0.2cm}  \Pi_i(q)\cap T_p \mathcal P
\end{align*}
 give  rise to local first integrals for the linear web 
$\boldsymbol{\mathcal L\hspace{-0.05cm} \mathcal W}_{\boldsymbol{S}_p}$
carried by Segre's cubic $\boldsymbol{S}_p$. Consequently:  one has
$
\big( \boldsymbol{\mathcal L\hspace{-0.05cm}\mathcal W}_{\mathcal P}\big)\lvert_{{\boldsymbol{S}_p}}  =
\boldsymbol{\mathcal L\hspace{-0.05cm} \mathcal W}_{\boldsymbol{S}_p}$
  locally at $p$ on 
$\boldsymbol{S}_p$; and the (germs at the points $\Pi_{i,p}(\tilde p)$'s of the) Fano surface $F_1({\boldsymbol{S}_p})$ can be taken as  `the' space(s) of leaves of the foliations of  $\boldsymbol{\mathcal L\hspace{-0.05cm}\mathcal W}_{\mathcal P}$  locally at $\tilde p$ on 
${\mathcal P}$. 
\mk 
\end{itemize}

We set $x_5={x_1x_2x_3}/{x_4}$ and $x_6=x_4$ and see both $x_5$ and $x_6$ as rational functions in the $x_k$'s with $k=1,\ldots,4$.  Then the following map is a birational affine parametrization of Perazzo's cubic
\begin{equation*}
\mathscr P: \quad \mathbf C^4  \dashrightarrow \mathcal P\subset \mathbf P^5\, , \quad 
(x_i)_{i=1}^4   \longmapsto 
\big[ x_s\big]_{s=1}^6
\end{equation*}
and the pull-back $\mathscr P^*\big( \boldsymbol{\mathcal L\hspace{-0.05cm}\mathcal W}_{\mathcal P}\big)$ (that we will still denote by $\boldsymbol{\mathcal L\hspace{-0.05cm}\mathcal W}_{\mathcal P}$, at bit abusively), can be described quite explicitly. Indeed  one verifies immediately that the six 2-planes included in $\mathscr P$ passing through a generic point of $\mathcal P$  are  cut out by the systems of equations 
\begin{equation}
\label{Eq:2-plane-nu}
 X_1/Y_{\nu(1)}=\lambda_1\, , \quad  X_2/Y_{\nu(2)}=\lambda_2\quad 
 \mbox{ and } \quad  X_3/Y_{\nu(3)}=\lambda_1\lambda_2
 \end{equation}
  for some generic parameters $\lambda_1,\lambda_2\in \mathbf C$, where $\nu$ ranges in  the set $\mathfrak S_3$ of permutations of $\{1,2,3\}$. \mk 
  
  For each $\nu \in \mathfrak S_3$, the pull-back under $\mathscr P$ of the foliations on 
$\mathcal P$   
  by  the 2-planes with equations \eqref{Eq:2-plane-nu}, denoted by $\mathcal F_\nu$,  admits the rational map $x\dashrightarrow  (x_1/x_{3+\nu(1)}, x_2/x_{3+\nu(2)})$ as first integral.  After some elementary simplifications,  we obtain that the following very simple rational functions can be taken as first integrals for $ \boldsymbol{\mathcal L\hspace{-0.05cm}\mathcal W}_{\mathcal P}$ on $\mathbf C^4$: 
 \begin{align*}
\Pi_{\boldsymbol{1}}= & \left(\, x_3 \, , \, \frac{x_1}{x_4} \right)  
 && \Pi_{(23)} =
  \left(\, {x_2} \, , \, \frac{x_1}{x_4}\,\right)
  && \Pi_{{(1,2)}}= 
\left(\, x_3\, , 
 \,\frac{x_2}{x_4}\,
 \right) 
   \bk \\
 \Pi_{(13)} =  &\, 
\left(\,{x_1}\, , \,\frac{x_3}{x_4}\,\right)
&&
\Pi_{(123)}= 
 \left(\, {x_2}\,, \,\frac{x_3}{x_4}\, \right)
&& \Pi_{(132)} =  \left(\, x_1\, , \, \frac{x_2}{x_4}\, \right) \, .\sk
\end{align*}

Remark that $\Pi_{\boldsymbol{1}}$ and $\Pi_{(23)} $ both have $x_1/x_4$ as a component. Thus the tangent distribution $\langle T_{\mathcal F_{\boldsymbol{1}}}, T_{\mathcal F_{(23)}}\rangle$  is not of dimension 4 but is the integrable 3-dimensional one  defined by $d(x_1/x_4)=0$.  This shows that $ \boldsymbol{\mathcal L\hspace{-0.05cm}\mathcal W}_{\mathcal P}$, which  is a 6-web of codimension 2 on a space of dimension 4, does not satisfy the general position hypothesis usually assumed when working classically with webs of this kind. 
Consequently the classical results about the ranks of webs of this type ({\it cf.}\,the material of 
Chapter 8 of \cite{Goldberg}) do not apply to 
$ \boldsymbol{\mathcal L\hspace{-0.05cm}\mathcal W}_{\mathcal P}$.  Considering the remarkable properties enjoyed by $  \boldsymbol{W}_{0,6}=\boldsymbol{\mathcal L\hspace{-0.05cm}\mathcal W}_{\boldsymbol{S}}$, one can wonder about the status of their direct generalizations to the web by planes associated to Perazzo's cubic fourfold. Here is a sample of questions we find interesting: 
\begin{questions}
 Let $p$ be a regular point of $\mathcal P$, set $F=F_1(\boldsymbol{S}_p)$ and let $k$ stand for 1 or 2. 
\begin{enumerate}
\item[1.] What precisely is the Fano scheme of planes $F_2\big(\mathcal P\big)\subset G_2\big(\mathbf P^5\big)$? How is it related to the Fano surface of lines  $F_1(\boldsymbol{S}_p)\subset  
G_1\big(\mathbf P^4\big)$ of Segre's cubic?
\mk
\item[2.] What is the $k$-rank of $ \boldsymbol{\mathcal L\hspace{-0.05cm}\mathcal W}_{\mathcal P}$?\mk  
\item[3.] Does the following assertion hold true: `for any  $\eta\in \boldsymbol{H}^0\big(
F, \omega^k_F
\big)$, the trace ${\rm Tr}(\eta)$, locally defined as ${\rm Tr}(\eta)=\sum_{i=1}^6 \Pi_{i,p}^*(\eta)$, vanishes identically'?  \mk 
\item[4.] If the answer to the previous question is affirmative, what can be said about the then well-defined associated linear map 
$
\boldsymbol{H}^0\big(F, \omega^k_F
\big) \rightarrow  \boldsymbol{AR}^{(k)}\big( 
\boldsymbol{\mathcal L\hspace{-0.05cm}\mathcal W}_{\mathcal P}\big),\, \eta\mapsto 
\big( \Pi_{i,p}^*(\eta)\big)_{i=1}^6$? 
Is it surjective? Is it an isomorphism?
\mk 

\end{enumerate}
\end{questions}

\newpage

\section*{\bf Appendix A: the abelian 1-ARs of $\mathcal W_{{0,6}}$.}
In \S\ref{S:n=3} (see Proposition  \ref{P:LW-X} more specifically),  we have translated into terms of webs   
some classical results about smooth cubic threefolds and their Fano surface.  
We then specialized some of these results 
to the case of Segre's cubic $\boldsymbol{S}$ 
 to obtain a conceptual way to describe the 2-abelian relations of $
  \boldsymbol{\mathcal L\hspace{-0.05cm}\mathcal W}_{\boldsymbol{S}} \simeq 
 \mathcal W_{{0,6}}$: the suitable version of `Abel's theorem' holds true for Segre's cubic and its Fano surface $\Sigma=F_1(\boldsymbol{S})$ as well and  gives us a natural linear isomorphism $\boldsymbol{H}^0( \Sigma , \omega^2_{\Sigma}\big)\simeq \boldsymbol{AR}^{(2)}
\big( 
\boldsymbol{\mathcal L\hspace{-0.05cm}\mathcal W}_{X}
\big)$.  Then in \S\ref{SSS:2-ARs-W-06}, by means of direct computations, we gave another 
much more down to earth and elementary proof of this result . \sk

For a smooth cubic $X\subset \mathbf P^4$, the 1-ARs of $\boldsymbol{\mathcal L\hspace{-0.05cm}\mathcal W}_{\boldsymbol{S}}$ all come from  holomorphic 1-forms on the corresponding Fano surface $F_1(X)$: see the content of Proposition  \ref{P:LW-X}
corresponding to $k=1$.  
It is natural to  wonder whether this can be specialized to the pair $(\boldsymbol{S},\Sigma)$ as well.   This question has already be discussed before (page \pageref{Abel's-1-form-Fano-Surface}) in the more general context of cubic threefolds with isolated singularities. A strategy relying on deformations is evokated page  \pageref{Colino-Desingularization} where it is mentioned that it applies to Segre's cubic. 
This approach relies on some unpublished results due to Collino. 

In this Appendix, we are going to explain 
the main points of the abstract approach via deformations in order to construct 1-ARs for 
$\boldsymbol{\mathcal L\hspace{-0.05cm}\mathcal W}_{\boldsymbol{S}}$ 
(identified with $\boldsymbol{\mathcal W}_{0,6})$ from 
the global {\it abelian} differential 1-forms on a certain modification $\widetilde \Sigma$ of $\Sigma$. Because we find it justified and funny, 
the ARs obtained that way will be called   `{\it Abelian 1-ARs}'. 
Although conceptually interesting, the lack 
of explicitness of the first construction  of these ARs 
is  not fully satisfying. 
We will remedy to this in a second stage.   Proceeding as in \S\ref{SSS:2-ARs-W-06}, using explicit computations and some basic facts of the theory of representations of 
$\mathfrak S_6$, we will construct  an explicit basis of a 5-dimensional subspace of 
$\boldsymbol{AR}^{(1)}
\big( 
\boldsymbol{\mathcal W}_{0,6}
\big)$ which naturally corresponds to the space of abelian 1-forms on $\widetilde \Sigma$. 
 \begin{center}
$\star$
\end{center}

Before entering into the considerations below, let us recall the following elementary fact about the space of 1-ARs of $\boldsymbol{\mathcal W}_{0,6}$. We know that this web is totally  non skew: two distinct of its foliations are tangent to a foliation of codimension 1 (admitting 
 for a first integral a forgetful map $\mathcal M_{0,6}\rightarrow \mathcal M_{0,4}\simeq \mathbf P^1\setminus \{0,1,\infty\}$).  Its follows that any of its 3-subwebs, and even more so 
 $\boldsymbol{\mathcal W}_{0,6}$ itself, has a space of 1-ARs of infinite dimension.  
The abelian relations we are interested in in this Appendix are those coming 
from particular global rational 1-forms on the Fano surface of Segre's cubic. Thus these are 1-ARs  of a very specific kind. In particular, they span a space of finite dimension which makes undertaking their study more reasonable. 
\subsubsection*{\bf A.1. An abstract construction of the 
Abelian 1-ARs of $\boldsymbol{\mathcal W}_{0,6}$}
We freely use below the results of \cite{Collino2} (especially from \S2.3 and  \S3 therein) to which we refer the reader for more details.\footnote{Note however that some of the results 
we use are not fully proved in \cite{Collino2}.}
We explain how the approach discussed page \pageref{Abel's-1-form-Fano-Surface} 
above can be applied to get some 1-ARs for $\boldsymbol{\mathcal W}_{0,6}$.

Let $S$ and $T$ be two homogeneous cubic forms in five variables  such that $S=0$ cuts out  $\boldsymbol{S}$ and where $T$ is sufficiently generic so that $X_t=\{S+t T=0\}\subset \mathbf P^4$ be an equation of a smooth cubic hypersurface for all $t\in \mathbf A$ such that $0<
\lvert t\lvert\,  <\hspace{-0.1cm}<1$.   For any $t$, one denotes by $F_t$ the Fano surface of lines in $X_t$ and 
$\mathscr F\subset G_1(\mathbf P^4)\times \mathbf A^1$  stands for the total space of the associated family of Fano surfaces.  We denote by 
$f$ the restriction along $\mathscr F$ of the second projection 
$G_1(\mathbf P^4)\times \mathbf A^1\rightarrow 
\mathbf A^1$.

Our goal is to establish a property (namely, the vanishing of the trace) of some global rational 1-forms on $F_0=\Sigma$ by using the fact that it can be obtained as a degeneration of 
the smooth Fano surfaces $F_t$ (with $0<
\lvert t\lvert\,  <\hspace{-0.1cm}<1$). For this reason, we are only interested in the geometry of the affine pencil $\mathscr F$ on a vertical (analytic) neighbourhood of 
the fiber $\Sigma=f^{-1}(0)$.  According to \cite[Prop.\,2.4]{Collino2}, near $t=0$ the pencil 
$\mathscr F$ is singular exactly at the points of $\Sigma$ denoted by 
$\ell_{ij,kl}$ above (and $L[(ij),(kl)]$ in Collino's paper). Moreover, 
these points, which are 45 in number, all  are ordinary quadratic singularities.  Then Collino explains that when performing a particular small blow-up at each of these isolated singular points, one gets a modified family 
$$ \tilde f= f\circ \nu \, : \, \widetilde{\mathscr F} 
\stackrel{\nu} {\rightarrow}  \mathscr  F 
\stackrel{f}{\rightarrow}
 \mathbf A
 $$
 enjoying nice properties. Indeed,  locally near the origin of the base, $ 
  \widetilde{\mathscr F} {\rightarrow}
 \mathbf A
$  is a semi-stable degeneration of surfaces: the total space $  \widetilde{\mathscr F}$ is 
smooth and the central fiber $\widetilde \Sigma= \tilde f^{-1}(0)=\nu^{-1}(\Sigma)$ is 
a union of smooth (rational) surfaces which have normal crossings. 
The surface $\widetilde \Sigma$ is a birational modification of $\Sigma$ whose abelian forms can be studied via deformations from those of the smooth  fibers $F_t$ 
for $t\neq 0$ sufficiently close to $0\in \mathbf A$, which was a priori not  doable  for $\Sigma$, this  surface not having normal crossings (see  \S\ref{SSS:Fano-Segre-cubic}). 

As mentioned in \S\ref{SSS:Fano-Segre-cubic}, the covering families of lines on $\boldsymbol{S}$ give rise to six components of $\Sigma$ isomorphic to $\overline{\mathcal M}_{0,5}$, denoted by $\Sigma(i)$ for $i=1,\ldots,6$.  If 
$\widetilde \Sigma(i)$ stands for the strict transform of  $\Sigma(i)$ in $  \widetilde{\mathscr F}$ under 
$\nu$, then it can by verified that $\nu: \widetilde \Sigma(i)\rightarrow  \Sigma(i)
\simeq 
\overline{\mathcal M}_{0,5}$ is an isomorphism. Let $\psi_i$ be the rational map associating  
to a generic point of $\boldsymbol{S}$ the line of $\Sigma(i)
$ passing through it. Up to the natural identification of $\boldsymbol{S}$ and $ \Sigma(i)$
with $\overline{\mathcal M}_{0,6}$  and $\overline{\mathcal M}_{0,5}$ respectivement, 
$\psi_i:
\boldsymbol{S}\dashrightarrow 
\Sigma(i)
$ corresponds to the map $\overline{\mathcal M}_{0,6}\dashrightarrow \overline{\mathcal M}_{0,5}$ consisting in forgetting the $i$-th point.  
\sk

The family $ 
  \widetilde{\mathscr F} {\rightarrow}
 \mathbf A
$  being a semi-stable deformation of $\widetilde \Sigma$, the general theory of 
\cite{Friedman} applies locally on the pointed germ of curve $\Delta=(\mathbf A,0)$: 
denoting by $\Omega^1_{\widetilde{\mathscr F}}\big({\rm Log}(\widetilde \Sigma\big) \big)$ the associated relative log complex, one sets 
$\Lambda^1_{\widetilde \Sigma}=
\Omega^1_{\widetilde{\mathscr F}}\big({\rm Log}(\widetilde \Sigma\big) \big)\big\lvert_{ 
\widetilde \Sigma}$. 
This sheaf of meromorphic 1-forms on $\widetilde \Sigma$ 
enjoys nice properties which we will now describe. We denote by $\partial \widetilde \Sigma$ the 1-boundary of $\widetilde \Sigma$, defined as the union of the irreducible curves
lying in the 
intersection  of two distinct components of $\widetilde \Sigma$.   For any irreducible component $\widetilde F$, one sets $\partial \widetilde F=\widetilde F\cap \partial \widetilde \Sigma$ and 
$\Omega^1_{\widetilde F}\big( 
{\rm Log}(\partial \widetilde F)
\big)$ stands for the sheaf of meromorphic forms on $\widetilde F$, holomorphic 
on $\widetilde F\setminus \partial \widetilde F$ and 
 with logarithmic poles along $\partial \widetilde F$.  
For any such $\widetilde F$, we denote the natural inclusion by $a: \widetilde F\hookrightarrow \widetilde \Sigma$. 
Then  from the general theory of Friedman, 
  one gets that $ \Lambda^1_{\widetilde \Sigma}$: 
\begin{enumerate}
\item[$(i).$]  is a locally free sheaf of $\mathcal O_{\widetilde \Sigma}$-modules;
\sk 
\item[$(ii).$] coincides with the sheaf 
 of abelian 1-forms on 
$\widetilde \Sigma$: one has $\Lambda^1_{\widetilde \Sigma}=\omega^1_{\widetilde \Sigma}$;
\sk 
\item[$(iii).$]  is such that there is an embedding of sheaves $\Lambda^1_{\widetilde \Sigma}
\hookrightarrow \bigoplus_{  \widetilde F} 
  a_* \Big(\Omega^1_{\widetilde F}\big( 
{\rm Log}(\partial \widetilde F)
\big)\Big)$ 
where the sum 
is taken over all the irreducible components $\widetilde F$ of 
$\widetilde \Sigma$ in the target sheaf.  
\end{enumerate}

But much more can be said in the  case under scrutiny. Indeed, from  \cite[\S3]{Collino2}, one gets: 
\begin{enumerate}
\label{(iv)}
\item[$(iv).$]  For any $i$, the pair $\big(\, \widetilde \Sigma(i)\, , \, \partial \widetilde \Sigma(i)\, \big)$ 
is isomorphic to $\big(\, \overline{\mathcal M}_{0,5}, \partial \overline{\mathcal M}_{0,5}\, \big)$ hence through $a$, one gets a map 
$\boldsymbol{H}^0\big({\widetilde \Sigma},  \omega^1_{\widetilde \Sigma}\big)
\rightarrow 
\boldsymbol{H}^0\big( 
\, \overline 
{\mathcal M}_{0,5}, 
\Omega^1_{\overline{\mathcal M}_{0,5}}\big({\rm Log}(
 \partial \overline{\mathcal M}_{0,5}
\big) \big)$ which is an isomorphism;
\mk 
\item[$(v).$] 
\vspace{-0.25cm}
Consequently, one has $h^0\big( \widetilde \Sigma, \omega^1_{\widetilde \Sigma}\big)=5$  
({\it cf.}\,the proof of \cite[Prop.\,3.1]{Collino2}
\sk
\item[$(vi).$] Hence 
$\nu_*\Big(\Omega^1_{\widetilde{\mathscr F}}\big({\rm Log}(\widetilde \Sigma\big) \big)\Big)$ is a locally free sheaf of rank 5 on $(\mathbf A,0)$. 
\sk 
\item[$(vii).$] 
The $\mathfrak S_6$-action lifts to $\tilde \Sigma$ and acts preserving $ \omega^1_{\widetilde \Sigma}$. Consequently $\boldsymbol{H}^0\big( \tilde \Sigma, \omega^1_{\widetilde \Sigma}
\big)$  naturally carries a structure of $\mathfrak S_6$-module.
\end{enumerate}

We obtain that the points {\bf 2.}\,when  $k=1$ and {\bf 3.b} of Proposition \ref{P:LW-X} 
hold true in the case when $X$ is specialized to Segre's hypercubic $\boldsymbol{S}$.
\begin{prop} 
{\rm 1.} For any $\eta\in \boldsymbol{H}^0\Big(\,  \widetilde \Sigma, \omega^1_{\widetilde \Sigma}\, \Big)$, its trace ${\rm Tr}(\eta)=\sum_{i=1}^{5} \psi_i^*(\eta)$ vanishes identically. 
\sk
\begin{enumerate}
\item[{\rm 2.}]
Hence there is an injective well-defined map 
\begin{align}
\label{Eq:Tr-(1)}
{\rm Tr}^{(1)}: \, \boldsymbol{H}^0\Big( \, \widetilde \Sigma \, , \,  \omega^1_{\widetilde \Sigma}\, \Big)  \rightarrow 
\boldsymbol{AR}^{(1)}\big( \boldsymbol{\mathcal W}_{0,6}\big)\, , \quad 
\eta  \longmapsto \Big( \psi_i^*\big(\eta\big)\Big)_{i=1}^6\, .
\end{align}
\sk
\item[{\rm 3.}] The image $\boldsymbol{AR}_{ab}^{(1)}
\big( \boldsymbol{\mathcal W}_{0,6}\big)={\rm Im}\big( 
{\rm Tr}^{(1)}
\big)$ is a subspace of dimension 5 of
$\boldsymbol{AR}^{(1)}
\big( \boldsymbol{\mathcal W}_{0,6}\big)$
  such that the wedge 
map $\wedge ^2 \boldsymbol{AR}_{ab}^{(1)}
\big( \boldsymbol{\mathcal W}_{0,6}\big)\rightarrow 
\boldsymbol{AR}^{(2)}
\big( \boldsymbol{\mathcal W}_{0,6}\big)$ is a well-defined isomorphism. 
\mk 
\item[{\rm 4.}] For any $i$,
the $i$-th projection  $\boldsymbol{AR}^{(1)}_{ab}
\big( \boldsymbol{\mathcal W}_{0,6}\big)
\rightarrow  \boldsymbol{AR}_{ab}^{(1)}
\big( \boldsymbol{\mathcal W}_{0,6}\big)[i]$ is an isomorphism where 
$\boldsymbol{AR}^{(1)}_{ab}
\big( \boldsymbol{\mathcal W}_{0,6}\big)[i]$ stands for the space of 
$i$-th components of the abelian 1-ARs of  $\boldsymbol{\mathcal W}_{0,6}$.  
\mk
\item[{\rm 5.}] 
 The map ${\rm Tr}^{(1)}$ induces an isomorphism of 
$\mathfrak S_6$-representations $
\boldsymbol{H}^0\big( \omega^1_{\widetilde \Sigma}\big)  \simeq 
\boldsymbol{AR}^{(1)}_{ab}\big( \boldsymbol{\mathcal W}_{0,6}\big)$.  
Hence $\boldsymbol{AR}^{(1)}_{ab}\big( \boldsymbol{\mathcal W}_{0,6}\big)$ is an irreducible $\mathfrak S_6$-module with associated Young symbol $\big[33\big]$.
\end{enumerate}
\end{prop}
\begin{proof}[Sketch of a proof]
Since the sheaf in $(v).$\,above is locally free, one can argue using a simple
deformation argument, 
similarly as in the proof 
of Proposition \ref{P:LW-X-singular} page \pageref{zolo}. 
The points 1., 2.\,and the first part of 3.\,(namely that $\wedge ^2 \boldsymbol{AR}_{ab}^{(1)}
\big( \boldsymbol{\mathcal W}_{0,6}\big)\rightarrow 
\boldsymbol{AR}^{(2)}
\big( \boldsymbol{\mathcal W}_{0,6}\big)$ is  indeed well-defined)  of the Proposition follow rather directly from this.  The point 4.\,is a direct translation into terms of 
$\boldsymbol{\mathcal W}_{0,6}$ of the point $(iv)$ above.

As for 5., the first part is easy and left to the reader. To get the second part, for any $i$ we consider  the induced representation of 
$\mathfrak S_5={\rm Fix}(i)\subset \mathfrak S_6$ on  
$\boldsymbol{H}^0\big(\widetilde \Sigma, \omega^1_{\widetilde \Sigma}\big)$. 
Then it can be verified that the map in $(iv).$ is an isomorphism of $\mathfrak S_5$-representation.  According to Lemma 2.1 of \cite{DolgachevFarbLooijenga}, 
$\boldsymbol{H}^0\big(\overline {\mathcal M}_{0,5}, 
\Omega^1_{\overline{\mathcal M}_{0,5}}\big({\rm Log}(
 \partial \overline{\mathcal M}_{0,5}
\big) \big)$ is an irreducible $\mathfrak S_5$-module, with associated Young symbol $[32]$. Considering the branching rules for representations of the symmetric groups, one deduces that  $
\boldsymbol{H}^0\big( \omega^1_{\widetilde \Sigma}\big) $ is an irreducible 
$\mathfrak S_6$-module with Young symbol $[33]$, which gives us the fifth point.  
\sk 

The single point remaining to be proved is that the wedge map of 3.\,is an isomorphism. 
Unfortunately, we do not have a conceptual argument to offer for this. This will follow from our explicit computations in the next part of this appendix. 
\end{proof}

Although interesting,  this is the expected result, which somewhat reduces its relevance. 
 Below we will recover most of its content by means of an elementary explicit approach. However, taking an abstract approach naturally leads to ask several questions, that we find interesting.
\begin{questions}
\begin{enumerate}
\item[1.] We have given a description of some 1-ARs of $\boldsymbol{\mathcal W}_{0,6}$ by means of the abelian 1-forms on $\widetilde \Sigma$. But this birational model of 
$\Sigma$ relies on some choices (namely, the way the singular points of 
$\mathcal F$ near $t=0$ are blown up) hence is non canonical. 
It would be interesting to have a more intrinsic description of the 1-ARs of $\boldsymbol{\mathcal W}_{0,6}$ under scrutiny. For any element $\tilde \omega 
\in 
\boldsymbol{H}^0\big( \widetilde \Sigma , \omega^1_{\widetilde \Sigma}\big) $, the push-forward $\omega=\nu_*(\tilde \omega )$ is a rational 1-form on $\Sigma$ whose restriction
 on each 
$\Sigma(i)$ is holomorphic on  $\Sigma(i)\setminus \partial \Sigma(i)\simeq \mathcal M_{0,5} $ with logarithmic singularities along 
$\partial \Sigma(i)=\partial \overline{ \mathcal M}_{0,5}$. Is such an $\omega$ an abelian differential on $\Sigma$? If so, does the trace of any abelian 1-form on $\Sigma$ vanish? Or in other terms, does $\nu_*$ give rise to an isomorphism $
\boldsymbol{AR}^{(1)}_{ab}
\big( \boldsymbol{\mathcal W}_{0,6}\big) \simeq 
\boldsymbol{H}^0\big( \widetilde \Sigma , \omega^1_{\widetilde \Sigma}\big) \rightarrow 
\boldsymbol{H}^0\big(\Sigma,\omega^1_\Sigma\big)
$?
\sk 
\item[2.]  According to \cite{DolgachevFarbLooijenga} (see the paragraph just before Lemma 2.1 therein), the pull-back under the composition 
${\Sigma}(i)\subset \Sigma\subset G_1\big(\mathbf P^4\big)$ of the 
dual of the tautological bundle  $\mathcal T$ on the grassmannian $G_1\big(\mathbf P^4\big)$ identifies with 
$\Omega^1_{{\Sigma}(i)}\big( {\rm Log}(\partial {\Sigma}(i))\big)$ which 
coincides with the restriction of $ \omega_{\widetilde \Sigma}^1$ on 
${\Sigma}(i)$. It is then  natural to wonder whether the isomorphism
$\big(\mathcal T^\vee\big)\lvert_{\Sigma_{reg}}\simeq \Omega^1_{\Sigma_{reg}}$ 
of  \cite[Theorem 1.10.ii]{AltmanKleiman} extends to an identification 
$\big(\mathcal T^\vee\big)\lvert_{\Sigma}\simeq \omega^1_{\Sigma}$  
on the whole $\Sigma$, although this Fano surface is singular. 
\end{enumerate}
\end{questions}

\subsubsection*{\bf A.2. An explicit description of the 
abelian 1-ARs of $\boldsymbol{\mathcal W}_{0,6}$}
We now turn to another approach to describe the abelian 1-ARs of $\boldsymbol{\mathcal W}_{0,6}$ which is much more elementary and explicit than the previous one. \mk 

The birational model of $\boldsymbol{\mathcal W}_{0,6}$ we are going to work with is the following: relatively to some rational coordinates $x_1,x_2,x_3$, we identify $\boldsymbol{\mathcal W}_{0,6}$ with the web defined by the six following rational first integrals: $U_1=(x_2,x_3)$, $U_2=(x_1,x_3)$, $U_3=(x_1,x_2)$ and 
$$
U_4=\left(\frac{\mathit{x_1} -1}{\mathit{x_3} -1},\frac{\mathit{x_2} -1}{\mathit{x_3} -1}
\right)\, , \qquad 
U_5=\left(
\frac{\mathit{x_1}}{\mathit{x_3}},\frac{\mathit{x_2}}{\mathit{x_3}}\right)
\quad 
\mbox{ and }
\quad 
U_6=\left(
\frac{\mathit{x_1} (\mathit{x_3} -1)}{\mathit{x_3} (\mathit{x_1} -1)},\frac{\mathit{x_2} (\mathit{x_3} -1)}{\mathit{x_3} (\mathit{x_2} -1)}
\right)\, .
$$
We denote by $U_{i,1}$ and $U_{i,2}$  the components of $U_i$. 
Recall that for any $n\geq 2$,  $A_n$ stands for the braid arrangement in $\mathbf C^n$ which is such that there is a natural identification 
\begin{equation}
\label{Eq:M0n-Cn-An}
 \mathbf C^n\setminus A_n\simeq \mathcal M_{0,n+3}
\end{equation} 
  Then each $U_i$ induces a regular map $\mathbf C^3\setminus A_3\rightarrow \mathbf C^2\setminus A_2$ which corresponds to the forgetful mapping $\mathcal M_{0,6}\rightarrow \mathcal M_{0,5}$ modulo the preceding identifications. 
 
Up to the identification \eqref{Eq:M0n-Cn-An} when $n=2$, the 
5-tuple of rational 1-forms in two variables
$$\big(\boldsymbol{\alpha}_s\big)_{s=1}^5=\left(\, 
\frac{du}{u}\, , \, \frac{dv}{v}
\, , \, \frac{du}{u-1}\, , \, \frac{dv}{v-1}\, ,  \, \frac{du-dv}{u-v}\, 
\right)
$$
is a basis of a complex vector space denoted by $\boldsymbol{{\rm Log}\Omega}^1_{{0,5}}$, 
which is naturally isomorphic to the space of 
global holomorphic 1-forms on ${\mathcal M}_{0,5}$ with logarithmic poles 
along $\partial \overline{\mathcal M}_{0,5}$.
 
Consequently,  for any $i=1,\ldots,6$, the 5-tuple of rational 1-forms in $x_1,x_2$ and $x_3$ 
$$ 
\big(\omega_{i,s}\big)_{s=1}^5 =
\big(U_i^*( \boldsymbol{\alpha}_s) \big)_{s=1}^5=
 \left( \, 
\frac{dU_{i,1}}{U_{i,1}}\, ,\, 
\frac{dU_{i,2}}{U_{i,2}}\, ,\, 
\frac{dU_{i,1}}{U_{i,1}-1}\, ,\, 
\frac{dU_{i,2}}{U_{i,2}-1}\, ,\, 
\frac{dU_{i,1}-dU_{i,2}}{U_{i,1}-U_{i,2}}
\, 
\right)$$
is a basis of a vector space  denoted by $U_i^*\boldsymbol{{\rm Log}\Omega}^1_{{0,5}}$ 
identifying with $\psi_i^*\Big( \boldsymbol{H}^0\big(\Sigma(i), \Omega^1_{ \Sigma(i)} 
\big(  {\rm Log}( \partial \Sigma(i)  \big) \big)\Big)$. It follows that the abelian 1-ARs of $\boldsymbol{\mathcal W}_{0,6}$ have to be looked for inside the space 
$$
\boldsymbol{{\rm Log}{\mathcal A\mathcal R}}^{(1)}=\left\{  \big( c_{i,s}\big)_{(i,s)\in J}\in \big( \mathbf C^5\big)^6\hspace{0.15cm}\Big\lvert 
\hspace{0.15cm} \sum_{i=1}^6 \sum_{s=1}^5 c_{i,s} \omega_{i,s}=0
\hspace{0.15cm}
\right\}
$$
where $J$ stands for the set of indices $J=\big\{ (i,s)\, \big\lvert \, i=1,\ldots,6, \, s=1,\ldots,5\, \big\}$.

\begin{prop}
\begin{enumerate}
\item[1.] The space $\boldsymbol{{\rm Log}{\mathcal A\mathcal R}}^{(1)}$ has dimension 21.\mk 
\item[2.] As a $\mathfrak S_6$-module, its decomposition into irreducibles is $
\boldsymbol{{\rm Log}{\mathcal A\mathcal R}}^{(1)}=V_{(3,2,1)}\oplus  V_{(3,3)}$.\mk 
\item[3.] The 5-dimensional component $V_{(3,3)}$ of $\boldsymbol{{\rm Log}{\mathcal A \mathcal R}}^{(1)}$ is precisely the image by the trace map ${\rm Tr}^{(1)}$ of the space of abelian differentials on 
$\widetilde \Sigma$: one has 
$V_{(3,3)}= 
{\rm Tr}^{(1)}\Big( \boldsymbol{H}^0\big(\widetilde \Sigma,\omega^1_{\widetilde \Sigma}\big)\Big)$.   
\mk
\item[4.]
A basis for $V_{(3,3)}=\boldsymbol{AR}^{(1)}_{ab}
\big( \boldsymbol{\mathcal W}_{0,6}\big)
$  can be explicited: for instance, 
such a basis is provided by the five 1-ARs associated to the explicit functional relations of Table \ref{T:1-RAs} below. 
\end{enumerate}
\end{prop}

Let us discuss how this proposition can be established.  Our proof relies on some explicit computations (performed on Maple) that we will not give here\footnote{We can make the corresponding Maple worksheets available to anyone requesting them.} but only describe. 
\begin{proof}
We set $\zeta_i=x_i$  for $i=1,2,3$ and $\zeta_4=0$, $\zeta_5=1$ and $\zeta_6=\infty$.  
For $k$ ranging from 1 to 3 and $l$ from $k+1$ to 5, we set $\zeta_{kl}=\zeta_k-\zeta_l$ and 
$\eta_{k,l}=d{\rm Log} \zeta_{kl}= (d\zeta_k-d\zeta_l)/(\zeta_k-\zeta_l)$.  The $\eta_{k,l}$'s are rational logarithmic 1-forms on $\mathbf C^3$ which are linearly independent over $\mathbf C$. Their span is  a complex vector space of dimension $9$, denoted by $\boldsymbol{{\rm Log}\Omega}^1_{0,6}$,  which is naturally isomorphic to 
$$
 \boldsymbol{H}^0\left(\overline{\mathcal M}_{0,6}, \Omega^1_{ \overline{\mathcal M}_{0,6}} 
\Big(  {\rm Log}\big( \, \partial \overline{\mathcal M}_{0,6}\, \big) \Big)\right) \, . 
$$

By definition of $\boldsymbol{{\rm Log}{\mathcal A\mathcal R}}^{(1)}$, one has a short complex of vector spaces
\begin{equation}
\label{Eq:SES-C-ev}
0\rightarrow 
\boldsymbol{{\rm Log}{\mathcal A\mathcal R}}^{(1)}\longrightarrow \oplus_{i=1}^6 U_i^*
\Big(\boldsymbol{{\rm Log}\Omega}^1_{\mathcal M_{0,5}}\Big)
\stackrel{\tau}{\longrightarrow} \boldsymbol{{\rm Log}\Omega}^1_{\mathcal M_{0,6}}\rightarrow 0
\end{equation}
where $\tau$ stands for the map given by $ (\omega_i)_{i=1}^6 \mapsto \sum_{i=1}^6 \omega_i$. By plain 
 linear algebra computations, it can be verified that the sequence  \eqref{Eq:SES-C-ev} is exact ({\it i.e.}\,$\tau$ is surjective) hence $\dim_{\mathbf C} \boldsymbol{{\rm Log}{\mathcal A\mathcal R}}^{(1)}=21$.  From now on, we will work with a fixed basis of $\boldsymbol{{\rm Log}{\mathcal A\mathcal R}}^{(1)}$,  denoted by $\mathfrak B$. 
\sk 

In order to prove the second point of the proposition, one starts by determining the character of $\boldsymbol{{\rm Log}{\mathcal A\mathcal R}}^{(1)}$ as a representation of $\mathfrak S_{6}$.  
We describe how we have proceeded. 
the following is a  complete set of representatives of the non trivial conjugacy classes of 
$\mathfrak S_{6}$:  
\begin{align*}
\mathscr S= \Big\{ \, 
 (12)\, , \, (12)(34)
\, , \, 
 & (12)(34)(56)
\, , \, 
(123)
\, , \, 
(123)(45)
\, , \,  \\
&(123)(456)
\, , \,  
 (1234)
\, , \, 
(1234)(56)\, , \,  (12345) \, , \,  (123456)\,  \Big\}\, . 
\end{align*}

We will use the following composition of group morphism
\begin{equation}
\label{Eq:pi-I-J}
\xymatrix@R=0.1cm@C=0.4cm{
\mathfrak S_6  \ar@{->}[r]& {\bf Bir}(\mathbf C^3)
\ar@{->}[r]&{\rm Aut}\Big( \boldsymbol{{\rm Log}{\mathcal A\mathcal R}}^{(1)}\Big)
\ar@{->}[r]& {\rm GL}_{21}\big(\mathbf C\big)
 \\
\sigma  \ar@{->}[r] & G_\sigma  \ar@{->}[r]  &G_\sigma^* 
\ar@{->}[r]  & M_\sigma
}
\end{equation}
with $G_\sigma$  being the birational map defined in  \eqref{Eq:G-sigma} and 
where 
$M_\sigma$ stands for the matrix of $G_\sigma^*$ acting on
$ \boldsymbol{{\rm Log}{\mathcal A\mathcal R}}^{(1)}$ expressed 
 in the basis $\mathfrak B$ ({\it i.e.}\,$M_\sigma={\rm Mat}_{\mathfrak B}\big(G_\sigma^* \big)$).
\mk

For any permutation $\sigma$, it is just a computational matter (that we have handled using Maple) to get an explicit expression for the birational map $G_\sigma$, hence for the matrix $M_\sigma$, from which one can take the trace ${\rm Tr}(M_\sigma)$. 
For each $\sigma$ in the set $\mathscr S$, $G_\sigma$ and ${\rm Tr}(M_\sigma)$ are given in Table \ref{Table=Tr(Msigma)} below.


\begin{table}[!h]
\rotatebox{00}{\scalebox{0.8}{ 
\begin{tabular}{|c|c|c|c|c|c|}
\hline
\begin{tabular}{c} \vspace{-0.3cm}\\
$\boldsymbol{\sigma}$ \vspace{0.1cm}
\end{tabular}
& 
$(12)$ 
& 
$(12)(34) $ 
& 
$(12)(34)(56)$ 
& 
$(123)$ 
& 
$(123)(45)$ 
 \\ \hline   
\begin{tabular}{c} \vspace{-0.3cm}\\
$\boldsymbol{G_\sigma}$ \vspace{0.1cm}
\end{tabular}
& 
$\Big(  \frac{x_i}{x_i-1}
\Big)_{i=1}^3$
&
$\Big(  {x_1}\, , \, 
\frac{x_1(x_2-1)}{x_2-x_1}\,, \, 
\frac{x_1(x_3-1)}{x_3-x_1}\,
\Big)$
& 
$\Big(  {x_1}\, , \, 
\frac{x_1(x_3-1)}{x_3-x_1}\,, \, 
\frac{x_1(x_2-1)}{x_2-x_1}\,
\Big)$
& 
$\Big(  \frac{x_i-1}{x_i}
\Big)_{i=1}^3$
& 
$ \Big(  \frac{x_1-1}{x_1},  \frac{x_3-1}{x_3}, 
 \frac{x_2-1}{x_2} \Big)$
\\ \hline  
\begin{tabular}{c} \vspace{-0.3cm}\\
$\boldsymbol{{\rm Tr}\big(M_\sigma\big)}$ \vspace{0.1cm}
\end{tabular}
& 1
& 1
& -3
& -1
& 1 
\\ \hline    \hline
\begin{tabular}{c} \vspace{-0.3cm}\\
$\boldsymbol{\sigma}$ \vspace{0.1cm}
\end{tabular}
& 
$(123)(456)$ 
& 
$(1234) $ 
& 
$(1234)(56)$ 
& 
$(12345)$ 
& 
$(123456)$ 
 \\ \hline   
 \begin{tabular}{c} \vspace{-0.3cm}\\
$\boldsymbol{G_\sigma}$ \vspace{0.1cm}
\end{tabular}
& 
$ \Big(  \frac{x_2-1}{x_2},  \frac{x_3-1}{x_3}, 
 \frac{x_1-1}{x_1} \Big)$
&
$ \Big(  \frac{1}{x_1},  \frac{x_2-1}{x_2-x_1}, 
 \frac{x_3-1}{x_3-x_1} \Big)$
& 
$ \Big(  \frac{1}{x_1},  \frac{x_3-1}{x_3-x_1}, 
 \frac{x_2-1}{x_2-x_1} \Big)$
&
$ \Big(  \frac{1}{x_2},  \frac{x_1-x_2}{x_2(x_1-1)}, 
 \frac{x_3-x_2}{x_2(x_3-1)} \Big)$
&
$ \Big(  \frac{1-x_2}{x_1-x_2},  \frac{1-x_3}{x_1-x_3}, 
 \frac{1}{x_1} \Big)$
\\ \hline  
\begin{tabular}{c} \vspace{-0.3cm}\\
$\boldsymbol{{\rm Tr}\big(M_\sigma\big)}$ \vspace{0.1cm}
\end{tabular}
& 2
& -1
& -1
& 0   
& 0
\\ \hline   
\end{tabular}
}}\bk 
\caption{}
\label{Table=Tr(Msigma)}
\end{table}

Table \ref{Table=Tr(Msigma)} characterizes entirely the character of 
$\boldsymbol{{\rm Log}{\mathcal A\mathcal R}}^{(1)}$ as a $\mathfrak S_6$-module, denoted by 
$\chi^{(1)}$.  We write 
 $\chi^{(1)}=\sum_{\lambda \vdash 6} n_\lambda\,  \chi_{\lambda}$ for some 
non negative integers $n_\lambda $,  
where $\chi_\lambda$ stands for the character of the irreducible $\mathfrak S_6$-representation $V_\lambda$ associated to  the partition $\lambda$. Using \href{https://groupprops.subwiki.org/wiki/Linear_representation_theory_of_symmetric_group:S6}{the character table of $\mathfrak S_6$} \href{https://groupprops.subwiki.org/wiki/Linear_representation_theory_of_symmetric_group:S6}{-representations}, it is not difficult to conclude that necessarily 
$\chi^{(1)}=\chi_{(3,3)}\oplus \chi_{(3,2,1)}$.  
 The second point of the proposition follows immediately from that. 
Since $V_{(3,2,1)}$ has dimension 16 and because ${\rm Tr}^{(1)}\big( \boldsymbol{H}^0\big(\widetilde \Sigma,\omega^1_{\widetilde \Sigma}\big)\big)$ is a $\mathfrak S_6$-submodule of $\boldsymbol{{\rm Log}{\mathcal A\mathcal R}}^{(1)}$ of dimension 5, 
we get point 3.
\mk

The preceding arguments used to show that $
\boldsymbol{AR}_{ab}^{(1)}\big( 
\boldsymbol{\mathcal W}_{0,6}\big)=V_{(3,3)}$ are fully non constructive and it is interesting and natural to seek for a basis of this 5-dimensional space.  Such a basis is provided by the 
1-ARs associated to the five functional identities of Table \ref{T:1-RAs} below.  We were able to determine these explicit functional relations using the following approach: first, 
thanks to formula \eqref{Eq:G-sigma}, the group morphism $\mathfrak S_{6}\rightarrow
{\bf Bir}(\mathbf C^3)$, $\sigma\mapsto  G_\sigma$ can be made explicit.  
Then using Maple's {\it `DifferentialGeometry'} package, one can compute the action of 
the pull-back by $G_\sigma$ on  $\boldsymbol{Log{\mathcal A\mathcal R}}^{(1)}$ with respect to the basis $\mathfrak B$. Eventually we   computed (on a computer) the $21\times 21$ square matrix $M_\sigma$ for any of the $6!=720$ permutations  $\sigma$ element of  $ \mathfrak S_6$.   

Next, for $\chi=\chi_{(3,3)}$ or $\chi=\chi_{(3,2,1)}$, we have  considered the matrix
$$
\Pi_{\chi}=\frac{\chi(
{\bf 1})}{\lvert \mathfrak S_6\lvert}
\sum_{\sigma\in \mathfrak S_6} \chi\big(\sigma^{-1}\big) \, M_\sigma \in {\rm Mat}_{21}(\mathbf C)
$$
(where ${\bf 1}$ stands for the identity element of $\mathfrak S_6$). From
a classical result of 
 the general theory of representations of finite group (see \cite[Chap.\,XVIII,\S4]{Lang}  for instance), we know that $\Pi_{\chi}$  is the matrix in the basis $\mathfrak B$ of the projection map 
$ \boldsymbol{{\rm Log}{\mathcal A}}^{(1)}=V_{(3,2,1)}\oplus  V_{(3,3)}\rightarrow V_\chi$.  
Since $\chi_{(3,3)}$ was known and because all the $M_\sigma$ were previously computed, we have been able to construct $\Pi_{\chi_{(3,3)}}$ explicitly.  Considering  suitable linear combinations of its columns led us to the 1-ARs associated to the five 
functional relations of Table \ref{T:1-RAs}.
\end{proof}

\begin{rem}
1. Using the same computational approach as in the preceding proof, one can show that 
the character $\chi_9$ of the 9-dimensional 
$\mathfrak S_6$-representation $\boldsymbol{{\rm Log}\Omega}^1_{0,6}$ is given by the following table, 
where the symbols in the first line stand for the possible decompositions into cycles with pairwise distinct supports of the 
conjugacy classes in $\mathfrak S_6$ (that is $\langle 2\rangle $ means a cycle of length 2, $\langle 2  3\rangle$ stands for a product of  a 2-cycle with a 3-cycle, etc.): \bk 
$$
\begin{tabular}{|c|c|c|c|c|c|c|c|c|c|c|}\hline
$\boldsymbol{\sigma}$ & $\langle 2\rangle$ & 
$\langle 2  2\rangle $ & 
$ \langle 2 2 2\rangle$ 
& $ \langle 3\rangle$ &  
$ \langle 2 3\rangle$ 
& $ \langle 3 3\rangle$ & 
$\langle 4\rangle$ 
& 
$ \langle 2 4\rangle$ 
& 
$\langle 5\rangle$ & 
$\langle 6\rangle$
 \\ \hline   
$\boldsymbol{\chi_9(\sigma)}$ & $3$&  $1$&$ 3$& $0$ & $0$ & $0$ & $-1$ & $1$ & $-1$& $0 $
  \\ \hline   
\end{tabular}\bk 
$$
We see that $\chi_9$ coincides with the character associated to the partition $(4,2)$, from which it follows that $\boldsymbol{{\rm Log}\Omega}^1_{0,6}$ is isomorphic to the corresponding Specht module $V_{(4,2)}$. \mk\\ 
\end{rem}


\begin{table}[h!]
\scalebox{1.08}{
\begin{tabular}{l}
${}^{}$\hspace{-0.4cm} $0= \,  \ln \Big(U_{1,1}\Big)+\ln \left( \frac{U_{2,1}}{U_{2,2}}\right)-\ln \Big(U_{3,2}\Big)+\ln \left( \frac{U_{4,1}-U_{4,2}}{U_{4,1}(U_{4,2}-1)} \right)-\ln \Big(U_{5,1}\Big)+\ln 
\left( \frac{U_{6,2}-1}{U_{6,1}-U_{6,2}}\right)$\bk 
\\ 
${}^{}$ \hspace{-0.4cm} $0= \ln \Big(U_{1,2}\Big)-\ln \Big(U_{2,2}\Big)+\ln \left( \frac{U_{3,1}}{U_{3,2}}
\right)
+\ln \left( \frac{U_{4,2} (U_{4,1}-1)}{U_{4,1}(U_{4,2}-1)}
\right)
+\ln \left( \frac{U_{5,2}}{U_{5,1}}
\right)
+\ln \left( \frac{U_{6,2}-1}{U_{6,1}-1}
\right)$  \bk \\ 
${}^{}$ \hspace{-0.4cm} $0=   \ln \big(U_{1,1}-1\big)
+\ln \left( \frac{U_{2,1}-1}{U_{2,2}-1}\right)
-\ln \big(U_{3,2}-1\big)
-\ln \big(U_{4,1}\big)
+\ln \left( \frac{U_{5,1}-U_{5,2}}{U_{5,1}(U_{5,2}-1)}\right)
+\ln \left( \frac{U_{6,1}(U_{6,2}-1)}{ U_{6,1}-U_{6,2}}\right)$ 
\bk 
\\ 
${}^{}$ \hspace{-0.4cm} $0=  \ln \big(U_{1,2}-1\big)-\ln \big(U_{2,2}-1\big)
+\ln \left( \frac{U_{3,1}-1}{U_{3,2}-1}\right)
+\ln \left( \frac{U_{4,2}}{U_{4,1}}\right) 
+\ln \left( \frac{U_{5,2}(U_{5,1}-1)}{U_{5,1}(U_{5,2}-1)}\right)
+\ln \left( \frac{U_{6,1}(U_{6,2}-1)}{U_{6,2}(U_{6,1}-1)}\right)$
\bk 
\\ 
${}^{}$ \hspace{-0.8cm} $-I\pi 
= 
\ln \Big(U_{1,1}-U_{1,2}\Big)
+\ln \left( \frac{U_{2,1}-U_{2,2}}{
U_{2,2}(U_{2,2}-1)}\right)
+\ln \left( \frac{U_{3,1}-U_{3,2}}{U_{3,2}(U_{3,2}-1)}\right)
+\ln \left(\frac{U_{4,2}}{U_{4,1}(U_{4,2}-1)}\right)
+\ln \left(\frac{U_{5,2}}{ U_{5,1}(U_{5,2}-1)}\right) 
$ \bk \\ 
${}^{}$ \hspace{10cm}$ +\ln \left(\frac{U_{6,1} (U_{6,2}-1)}{ 
(U_{6,1}-1)(U_{6,1}-U_{6,2})}\right)$. \bk
\end{tabular}}
\caption{A basis (in functional form) of the space of 1-abelian relations of
$\boldsymbol{\mathcal W}_{0,6}$ coming from abelian 1-differentials on 
the birational model $\widetilde \Sigma$ of the 
the Fano surface 
$\Sigma$ 
 of Segre's cubic $\boldsymbol{S}$.}
 \label{T:1-RAs}
\end{table}
%
%
%
%

\noindent {\it 2.   For any $s=1,\ldots,5$,  let $\boldsymbol{\mathcal AR}_s \in 
 \boldsymbol{AR}^{(1)}_{ab}\big(  \boldsymbol{\mathcal W}_{0,6}  \big)
$
   be the 1-AR of $\boldsymbol{\mathcal W}_{0,6}$ corresponding to the functional relation   of the $s$-th line in Table  \ref{T:1-RAs}.   
Each    $\boldsymbol{\mathcal AR}_s$ is a 6-tuple  $(\boldsymbol{ar}_{s,i}(U_{i,1},U_{i,2}))_{i=1}^6$ of functions of two variables, such that the sum $\sum_{i=1}^6 
\boldsymbol{ar}_{s,i}(U_{i,1},U_{i,2})$ is (locally) constant.
For any $i$, let   $\Psi_i : \boldsymbol{AR}^{(1)}_{ab}\big(  \boldsymbol{\mathcal W}_{0,6}  \big) 
\rightarrow \boldsymbol{{\rm Log}\Omega}^1_{0,5}
$ 
be the linear map obtained as the following composition 
$$
  \xymatrix@R=0.4cm@C=0.7cm{ 
 \boldsymbol{AR}^{(1)}_{ab}\big(  \boldsymbol{\mathcal W}_{0,6}  \big) 
\ar@{->}[r] &
 \boldsymbol{H}^0\Big(\, \widetilde \Sigma\, ,\, \omega^1_{\widetilde \Sigma}\, \Big)
 \ar@{->}[r] &
  \boldsymbol{H}^0\left(\,  \Sigma(i)\, ,\,  \Omega^1_{\Sigma(i)}\Big( {\rm Log}\big(\partial \Sigma(i)\big)
 \Big)  \right)  \eq[d] &  \\
 & &  \boldsymbol{H}^0\left(\overline{\mathcal M}_{0,5}, \Omega^1_{ \overline{\mathcal M}_{0,5}} 
\Big(  {\rm Log}\big( \partial \overline{\mathcal M}_{0,5}\big) \Big)\right) 
\ar@{->}[r]  & 
\boldsymbol{{\rm Log}\Omega}^1_{0,5}
}
$$
where the first map is the inverse of \eqref{Eq:Tr-(1)},  the second is that of point $(iv)$ page 
\pageref{(iv)}, the third (vertical) and the last maps being the natural ones respectively induced by the isomorphisms of pairs $\big(\, \Sigma(i), \partial  \Sigma(i)\, \big)\simeq \big( 
\overline{\mathcal M}_{0,5},  \partial \overline{\mathcal M}_{0,5}\, \big) \simeq \big( \, \mathbf C^2, A_2\big)$.  Then $\Psi_i$ is an isomorphism which 
is fully characterized by the relations 
$\Psi_i\big(\boldsymbol{\mathcal AR}_s\big)=
d\big(\boldsymbol{ar}_{s,i}(u,v)\big)\in \boldsymbol{{\rm Log}\Omega}^1_{0,5}$ for  any $s$.
 In particular, one has $\Psi_1\big(\boldsymbol{\mathcal AR}_s\big)=\boldsymbol{\alpha}_s$ for any $s=1,\ldots,5$. 
}

\newpage

\newpage
\section*{\bf Appendix B: the 1-ARs of  the 3-web of lines on the chordal cubic.}
In this appendix, we study the 
1-ARs of the web of lines on the chordal cubic $\mathscr C$ (already discussed in \S\ref{SS:je-sais-pas-quoi} above) and explain how this web can be seen as a particular case of a nice family of curvilinear 3-webs studied by Blaschke and Walberer. We will take advantage of this to recall some of their results which are among the most striking ones in web geometry although not very well known by modern geometers. 

\subsection*{\bf B.1.  the 1-ARs of  the 3-web of lines on the chordal cubic.}
As explained in \S\ref{SS:je-sais-pas-quoi}, 
the web by lines on $\mathscr C$ is degenerated: 
$ \boldsymbol{\mathcal L\hspace{-0.05cm}\mathcal W}_{\boldsymbol{\mathscr C}}$ is a skew 3-web, and not a 6-web as for most of the cubic threefolds. However, one can ask about the ARs of this web. It turns out that concerning 1-ARs, this web is not less interesting than those associated to smooth cubic hypersurfaces.  We use below the notations introduced  in \S\ref{SS:je-sais-pas-quoi}: 
$U''_+$, $U''_-$ and $U'$ stand for the explicit rational first integrals of (a birational model of) $\boldsymbol{\mathcal L\hspace{-0.05cm}\mathcal W}_{\hspace{-0.05cm}{\mathscr C}}$ given in \eqref{Eq:UUU} and $X''_{+}$, $X''_{-}$ and $X'$ denote the associated vector fields defined just after. 
\sk 

According to a classical bound due to Kh\"aler,\footnote{That this bound is due to Kh\"aler is mentioned in \cite[Satz $\boldsymbol{S}_3$]{Blaschke-1-Rank}. A published proof is given in \cite[\S2]{BlaschkeK}.} one has 
${\rm rk}^{(1)}\big(  \boldsymbol{\mathcal L\hspace{-0.05cm}\mathcal W}_{\boldsymbol{\mathscr C}}\big)\leq 5$ and the question is whether this  is actually an equality. 
This can be answered by determining an explicit basis for $\boldsymbol{AR}^{(1)}\big( \boldsymbol{\mathcal L\hspace{-0.05cm}\mathcal W}_{\boldsymbol{\mathscr C}} \big)$.  Since 1-ARs can be integrated into functional identities, one has to find a basis of 
non trivial 3-tuples  of holomorphic functions of two variables $(M,N,R)$ such that the functional relation
\begin{equation}
\label{E:RAmunurho}
M\big(U''_+\big)+N\big(U''_-\big)+R\big(U'\big)=0
\end{equation}
is identically satisfied.
In \S\ref{SS:En-n-even}, we  briefly explained how 
Abel's method described in \cite{PirioSelecta} for determining in an effective way the ARs of planar webs can be generalized to the determination of the 2-ARs of curvilinear webs (in dimension 3). 
It turns out that Abel's method generalizes quite straightforwardly to the case of the 1-ARs of such webs as well, and we will apply this 
quite effectively 
to the web under scrutiny. 

We denote  the components of $U_-''$ by $U_{-i}''$ with $i=1,2$ and $N_i$ stands for the partial derivative of $N$ with respect to the $i$-th variable.  
Then one proceeds as follows for determining the solutions $(M,N,R)$ of \eqref{E:RAmunurho}: one first applies $X'$ to the LHS of 
\eqref{E:RAmunurho} in order to eliminate $R(U')$. 
Next one divides by $X'(U''_{-2})$ and  applies $X''_-$ : this kills the remaining $N_2(U_-'')$. Similarly, one divides 
   the expression just obtained by the coefficient of $N_1(U''_-)$ in it and then one applies $X_-''$ again: we have eliminated all the terms in $N$ and $R$ (and in their partial derivatives) thus what remains can be seen as a PDE in $M(U''_+)$ with variable  coefficients. This gives us 
 a  system of linear PDEs of the third order in $M$ which is not difficult to solve with the help of a computer algebra system.
 
  One gets that, up to the addition of a complex constant,   any function $M$ appearing in an 1-AR 
 \eqref{E:RAmunurho} is  necessarily a linear combination of the five following functions $M_i=M_i(u_1,u_2)$:  
 $$  M_1=\frac{1}{u_1-u_2} \, , \quad 
   M_2=\frac{2 u_1 + u_2}{u_1-u_2}    
   \, , \quad   M_3=\frac{1}{u_1(u_1-u_2)} \, , \quad 
M_4=  \frac{u_1u_2}{u_1-u_2}\quad \mbox{ and } \quad M_5= \frac{u_1u_2^2}{u_1-u_2}\, .
  $$
  One verifies that the same holds true for any function $N$ appearing in 
 \eqref{E:RAmunurho} :  up to a constant, it is a linear combination of the five previous functions. \mk 
  
Similarly, one obtains that, up to a constant,   any function $R=R(t_1,t_2)$ appearing in an identity 
 \eqref{E:RAmunurho}  is a linear combination of the following five rational functions $R_i=R_i(t_1,t_2)$: 
 $$  
   R_1=\frac{2(t_1+t_2)}{(t_1-t_2)^2} 
    \, , \hspace{0.2cm}  
    R_2= \frac{4\big(t_1^2 + t_1t_2 + t_2^2\big)}{(t_1-t_2)^2} \, , \hspace{0.2cm}
    R_3=\frac{4}{(t_1-t_2)^2}  
     \, , \hspace{0.2cm}
R_4=  \frac{2t_1t_2(t_1+t_2)}{(t_1-t_2)^2}
 \, ,  \hspace{0.2cm}  R_5=\frac{4t_1^2t_2^2}{(t_1-t_2)^2} \, .   $$
  Then from elementary computations, one obtains that 
  the functional identity  
  \begin{equation}
 M_i\big(U''_+\big)+M_i\big(U''_-\big)+R_i\big(U'\big)=0
\end{equation}
holds true  identically  
 for any $i=1,\ldots,5$.  One then verifies that the five tuples $(M_i,M_i,R_i)$ are linearly 
 independent. Considering K\"ahler's bound above, this implies that 
 ${\rm rk}^{(1)}\big( \boldsymbol{\mathcal L\hspace{-0.05cm}\mathcal W}_{{\mathscr C}}
 \big)=5
 $ hence $\boldsymbol{\mathcal L\hspace{-0.05cm}\mathcal W}_{{\mathscr C}}$ has maximal 1-rank.\sk 
 
   Considering Proposition A.1, it is natural to ask whether the 1-ARs of this web can be described by means of some abelian differential 1-forms  of the Fano surface $F_1({\mathscr C})$ (or possibly of a suitable birational model of it) or not. But to do so one has to face some difficulties: 
 \begin{itemize}
 \item first, since $F_1(\mathscr C)$ is not reduced as a scheme, what could be an abelian differential on it is not clear, at least to the author;
 \footnote{However, the contents 
 \cite{AL} may allow this point to be answered.} 
  \sk
  \item  assuming that the first point might be settled, one expects an abelian 1-form on $F_1(\mathscr C)$ to deform into a holomorphic 1-form on the Fano surfaces $F_1(X_t)$ of 
 any 
  regular analytic 
 family of smooth hypercubics $X_t\subset \mathbf P^4$ (with $t\in (\mathbf C^*,0)$)
 degenerating onto $X_0=\mathscr C$; 
 \sk
   \item finally, for a smoothing family as just above, one has to understand how the  $6$-webs $\boldsymbol{\mathcal L\hspace{-0.05cm}\mathcal W}_{X_t}$ for  $t\in (\mathbf C^*,0)$ degenerate 
   to the 3-web $\boldsymbol{\mathcal L\hspace{-0.05cm}\mathcal W}_{{\mathscr C}}$ when $t\rightarrow 0$.
 \sk
  \end{itemize}
Addressing these points here would require too much space. Instead, we will 
 discuss in the next sub-appendix below another way to understand  the ARs of 
 $\boldsymbol{\mathcal L\hspace{-0.05cm}\mathcal W}_{{\mathscr C}}$, 
 by considering this web from another point of view. \sk

 Finally, note that since the 2-rank of any skew curvilinear 3-web in dimension 3 is zero according to \cite[${\mathbf S}_{23}$]{Blaschke-1-Rank}, 
 there is no point in considering the 2-ARs of $\boldsymbol{\mathcal L\hspace{-0.05cm}\mathcal W}_{{\mathscr C}}$  hence  we will not talk more about it.

\subsection*{\bf B.2.  Blaschke-Walberer theory and its application to ${\mathcal L\hspace{-0.05cm}\mathcal W}_{{\mathscr C}}$}
We briefly discuss the beautiful theory established  in  \cite{BW} before explaining in {\bf B.2.2} how   $\boldsymbol{\mathcal L\hspace{-0.05cm}\mathcal W}_{{\mathscr C}}$  can be studied using it.  In the last and short subsection {\bf B.2.3}, we will finish 
by describing another approach to study this web, which is similar to that considered in {\bf A.1} and relies on results due to Collino as well. 
 
\subsubsection*{\bf B.2.1.  Blaschke-Walberer theory}
 In \cite{BW}, Blaschke and Walberer associate a curvilinear 3-web to a sufficiently generic  cubic hypersurface and study its properties,  in particular its 1-ARs.  \sk

 \vspace{-0.5cm}
 Their construction is as follows: 
let  $X\subset \mathbf P^4$ be a given 
 irreducible cubic hypersurface and denote by $F=F_1(X)\subset G_1(\mathbf P^4)$ its Fano surface. In the following we do not consider $F$ with its structure of scheme: we actually work with the underlying reduced surface $F_{red}$, that we will still denote abusively by $F$ to simplify the writing. By definition, the associated (labeled) {\bf `triangle variety'}
 $\overline{\Delta}=\overline{{\Delta}(X)}$
 is  the closed algebraic subvariety of $F^3$ formed by 
 `triangles included in $X$', namely 
 triples $T=(\ell_1, \ell_2,\ell_3)$ of lines included in $X$ such that there exists a 2-plane $\pi \in G_2(\mathbf P^4)$ 
 satisfying  $X\cdot\pi=\ell_1+\ell_2+\ell_3$ (equality between 1-cycles on $X$).  
 Such a triangle $T=(\ell_i)_{i=1}^3$ is `{\bf non-degenerate}' if the $\ell_i$'s are pairwise distinct with their union in $\mathbf P^4$ abstractly isomorphic to the model triangle
  formed by three  non concurrent lines  in $\mathbf P^2$ (see the picture below).

\begin{figure}[h!]
\centering
\includegraphics[width=30mm]{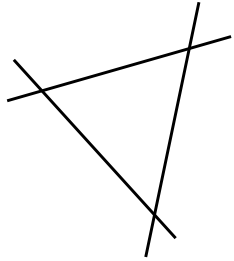}
\caption{A non degenerate triangle in the plane.}
\label{fig:method-triangle}
\end{figure}

One verifies that $\overline{\Delta}\subset F^3$ is  of dimension 3 and that the non-degenerate triangles form a Zariski open subset, denoted by 
${\Delta}={\Delta}(X)$, that we will also call abusively the `triangle variety' of $X$ when it is not empty.  By restriction, the natural projections $F^3\rightarrow F$ onto the three distinct factors give rise to three dominant rational maps $P_i: \overline{\Delta}\dashrightarrow F$ which define three foliations by algebraic curves on $\Delta$. These foliations form what we call {\bf Blaschke-Walberer triangle web} on $\Delta$, denoted by 
\begin{equation}
\label{Def:DW-X}
\boldsymbol{\Delta\hspace{-0.05cm}\mathcal W}_{\hspace{-0.03cm}X}=
\boldsymbol{ \mathcal W}\big( \, P_1\, , \, P_2\, , \, P_3\, \big) \, .
\end{equation}

Let us say that  $X$ is {\bf $\boldsymbol{\mathcal W}$-admissible} if for $T\in \overline{\Delta}$ generic: 
(1) $T$ is a non-degenerate triangle, {\it i.e.}\, ${\Delta}$ is non empty (hence it is dense in $\overline{\Delta}$);   (2)  the tangent spaces at $T$ of the level subsets of the $p_i$'s through this point are in general position,  {\it i.e.}\,$\boldsymbol{\Delta\hspace{-0.05cm}\mathcal W}_{\hspace{-0.03cm}X}$ is a genuine curvilinear web on $\Delta$ at $T$; and (3) this web is (generically) skew.

\begin{exm}
{\rm 1.} A smooth cubic threefold is $\boldsymbol{\mathcal W}$-admissible (see the proposition just below).\sk

\noindent {\rm 2.}  On the contrary, Segre's cubic $\boldsymbol{S}$ is not $\boldsymbol{\mathcal W}$-admissible.
 This can be verified easily using its model in $\mathbf P^3$ described in \S\ref{SSS:W0n+3} (see especially \eqref{Eq:Exceptional-web}). 
Let us describe the triangles in $\boldsymbol{S}\simeq \mathbf P^3$ passing through a generic point $q$. Such a triangle is determined by two distinct lines passing through $q$. In the model we are working with, two such lines correspond to the lines 
$\ell_s(q)=\langle q,p_s\rangle\subset  \mathbf P^3$ for $s=1,2$ say.  The third edge  of the triangle determined by $(\ell_1(q),\ell_2(q))$ is the 
line $l_3(q)=\langle q,p_1,p_2\rangle \cap \Pi_{123}$ 
with  $ \Pi_{123}=\langle p_3,p_4,p_5\rangle\simeq \mathbf P^2$. 
The map $q\mapsto l_3(q)$ is the map $P_3$ appearing in \eqref{Def:DW-X}.  It takes values into the dual projective plane 
$\check{\Pi}_{123}$ of lines contained in 
${\Pi}_{123}$. 

Noticing that any line $l_3(q)$ passes through the intersection point 
 of   
$\langle p_1,p_2\rangle$ with ${\Pi}_{123}$, we get that actually ${\rm Im}(P_3)$ is a line in $\check{\Pi}_{123}$. This map is then of rank 1 and its generic fiber is a surface (actually, the fiber of $P_3$ through $q$ is clearly the 2-plane $\langle q,p_1,p_2\rangle$). Thus $P_3$ defines a foliation by 2-planes,  hence 
$\boldsymbol{\Delta\hspace{-0.05cm}\mathcal W}_{\hspace{-0.03cm}\boldsymbol{S}}$ 
is not a curvilinear web, which shows that $\boldsymbol{S}$ is not 
$\boldsymbol{\mathcal W}$-admissible. 

\noindent {\rm 3.} 
But the property to be $\boldsymbol{\mathcal W}$-admissible for a cubic threefold is not directly related to the fact that it is singular or not. For instance, the chordal cubic $\mathscr C$ is 
$\boldsymbol{\mathcal W}$-admissible although having a singular locus of dimension 1. 
\end{exm}

Remark that one could have worked without labeling the edges of the triangles included in $X$. This would have given a more intrinsic construction of a variety of `unlabeled triangles' and of a non-ordered version of $\boldsymbol{\Delta\hspace{-0.05cm}\mathcal W}_{\hspace{-0.03cm}X}$. Since we are interested in local analytic properties of this web, it is pointless and actually much more convenient to work on $\Delta_X\subset F^3$  since on it, the web we are interested in is  defined by global rational first integrals. \sk

\begin{prop} Let $X$ be a smooth cubic threefold in $\mathbf P^4$. 
\label{P:BW-X-smooth}
\begin{enumerate}
\item[1.]  
\vspace{-0.15cm}
The hypersurface $X$ is  $\boldsymbol{\mathcal W}$-admissible. Moreover, 
the skew 3-web $\boldsymbol{\Delta\hspace{-0.05cm}\mathcal W}_{\hspace{-0.03cm}X}$ is linearizable.
\mk
\item[2.]   The first trace map ${\rm Tr}^{(1)}: \omega\mapsto \sum_{i=1}^3 P_i^*(\omega)$  induces a well-defined  isomorphism 
\begin{align*}
{\rm Tr}^{(1)}\, :\, \boldsymbol{H}^0\left(F, \Omega^1_{F}\right)& \longrightarrow \boldsymbol{AR}^{(1)}\big( \boldsymbol{\Delta\hspace{-0.05cm}\mathcal W}_{X} \big)\,. 
\end{align*}
\item[3.]  
Consequently ${\rm rk}^{(1)}\big(  \boldsymbol{\Delta\hspace{-0.05cm}\mathcal W}_{X}\big)=h^0\big(F,\Omega^1_{F}\big)=5$ hence the 3-web  
$\boldsymbol{\Delta\hspace{-0.05cm}\mathcal W}_{X}$
has maximal 1-rank.
  \end{enumerate}
\end{prop}
\begin{proof}[Proof (sketched)]
Assuming that $X$ is smooth, one verifies easily that $\Delta_X$ is non empty and generically reduced and the fact that $\boldsymbol{\Delta\hspace{-0.05cm}\mathcal W}_{X}$ is a genuine web can be obtained using the same arguments as in the  proof of Lemma \ref{L:X-6lines-PG}. 
As for the skewness of $\boldsymbol{\Delta\hspace{-0.05cm}\mathcal W}_{X}$, it will be established  later (see Lemma \ref{Lem:lemo}.2).  And a nice geometric argument gives its linearizability. 
For $T\in \Delta_X$, one denotes by $\ell_i^T\in F$ its edges (with $i=1,2,3$) and one sets $\langle T\rangle$ for the 2-plane in $\mathbf P^4$ spanned by the $\ell_i^T$'s.  We 
fix a generic linear projection $\pi: \mathbf P^4\dashrightarrow \mathbf P^3$. 
Then for $T_0\in \Delta_X$ generic, one considers 
the germ of analytic map $\Xi : (\Delta_X,T_0)\rightarrow  \check{\mathbf P}^3 : 
\, T \mapsto \pi(\langle T\rangle)$. One verifies that it is a local biholomorphism. Moreover, 
for a generic line $\ell\in G_1({\mathbf P}^4)$,  $\pi(\ell)$ is a line in $\mathbf P^3$ whose projective dual, denoted by 
$\pi(\ell)^\vee$, 
 is a line in  $\check{\mathbf P}^3$  as well. For any $T=(\ell_i)_{i=1}^3 \in 
 (\Delta_X,T_0)$, the three lines $\pi(\ell_i)^\vee$ for $i=1,2,3$ are concurrent at  $\Xi(T)$, and 
 are easily seen to be the leaves of the push-forward of 
$ \boldsymbol{\Delta\hspace{-0.05cm}\mathcal W}_{X}$
by 
$\Xi$.  Thus $\Xi_*\big(  \boldsymbol{\Delta\hspace{-0.05cm}\mathcal W}_{X}\big)$ is a germ of linear 3-web on $\check{\mathbf P^3}$ at $\Xi(T_0))$ hence,  in particular,  $\boldsymbol{\Delta\hspace{-0.05cm}\mathcal W}_{X}$ is linearizable. 
\sk 

The third and last points of the proposition follow immediately from the second which itself is a simple rephrasing of the last statement of 
\eqref{Eq:a-Isom}
in terms of webs.
\end{proof}
\sk 

A remarkable result obtained by Blaschke and Walberer in their paper is that the 
web-theoretic content of the preceding proposition actually holds true for any cubic threefold  
as soon as its web of triangles  is $\boldsymbol{\mathcal W}$-admissible: 


\begin{prop}[Blaschke-Walberer] 
\label{Prop:Blaschke-Walberer}
Let $X$ be an irreducible  $\boldsymbol{\mathcal W}$-admissible cubic in $\mathbf P^4$. 
\begin{enumerate}
\item[1.]  
\vspace{-0.15cm}
There exists a certain 5-dimensional space $\boldsymbol{H}_F$ of rational 1-forms on $F$ whose traces vanish. This gives rise to an isomorphism ${\rm Tr}^{(1)}\, :\, \boldsymbol{H}_F \longrightarrow \boldsymbol{AR}^{(1)}\big( \boldsymbol{\Delta\hspace{-0.05cm}\mathcal W}_{X} \big)\,. $
\sk 
\item[2.]  
Consequently ${\rm rk}^{(1)}\big(  \boldsymbol{\Delta\hspace{-0.05cm}\mathcal W}_{X}\big)=\dim \big( \boldsymbol{H}_F \big)=5$ hence the 3-web  
$\boldsymbol{\Delta\hspace{-0.05cm}\mathcal W}_{X}$
has maximal 1-rank.
  \end{enumerate}
\end{prop}
We think that this result deserves some comments. 
It is proved in the sixth and  last section of \cite{BW} by means of a rather explicit approach. Very roughly, from a  cubic equation defining a $\boldsymbol{\mathcal W}$-admissible cubic  $X\subset \mathbf P^4$ and after six pages of explicit algebraic computations, the authors build rational 1-forms $\mathfrak u_i$ for $i$ ranging from 1 to 5 (see formula (68) in \cite{BW}) which they show to give rise to five linearly independent 1-abelian relations for $ \boldsymbol{\Delta\hspace{-0.05cm}\mathcal W}_{X}$. 
They do not establish that the  corresponding integrals $\int \mathfrak u_i$ are precisely the `{\it Abelsche  Integrale erster Gattung}'\footnote{Translated by `{\it Abelian integrals of the first kind}' into English.} of the Fano surface of lines included in $X$, but this is strongly suggested at the very end. Thanks to 
Proposition \ref{P:BW-X-smooth}, we know that it is indeed the case, at least when $X$ is smooth.\sk

Although computational, Blaschke and Walberer's approach is very interesting according to us, essentially for two reasons: first, being algebraic in nature, it allows a priori to handle the case of singular cubics as well; the proof of the second point of Proposition \ref{P:BW-X-smooth} relies on point {\bf xvi.} of \S\ref{SSS:PropertiesOfF} (with $k=1$), the only modern proof of which we are aware of (p.\,332 of \cite{ClemensGriffiths}) relies on a basic but non-explicit principle of complex analysis;\footnote{Namely, the maximal modulus principle in complex analysis which admits as a direct corollary that any holomorphic map from a simply connected compact  manifold into a complex torus is necessarily constant.} this leads us to the second reason, which is that the approach in \cite[\S6]{BW} gives rise to explicit algebraic formulas for the abelian 1-forms on $F_1(X)$ which is interesting and has nothing comparable in modern literature to our knowledge.   Since one can take great benefit from explicit formulas for the holomorphic or abelian differentials on a projective variety\footnote{As a example of such formulas, one can think to the description of the abelian differentials on a complete intersection $V\subset \mathbf P^N$ by means of successive Poincar\'e residues of some rational forms on  $\mathbf P^N$.} 
it would be very interesting to revisit the results of Blaschke and Walberer by taking a more modern (and therefore perhaps more accurate and rigourous) approach.  
\mk 

But the most striking result in \cite{BW} is actually not the preceding proposition but rather its following converse which can be considered as one of the most interesting results in what concerns the algebraization of webs\footnote{In \cite{CChern}, Chern writes  that \cite{BW} `{\it is perhaps Blaschke's deepest paper}' in web geometry.}



\noindent{\bf Blaschke-Walberer's algebraization theorem.} {\it 
Let $\boldsymbol{\mathcal W}$ be a skew curvilinear 3-web in $\mathbf C^3$ with maximal 1-rank. Then $\boldsymbol{\mathcal W}$ is equivalent to the triangle web $\boldsymbol{\Delta\hspace{-0.05cm}\mathcal W}_{X}$ of a cubic hypersurface  $X\subset \mathbf P^4$.}

The proof of this theorem consists in many pages of normalisations and computations that we do not have considered in detail yet. Our impression is that this is a genuine  computational {\it tour de force}. Still,  as for the previous result,
we believe it would be interesting and even necessary to go over Blaschke-Walberer's proof again, with the modern standard of mathematical rigour in order to certify that any of their arguments or computations is indeed correct.

Our opinion that the content of \cite{BW} should be revisited should not lead one to believe that we doubt the validity of the above results: this is not the case.  In the next subsection we use them to get another and better understanding of 
$\boldsymbol{\mathcal L\hspace{-0.05cm}\mathcal W}_{\mathscr C}$ and of its 1-ARs.
\sk 

Ending this sub-section, it is interesting to point out that a notion of `triangle varieties' generalizing  the one discussed here has been recently considered by algebraic geometers for hyper-K\"ahler manifolds (see \cite{Bazhov,Voisin}) and first results show that it is a relevant notion in what regards the  study of these varieties. It would be interesting to figure out whether this notion might give relevant outputs in web geometry as well.

\subsubsection*{\bf B.2.2. Another point of view on ${\mathcal L\hspace{-0.05cm}\mathcal W}_{\mathscr C}$}
The key idea here is rather simple and relies on the fact that ${\mathcal L\hspace{-0.05cm}\mathcal W}_{\mathscr C}$ can be interpreted as 
$\boldsymbol{\Delta\hspace{-0.05cm}\mathcal W}_{\mathscr C}$. 
Blaschke-Walberer theory will then apply and 
Proposition \ref{Prop:Blaschke-Walberer} will provide an 
 interpretation of 
the 1-ARs of this web.  \mk 

Let $X$ be an irreducible cubic threefold in $\mathbf P^4$ containing a non degenerate triangle. 
By definition, for $k=1,2,3$, its {\bf $\boldsymbol{k}$-th corner map} is the rational map $c_k : \Delta_X\dashrightarrow X$ such that for any 
generic triangle $T=(\ell_1,\ell_2,\ell_3) \in \Delta_X$, $c_k(T)$ is equal to the intersection point of $\ell_i$ with $\ell_j$ if $i$ and $j$ are such that $\{i,j,k\}=\{1,2,3\}$.
From the obvious fact that a generic triangle is determined by two of its edges, 
one easily gets that each $c_k$ is dominant onto $X$ and generically finite. Consequently, given a generic base point $T^*=(\ell_1^*,\ell_2^*,\ell_3^*)$ and setting $x^*=c_k(T^*)$, each corner map $c_k$ induces a local biholomorphism $c_k: (\Delta_X,T^*)\rightarrow (X,x^*)$ which can be used to locally push-forward $\boldsymbol{\Delta\hspace{-0.05cm}\mathcal W}_{X}$ onto $X$.

Assume that $i,j,k$ are as above.   For a triangle $T=(\ell_s)_{s=1}^3$
 sufficiently close to $T^*$ in $\Delta_X$ (and in particular non degenerate), the leaves 
of 
$\boldsymbol{\Delta\hspace{-0.05cm}\mathcal W}_{X}$ passing through it, 
denoted by  $\mathcal L_m(T)$, 
 are the sets of triangles $T'=(\ell'_s)_{s=1}^3$ such that $\ell_m'=\ell_m$ for $m=1,2,3$. Since 
 $c_k(T')\in \ell_i\subset X$, we deduce that $c_k\big( \mathcal L_i(T)
 \big)\subset \ell_i$ and similarly $c_k\big( \mathcal L_j(T)
 \big)\subset \ell_j$.  It follows that 
 $(c_k)_*\big( \boldsymbol{\Delta\hspace{-0.05cm}\mathcal W}_{X} \big)$ is a germ of 3-web on $X$ at $x^*$, two of the local foliations of which are formed by lines included in $X$ hence form a 2-subweb of 
 $\boldsymbol{\mathcal L\hspace{-0.05cm}\mathcal W}_{X}$.  
 This gives us the 
 \begin{lem} 
 \label{Lem:lemo}
If $\boldsymbol{\mathcal L\hspace{-0.05cm}\mathcal W}_{X}$ is skew then so is $\boldsymbol{\Delta\hspace{-0.05cm}\mathcal W}_{X}$. In particular, this applies when $X$ is smooth.  
 \end{lem}
 Note that the first statement in this lemma holds true even if the lines trough a general point of $X$ come with multiplicities that is, even if $\boldsymbol{\mathcal L\hspace{-0.05cm}\mathcal W}_{X}$ is a $k$-web with $k<6$. 
 This  remark applies in the case of the chordal cubic. Although $\boldsymbol{\mathcal L\hspace{-0.05cm}\mathcal W}_{\mathscr C}$ is only a 3-web, it has been verified  that this web is skew in \S\ref{SS:je-sais-pas-quoi} (more precisely, see \eqref{Eq:LW-C--skew}). From the 
  preceding lemma, we get that  the triangle web 
 $\boldsymbol{\Delta\hspace{-0.05cm}\mathcal W}_{\mathscr C}$ is skew, hence 
 Proposition \ref{Prop:Blaschke-Walberer} applies and gives us a description of the 1-ARs in terms of rational 1-differentials on the Fano surface $F_1(\mathscr C)$.
 \sk
 
 At this point, since both  $\boldsymbol{\mathcal L\hspace{-0.05cm}\mathcal W}_{\mathscr C}$ and $\boldsymbol{\Delta\hspace{-0.05cm}\mathcal W}_{\mathscr C}$ are skew curvilinear 3-web with maximal 1-rank constructed from the same cubic hypersurface, one can wonder how these two webs are related. It turns out that they are actually the same as an easy computational verification shows:
\begin{lem} 
\label{Lem:Corner-map-Chordal-Cubic}
 In the case of the chordal cubic $\mathscr C$, any corner map $c_k$ gives rise to an equivalence
$$
\boldsymbol{\Delta\hspace{-0.05cm}\mathcal W}_{\mathscr C}
\stackrel{\sim}{\longrightarrow}
 (c_k)_*\big(
\boldsymbol{\Delta\hspace{-0.05cm}\mathcal W}_{\mathscr C}
\big)=\boldsymbol{\mathcal L\hspace{-0.05cm}\mathcal W}_{\mathscr C}\, .$$
\end{lem}
The proof is left to the reader. Note that this result is specific to the case of $\mathscr C$. In general, the push-forward by a corner map $c_k$ of the foliation induced by $P_k$ is not linear on the considered cubic.\sk 

 Taking the pull-backs under a corner map of the total derivatives of the functional identities \eqref{E:RAmunurho}, one gets a rational basis of $\boldsymbol{AR}^{(1)}\big(  \boldsymbol{\mathcal L\hspace{-0.05cm}\mathcal W}_{\mathscr C} \big) $ which enjoys the interesting feature of being completely explicit.  On the other hand, 
the first point of Proposition \ref{Prop:Blaschke-Walberer} offers another view on the 1-ARs of $\boldsymbol{\Delta\hspace{-0.05cm}\mathcal W}_{\mathscr C}$, more abstract  and conceptual but not really explicit. It would be interesting to relate more concretely these two descriptions of the abelian relations under scrutiny and to get a description of them which would be intrinsic and conceptual as well as explicit.
 
 In the next subsection, we describe another approach to describe the 1-ARs of  $
\boldsymbol{\Delta\hspace{-0.05cm}\mathcal W}_{\mathscr C}\simeq  
 \boldsymbol{\mathcal L\hspace{-0.05cm}\mathcal W}_{\mathscr C}$ which is conceptual and can be made explicit but lacks being intrinsic.


\subsubsection*{\bf B.2.3. Deforming ${\mathscr C}$ following Collino to better understand the 1-ARs of ${\mathcal L\hspace{-0.05cm}\mathcal W}_{\mathscr C}$}
The same approach as the one described in Appendix A.1 in the case of Segre's cubic can be implemented in the case of the chordal cubic as well. And again, all the results needed to justify it can be found in a paper by Collino, namely in \cite{Collino1}. 
Since this is quite similar to what has been done in A.1, we will be rather quick 
and will not give any details. \sk 

Let $C$ be the cubic form appearing in the equation \eqref{Eq:Chordal-cubic-equation} of  ${\mathscr C}$. Given another generic cubic form $G$, one considers the pencils of cubic hypersurfaces $g : \mathcal G\subset \mathbf P^4\times \mathbf A^1\rightarrow \mathbf A^1 $ whose fiber at $s\in \mathbf A$ is $\mathcal G_s=\{ \, C+s^2 G=0\}\subset \mathbf P^4$.\footnote{See around (1.4) in \cite{Collino1} for an explanation of the choice of $s^2$ as a deformation parameter}  
Let $h: \mathcal H\subset G_1(\mathbf P^4)\times \mathbf A^1\rightarrow \mathbf A^1 $ be the associated family of Fano surfaces.  For $t$ with a non zero but very small modulus, the cubic threefold $G_s=g^{-1}(s)$ is smooth hence $F_s=F(G_s)=h^{-1}(s)$ is a smooth irreducible surface with $h^0(F_s,\Omega^1_{F_s})=5$.  The cohomology spaces $\boldsymbol{H}^1( F_s, \mathbf C)$'s for $s\in (\mathbf C^*,0)$ form a variation of pure Hodge structures and the basic question here is how its degeneration as the deformation parameter $s$ goes to the origin. This is nicely answered in \cite{Collino1}, to which we refer the reader for details. \sk 

Recall that the singular set  of ${\mathscr C}$ is a rational quartic curve that we denote by $\Gamma$. 
For $s\neq 0$ sufficiently close to the origin, the cubic $G_s$ 
intersects this curve  in 12 points which,  as $t\rightarrow 0$, converge onto 12 pairwise distinct limit points on $\Gamma$.  Let $K\rightarrow \Gamma$ be the 2-to-1 covering of $\Gamma$ ramified at the 12 limit points on $ \Gamma$: $K$ is a hyperelliptic curve of genus 5.  
Recall that $F_0=F(\mathscr C)$ has two irreducible components, the symmetric product $F'=\Gamma^{[2]}$ and another one denoted by  $F''$ 
(described in \S\ref{SS:je-sais-pas-quoi}).  Let 
 $\widetilde{\mathcal H}$ be the proper transform of the family of Fano surfaces 
 $\mathcal H\subset G_1(\mathbf P^4)\times \mathbf A^1$ by the blow-up of 
$G_1(\mathbf P^4)\times \mathbf A^1$ along $F'\times \{0\}$ and denote by 
$\tilde h: \widetilde{\mathcal H}\rightarrow \mathbf A^1$ the natural projection. Then according to  \cite[Prop.\,2.1]{Collino1}, $\widetilde{\mathcal H}\rightarrow \mathbf A^1$ resolves the singularities of $h: {\mathcal H}\rightarrow \mathbf A^1$ 
along the fiber over 0; and the central fiber $\widetilde F_0=\tilde h^{-1}(0)$ is a reduced divisor with normal crossing with two irreducible components: one has $\widetilde F_0=F''\cup K^{[2]}$ and the intersection curve $F''\cap K^{[2]}$ corresponds to a conic in $F''\simeq \mathbf P^2$ and to the $g^1_2$ in the symmetric product $K^{[2]}$ associated to the covering $K\rightarrow \Gamma\simeq \mathbf P^1$. 

Using the resolution $\widetilde{\mathcal H}\rightarrow \mathbf A^1$, Collino proved the following 
%

\begin{prop}[Proposition 2.2.1 in \cite{Collino1}] 
${}^{}$

\vspace{-0.15cm}
\begin{enumerate}
\item[{\rm 1.}] The limiting Hodge structure at the origin of the $\boldsymbol{H}^1\big( F_s, \mathbf C\big)$'s for $s\in (\mathbf C^*,0)$ exists and  is equal to the one of $\boldsymbol{H}^1\big(  K^{[2]} , \mathbf C\big)$. In particular, this limit  is pure and of dimension 5. 
\mk
\item[{\rm 2.}] The family of associated Albanese varieties has good reduction 
at $0$ 
and the special fiber 
$\widetilde A_0$
at this point is the Albanese variety of $K^{[2]}$, namely 
the Jacobian variety 
$J(K)$ of $K$. 
\end{enumerate}
\end{prop}
For $s\in \mathbf C$, denote by $P_{i,s}: \Delta(G_s)\dashrightarrow F_s$ (for $i=1,2,3$) the three rational maps defining 
$\boldsymbol{\Delta\hspace{-0.05cm}\mathcal W}_{G_s}$ ({\it cf.}\,\eqref{Def:DW-X}).  Each $\{ P_{i,s}\}_{ s\in \mathbf C}$ is a smooth family of  maps from which we get that 
the triangles webs $\boldsymbol{\Delta\hspace{-0.05cm}\mathcal W}_{G_s}$'s  form a family of skew curvilinear 3-webs which is smooth at the origin (verification left to the reader). For any $s$ such that $G_s$ is smooth, the Albanese map ${\rm alb}_{F_s}: F_s\rightarrow {\rm Alb}(F_s)$ is such that $\Psi_s= \sum_{i=1}^3 {\rm alb}_{F_s}(P_{i,s}): \Delta(G_s)\rightarrow {\rm Alb}(F_s)$ is constant.  Up to some choice of a smooth family of base points that we will not detail here, one deduces from the second point of the proposition above that the $\Psi_s$'s extend at the origin to a smooth family of maps, with 
$$\Psi_0=
\sum_{i=1}^3 {\rm alb}_{\widetilde F_0}\big(P_{i,0}\big): \Delta(\mathscr C)\longrightarrow 
{\rm Alb}\big( 
 \widetilde F_0\big)=
\widetilde A_0
\simeq J(K)$$
 being constant as well. In terms of webs, this gives us the
\begin{cor}
The trace induces an isomorphism between 
the space of global 1-forms on $\widetilde A_0 \simeq J(K)$, which is isomorphic to $\boldsymbol{H}^0\big(K,\Omega_K^1\big)$,  and the 
space of 1-ARs of 
$\boldsymbol{\Delta\hspace{-0.05cm}\mathcal W}_{\mathscr C}\simeq 
\boldsymbol{\mathcal L\hspace{-0.05cm}\mathcal W}_{\mathscr C}$: 
\begin{equation}
\label{Eq:Isom-H0(Omega1K)-AR1}
\boldsymbol{H}^0\Big(K,\Omega_K^1\Big)\simeq 
\boldsymbol{H}^0\Big(\, \widetilde A_0,\Omega_{\widetilde A_0}^1\, \Big) \stackrel{\sim}{\longrightarrow}
\boldsymbol{AR}^{(1)}\Big( 
\boldsymbol{\Delta\hspace{-0.05cm}\mathcal W}_{\mathscr C}
\Big)\, .
\end{equation}
\end{cor}
Since $\boldsymbol{\Delta\hspace{-0.05cm}\mathcal W}_{\mathscr C}\simeq 
\boldsymbol{\mathcal L\hspace{-0.05cm}\mathcal W}_{\mathscr C}$, this corollary provides a conceptual description of the 1-ARs of $\boldsymbol{\mathcal L\hspace{-0.05cm}\mathcal W}_{\mathscr C}$. However, 
although interesting 
this description has the disadvantage  of not being canonical: 
getting a smoothing of $C$ depends of the pencil of cubics considered. Indeed, as explained in \cite[\S1]{Collino1}, another choice for this pencil would give another limit configuration of 12 points on  $\Gamma$ hence another 
hyperelliptic genus 5 ramified covering $K'\rightarrow \Gamma$, in general non isomorphic to $K$. This would give another isomorphism 
$\boldsymbol{H}^0\big(K',\Omega_{K'}^1\big)\simeq 
\boldsymbol{AR}^{(1)}\big( 
\boldsymbol{\Delta\hspace{-0.05cm}\mathcal W}_{\mathscr C}
\big)$, which shows to what extent \eqref{Eq:Isom-H0(Omega1K)-AR1} is not canonical.

\newpage




\begin{thebibliography}{99}



 \bibitem[\bf Al]{Allcock}
 {\scshape \bf  D. Allcock}. 
\href{http://www.ams.org/journals/jag/2003-12-02/S1056-3911-02-00313-2/home.html}{\it The moduli space of cubic threefolds}. J. Algebraic Geom. {\bf 12} (2003), pp. 201--223. 
\mk 

 \bibitem[\bf AK]{AltmanKleiman}
 {\scshape \bf  A. Altman, S. Kleiman}. 
 \href{http://www.numdam.org/item?id=CM_1977__34_1_3_0}{\it 
Foundations of the theory of Fano schemes}.
Compositio Math. {\bf 34} (1977), pp. 3--47. 
\mk 



 \bibitem[\bf AL]{AL}
 {\scshape \bf  M. Andersson,  R. L\"ark\"ang}.  
\href{https://doi.org/10.1007/s00208-018-1678-8} 
{\it The $\overline{\partial}$-equation on a non-reduced analytic space}.
 Math. Ann. {\bf 374} (2019), pp.  553--599. 
\mk 


 \bibitem[\bf AHL]{AHL} 
{\scshape  \bf  N. Arkani-Hamed, S. He, T. Lam}.
\href{https://arxiv.org/abs/2005.11419}
{\it Cluster configuration spaces of finite type}. 
Preprint  arXiv:2005.11419 (2020). 
\mk 




 \bibitem[\bf Bak]{Baker} 
 {\scshape \bf   H.\,F.  Baker}. 
\href{https://doi.org/10.1112/jlms/s1-6.3.176}{\it Segre's ten nodal cubic primal in space of four dimensions and Del Pezzo's surface in five dimen-} \href{https://doi.org/10.1112/jlms/s1-6.3.176}{\it  sions}.
J. London Math. Soc. {\bf 6} (1931), pp. 176--185. 
\mk 


 \bibitem[\bf Bar]{Barlet} 
 {\scshape \bf   D. Barlet}. 
\href{https://doi.org/10.1007/BFb0064400}{\it Le faisceau $\omega^\cdot_X$ sur un espace analytique $X$ de dimension pure}. 
Lecture Notes in Math. {\bf 670} (1978), pp. 187--204.
\mk 


 \bibitem[\bf Baz]{Bazhov} 
 {\scshape \bf   I. Bazhov}. 
\href{https://doi.org/10.1016/j.jpaa.2018.09.005}
{\it On the variety of triangles for a hyper-K\"ahler fourfold constructed by Debarre and Voisin}.  J. Pure Appl. Algebra {\bf 223} (2019), pp. 2530--2542. 
\mk 



 \bibitem[\bf B]{Bertini} 
 {\scshape \bf   E.  Bertini}. 
\href{https://archive.org/details/introduzioneall00bertgoog}{Introduzione alla geometria proiettiva degli iperspazi}.
 E. Spoerri, Pisa  (1907), pp. 426.
\mk 


 \bibitem[\bf Ber]{Bertram} 
 {\scshape \bf   W. Bertram}. 
\href{https://doi.org/10.1007/b76884}{The geometry of Jordan and Lie structures}.  Lectures Notes  in Math.  {\bf 1754},  Springer-Verlag, Berlin, 2000,  269 pp.



 \bibitem[\bf Bla1]{Blaschke1933} 
 {\scshape \bf  W. Blaschke}.   
\href{http://www.digizeitschriften.de/dms/resolveppn/?PID=GDZPPN002130351}
{\it Textilgeometrie und Abelsche Integrale}. 
Jahr. Deutsch. Math.-Ver. 
  {\bf 43} (1933), pp. 87--97. 
\mk 



 \bibitem[\bf Bla2]{BlaschkeK} 
 {\scshape \bf  W. Blaschke}.
 \href{https://doi.org/10.1007/BF02940655}{\it \"Uber gewebe von kurven im R3}. 
 Abh. Math. Semin. Univ. Hamb. {\bf 9}  (1933),  pp. 291--298.
\mk 


 \bibitem[\bf Bla3]{Blaschke-1-Rank}
 {\scshape \bf  W. Blaschke}.
 \href{https://doi.org/10.1007/BF02940656}{\it Abz\"ahlungen f\"ur kurvengewebe und fl\"achengewebe}.  Abh. Math. Semin. Univ. Hamb. {\bf 9}  (1933),  pp. 299--312.
\mk 



 \bibitem[\bf BW]{BW} 
 {\scshape \bf  W. Blaschke, P. Walberer}.
 \href{https://doi.org/10.1007/BF02940673}{\it Die kurven-3-gewebe h\"ochsten ranges im R3}
Abh. Math. Semin. Univ. Hamb. {\bf 10}  (1934),  pp.180--200. 
\mk 


 \bibitem[\bf Bla4]{Blaschke1955} 
 {\scshape \bf  W. Blaschke}.
 \href{https://link.springer.com/book/10.1007/978-3-0348-6952-2}{Einf\"uhrung in die Geometrie der Waben}.  
 Elem. der Math. vom h\"oheren Standpunkt  Bd. {\bf IV}
Birkh\"auser Verlag, 1955, pp. 108.
 \mk 




 



 \bibitem[\bf BB]{BB} 
 {\scshape \bf  W. Blaschke, G.  Bol}.   
\href{http://eudml.org/doc/203712}{Geometrie der Gewebe}.  
Berlin, Springer 1938.
\mk 







 \bibitem[\bf Bol]{Bol} 
 {\scshape \bf G.  Bol}.   
 \href{https://doi.org/10.1007/BF02940735}{\it \"Uber ein bemerkenswertes F\"unfgewebe in der Ebene}.  Abh. Math. Hamburg Univ. {\bf 11} 
(1936),  pp. 387--393. 
 \mk 


\bibitem[\bf BM]{BM} 
{\scshape  \bf  M. Bolognesi, A. Massarenti}. 
 \href{https://msp.org/ant/2021/15-2/ant-v15-n2-p06-s.pdf}{\it Birational geometry of moduli spaces of configurations of points on the line}. Algebra Number Theory {\bf 15} (2021), pp. 513--544.
 \mk 
 
 \bibitem[\bf BS]{BombieriSwinnerton-Dyer}
{\scshape  \bf E. Bombieri, H.\,P Swinnerton-Dyer}. 
\href{http://www.numdam.org/item?id=ASNSP_1967_3_21_1_1_0}
{\it On the local zeta function of a cubic threefold}. 
Ann. Scuola Norm. Sup. Pisa {\bf 21} (1967), pp. 1--29.
\mk 




\bibitem[\bf Br]{Brown} 
{\scshape  \bf  F. Brown}. 
\href{http://www.numdam.org.ezproxy.math-info-paris.cnrs.fr/item?id=ASENS_2009_4_42_3_371_0}{\it Multiple zeta values and periods of moduli spaces $\overline{\mathfrak M}_{0,n}$}.  Ann. Sci. \'ENS. {\bf 42} (2009), pp.\,371--489. 
\mk 





\bibitem[\bf Bu]{Burau} 
{\scshape  \bf  W. Burau}. 
\href{https://doi.org/10.1007/BF02849407}{\it 
On certain models for congruences of rational normal curves}. 
Rend. Circ. Mat. Palermo II. Ser. {\bf 15} (1966), pp. 41--50.
\mk 



\bibitem[\bf CGHL]{Casalaina-Martin}
{\scshape  \bf  S. Casalaina-Martin, S. Grushevsky, K.  Hulek, R. Laza}. 
\href{https://doi.org/10.1112/plms.12375}
{\it Complete moduli of cubic threefolds and their}
\href{https://doi.org/10.1112/plms.12375}
{\it intermediate Jacobians}. Proc. Lond. Math. Soc. {\bf 122} (2021), pp.  259--316.
\mk 




\bibitem[\bf ChG]{CG2}
{\scshape \bf S.-S. Chern, P.\,A. Griffiths}. 
 \href{https://www.math.uni-bielefeld.de/jb_dmv/JB_DMV_083_2.pdf}{\it Corrections and addenda to our paper: ``Abel's theorem and webs''},  
Jahresbe\-richt 
Deutsch. Math.-Ver. \textbf{83} (1981), pp.  78--83.
\mk 


\bibitem[\bf Ch]{CChern}
{\scshape \bf S.-S. Chern}. {\it Wilhelm Blaschke and Web Geometry}. 
In `Wilhelm Blaschke Gesammelte Werke Band {\bf 5}' (1985), pp. 25--27.
\mk 



\bibitem[\bf ClG]{ClemensGriffiths}
{\scshape \bf C.\,H. Clemens, P.\,A. Griffiths}.
\href{https://doi.org/10.2307/1970801}{\it The intermediate Jacobian of the cubic threefold}. 
Ann. of Math. {\bf 95} (1972), pp. 281--356. 
\mk



\bibitem[\bf Cob]{Coble}
{\scshape  \bf  A. Coble}.
\href{https://doi.org/10.2307/2370617}{\it 
A Generalization of the Weddle Surface, of Its Cremona Group, and of Its Parametric Expression in}
\href{https://doi.org/10.2307/2370617}{\it 
 Terms of Hyperelliptic Theta Functions}. 
Amer. J. Math. {\bf 52} (1930), pp.  439--500.
\mk 


\bibitem[\bf CM]{CollinoMurre}
{\scshape  \bf  A. Collino, J.\,P. Murre}. 
\href{https://doi.org/10.1016/S1385-7258(78)80003-1}{\it The intermediate Jacobian of a cubic threefold with one ordinary
double point; an}
\href{https://doi.org/10.1016/S1385-7258(78)80003-1}{\it algebraic-geometric approach (I) \& (II)}. Indag. Math. {\bf 40} (1978), pp. 43--45
and pp. 56--71.
\mk 

\bibitem[\bf Col1]{Collino1}
{\scshape  \bf  A. Collino}. 
\href{https://doi.org/10.1007/BFb0093590}{\it The fundamental group of the Fano surface I \& II}. In `Algebraic threefolds (Varenna, 1981)', pp. 209--218 and pp. 219--220, Lecture Notes in Math. {\bf 947}, Springer, 1982.
\mk 


\bibitem[\bf Col2]{Collino2}
{\scshape  \bf  A. Collino}. 
\href{https://arxiv.org/abs/1211.2621}{\it Remarks On the Topology of the Fano surface}. 
Preprint arXiv:1211.2621 (2012).
\mk 


\bibitem[\bf Da1]{DThesis}
{\scshape  \bf  D.\,B. Damiano}. 
\href{https://www.proquest.com/dissertations-theses/webs-abelian-equations-characteristic-classes/docview/303001457/se-2?accountid=15867}{\it Webs, abelian equations, and characteristic classes}. 
Ph.D. Thesis, Brown University 1980. 
\mk 

\bibitem[\bf Da2]{D}
{\scshape  \bf  D.\,B. Damiano}. 
\href{https://doi.org/10.2307/2374443}{\it Webs and characteristic forms of Grassmann manifolds}. 
Amer. J. Math. {\bf 105} (1983), pp. 1325--1345.
\mk 

\bibitem[\bf Do1]{DolgachevSegre}
{\scshape  \bf  I. Dolgachev}. 
\href{https://doi.org/10.1007/978-3-319-32994-9_11}{\it Corrado Segre and nodal cubic threefolds}. In `From classical to modern algebraic geometry', 
Trends Hist. Sci., Birkh\"auser/Springer 2016, pp. 429--450. 
\mk 

\bibitem[\bf DFL]{DolgachevFarbLooijenga}
{\scshape  \bf  I. Dolgachev, B. Farb, E. Looijenga}. 
\href{https://doi.org/10.1007/s40879-018-0231-3}{\it  Geometry of the Wiman-Edge pencil, I: algebro-geometric aspects}.
 Eur. J. Math. 4 (2018), pp. 879--930. 
\mk 


\bibitem[\bf Do2]{Dolgachev15}
{\scshape  \bf  I. Dolgachev}. 
\href{https://arxiv.org/abs/1906.12295}{\it 15-nodal quartic surfaces I: quintic del Pezzo surfaces and congruences of lines in $\mathbf  P^3$}. Preprint arXiv:1906.12295 (2019). 
\mk 



\bibitem[\bf Fa1]{Fano1}
{\scshape  \bf  G. Fano}. 
\href{http://www.bdim.eu/item?id=GM_Fano_1904_4}{\it  
Sulle superficie algebriche contenute in una variet\`a cubica dello spazio a quattro dimensioni}. 
Atti R. Acc. Sci. Torino  {\bf 39} (1904), pp. 597-613.
\mk 




\bibitem[\bf Fa2]{Fano}
{\scshape  \bf  G. Fano}. 
\href{http://www.bdim.eu/item?id=GM_Fano_1904_1}{\it  
Sul sistema $\infty^2$ di rette contenuto in una variet\`a cubica generale dello   spazio a quattro dimensioni}. 
Atti R. Acc. Sci. Torino  {\bf 39} (1904), pp. 778--792.
\mk 

 \bibitem[\bf FG]{FockGoncharov-XInfinity}
{\scshape  \bf V.  Fock,  A.\,B. Goncharov}.
\href{https://doi.org/10.1007/s00029-016-0282-6}{\it Cluster Poisson varieties at infinity}.  Selecta Math. {\bf 22} (2016),  pp. 2569--2589.
\mk 



 \bibitem[\bf FSM]{FreitagSalvatiManni}
{\scshape  \bf E. Freitag, R. Salvati Manni}.
\href{https://www.jstor.org/stable/40997331}{\it The modular variety of hyperelliptic curves of genus three}.  
Trans. Amer. Math. Soc. {\bf 363} (2011), pp. 281--312. 
\mk


 \bibitem[\bf Fr]{Friedman}
{\scshape  \bf R. Friedman}.
\href{https://www.jstor.org/stable/2006955}
{\it Global smoothings of varieties with normal crossings}. 
Ann. of Math. {\bf 118} (1983), pp. 75--114. 
\mk 


 \bibitem[\bf FH]{FultonHarris}
{\scshape  \bf W. Fulton, J. Harris}.
\href{https://link.springer.com/book/10.1007/978-1-4612-0979-9}{Representation theory. A first course}. Graduate Texts in Math. {\bf 129}. Springer-Verlag, 1991. 
\mk 





\bibitem[\bf GK]{GeerKouvidakis}
{\scshape  \bf G. Van der Geer, A. Kouvidakis}
\href{https://doi.org/10.2969/aspm/05810027}{\it A note on Fano surfaces of nodal cubic threefolds}. 
Adv. Studies Pure Math. {\bf 58}  (2010), 
pp. 27-45.
\mk 


\bibitem[\bf GM]{GelfandMacPherson}
{\scshape  \bf  I. Gelfand, R.  MacPherson}. 
\href{https://doi.org/10.1016/0001-8708(82)90040-8}{\it Geometry in Grassmannians and a generalization of the dilogarithm}. 
Adv. in Math. {\bf 44} (1982), pp.  279--312.
\mk 




\bibitem[\bf Ghe]{Gherardelli}
{\scshape  \bf  F. Gherardelli}. 
\href{https://doi.org/10.1007/BF02925643}{\it Un' osservazione sulla variet\'a cubica di $P4$}. 
 Sem. Mat. Fis. Milano {\bf 37}  (1967), pp. 157--160.
\mk 



\bibitem[\bf GMZ]{GMZ}
{\scshape  \bf   H. Gluck, F. Morgan, W. Ziller}. 
\href{https://doi.org/10.1007/BF02564674}{\it Calibrated geometries in Grassmann manifolds}. 
Comment. Math. Helv. {\bf 64} (1989), pp. 256--268. 
\mk 


\bibitem[\bf Go]{Goldberg}
{\scshape  \bf V.\,V. Goldberg}. 
\href{https://link.springer.com/book/10.1007/978-94-009-3013-1}{Theory of multicodimensional $(n+1)$-webs}. 
Mathematics and its Applications {\bf 44}. Kluwer Acad. Pub. 1988, 466 pp. 
\mk 



\bibitem[\bf Gw1]{GwenaThesis}
{\scshape  \bf T. Gwena}. 
\href{https://esploro.libs.uga.edu/esploro/outputs/9949334576102959}
{\it Degenerations of Prym Varieties and cubic threefolds}. PhD thesis, University of Georgia
(2004). 
\mk


\bibitem[\bf Gw2]{Gwena}
{\scshape  \bf T. Gwena}. 
\href{https://www.jstor.org/stable/4097782}
{\it Degenerations of cubic threefolds and matroids}. 
Proc. AMS {\bf 133} (2005), pp. 1317--1323. 
\mk



\bibitem[\bf He]{He}
{\scshape  \bf C. He}. 
\href{https://doi.org/10.1016/j.topol.2020.107239}
{\it Cohomology rings of the real and oriented partial flag manifolds}. 
Topology Appl. {\bf 279} (2020), 107239. 



\bibitem[\bf HP]{HenkinPassare}
{\scshape  \bf  G. Henkin, M. Passare}. 
\href{https://link.springer.com/article/10.1007/s002220050287}
{\it  Abelian differentials on singular varieties and variations on a theorem of Lie-Griffiths}. 
Invent. Math. {\bf 135},  (1999), pp. 297--328.
\mk


\bibitem[\bf HMSV]{HowardAL}
{\scshape  \bf  B. Howard, J. Millson, A. Snowden, R. Vakil}. 
\href{https://doi.org/10.1112/plms/pds016}
{\it  The geometry of eight points in projective space: represen-} 
\href{https://doi.org/10.1112/plms/pds016}
{\it tation theory, Lie theory and dualities}. 
Proc. Lond. Math. Soc. {\bf 105} (2012), pp. 1215--1244. 
\mk


\bibitem[\bf Hu]{Hunt}
{\scshape \bf B. Hunt}. 
\href{https://doi.org/10.1007/BFb0094399}{The geometry of some special arithmetic quotients}. 
Lect. Notes in Math. {\bf 1637}. 
Springer-Verlag, 1996.
\mk

\vspace{-0.3cm}
\bibitem[\bf Hw]{Hwang}
{\scshape \bf J.-M. Hwang}. 
\href{https://doi.org/10.1215/00127094-3715296}{\it Geometry of webs of algebraic curves}. 
Duke Math. J. {\bf 166} (2017), pp. 495--536. 
\mk


\bibitem[\bf K]{Kumar2003}
{\scshape  \bf C. Kumar}. 
\href{https://doi.org/10.1112/S0024609302001844}{\it Linear systems and quotients of projective space}. 
Bull. Lond. Math. Soc. {\bf 35} (2003), pp. 152--160.
\mk 


\bibitem[\bf Lan]{Lang}
{\scshape  \bf S. Lang}. 
\href{https://doi.org/10.1007/978-1-4613-0041-0_18}{Algebra}.  Graduate Texts in Math.  {\bf 211}. Springer-Verlag, 2002.
\mk



\bibitem[\bf Loo]{Looijenga}
{\scshape  \bf E. Looijenga}. 
\href{https://doi.org/10.1007/s00222-009-0178-6}{\it The period map for cubic fourfolds}
Invent Math. {\bf 177} (2009) pp. 213--233. 
\mk



\bibitem[\bf Mum]{Mumford}
{\scshape  \bf D. Mumford}. 
\href{https://projecteuclid.org/download/pdf_1/euclid.kjm/1250523940}{\it Rational equivalence of 0-cycles on surfaces}. 
J. Math. Kyoto Univ. {\bf 9} (1968), pp. 195--204. 
\mk 



\bibitem[\bf Mur]{Murre}
{\scshape  \bf J. Murre}. 
\href{http://www.numdam.org/item?id=CM_1972__25_2_161_0}{\it Algebraic equivalence modulo rational equivalence on a cubic threefold}. 
 Compositio Math. {\bf 25} (1972), pp. 161--206. 
\mk








\bibitem[\bf N]{Nag}
{\scshape  \bf S. Nag}. 
\href{https://doi.org/10.1215/S0012-7094-81-04821-3}{\it The Torelli spaces of punctured tori and spheres}.  Duke Math. Journal {\bf 48} (1981), pp. 359--388.
\mk



\bibitem[{\bf Pera}]{Perazzo} 
{\scshape  \bf U.  Perazzo}.
\href{}{\it  Sopra una forma cubia con 9 rette doppie dello spazio a cinque dimensioni, e i correspondenti complessi cubici di rette nello spazio ordinario}. 
Atti Accad. Reale Torino  {\bf 36} (1901), pp. 891--895.
%
\mk



\bibitem[{\bf Per}]{Pereira} 
{\scshape  \bf J.\,V. Pereira}.  
\href{https://doi.org/10.1142/e033}{\it Resonance webs of hyperplane arrangements}.  In `Arrangements of hyperplanes - Sapporo 2009',  Adv. Stud. Pure Math., {\bf 62}, Math. Soc. Japan  2012, pp. 261--291.
\mk 





\bibitem[{\bf PP}]{Coloquio} 
{\scshape  \bf J.\,V. Pereira,  L. Pirio}.  
\href{http://www.springer.com/us/book/9783319145617}{An invitation to web geometry}.  
IMPA Monographs, Vol. 2 Springer 2015.
\mk 


\bibitem[{\bf P1}]{PirioThese} 
{\scshape  \bf L. Pirio}.  
   \href{https://tel.archives-ouvertes.fr/tel-00335195}
   {\'Equations fonctionnelles ab\'eliennes et th\'eorie des tissus}.
Th\`ese  Univ.\,Pierre et Marie Curie, 2004. 
\mk 

\bibitem[{\bf P2}]{PirioSelecta} 
{\scshape  \bf L. Pirio}.  
   \href{https://doi.org/10.1007/s00029-005-0012-y}
   {\it  Abelian functional equations, planar web geometry and polylogarithms}.
Selecta Math. {\bf 11} (2005), pp. 453--489. 
\mk 

\bibitem[{\bf P3}]{EAW} 
{\scshape  \bf L. Pirio}.  
   \href{https://doi.org/10.1007/s00208-015-1244-6}{\it Tissus alg\'ebriques exceptionnels}. 
   Math. Ann. {\bf 364}  (2016), pp. 1135--1166.
\mk 

\bibitem[{\bf P4}]{ClusterWebs} 
{\scshape  \bf L. Pirio}.  
   \href{https://arxiv.org/abs/2105.01543}
   {\it On  webs, polylogarithms and cluster algebras}.  
Preprint arXiv:2105.01543.
\mk 



\bibitem[{\bf Pu}]{Pukhlikov} 
{\scshape  \bf A.\,V. Pukhlikov}.  
   \href{http://dx.doi.org/10.1070/RM2007v062n05ABEH004454}
   {\it Birationally rigid varieties I. Fano varieties}. Russian Math. Surveys {\bf 62} (2007), pp. 857--942.
   \mk 



\bibitem[{\bf R1}]{Room1934} 
{\scshape  \bf T.\,G. Room}.  
   \href{https://doi.org/10.1112/plms/s2-37.1.292}{\it A generalization of the Kummer $16_6$
 configuration I}. 
Proc. London Math. Soc. {\bf 37} (1934), pp. 292--337.
\mk 



\bibitem[{\bf R2}]{RoomBook} 
{\scshape  \bf T.\,G. Room}.  
   \href{}{The geometry of determinantal loci}. 
Cambridge Univ. Press. XXVIII 1938, 483 pp.
\mk 



\bibitem[{\bf Sad}]{Sadykov} 
{\scshape  \bf R. Sadykov}.  
   \href{https://msp.org/pjm/2017/289-2/pjm-v289-n2-p07-s.pdf}{\it Elementary calculation of the cohomology rings of real Grassmann manifolds}. Pacific J. Math. {\bf 289} (2017), pp. 443--447.  
\mk 


\bibitem[{\bf Sag}]{Sagan} 
{\scshape  \bf B. Sagan}.  
   \href{https://link.springer.com/book/10.1007/978-1-4757-6804-6}{The symmetric group. Representations, combinatorial algorithms, and symmetric functions}. Second edition. Grad. Texts Math. {\bf 203} Springer-Verlag, 2001.
\mk 




\bibitem[{\bf Se}]{Segre} 
{\scshape  \bf C. Segre}.  
\href{http://www.bdim.eu/item?id=GM_Segre_CW_4_99}{\it Sulle variet\`a cubiche dello spazio a quattro dimensioni e su certi sistemi di rette e certe superficie dello}
\href{http://www.bdim.eu/item?id=GM_Segre_CW_4_99}{\it
 spazio ordinario}.
Mem. R. Acc. Scienze Torino {\bf 2} (1887), pp. 3--48.
\mk 

\bibitem[{\bf SR}]{SempleRoth} 
{\scshape  \bf J. Semple, L. Roth}.  
\href{https://archive.org/details/introductiontoal0000semp/page/n3/mode/2up}{Introduction to Algebraic Geometry}. Clarendon Press, 1949, 446 pp. 
\mk

\bibitem[{\bf Sn}]{Snyder} 
{\scshape  \bf V. Snyder}.
\href{https://www.jstor.org/stable/1988722}
{\it Surfaces Derived from the Cubic Variety Having Nine Double Points in Four Dimensional Space}.  Trans. AMS  {\bf 10} (1909), pp. 71--78.
\mk




\bibitem[{\bf Voi}]{Voisin} 
{\scshape  \bf C. Voisin}.
\href{https://arxiv.org/abs/1810.11848}
{\it Triangle varieties and surface decomposition of hyper-K\"ahler manifolds}. 
Preprint arXiv:1810.11848 (2018).
\mk 

\bibitem[{\bf T}]{Tyurin} 
{\scshape  \bf  A. Tyurin}. 
\href{https://doi.org/10.1070/RM1972v027n05ABEH001384}{\it 
Five lectures on three-dimensional varieties}. 
Uspehi Mat. Nauk {\bf 27} (1972), pp. 3--50.\mk



\end{thebibliography}
\end{document}